\documentclass[reqno]{amsart}
\usepackage{amssymb,amsmath,amsthm}
\usepackage{mathrsfs}
\usepackage{amscd}
\usepackage{amsfonts}
\usepackage{amssymb}
\usepackage{latexsym}
\usepackage{color}
\usepackage{esint}
\usepackage{graphicx}
\usepackage{float}
\graphicspath{{Figures/}}

\usepackage{todonotes}
\usepackage[colorlinks,linkcolor=blue,anchorcolor=red,citecolor=blue]{hyperref}
\usepackage{cleveref}
\usepackage[margin=1in]{geometry} 
\usepackage{marginnote}

\usepackage{color}

\allowdisplaybreaks

\usepackage[makeroom]{cancel}

\newtheorem{theorem}{Theorem}

\newtheorem{corollary}[theorem]{Corollary}
\newtheorem{definition}{Definition}
\newtheorem{lemma}{Lemma}
\newtheorem{proposition}[theorem]{Proposition}
\newtheorem{remark}{Remark}

\let\e=\varepsilon
\let\d=\delta
\let\h=v

\let\p=\partial

\let\O=\Omega

\numberwithin{equation}{section}

\let\hide\iffalse

\newcommand{\R}{\mathbb{R}}

\renewcommand{\S}{\mathbb{S}}

\newcommand{\be}{\begin{equation}}
\newcommand{\bm}{\begin{multline}}
\newcommand{\ee}{\end{equation}}
\newcommand{\dd}{\mathrm{d}}

\newcommand{\xb}{x_{\mathbf{b}}}
\newcommand{\tb}{t_{\mathbf{b}}}

\newcommand{\vb}{v_{\mathbf{b}}}

\newcommand{\tf}{t_{\mathbf{f}}}

\newcommand{\Bes}{\begin{eqnarray*}}
\newcommand{\Ees}{\end{eqnarray*}}
\newcommand{\Be}{\begin{equation} }
\newcommand{\Ee}{\end{equation}}

\def\p{\partial}

\def\O{\Omega}
\def\R{\mathbb{R}}
\def\d{\mathrm{d}}
\def\B{\begin{equation}}
\def\E{\end{equation}}
\def\BN{\begin{eqnarray*}}
\def\EN{\end{eqnarray*}}

\usepackage{cite}

\def\N{\mathbb{N}}

\DeclareMathOperator{\spn}{span}

\makeatletter
\@namedef{subjclassname@2020}{%
  \textup{2020} Mathematics Subject Classification}
\makeatother

\begin{document}

\title{Stationary Vlasov-Poisson-Boltzmann system in a convex domain}

\author[H. Chen]{Hongxu Chen}
\address[HC]{School of Mathematical Sciences, Shenzhen University, Shenzhen, Guangdong 518060, China}
\email{hongxuchen.math@gmail.com}

\author[J. Jin]{Jiaxin Jin}
\address[JJ]{Department of Mathematics, University of Louisiana at Lafayette, Louisiana, USA}
\email{jiaxin.jin@louisiana.edu}

\begin{abstract}

We study the stationary and dynamical Vlasov-Poisson-Boltzmann system in a bounded, convex domain subject to a confining external potential field. For the stationary problem, we construct a unique stationary solution with an inflow boundary condition. A key difficulty is to obtain pointwise regularity for stationary solutions due to the intricate coupling between the self-consistent electric field and the Boltzmann collision operator. To overcome this issue, we establish a $W^{1,p}_{x,v}$--$\alpha C^1_{x,v}$ bootstrap framework and derive an unweighted $C^1_v$ estimate by exploiting the structure of the external potential field.
We then investigate the dynamical Vlasov-Poisson-Boltzmann system near the stationary solution. We prove the global existence and uniqueness of solutions for small perturbations and establish exponential convergence toward the stationary state in weighted $L^\infty$ norms. Our results reveal the stabilizing effect of the external potential field and provide a framework for the stationary and dynamical theories of the Vlasov-Poisson-Boltzmann system in bounded domains.
\end{abstract}

\maketitle

\tableofcontents

\section{Introduction}
\label{sec:introduction}

The Vlasov–Poisson–Boltzmann (VPB) system is a fundamental kinetic model describing the evolution of dilute charged particles under self-consistent electrostatic interactions and binary collisions \cite{desvillettes1991asymptotics, glassey1999decay, GVPB, guo2010global, mischler2000initial, CKL}.
It plays a central role in plasma physics and kinetic theory, combining two main mechanisms: nonlinear Boltzmann collisions and self-consistent electrostatic forces governed by the Poisson equation.

In this paper, we investigate the stationary VPB system in a three-dimensional convex and bounded domain $\O\subset \R^3$. The stationary Vlasov-Boltzmann equation is formulated as
\begin{equation}\label{proF}
v\cdot\nabla_x F_s- E\cdot \nabla_v F_s=Q(F_s,F_s), 
\quad (x,v) \in \O \times \mathbb{R}^3. 
\end{equation}
Here, $F_s=F_s(x,v)\geq 0$ stands for the stationary velocity distribution function of the charged particles with velocity $v=(v_1,v_2,v_3)\in \R^3$ at position $x \in \Omega$, the boundary conditions are to be specified later. $E$ stands for the field. The Boltzmann collision term $Q$ is a bilinear integral operator acting only on the velocity variable, and in this manuscript, we only consider the hard sphere model, reading as 
\begin{equation*}
Q(F,G):=\int_{\mathbb{R}^3}\int_{\mathbb{S}^2}|(v-u)\cdot \omega|[F(u')G(v')-F(u)G(v)]\,\d\omega \dd u,
\end{equation*}
where the velocity pairs $(v,u)$ and $(v',u')$ are given in the $\omega$-representation by 
\begin{equation*}
v' =v+[(u-v)\cdot\omega]\omega,
\quad u'=u-[(u-v)\cdot\omega]\omega,
\quad \omega\in\mathbb{S}^2,
\end{equation*}
which corresponds to conservation of momentum and energy for elastic collisions between molecules:
\begin{equation*}
  v+u=v'+u',\quad |v|^2+|u|^2=|v'|^2+|u'|^2.
\end{equation*}
It is well known that the equilibrium state of the Boltzmann equation is given by Maxwellians, which we denote by
\begin{equation*}
\mu(v):=\frac{1}{(2\pi)^{3/2}}e^{-\frac{|v|^2}{2}}.
\end{equation*}

The mathematical theory of kinetic equations of Boltzmann and Vlasov type has been extensively developed. In the whole space or torus setting, global perturbative theories near Maxwellians have been established through high-order energy methods, velocity weights, and hypocoercivity structures, including the Boltzmann equation, VPB system, and Vlasov-Maxwell-Boltzmann system \cite{guo2004boltzmann,guo2002vlasov,ukai2006boltzmann,duan2011optimal,guo2003vlasov}.
In bounded domains, however, the presence of boundaries introduces additional difficulties. The transport dynamics interact with boundary conditions (inflow, specular reflection, or diffuse reflection), leading to possible loss of regularity near grazing trajectories. To resolve this issue, a low-regularity space has to be considered for well-posedness. In \cite{G}, Guo proposed the $L^2-L^\infty$ framework to construct a unique solution to the Boltzmann equation that converges to the equilibrium exponentially fast, with several boundary conditions. Later, Esposito-Guo-Kim-Marra applied the $L^2-L^\infty$ argument and constructed a unique stationary Boltzmann solution with dynamical stability under non-isothermal boundary in \cite{EGKM, EGKM2}. We refer to \cite{chen, KL,duan2019effects} for further developments under different boundary conditions.

In the bounded domain, the regularity issue remains unknown in the $L^2-L^\infty$ framework. In fact, the geometry of the domain plays a crucial role, for instance, in the non-convex domain, singularity may propagates into the interior domain through characteristics, leading to discontinuity, thus BV regularity becomes the optimal regularity as demonstrated in \cite{kim2011formation,GKTT2}. Even in a convex domain, the regularity issue is challenging due to the mixing of the singularity along the characteristic and the velocities interchange induced by the nonlocal Boltzmann operator. A major breakthrough was achieved by Guo-Kim-Tonon-Trescases in \cite{GKTT}, where a kinetic weight was proposed to compensate the singularity, and the weight $C^1$ estimate was achieved under several boundary condition in a strictly convex domain. This work inspired further developments in the regularity theory for stationary problems, including the linear problem studied in \cite{Ikun, Ikun2}, and the nonlinear problem with higher $C^{1,\beta}$ regularity and Sobolev regularity in \cite{CK, chen2024gradient}. The stationary regularity estimate provides important methodology for the investigation of the stationary VPB system in the current manuscript. We also mentioned recent progress on regularity in non-convex domains with specular boundary conditions in \cite{kim2024holder,an2025optimal}.

For the dynamical VPB system in bounded domains, a recent advancement was achieved by Cao-Kim-Lee in \cite{CKL}, who constructed a unique global-in-time solution near equilibrium in convex domains with diffuse boundary conditions. 
In particular, they derived weighted $W^{1,p}$ regularity for $p<6$ to establish the existence and uniqueness. Additional studies on the VPB system under different boundary conditions and potential settings can be found in \cite{CKQ, li2021global}. To refine the weighted $W^{1,p}$ regularity, Cao imposed an external potential field with a favorable sign condition on the boundary in \cite{cao2019regularity}, and constructed a local solution with the weighted $C^1$ regularity. This external potential field can enhance the regularity by preventing characteristics from becoming tangential to the boundary. A similar type of external field was adopted for the Vlasov–Poisson system in the earlier work \cite{hwang2010global}. We further note the dominant gravity field considered in \cite{jin2026asymptotic, jang2025long} also exhibits a similar regularity-enhancing effect.

Despite these developments, the well-posedness of the stationary VPB system remains an open problem. The stationary problem is more challenging than the time-dependent one. In time-dependent settings, time decay provides a natural mechanism to control nonlinear growth, while the Gronwall-type argument can be applied to study the regularity estimate and the stability estimate, such as in \cite{CKL}. Moreover, the global well-posedness can be justified through a combination of global a priori estimate and local well-posedness argument. In contrast, stationary systems lack time dissipation and Gronwall's inequality, and one must rely on structural coercivity, elliptic coupling through the Poisson equation, and precise control of the Vlasov operator to construct the regularity and well-posedness. These difficulties are particularly pronounced in three-dimensional domains, which are also the most relevant setting for physical applications. 

In this manuscript, we consider the stationary Vlasov–Poisson–Boltzmann system in a three-dimensional, strictly convex, bounded domain with an external potential field. As a first attempt at the stationary VPB problem, we focus on inflow boundary conditions in this work, while leaving the diffuse reflection boundary condition for future work. The main results of this paper are as follows:

\begin{enumerate}
\item Construction of a unique stationary solution near the Maxwellian equilibrium.

\medskip

\item Dynamical stability of the stationary solution.
\end{enumerate}

In the remainder of this introduction, we present the basic setting of the problem, state the main results for both the stationary and dynamical systems, and discuss the main difficulties and proof strategy.
Finally, we provide an outline of the paper and introduce the notation used throughout the paper.

\subsection{Basic setting}
\label{sec:basic_setting}

The three-dimensional domain $\O\subset \mathbb{R}^3$ is assumed to be strictly convex, and bounded with $C^3$ compact boundary $\partial\Omega$. Below, we provide some geometric descriptions. We fix a sufficiently small parameter $0<\delta_1\ll 1$.
We then choose a finite number of points $p \in \tilde{\mathcal{P}} \subset\p\O$ and $0<\delta_2\ll 1$ such that 
\[
\mathcal{O}_p=\eta_p(
B_+(0; \delta_1)) \subset B(p;\delta_2) \cap \bar{\O},
\]
and $\{\mathcal{O}_p \}$ forms a finite covering of $\partial \Omega$.
We further choose an interior covering $\mathcal{O}_0 \subset \O$ such that $\{ \mathcal{O}_p\}_{p \in \mathcal{P}}$ with $\mathcal{P} = \tilde{\mathcal{P}}\cup \{0\}$ forms an open covering of $\bar\O$.
We define a partition of unity
\be \notag
\mathbf{1}_{\bar\O} (x)=
\sum_{p \in \mathcal{P}} \iota_p(x)
\ \text{ such that } \ 
0 \leq \iota_p(x) \leq 1, \ \
\iota_p(x) \equiv 0 
\ \text{ for } \ 
x \notin \mathcal{O}_p.
\ee

Without loss of generality (see \cite{KL}) we can always reparametrize $\eta_p$ such that $\partial_{\mathbf{x}_{p,i}} \eta_p \neq 0$ for $i=1,2,3$ at $\mathbf{x}_{p,3}=0$, and an \textit{orthogonality} holds as
\be \notag
\partial_{\mathbf{x}_{p,i}}\eta_p \cdot \partial_{\mathbf{x}_{p,j}}\eta_p =0 
\ \text{ at } \ \mathbf{x}_{p,3}=0, \quad 
i, j \in \{ 1, 2, 3 \} \text{ and } i \neq j.
\ee
And at $\mathbf{x}_{p,3}=0$, the $\mathbf{x}_{p,3}$ derivative gives the outward normal: 
\begin{equation*} 
{\color{black}n_p(\mathbf{x}_p) = \frac{\partial_{\mathbf{x}_{p,3}}\eta_p}{\langle \partial_{\mathbf{x}_{p,3}}\eta_p,\partial_{\mathbf{x}_{p,3}}\eta_p\rangle}.}    
\end{equation*}
 For simplicity, we denote
\begin{equation*}
\partial_i \eta_p(\mathbf{x}_p): = \partial_{\mathbf{x}_{p,i}} \eta_p.
\end{equation*}
For $x\in \p\O$, we choose $p\in \mathcal{P}$ such that $x \in \mathcal{O}_p$. We denote
\begin{align*}
\mathbf{x}_p: = (\mathbf{x}_{p,1},\mathbf{x}_{p,2},0) 
\ \text{ such that } \ 
\eta_{p}(\mathbf{x_p}) = x.
\end{align*}
The tangential derivatives on the boundary can be interpreted by the new coordinates:
\begin{align*}
    &   \p_{\mathbf{x}_p} f(x,v) := (\p_{\mathbf{x}_{p,1}}f(x,v),\p_{\mathbf{x}_{p,2}}f(x,v) ), \ x \in \p\O, \\
    & \p_{\mathbf{x}_{p,i}}[f(x,v)]= \p_{\mathbf{x}_{p,i}}[
f( \eta_{p} ( \mathbf{x}_{p}  ),    {v} ) ]
 :=
\frac{\p \eta_{p}(\mathbf{x}_{p,i})}{\p \mathbf{x}_{p,i}}
\cdot \nabla_x f ( \eta_{p} ( \mathbf{x}_{p}  ), v), \quad i=1,2, \ \ x \in \p\O.
\end{align*}

Next, since $\O$ is strictly convex with $C^3$ boundary $\partial\Omega$, there exists a $C^3$ function $\xi(x)$ such that
\be \label{xi_def}
\O = \{x\in \mathbb{R}^3 \mid \xi(x)<0\}
\ \text{ and } \
\p\O=\{x\in \mathbb{R}^3 \mid \xi(x)=0\}. 
\ee
On the boundary $\p\O$, we have $|\nabla_x \xi(x)| > 0$ for all $x\in \p\O$, and the outward normal vector is given by
\be \notag
n(x) = \frac{\nabla_x \xi(x)}{|\nabla_x \xi(x)|}, \quad x\in \p\O.
\ee
Moreover, the strict convexity of $\O$ ensures that there exists $C_\xi > 0$ such that for any $x \in \p\O$,
\be \label{xi_convex}
y^{\intercal} \cdot \nabla_x^2 \xi(x) \cdot y \geq C_\xi |y|^2, \quad y \in \mathbb{R}^3.
\ee
We refer to the construction of such a defining function $\xi(x)$ in \cite{EGKM2}. In particular, for some sufficiently small $\delta \ll 1$, we have
\be \label{xi_dist}
\xi(x) = -dist(x,\p\O) 
\ \text{ when } \
dist(x,\p\O) < \delta.    
\ee
Let $d>0$. We define the boundary layer region $\O_d \subset \O$ by
\be \label{def:O_d}
\O_d := \{ x \in \O \mid dist (x, \p\O) < d \}.
\ee

\begin{lemma}[\cite{cao2019regularity}]
\label{lemma:dist_unique}

Suppose that $\p \O$ is $C^2$. Then there exists a sufficiently small constant $0 < \delta = \delta (\O) \ll 1$ such that for any $x \in \O_{\delta}$, there exists a unique $\tilde{x} \in \p \O$ satisfying that 
\[
dist (x, \tilde{x}) = dist (x, \p \O)
\ \text{ and } \
\sup_{\O_{\delta}} | \nabla_x \tilde{x} | < \infty.
\]
\end{lemma}

\begin{remark}

Lemma \ref{lemma:dist_unique} implies that each point $x \in \Omega$ sufficiently close to the boundary can be uniquely associated with its closest boundary point $\tilde{x} \in \partial \Omega$. 
This structure plays a key role in the kinetic weight (see Definition~\ref{def:alpha_weight_steady}), which is used to control the behavior of characteristics near grazing trajectories.
\end{remark}

In a bounded domain $\Omega$, charged particles interact with the physical boundary $\p\Omega$. We now introduce the boundary structure of the phase space. Let $n(x)$ denote the outward normal vector at $x\in \p\O$. We split the boundary phase space $\partial\Omega\times \R^3_v$ into the outgoing set $\gamma_+$, incoming set $\gamma_-$ and grazing set $\gamma_0$ by
\begin{align*}
\gamma_{+} := & \{(x,v)\in \p\O \times\mathbb{R}^{3}: \ n(x)\cdot v >0 \} ,   \\
\gamma_{-} := & \{(x,v)\in \p\O \times\mathbb{R}^{3}: \ n(x)\cdot v < 0\},  \\
\gamma_{0} := & \{(x,v)\in \p\O \times \mathbb{R}^{3}: \ n(x)\cdot v=0\}.
\end{align*}

\subsection{Stationary problem and well-posedness}

We consider the stationary Vlasov–Poisson–Boltzmann system in the phase space $(x,v) \in \Omega \times \mathbb{R}^3$, where $\Omega \subset \mathbb{R}^3$ is the strictly convex and bounded domain introduced in Section~\ref{sec:basic_setting}, in the presence of an external potential field $\phi_E$. By substituting $E$ in \eqref{proF} with the self-consistent electric field and the external potential field, the problem is formulated as
\begin{align}
\begin{cases}
&   v\cdot\nabla_x F_s - \nabla_x (\phi_{F_s}+\phi_E)\cdot \nabla_{v} F_s  =  Q(F_s,F_s), \\
& F_s|_{\gamma_-} = F_b, \\
& - \Delta_{x} \phi_{F_s} = \int_{\mathbb{R}^3} F_s \dd v - \rho_0 e^{-\phi_E}  \text{ in } \O, \quad \phi_{F_s} = 0 \text{ on } \p\O, \\
& - \p_n \phi_E > C_E > 0 \text{ on } \p\O. 
\end{cases} \label{nonlinear_SVPB_system}
\end{align}
Here, $\nabla_x \phi_{F_s}$ represents the self-consistent electric field generated from the stationary charged particles. It can be determined by solving the Poisson equation with the zero Dirichlet boundary condition. In the Poisson equation, $\rho_0$ represents the background density, and we assume the doping profile is given by $\rho_0 e^{-\phi_E}$ under the influence of the external potential $\phi_E$. 
Without loss of generality, we normalize $\rho_0 \equiv 1$. Moreover, for the external potential $\phi_E$, we impose the favorable sign condition $- \p_n\phi_E > C_E >0$ on the boundary.

With the extra $\phi_E$, the equilibrium state of \eqref{nonlinear_SVPB_system} is given by 
\[
F_s = e^{-\phi_E}\mu
\ \text{ when }
F_b = e^{-\phi_E}\mu.
\]
To construct a non-trivial stationary profile, we choose a non-trivial inflow data $F_b$, which is a small perturbation around $e^{-\phi_E}\mu$, given by
\be \notag
F_b = e^{-\phi_E}\mu + e^{-\phi_E/2}\sqrt{\mu} f_b.
\ee
Accordingly, we seek a stationary solution of the form
\[
F_s = e^{-\phi_E}\mu + e^{-\phi_E/2}\sqrt{\mu}h.
\]
Substituting this ansatz into the system, the equations for the perturbation $h$ are given by
%
\begin{align}
\begin{cases}
& v \cdot \nabla_x h - \nabla_x (\phi_h + \phi_E) \cdot \nabla_{v}  h + \frac{v\cdot \nabla_x\phi_{h}}{2}h + e^{-\phi_E}\mathcal{L}h \\
& \ \ \ \  = -(v\cdot \nabla_x\phi_{h})e^{-\phi_E/2}\sqrt{\mu}  +  e^{-\phi_E/2}\Gamma(h,h), \\
& h(x,v)|_{\gamma_-} = f_b, \\
&  -\Delta_x \phi_{h} = e^{-\phi_E/2}\int_{\mathbb{R}^3}  h \sqrt{\mu} \dd v \text{ in }\O, \quad \phi_h = 0 \text{ on } \p\O, \\
&  -\p_n\phi_E > C_E >0 \text{ on }\p\O.   
\end{cases}  \label{eqn:h}
\end{align}
Here, the linearized collision operator $\mathcal{L}$ and the nonlinear collision operator $\Gamma$ are defined by
\be \label{def:collision_operator}
\mathcal{L} f : = \frac{Q(\mu,\sqrt{\mu}f) + Q(\sqrt{\mu}f,\mu)}{\sqrt{\mu}}
\ \text{ and } \
\Gamma (f,f) := \frac{Q(\sqrt{\mu}f,\sqrt{\mu}f)}{\sqrt{\mu}}.
\ee
The boundary perturbation is defined as
\be\label{L_Gamma}
f_b := \frac{e^{\phi_E/2}F_b - e^{-\phi_E/2} \mu}{\sqrt{\mu}} 
\ \text{ and } \ 
\phi_{h} = \phi_{F_s}.
\ee

In the following, we provide a precise definition of weak solutions to the stationary problem. 

\begin{definition} \label{def:weak_sol}

(a) We say that $(h, \nabla_x \phi_h) \in \big( L^2_{loc}(\O \times \R^3) \; \cap \; L^2_{loc}(\p\O \times \R^3; \dd \gamma) \big) \times L^2_{loc}(\O)$ satisfies the weak formulation of the kinetic equation in \eqref{eqn:h}, if all terms are bounded and the following condition holds for any test function $\psi \in C^\infty_c (\bar \O \times \R^3)$, 
\Be \label{weak_form_1}
\begin{split}
& \int_{ \gamma_+ } h (x, v) \psi(x, v) \dd \gamma - \int_{ \gamma_- } f_b (x, v) \psi(x, v) \dd \gamma - \iint_{\O \times \R^3} h (x,v) v \cdot \nabla_x \psi(x, v) \dd v \dd x
\\& \ \ \ \ + \iint_{\O \times \R^3} h (x, v) \nabla_x (\phi_h + \phi_E) \cdot \nabla_v \psi (x, v)  \dd v \dd x + \iint_{\O \times \R^3} \frac{v\cdot \nabla_x \phi_{h}}{2} h \psi(x, v) \dd v \dd x 
\\& \ \ \ \ + \iint_{\O \times \R^3} e^{-\phi_E}\mathcal{L} h \psi(x, v) \dd v \dd x
\\& = \iint_{\O \times \R^3} - ( v \cdot \nabla_x \phi_{h} )e^{-\phi_E/2}\sqrt{\mu} \psi(x, v) \dd v \dd x + \iint_{\O \times \R^3} e^{-\phi_E/2} \Gamma (h, h) \psi(x, v) \dd v \dd x.
\end{split}
\Ee
	
(b) We say that $(h, \nabla_x \phi_h) \in L^2_{loc}(\O \times \R^3) \times W^{1,2}_{loc} (\O  )$ satisfies the weak formulation of the Poisson equation in \eqref{eqn:h}, if all terms below are bounded and the following condition holds for any test function $\varphi \in H^1_0 (\O) \cap C^\infty_c (\bar \O)$, 
\Be \label{weak_form_2}
\int_{\O} \nabla_x \phi_h \cdot \nabla_x \varphi \dd x 
= \int_{\O} e^{-\phi_E/2} \big( \int_{\mathbb{R}^3} h \sqrt{\mu} \dd v \big) \varphi \dd x.
\Ee
\end{definition}

Next, we introduce the kinetic and velocity weights for the stationary problem.

\begin{definition} \label{def:alpha_weight_steady}

(a)
Since the potential field is time-independent in the stationary problem \eqref{eqn:h}, we denote
\begin{align*}
& E_{h}(x)=- \nabla_x \big( \phi_{h}(x) + \phi_E(x) \big).
\end{align*}
Recall $\xi(x)$ and $\O_\delta$ introduced in \eqref{xi_dist} and \eqref{def:O_d}, respectively.
By Lemma~\ref{lemma:dist_unique}, for sufficiently small $0<\delta=\delta(\O)\ll1$, each $x\in\O_\delta$ admits a unique $\tilde{x}\in\partial\O$ such that $dist (x, \tilde{x}) = dist (x, \partial \O) = - \xi(x)$.
To compensate for the boundary singularity, we choose $0 < \delta' \ll \delta^{1/2}$ and define the kinetic weight $\alpha_h(x,v)$ by
{\small
\be \label{alpha_weight_steady}
\begin{split}
& \alpha_h (x,v)  
\\& :=
\begin{cases}
\chi_{\delta'}\Big(\Big[|v\cdot \nabla_x\xi(x)|^2 + \xi^2(x) - 2 (v\cdot \nabla^2\xi(x) \cdot v)\xi(x) - 2(E_{h}(\tilde{x})\cdot \nabla_x\xi(\tilde{x}))\xi(x) \Big]^{1/2}\Big), & \text{if } x\in \O_\delta, \\[5pt]
\delta', & \text{if } x\in \O \setminus \O_\delta.
\end{cases}    
\end{split}
\ee
}
Here $\chi_{\delta'}$ is an increasing cut-off function satisfying
\begin{equation} \label{cut_off_function}
\chi_{\delta'}(x) = 
    \begin{cases}
        x, & 0\leq x\leq \frac{\delta'}{4}, \\[5pt]
        \delta', & x\geq \frac{\delta'}{2},
    \end{cases}
    \quad \text{and} \quad 0\leq \chi_{\delta'}'(x)\leq 4.
\end{equation}

(b) We define the velocity weights $w(v), w_{\tilde{\theta}} (v)$ by 
\begin{equation*}
\begin{split}
& w(v) := e^{\theta |v|^2}, \ 0 < \theta < \frac{1}{4}, 
\ \text{ and } \
w_{\tilde{\theta}} (v) := e^{\tilde{\theta}|v|^2}, \ 0 < \tilde{\theta} \ll \theta.
\end{split}
\end{equation*}
\end{definition}

We now state our main result on the existence, uniqueness, and regularity of stationary solutions.

\begin{theorem} \label{thm:steady_wellpose}

Consider the kinetic and velocity weights $\alpha_h(x,v)$, $w(v)$, and $w_{\tilde{\theta}}(v)$ defined in Definition~\ref{def:alpha_weight_steady}.
There exist sufficiently small constants $\delta_1 \ll \delta_0 \ll 1$ such that, if the external field and inflow data $F_b = e^{-\phi_E}\mu + e^{-\phi_E/2}\sqrt{\mu}f_b \geq 0$ satisfy
\be \label{inflow_condition}
\begin{aligned}
\textbf{(Inflow condition)} \qquad &  
\begin{cases}
& \frac{\delta_0}{2}\leq C_E \leq \Vert  \nabla_x \phi_E  \Vert_{L^\infty_{x}} < \delta_0, \\[5pt]
& \Vert e^{-\phi_E/2} - 1 \Vert_{L^\infty_x} +\Vert \nabla_x^2 \phi_E\Vert_{L^\infty_x} + \Vert \nabla_x^3\phi_E\Vert_{L^p_x} < \delta_0, \\[5pt]
& | wf_b|_{L^\infty_{\p\O,v}} + | w \nabla_{v}f_b|_{L^\infty_{\p\O,v}}+ | w \p_{\mathbf{x}_{p}}f_b|_{L^\infty_{\p\O,v}} < \delta_1.
\end{cases}
\end{aligned}
\ee
Then there exists a unique solution $F_s = e^{-\phi_E}\mu + e^{-\phi_E/2}\sqrt{\mu}h\geq 0$ to \eqref{nonlinear_SVPB_system}.
Moreover, for $2<p<3$ and some constants $C_0, C_1 > 1$, it holds
\begin{align}
    & \Vert wh\Vert_{L^\infty_{x,v}} < C_0 \delta_1 \ll 1,  \notag\\
    & \Vert w_{\tilde{\theta}}\alpha_h \p_{x,v} h\Vert_{L^\infty_{x,v}} + \Vert w_{\tilde{\theta}} \nabla_v h\Vert_{L^\infty_{x,v}} + \Vert w_{\tilde{\theta}} \p_{x,v}h\Vert_{L^p_{x,v}} < C_1 \delta_1\ll 1. \label{regularity_steady}
\end{align}
\end{theorem}

Theorem~\ref{thm:steady_wellpose} is restated as Theorem~\ref{thm:steady_solution} in Section~\ref{sec:existence_construction}, where its proof is presented.
The positivity of the solution is established in the proof of Theorem \ref{thm:dynamical_stability} in Section \ref{sec:dynamical_proof}; see also Remark~\ref{rmk:positivity} for further discussion.


\begin{remark} \label{rmk:dimension}

The velocity regularity is crucial in proving the uniqueness. 
For the stationary problem, due to the absence of a Gronwall-type argument, uniqueness needs to be proved under the $L^2_{x,v}$ space; see the detailed discussion in Section~\ref{sec:difficulty}.
More precisely, let $h_1$ and $h_2$ be two solutions. In the $L^2$ energy estimate for $h_1-h_2$, one needs to control the term
\be \notag
\iint_{\O\times \mathbb{R}^3} |\nabla_v h_1 \nabla_x (\phi_{h_1}-\phi_{h_2})| |h_1-h_2| \dd x \dd v.
\ee
In the three-dimensional case, one needs to place
$\nabla_x (\phi_{h_1}-\phi_{h_2})$ in $L^6_x$ and apply the Sobolev embedding
\be \notag
\Vert \nabla_x (\phi_{h_1}-\phi_{h_2}) \Vert_{L^6_x} \lesssim \Vert \nabla_x (\phi_{h_1}-\phi_{h_2}) \Vert_{H^1_x} \lesssim \Vert h_1-h_2\Vert_{L^2_{x,v}}.
\ee
Consequently, one is led to require the regularity
$\Vert \nabla_v h_1\Vert_{L^3_xL^2_v}$, which lies beyond the range of the Sobolev regularity
$W^{1,p}_{x,v}$ with $p<3$.
Therefore, it becomes necessary to introduce the kinetic weight in order to derive higher regularity estimates, such as the point-wise regularity in \eqref{regularity_steady}.

In contrast, for a two-dimensional domain, this difficulty can be overcome by a refined Sobolev embedding. For instance,
\be \notag
\Vert \nabla_x (\phi_{h_1}-\phi_{h_2})\Vert_{L^8_x}
\lesssim \Vert \nabla_x (\phi_{h_1}-\phi_{h_2})\Vert_{H^1_x}
\lesssim \Vert h_1-h_2\Vert_{L^2_{x,v}}.
\ee
As a result, the weaker regularity
$\Vert \nabla_v h_1\Vert_{L^{8/3}_xL^2_v}$ is sufficient, and hence the Sobolev regularity
$W^{1,p}_{x,v}$ with $p<3$ already suffices. Nevertheless, without higher regularity, one cannot expect dynamical stability or positivity. Therefore, the $C^1_v$ regularity remains necessary; see Remark~\ref{Remark:dynamical} for further discussion.
\end{remark}

\begin{remark} \label{rmk:point-wise_regularity}

Without the external field, the weighted $C^1$ regularity in \eqref{regularity_steady} is not expected in general, as discussed in \cite{CKL}. 
Indeed, the variance of the kinetic weight
\[
\alpha_h^2 \sim |n(x)\cdot v|^2 + \mathrm{dist}(x,\partial\Omega)|v|^2
\]
along the characteristic becomes singular in the small-velocity regime $|v|\ll 1$ due to the electric field. In the presence of an external potential field, \cite{cao2019regularity} incorporated the distance information and introduced the weight \eqref{alpha_weight_steady}, thereby establishing local well-posedness with weighted $C^1$ regularity.

In this paper, we adapt this weight and further construct a unique stationary solution with weighted $C^1$ regularity. This regularity is particularly delicate in the stationary setting, and $W^{1,p}$ estimates are required to implement the bootstrap argument; see the detailed discussion in Section~\ref{sec:difficulty}. 
In addition, we establish $C^1_v$ regularity without kinetic weight, which plays an essential role in proving dynamical stability and deriving uniform-in-time estimates; see also Remark~\ref{Remark:dynamical} for further discussion.
\end{remark}

\begin{remark} \label{rmk:positivity}

The positivity of the stationary solution $F_s \geq 0$ follows from the positivity and asymptotic stability of the solution $F(t,x,v)$ to the dynamical problem. More precisely,
\[
F(t) \geq 0, \qquad
\Vert F(t)-F_s\Vert_{L^\infty_{x,v}} \to 0
\ \text{ as }
t \to \infty,
\]
as established in Theorem~\ref{thm:dynamical_stability}. 
We refer to the proof of Theorem~\ref{thm:dynamical_stability} in Section~\ref{sec:dynamical_proof} for further details.
\end{remark}

\subsection{Dynamical problem and stability analysis}

The dynamical problem is given by
\be \label{F_dynamical}
\begin{cases}
& \p_t F + v\cdot\nabla_x F - \nabla_x (\phi_F+\phi_E)\cdot \nabla_{v} F  =  Q(F,F), \\
& F(0,x,v) = F_0(x,v), \quad F|_{\gamma_-} = F_b, \\
& -\Delta_{x} \phi_F = \int_{\mathbb{R}^3} F \dd v -  e^{-\phi_E}  \text{ in } \O, \quad \phi_F = 0 \text{ on } \p\O, \\
& -\p_n \phi_E > C_E > 0 \text{ on } \p\O. 
\end{cases} 
\ee
To study the dynamical stability around the stationary solution $F_s$, we decompose
\be \notag
F=F_s+e^{-\phi_E/2}\sqrt{\mu}f = e^{-\phi_E}\mu + e^{-\phi_E/2}\sqrt{\mu}(h+f),
\ee
and obtain the equations for $f$:
\be \label{linear_f_dynamical}
\begin{cases}
& \p_t f + v \cdot \nabla_x f - \nabla_x (\phi_f+\phi_E)\cdot \nabla_{v} f + \frac{v \cdot \nabla_x \phi_f}{2}f + e^{-\phi_E}\mathcal{L} f \\
& =  -e^{-\phi_E/2}v\cdot \nabla_x (\phi_f-\phi_{h}) \sqrt{\mu} - \frac{v \cdot \nabla_x (\phi_f-\phi_{h})}{2}h + \nabla_x (\phi_f-\phi_{h}) \cdot \nabla_{v} h \\
& \ \ \ \ +   e^{-\phi_E/2}[\Gamma(f,f) + \Gamma(f,h)+\Gamma(h,f)], \\
& f(0,x,v)=f_0(x,v):= \frac{F_0(x,v)-F_s}{e^{-\phi_E/2}\sqrt{\mu}}, \quad f|_{\gamma_-} = 0, \\
& -\Delta_x \phi_f = e^{-\phi_E/2} \int_{\mathbb{R}^3} (h+f)\sqrt{\mu} \dd v \text{ in } \O, \quad \phi_f = 0 \text{ on } \p\O, \\
& -\p_n \phi_E > C_E > 0 \text{ on } \p\O.
\end{cases}  
\ee

Next, we provide a precise definition of weak solutions to the dynamical problem. 

\begin{definition} \label{def:weak_sol_dy} 

Suppose $(h, \nabla_x \phi_h)$ is a weak solution of \eqref{eqn:h} in the sense of Definition \ref{def:weak_sol}.  
	
(a) We say that $(f, \nabla_x \phi_f) \in \big( L^2_{loc}(\R_+ \times \O \times \R^3) \; \cap \; L^2_{loc}(\R_+ \times \p\O \times \R^3; \dd \gamma) \big) \times L^2_{loc}(\R_+ \times \O \times \R^3)$ satisfies the weak formulation of the kinetic equation in \eqref{linear_f_dynamical}, if all terms below are bounded and the following condition holds for any test function $\psi \in C^\infty_c (\R_+ \times \bar \O \times \R^3)$, 
\Be \notag 
\begin{split}
& \iint_{ \O \times \R^3} f_{0} (x,v) \psi(0,x,v) \dd v \dd x - \iint_{\R_+ \times \O \times \R^3} f (t,x,v) \p_t \psi(t, x, v) \dd v \dd x \dd t + \iint_{\R_+ \times \O \times \R^3} e^{-\phi_E}\mathcal{L} f \psi(t,x, v) \dd v \dd x \dd t
\\& + \iint_{\R_+ \times \gamma_+ } f (t,x,v) \psi(t,x,v) \dd \gamma \dd t  - \iint_{\R_+ \times \O \times \R^3} f (t,x,v) v \cdot \nabla_x \psi(t, x, v) \dd v \dd x \dd t
\\& + \iint_{\R_+ \times \O \times \R^3} f (t, x, v) \nabla_x (\phi_f + \phi_E) \cdot \nabla_v \psi (t,x,v)  \dd v \dd x \dd t + \iint_{\R_+ \times \O \times \R^3} \frac{v\cdot \nabla_x \phi_{f}}{2} f \psi(t, x, v) \dd v \dd x \dd t
\\& 
\\& = \iint_{\R_+ \times \O \times \R^3} - ( v \cdot \nabla_x (\phi_f-\phi_{h}) )e^{-\phi_E/2}\sqrt{\mu} \psi(t,x, v) \dd v \dd x \dd t - \iint_{\R_+ \times \O \times \R^3} \frac{v \cdot \nabla_x (\phi_f-\phi_{h})}{2} h \psi(t,x, v) \dd v \dd x \dd t
\\& \ \ \ \ - \iint_{\R_+ \times \O \times \R^3} \nabla_x (\phi_f-\phi_{h}) \cdot  h \nabla_{v} \psi(t,x, v) \dd v \dd x \dd t 
\\& \ \ \ \ + \iint_{\R_+ \times \O \times \R^3} e^{-\phi_E/2}[\Gamma(f,f) + \Gamma(f,h)+\Gamma(h,f)] \psi(t,x, v) \dd v \dd x \dd t.
\end{split}
\Ee

(b) We say that $(f, \nabla_x \phi_f) \in L^2_{loc}(\R_+ \times \O \times \R^3) \times W^{1,2}_{loc} (\R_+ \times \O)$ satisfies the weak formulation of the Poisson equation in \eqref{linear_f_dynamical}, if all terms below are bounded and the following condition holds for any $t \in \R_+$ and test function $\varphi \in H^1_0 (\O) \cap C^\infty_c (\bar \O)$,
\Be \notag
\int_{\O} \nabla_x \phi_f (t, x) \cdot \nabla_x \varphi (x) \dd x 
= \int_{\O} e^{-\phi_E/2} \big( \int_{\mathbb{R}^3} (h + f) \sqrt{\mu} \dd v \big) \varphi \dd x.
\Ee
\end{definition}

For the dynamical problem, the self-consistent electric field depends on time. We denote
\be \notag
E_f (t,x) := -[\nabla_x \phi_f(t,x) + \nabla_x \phi_E(x)]. 
\ee
The kinetic weight for the dynamical problem is denoted as
{\small
\be \label{alpha_weight_dyna}
\begin{split}
& \alpha_f(t,x,v) 
\\& :=
\begin{cases}
\chi_{\delta'}\Big(\Big[|v\cdot \nabla_x\xi(x)|^2 + \xi^2(x) - 2 (v\cdot \nabla^2\xi(x) \cdot v)\xi(x) - 2(E_{f}(t,\tilde{x})\cdot \nabla_x\xi(\tilde{x}))\xi(x) \Big]^{1/2}\Big), & \text{if } x\in \O_\delta, \\[5pt]
\delta', & \text{if } x\in \O \setminus \O_\delta.
\end{cases}
\end{split}
\ee
}

\begin{theorem} \label{thm:dynamical_stability}

Assume that all assumptions in Theorem \ref{thm:steady_wellpose} hold, and let $F_s$ be the stationary solution to \eqref{nonlinear_SVPB_system}.
There exists a constant $\delta_2 \ll 1$ such that, if the initial data $F_0 = F_s + e^{-\phi_E/2} \sqrt{\mu} f_0 \geq 0$ satisfies the compatibility condition $f_0 |_{\gamma_-} = 0$ and the smallness condition
\be \label{initial_condition}
\Vert wf_0\Vert_{L^\infty_{x,v}} + \Vert w_{\tilde{\theta}}\alpha_{f_0} \p_{x,v}f_0\Vert_{L^\infty_{x,v}} + \Vert w_{\tilde{\theta}} \p_{x,v}f_0\Vert_{L^p_{x,v}} + \Vert w_{\tilde{\theta}} \nabla_{v}f_0\Vert_{L^\infty_{x,v}} < \delta_2,  
\ee
then there exists a unique solution $F(t) = F_s + e^{-\phi_E/2}\sqrt{\mu}f(t) \geq 0$ to \eqref{F_dynamical}.
Moreover, there exist constants $\lambda \ll 1$ and $C_2, C_3 > 1$ such that for any $2 < p < 3$ and $t \geq 0$,
\begin{align}
& \Vert e^{\lambda t}wf(t)\Vert_{L^\infty_{x,v}} < C_2 \delta_2 \ll 1, 
\label{global_Linfty} \\
& \Vert w_{\tilde{\theta}} \alpha_f\p_{x,v}f(t)\Vert_{L^\infty_{x,v}} + \Vert w_{\tilde{\theta}} \p_{x,v}f(t)\Vert_{L^p_{x,v}} + \Vert w_{\tilde{\theta}} \nabla_v f(t)\Vert_{L^\infty_{x,v}} < C_3 (\delta_2 + \delta_1) \ll 1, 
\label{global_regularity}
\end{align}
where $\delta_1$ is the constant defined in Theorem \ref{thm:steady_wellpose}.
\end{theorem}


\begin{remark} \label{Remark:dynamical}


For the pure Boltzmann equation, the regularity is shown to decay exponentially fast in \cite{chen2024gradient}. However, in the presence of the additional Vlasov–Poisson coupling, the term $\nabla_v h$ remains on the right-hand side of the perturbation equation \eqref{linear_f_dynamical}. Due to the lack of control of second-order velocity derivatives of $h$, it is necessary to adopt an alternative formulation. 
In particular, we rewrite the solution $F (t,x,v)$ as \begin{equation} \notag
F (t,x,v)= e^{-\phi_E}\mu + e^{-\phi_E/2} \sqrt{\mu} \mathfrak{f}. 
\end{equation} 
This yields a reformulated system \eqref{eqn:mkf} in which the term involving $\nabla_v h$ is replaced by derivatives of the Maxwellian $\nabla_v \sqrt{\mu}$.
On the other hand, this formulation leads to a non-homogeneous boundary condition 
\be \notag
\mathfrak{f}(t,x,v)|{\gamma_-} = f_b,
\ee
and therefore decay in time for $\partial_{x,v}\mathfrak{f}$ cannot be expected. Consequently, the dynamical regularity result in \eqref{global_regularity} is only uniformly bounded in time, rather than decaying exponentially.

In addition, due to the presence of the term $\nabla_v h$ on the right-hand side, a $C^1_v$ estimate $\Vert \nabla_v h\Vert_{L^\infty_{x,v}}$ is required when applying the $L^2$–$L^\infty$ framework to prove dynamical stability. For the pure Boltzmann equation, the velocity derivative typically exhibits a singularity of order $1/|v|^2$, as shown in \cite{CK}. Such a $C^1_v$ estimate can only be recovered in the presence of the external potential field $\nabla_x \phi_E$. A similar type of $C^1_v$ estimate was also used in \cite{jin2026asymptotic}, where a dominant gravitational field is considered for the Vlasov–Poisson system. For a different type of stabilizing effect via an external electric field in the Vlasov–Poisson system, see \cite{einkemmer2025control}.

\end{remark}

\subsection{Main difficulties and proof strategy} \label{sec:difficulty}

Due to the presence of the self-consistent electric field, the main difficulty in constructing a unique stationary solution to the VPB system lies in obtaining suitable stationary regularity estimates. The existence and uniqueness arguments rely on these estimates together with a sequential construction scheme. 
The dynamical problem is closely related and follows similar ideas; for this reason, we mainly focus on the stationary case in the discussion of the key difficulties. The corresponding differences and additional ideas in the dynamical setting are described in Remark \ref{Remark:dynamical}

\begin{itemize}
    \item A priori $W^{1,p}$ estimate
\end{itemize}

In the stationary problem, a major difficulty in deriving stationary regularity estimates lies in controlling the integral operator $K$ and its associated kernel $\mathbf{k}$ (see Section~\ref{sec:collision_operator}).
For the stationary Boltzmann equation, the key idea in \cite{CK} is to exploit the mixing effect along characteristics to gain regularity.
Specifically, a double Duhamel formula is employed, leading to the expression
\begin{align*}
& \int^t_0 \dd s \int_{\mathbb{R}^3} \dd u \mathbf{k}(v,u) \int^s_0 \dd s' \int_{\mathbb{R}^3}\dd u' \mathbf{k}(u,u') \nabla_x h(x-(t-s)v-(s-s')u,u').
\end{align*}
The additional iteration introduces the key term $(s-s')u$ in the spatial variable. Through the change of variables
\begin{align*}
& \nabla_x h(x-(t-s)v-(s-s')u,u') = - (s-s')^{-1} \nabla_u [h(x-(t-s)v-(s-s')u,u')],
\end{align*}
the derivative can be removed through integration by parts in $\dd u$ when $s-s'$ is away from $0$.

In our stationary VPB problem, the characteristic is determined by the potential field, and the Duhamel formula becomes
\be \label{problem_term}
\int^t_0 \dd s \int_{\mathbb{R}^3} \dd u \mathbf{k}(V(s),u) \int^s_0 \dd s' \int_{\mathbb{R}^3}\dd u' \mathbf{k}(V(s';s,X(s),u),u') \nabla_x[ h(X(s';s,X(s),u),u')]. 
\ee
Applying the chain rule in this case yields
\be \notag
\nabla_x [h(X(s';s,X(s),u),u')] 
\sim \nabla_x X(s';s,X(s),u) \nabla_u [h(X(s';s,X(s),u),u')] \nabla_u^{-1}X(s';s,X(s),u).
\ee
However, integration by parts in $\dd u$ requires control of the second derivative $\nabla_{u} [\nabla_{x} X(s';s,X(s),u)]$, which depends on the third-order derivative $\Vert\nabla_x^3 \phi_h\Vert_{L^\infty_x}$.
This in turn leads to the need for control of second-order derivatives of $h$, which are not expected even for the pure Boltzmann equation \cite{GKTT}.

To overcome this difficulty, we observe that the third-order derivative control of $\nabla_x^3 \phi_f$ can be expected in $L^p$ spaces. As demonstrated in \cite{chen2024gradient}, in a convex domain, the optimal Sobolev regularity for the Boltzmann equation is $W^{1,p}$ with $p<3$ (see also \cite{chen2023geometric}). By the Calder\'on–Zygmund estimates,
\begin{align*}
& \Vert \nabla_x^3 \phi_h\Vert_{L^p_{x}} \lesssim \Vert h\Vert_{L^p_{x,v}} + \Vert \nabla_x h\Vert_{L^p_{x,v}}.
\end{align*}
This provides control of $\Vert\nabla_{x,v}X(s;t,x,v)\Vert_{L^p_{x,v}}$, as shown in the Lemma \ref{lemma:X_vv_Lp}.

We note that the $W^{1,p}$ estimate in \cite{chen2024gradient} is derived from a weighted point-wise $C^1$ estimate. In our setting, the $W^{1,p}$ regularity must be constructed independently. Due to the absence of a Gronwall inequality in the stationary problem, the $W^{1,p}$ estimate is constructed directly through the method of characteristics. In this framework, the problematic term \eqref{problem_term} can be controlled under the $L^p$ norm after integration by parts. 
The integrability in $v$ within $L^p_v$ arises from carefully applying the H\"older inequality to extract $\langle v\rangle^{-1}$ from the damping factor $\nu(v)$ in time integration and the integration against $\mathbf{k}(v,u) \dd u$. This imposes the condition $p > 2$, which holds in convex domains.
Another difficulty in the $W^{1,p}$ estimate arises from the singularity of the derivative of the backward exit time $\tb$ (see \eqref{def_tb}), given by
\[
\nabla_x \tb \sim \frac{1}{n(\xb)\cdot \vb}.
\]
To control the integral of the singularity in $L^p_{x,v}$ space, we crucially apply the change of variable in Lemma \ref{lemma:integrate_nv}, which reduces the singular exponent to $p-1$. 
Next, we exploit the external potential field to control the forward exit time $\tf$ (see \eqref{def_tb}) through the estimate in \eqref{est_tb}, 
\[
\tf \lesssim |n(x)\cdot v|,
\]
which further reduces the exponent to $p-2$.
Such an estimate generally fails without the external field, particularly in the small-velocity regime. We refer to Section \ref{sec:w1p_estimate} for further details.

\begin{itemize}
    \item $W^{1,p}$–$\alpha C^1_{x, v}$ bootstrap and $C^1_v$ estimate without kinetic weight
\end{itemize}

To control the second derivative $\Vert \nabla_x^2 \phi_h \Vert_{L^\infty_{x}}$ in the three-dimensional domain, weighted $C^1$ regularity remains necessary, as shown in Lemma \ref{lemma:nabla2_phi_bdd} and Lemma \ref{lemma:phi_C2}. To circumvent the aforementioned difficulty with $\Vert \nabla^3_x \phi_h\Vert_{L^\infty_x}$, our key observation is that the $W^{1,p}$ estimate allows us to implement a $W^{1,p}$–$\alpha C^1$ bootstrap argument, converting point-wise regularity into $W^{1,p}$ control. 
This argument avoids integration by parts and therefore does not require second derivatives of $X(s)$. We note that this bootstrap argument shares the spirit of the $L^2$–$L^\infty$ argument in \cite{G}, which applied to the double Duhamel formula without taking derivatives.

As a key component of this bootstrap argument, when the external potential field $\phi_E$ dominates near the boundary, distance information can be incorporated into the kinetic weight $\alpha_h$ (see Remark \ref{rmk:point-wise_regularity}). Such a weight was introduced in \cite{cao2019regularity} for local-in-time solutions, remaining constant at a distance $\delta$ from the boundary while compensating for the singularity $n(\xb) \cdot \vb$ near the grazing set $\gamma_0$.
To adapt this weight to the stationary and global setting, we observe that, due to the presence of $\phi_E$, the characteristic can stay near the boundary only for a finite time, as shown in \eqref{t2_t1_bdd} in Lemma \ref{lemma:est_tf}. This observation yields control of the variation of $\phi_E$ along the characteristic in Lemma \ref{lemma:velocity}. Consequently, the nonlocal-to-local estimate in Lemma \ref{lemma:nonlocal_to_local} remains valid for large-time integration.

Furthermore, justifying dynamical stability requires a $C^1_v$ estimate without kinetic weight (see Remark \ref{Remark:dynamical}). This is not generally true without an external field $\nabla_x\phi_E$. In our setting, the external field provides control of $\tb$ via the singularity $|n(x)\cdot v|$, as shown in \eqref{est_tb}. This estimate is crucial in controlling the $v$-derivative of the characteristic in Lemma \ref{lemma:deri_backward}, where the denominator singularity is compensated by $\tb$. We refer to Section \ref{sec:c1_estimate} for further details.

\begin{itemize}
    \item Uniqueness argument.
\end{itemize}

In \cite{CKL}, uniqueness is proved in the space $L^3_x L^{1+}_v$ by applying Gronwall's inequality in the dynamical setting. However, in the stationary problem, this approach is no longer available, as Gronwall's inequality cannot be used.
Since the linear collision operator $\mathcal{L}$ (see Section~\ref{sec:collision_operator}) is coercive only on $L^2_v$, we work in $L^2_v$ and control $\Vert h_1-h_2\Vert_{L^2_{x,v}}$. 

The microscopic component $\|(\mathbf{I}-\mathbf{P})(h_1-h_2)\|_{L^2_{x,v}}$ is controlled by the coercivity of $\mathcal{L}$. For the macroscopic component $\|\mathbf{P}(h_1-h_2)\|_{L^2_{x,v}}$, we employ the dual argument in \cite{EGKM,EGKM2}.
In the equation for $h_1-h_2$, the main problematic term is
\begin{align*}
& \int_{\O\times \mathbb{R}^3} |\nabla_x \phi_{h_1-h_2} \cdot \nabla_v h_2| |h_1-h_2| \dd x \dd v.
\end{align*}
In the $L^2_{x,v}$ framework, this term requires at least $L^2_v$ control of $\nabla_v h_2$, which is not available in \cite{CKL}. 
However, in our setting with external force, the $v$-derivative enjoys a $C^1_v$ estimate without kinetic weight, as mentioned above. Consequently, the problematic term can be controlled by 
\[\Vert h_1-h_2\Vert_{L^2_{x,v}} \Vert \nabla_x \phi_{h_1-h_2}\Vert_{L^2_x} \Vert \nabla_v h_2\Vert_{L^\infty_{x,v}} \lesssim \Vert h_1-h_2\Vert_{L^2_{x,v}}^2 \Vert \nabla_v h_2\Vert_{L^\infty_{x,v}}.\]
We refer to Section \ref{sec:stationary_uniqueness} for further details.

\begin{itemize}
    \item Existence argument.
\end{itemize}

To establish the existence of a stationary solution, we begin with the following sequential construction (see \eqref{eqtn:h^l}–\eqref{bdry:phi^l} for details): for any $\ell \in \N$,
\be \label{intro:h^l}
\begin{split}
v \cdot \nabla_x h^{\ell+1} 
& - \nabla_x (\phi^\ell_h + \phi_E) \cdot \nabla_{v} h^{\ell+1} 
+ \frac{v \cdot \nabla_x \phi^\ell_h}{2} h^{\ell+1} + e^{-\phi_E} \mathcal{L} h^{\ell+1}
\\& = - (v \cdot \nabla_x \phi^{\ell+1}_h) e^{-\phi_E/2} \sqrt{\mu} + e^{-\phi_E/2} \Gamma(h^{\ell}, h^{\ell}),
\\ - \Delta \phi^i_h & = e^{-\phi_E/2} \int_{\R^3} h^{i} \sqrt{\mu} \dd v \text{ in } \O,  \ i\in \{\ell,\ell+1\}.
\end{split}
\ee
We emphasize that the term $\nabla_x \phi^{\ell+1}_h$ appears on the right-hand side. This choice is crucial, as it enables us to obtain the key estimate $\Vert \mathbf{P} h^{\ell+1} \Vert_{L^2_{x,\nu}}^2$ in \eqref{est:Unif_h_L2_2}, which in turn yields uniform-in-$\ell$ bounds for $h^{\ell+1}$ (see Lemma~\ref{lemma:Unif_steady}).
We then show that $\{ h^{\ell} \}_{\ell=0}^\infty$ forms a Cauchy sequence in \( L^2_{x,\nu}(\Omega \times \mathbb{R}^3) \cap L^2(\gamma_+) \), and its limit is a weak solution to the steady problem.

The main difficulty lies in establishing the well-posedness of the iterative scheme \eqref{intro:h^l}. Indeed, classical well-posedness results typically apply only when the right-hand side depends solely on the $\ell$-th iterate, whereas here it also involves the $(\ell+1)$-th order term through $\phi_h^{\ell+1}$. 
To overcome this difficulty, we introduce a sufficiently small parameter $0 < \lambda \ll 1$, a pair of given external functions $(g, \phi_g)$, and a sequence of source terms $\{ S_k \}^{\infty}_{k=1}$. 
We then consider the following auxiliary system:
\be \label{intro:h_lambda}
\begin{split}
v \cdot \nabla_x h
& - \nabla_x (\phi_g + \phi_E) \cdot \nabla_{v} h 
+ \frac{v \cdot \nabla_x \phi_g}{2} h + e^{-\phi_E} \mathcal{L} h
\\& = - k \lambda (v \cdot \nabla_x \phi_h) e^{-\phi_E/2} \sqrt{\mu} + e^{-\phi_E/2} \Gamma(g, g) + S_k, 
\\ - \Delta \phi_h & = e^{-\phi_E/2} \int_{\R^3} h \sqrt{\mu} \dd v \text{ in } \O.
\end{split}
\ee

We begin by establishing the well-posedness of \eqref{intro:h_lambda} (see Proposition~\ref{prop:steady_small_lambda_step1}) in the case  $k=1$.
Thanks to the smallness of \( \lambda \), the well-posedness of this system follows from a sequential argument combined with a priori estimates.
Next, in Proposition~\ref{prop:h_step2}, we extend the result from $\lambda$ to $N \lambda = 1$ (i.e., $k = N$). 
The key idea is to rewrite the right-hand side of the first equation in \eqref{intro:h_lambda} as
\begin{equation} \notag
\begin{split}
& - k \lambda (v \cdot \nabla_x \phi_h) e^{-\phi_E/2} \sqrt{\mu} + e^{-\phi_E/2} \Gamma(g, g) + S_k
\\& = - (k-1) \lambda (v \cdot \nabla_x \phi_h) e^{-\phi_E/2} \sqrt{\mu} + e^{-\phi_E/2} \Gamma(g, g) 
+ \underbrace{\big( S_k - \lambda (v \cdot \nabla_x \phi_h) e^{-\phi_E/2} \sqrt{\mu} \big)}_{\tilde{S_k}}.
\end{split}
\end{equation}
This allows us to treat the system inductively. Specifically, assuming that for some $C > 0$,
\begin{equation} \label{intro:S_k}
\begin{split}
\Vert w S_k \Vert_{L^\infty_{x,v}} + \Vert S_k \Vert_{L^2_{x,v}}
\lesssim C[ |f_b|_{L^2_{\gamma_-}} + | w f_b|_{L^\infty_{\p\O,v}}].
\end{split}
\end{equation}
By a priori estimates and induction, we obtain that
\be \label{intro:h_S_k}
\begin{split}
\Vert w h \Vert_{L^\infty_{x,v}} + | h |_{L^2_{\gamma_+}} + \Vert h \Vert_{L^2_{x,\nu}}
& \lesssim (C+1) [ |f_b|_{L^2_{\gamma_-}} + | w f_b|_{L^\infty_{\p\O,v}}].
\end{split}
\ee
We note that the factor $C + 1$ in \eqref{intro:h_S_k} arises from the additional term $\lambda (v \cdot \nabla_x \phi_h) e^{-\phi_E/2} \sqrt{\mu}$ absorbed into $\tilde{S_k}$.
Moreover, similar regularity estimates for $\p_{x,v} h$ hold as in \eqref{regularity:wh_c1_step2} of Proposition~\ref{prop:h_step2}.
Finally, we set $S_N = 0$ for $k = N$ with $N \lambda = 1$. In this case, $S_N$ satisfies \eqref{intro:S_k} with $C = 0$, and the corresponding estimate \eqref{intro:h_S_k} holds.
Then, by replacing the prescribed function $g$ with the sequence $\{ h^{\ell} \}_{\ell=0}^\infty$ in \eqref{intro:h_lambda}, we recover the original scheme \eqref{intro:h^l}, thereby establishing its well-posedness (see Theorem \ref{thm:well_poseness}).
We refer to Section \ref{sec:existence_construction} for further details.

\subsection{Outline of the paper}

The rest of the paper is organized as follows:
We collect several preliminary estimates in Section \ref{sec:prelim}, including properties of the characteristics, the collision operators, and the kinetic weight.
Sections~\ref{sec:L2Linfty_estimate}, \ref{sec:w1p_estimate}, and \ref{sec:c1_estimate} establish a priori estimates for the stationary VPB system: the $L^2$--$L^\infty$ estimate, the $W^{1,p}$ estimate, and the $\alpha$-$C^1_{x,v}$ and $C^1_v$ estimates without kinetic weight, respectively.
In Section \ref{sec:stationary_uniqueness}, we prove the uniqueness of stationary solutions under a smallness condition on the unweighted $C^1_v$ estimate.
Based on the estimates and the uniqueness argument in Sections~\ref{sec:L2Linfty_estimate}–\ref{sec:stationary_uniqueness}, Section~\ref{sec:existence} constructs stationary solutions via a sequential scheme. For clarity, the a priori estimates, the well-posedness, and the existence argument are presented separately. 
Finally, Section~\ref{sec:dynamical} studies the dynamical problem around the stationary solution and establishes dynamical stability as well as uniform-in-time regularity, as stated in Theorem \ref{thm:dynamical_stability}.

\subsection{Notation}

We introduce the following notational conventions.
The relation $f \lesssim g$ (resp.\ $f \gtrsim g$) denotes that there exists a constant $C>0$ such that $f \leq C g$ (resp.\ $f \geq C g$). The notation $\partial_{x,v} f$ denotes either a first-order spatial derivative or a first-order velocity derivative of $f$. Thus, $|\partial_{x,v} f| \lesssim 1$ means $|\partial_x f| + |\partial_v f| \lesssim 1$, and $|\partial_{x,v} f| \lesssim |\partial_{x,v} g|$ means $|\partial_x f| + |\partial_v f| \lesssim |\partial_x g| + |\partial_v g|$. The operator $\mathbf{I}$ denotes the identity operator.
Throughout the paper, we  introduce the velocity weights  
\[
w(v) = e^{\theta |v|^2}, \qquad 
w_{\tilde{\theta}}(v) = e^{\tilde{\theta}|v|^2}, 
\ \text{ with }
0 < \tilde{\theta} \ll \theta < \frac{1}{4}.
\]
Define the boundary layer of the domain for $d > 0$ by
\[
\Omega_d := \{ x \in \Omega : \mathrm{dist}(x, \partial \Omega) < d \}.
\]
The boundary measure is given by
\begin{equation*}
\dd \gamma = |n(x)\cdot v| \dd v \dd S_x,
\end{equation*}
where $d S_x$ denotes the surface measure on $\p\Omega$.
For $1 \leq p < \infty$, we define the boundary norms
\be \notag
|f|_{L^p_{\gamma_\pm}}^p := \int_{\gamma_\pm} |f(x,v)|^p \, \dd \gamma, \qquad
|f|_{L^\infty_{\p\O,v}} := \operatorname*{ess\,sup}_{(x,v)\in \p\O\times \R^3} |f(x,v)|.
\ee
We also introduce the following norms:
\begin{align*}
& \Vert f\Vert_{L^2_\nu} := \Vert \nu^{1/2}f\Vert_{L^2_v}, \qquad  \Vert f\Vert_{L^2_{x,\nu}}:=\Vert \nu^{1/2}f\Vert_{L^2_{x,v}}, \\
& \Vert f\Vert_{L^2_{T,x}} :=  \Big(\int_0^T \Vert f(t)\Vert_{L^2_{x}}^2 \dd t\Big)^{1/2}, \quad  
\Vert f\Vert_{L^2_{T,x,v}} :=  \Big( \int_0^T \Vert f(t)\Vert_{L^2_{x,v}}^2 \dd t \Big)^{1/2}, \\
& \Vert f\Vert_{L^\infty_t L^p_{x,v}} := \sup_{0\leq s\leq t} \Vert f(s)\Vert_{L^p_{x,v}}, \quad
| f|_{L^2_{T,\gamma_\pm}} := \Big( \int_0^T |f|_{L^2_{\gamma_\pm}}^2 \dd t \Big)^{1/2}.
\end{align*}

\section{Preliminary} \label{sec:prelim}

In this section, we collect several preliminary estimates. All estimates hold for both the stationary problem \eqref{eqn:h} and the dynamical problem \eqref{linear_f_dynamical}. 
For simplicity, we use $f$ and $\phi_f$ to represent both settings throughout this section (except in Lemmas \ref{lemma:W3p}, \ref{lemma:phi_x_infinity}, \ref{lemma:nabla2_phi_bdd}, and \ref{lemma:phi_C2}, where the stationary and dynamical settings are treated separately)\footnote{In later sections, we will specify whether we are considering $h (x, v)$ in \eqref{eqn:h} or $f (t,x,v)$ in \eqref{linear_f_dynamical}.}.
The characteristics of \eqref{eqn:h} and \eqref{linear_f_dynamical} are determined by the Hamiltonian ODEs
\Be \label{characteristics}
\begin{cases}
\frac{\dd}{\dd s} X_f(s; t,x,v) = V_f(s; t,x,v), \\[5pt]
\frac{\dd}{\dd s} V_f(s; t,x,v) = - \nabla_x (\phi_f + \phi_E) \big( s, X_f(s; t,x,v) \big),
\end{cases}
\Ee
for $- \infty < s, t < \infty$, with $(X_f(t; t,x,v), V_f(t; t,x,v)) = (x,v)$.
For simplicity, we omit the subscript $f$ and write
\be \notag
X(s;t,x,v) := X_f(s;t,x,v), \qquad
V(s;t,x,v) := V_f(s;t,x,v).
\ee

We now define the backward and forward exit times by
\Be \label{def_tb}
\begin{split}
\tb (t,x,v) &:= \sup \left\{ s \geq 0 : X(t-\tau; t,x,v) \in \Omega \ \text{for all } \tau \in [0,s) \right\}, \\
\tf (t,x,v) &:= \sup \left\{ s \geq 0 : X(t+\tau; t,x,v) \in \Omega \ \text{for all } \tau \in [0,s) \right\}.
\end{split}
\Ee
Moreover, we define the backward exit position and velocity by
\begin{equation*}
\xb (t,x,v) := X (t - \tb; t,x,v), \qquad
\vb (t,x,v) := V (t - \tb; t,x,v).
\end{equation*}
Recall from Section~\ref{sec:basic_setting} that $\{\mathcal{O}_p\}_{p \in \mathcal{P}}$, with $\mathcal{P} = \tilde{\mathcal{P}} \cup \{0\}$, forms an open covering of $\bar{\Omega}$.
For every backward exit position $\xb(x,v) \in \partial\Omega$, we choose $p^1 \in \mathcal{P}$ satisfying $\xb(x,v) \in \mathcal{O}_{p^1}$, and denote
\be \notag
\mathbf{x}_{p^1}^1: = (\mathbf{x}^1_{p^1,1},\mathbf{x}^1_{p^1,2},0) 
\ \text{ such that } \ 
\eta_{p^1}(\mathbf{x}_{p^1}^1) = \xb(x,v).
\ee
In addition, for the stationary problem \eqref{eqn:h}, the corresponding backward and forward exit times, positions, and velocities are all time-independent.

\subsection{Estimate for the characteristic}
\label{sec:estimate_characteristic}

\begin{lemma} \label{lemma:deri_XV}

For some fixed $t>0$, assume that $\Vert \nabla_x^2 (\phi_f + \phi_E) \Vert_{L^\infty_x}t
+ \Vert \nabla_x (\phi_f + \phi_E) \Vert_{L^\infty_{x}} \ll 1$.

\begin{enumerate}
\item For any $\max \{ 0, t - \tb \} \leq s \leq t$,
\begin{align} 
& |\nabla_{v} X(s;t,x,v)| \lesssim |t-s| e^{\Vert \nabla_x^2 (\phi_f + \phi_E) \Vert_{L^\infty_x} \frac{(t-s)^2}{2} }  \lesssim |t-s| e^{o(1)(t-s)},  \label{est:X_v first} \\
& |\nabla_{v} V(s;t,x,v)| \lesssim 1 + |t - s| e^{\Vert \nabla_x^2 (\phi_f + \phi_E) \Vert_{L^\infty_x}\frac{(t-s)^2}{2} } \lesssim (1+|t-s|) e^{o(1)(t-s)}.
\label{est:V_v first}
\end{align}
Further, suppose that $\| \nabla_x^2 (\phi_f + \phi_E) \|_{L^\infty_x} (t)^2 e^{t} \ll 1$. Then we have
\begin{equation} \label{est:X_v second}
|\det (\nabla_v X(s;t,x,v))| \gtrsim |t-s|^3, \qquad  
| (\p_{v} X(s;t,x,v))^{-1} | \lesssim |t-s|^{-1}.
\end{equation}

\item For any $\max \{ 0, t - \tb \} \leq s \leq t$,
\begin{align} 
& |\nabla_{x} X(s;t,x,v)| \lesssim e^{\Vert \nabla_x^2 (\phi_f + \phi_E) \Vert_{L^\infty_x} \frac{(t-s)^2}{2} } \lesssim e^{o(1)(t-s)},  
\label{est:X_x first} \\
& |\nabla_{x} V(s;t,x,v)| \lesssim e^{\Vert \nabla_x^2 (\phi_f + \phi_E) \Vert_{L^\infty_x} \frac{(t-s)^2}{2} } \lesssim e^{o(1)(t-s)}.
\label{est:V_x first}
\end{align}
Moreover, suppose that $\| \nabla_x^2 (\phi_f + \phi_E) \|_{L^\infty_x} (t)^2 e^{t} \ll 1$. Then we have
\begin{equation} \label{est:X_x second}
\big| \p_{x} X(s;t,x,v) \big| \gtrsim 1.
\end{equation}
\end{enumerate}
\end{lemma}

\begin{proof}

\smallskip

First, we prove part $(a)$.
From \eqref{characteristics}, we have 
\Be \notag
\begin{split} 
\Big| \frac{\dd}{\dd s}\nabla_v X(s;t,x,v) \Big| 
& \lesssim |\nabla_v V(s;t,x,v)|,
\\ \Big| \frac{\dd}{\dd s}\nabla_v V(s;t,x,v) \Big| 
& \lesssim 
\Vert \nabla_x^2 (\phi_f + \phi_E) \Vert_{L^\infty_x} |\nabla_v X(s;t,x,v)|.
\end{split}
\Ee
Since $\nabla_v X(t;t,x,v) = \nabla_v x = 0$, $\nabla_v V(t;t,x,v) = \nabla_v v = Id_{3\times 3}$. Then we obtain
\Be \label{eq1:Xv first}
\begin{split} 
|\nabla_v X(s;t,x,v)| 
& \lesssim \int^{t}_{s} |\nabla_v V(s_1;t,x,v)| \dd s_1
\\& \lesssim |t-s| + \int^{t}_{s} \int^{t}_{s_1} \Vert \nabla_x^2 (\phi_f + \phi_E) \Vert_{L^\infty_x} |\nabla_v X(s_2;t,x,v)| \dd s_2 \dd s_1.
\end{split}
\Ee
Using Fubini's theorem, we derive
\Be \label{eq2:Xv first}
\begin{split}
\eqref{eq1:Xv first} 
& \lesssim |t-s| + \int^{t}_{s} \int^{s_2}_{s} \Vert \nabla_x^2 (\phi_f + \phi_E) \Vert_{L^\infty_x} |\nabla_v X(s_2;t,x,v)| \dd s_1 \dd s_2
\\& \lesssim |t-s| + \int^{t}_{s} (s_2-s) \Vert \nabla_x^2 (\phi_f + \phi_E) \Vert_{L^\infty_x} |\nabla_v X(s_2;t,x,v)| \dd s_2.
\end{split}
\Ee
Since $\max \{ 0, t - \tb \} \leq s \leq t $, the Gronwall's inequality implies that
\begin{align}
|\nabla_v X(s;t,x,v)| 
\lesssim |t-s| \exp \big( \Vert \nabla_x^2 (\phi_f + \phi_E) \Vert_{L^\infty_x}
\int^{t}_{s} (s_2-s) \dd s_2 \big)
\leq |t-s| e^{\Vert \nabla_x^2 (\phi_f + \phi_E) \Vert_{L^\infty_x} \frac{(t-s)^2}{2} }. \notag
\end{align}

Similar to \eqref{eq1:Xv first}, we also obtain 
\Be \notag
\begin{split} 
|\nabla_v V(s;t,x,v)| 
& \lesssim 1 + \int^{t}_{s} \Vert \nabla_x^2 (\phi_f + \phi_E) \Vert_{L^\infty_x} |\nabla_v X(s_1;t,x,v)| \dd s_1 < \infty.
\end{split}
\Ee
This, together with \eqref{eq2:Xv first} and the Gronwall's inequality, implies that
\Be \notag
\begin{split} 
|\nabla_v X(s;t,x,v)| + |\nabla_v V(s;t,x,v)|
& \lesssim 1 + |t-s| + \int^{t}_{s} (s_2 - s + 1) \Vert \nabla_x^2 (\phi_f + \phi_E) \Vert_{L^\infty_x} |\nabla_v X(s_2;t,x,v)| \dd s_2
\\& \lesssim |t - s + 1| e^{\Vert \nabla_x^2 (\phi_f + \phi_E) \Vert_{L^\infty_x} \big( \frac{1}{2} (t-s)^2 + (t-s) \big)}.
\end{split}
\Ee
and thus we conclude \eqref{est:X_v first} and \eqref{est:V_v first}.

\smallskip

We now prove \eqref{est:X_v second}. 
The characteristics \eqref{characteristics} implies that
\Be \label{eq1:Xv second}
\begin{split}
\nabla_v X(s;t,x,v)
& = \nabla_v \big( x - \int^{t}_{s} \big( v - \int^{t}_{s_1} - \nabla_x (\phi_f + \phi_E) (X(s_2; t,x,v)) \dd s_2 \big) \dd s_1 \big)
\\& = - \int^{t}_{s} \big( Id_{3\times 3} + \int^{t}_{s_1} \nabla_x^2 (\phi_f + \phi_E) (X(s_2; t,x,v)) \nabla_v X(s_2;t,x,v) \dd s_2 \big) \dd s_1
\\& = - (t-s)Id_{3\times 3} - 
\underbrace{
\int^{t}_{s} \int^{t}_{s_1} \nabla_x^2 (\phi_f + \phi_E) (X(s_2; t,x,v)) \nabla_v X(s_2;t,x,v) \dd s_2 \dd s_1
}_{\eqref{eq1:Xv second}^*}.
\end{split}
\Ee
Using \eqref{est:X_v first} and $0 \leq s \leq t$, we have
\Be \label{eq2:Xv second}
\begin{split}
| \eqref{eq1:Xv second}^* |
& \leq \int^{t}_{s} \int^{t}_{s_1} | \nabla_x^2 (\phi_f + \phi_E) (X(s_2; t,x,v)) | | \nabla_v X(s_2;t,x,v) | \dd s_2 \dd s_1
\\& \leq \| \nabla_x^2 (\phi_f + \phi_E) \|_{L^\infty_x} \int^{t}_{s} \int^{t}_{s_1}  |t - s_2| e^{\Vert \nabla_x^2 (\phi_f + \phi_E) \Vert_{L^\infty_x} \frac{(t-s_2)^2}{2} } \dd s_2 \dd s_1
\\& \lesssim (t-s) e^{\Vert \nabla_x^2 (\phi_f + \phi_E) \Vert_{L^\infty_x} \frac{(t)^2}{2} } \| \nabla_x^2 (\phi_f + \phi_E) \|_{L^\infty_x} (t)^2.
\end{split}
\Ee
Further, assume that $\| \nabla_x^2 (\phi_f + \phi_E) \|_{L^\infty_x} (t)^2 e^{t} \ll 1$.  
From \eqref{eq1:Xv second} and \eqref{eq2:Xv second}, we derive that
\be \notag
\nabla_v X(s;t,x,v) = -(t-s)\Big(Id_{3\times 3} + o(1) M \Big),
\ee
where $M \in \R^{3\times 3}$ is a matrix satisfying $|M|\lesssim 1$. This concludes \eqref{est:X_v second}.

\smallskip

Second, we prove part $(b)$.
From \eqref{characteristics}, we have 
\Be \notag
\begin{split} 
\Big| \frac{\dd}{\dd s} \nabla_x X(s;t,x,v) \Big| 
& \lesssim |\nabla_x V(s;t,x,v)|,
\\ \Big| \frac{\dd}{\dd s} \nabla_x V(s;t,x,v) \Big| 
& \lesssim 
\Vert \nabla_x^2 \phi_f\Vert_{L^\infty_x} |\nabla_x X(s;t,x,v)|.
\end{split}
\Ee
Since $\nabla_x X(t;t,x,v) = \nabla_x x = 1$, $\nabla_x V(t;t,x,v) = \nabla_x v = 0$. Then we obtain
\Be \label{eq1:Xx first}
\begin{split} 
|\nabla_x X(s;t,x,v)| 
& \lesssim 1 + \int^{t}_{s} |\nabla_x V(s_1;t,x,v)| \dd s_1
\\& \lesssim 1 + \int^{t}_{s} \int^{t}_{s_1} \Vert \nabla_x^2 (\phi_f + \phi_E) \Vert_{L^\infty_x} |\nabla_x X(s_2;t,x,v)| \dd s_2 \dd s_1.
\end{split}
\Ee
Using Fubini's theorem, we derive
\Be \label{eq2:Xx first}
\begin{split}
\eqref{eq1:Xx first} 
& \lesssim 1 + \int^{t}_{s} \int^{s_2}_{s} \Vert \nabla_x^2 (\phi_f + \phi_E) \Vert_{L^\infty_x} |\nabla_x X(s_2;t,x,v)| \dd s_1 \dd s_2
\\& \lesssim 1 + \int^{t}_{s} (s_2-s) \Vert \nabla_x^2 (\phi_f + \phi_E) \Vert_{L^\infty_x} |\nabla_x X(s_2;t,x,v)| \dd s_2.
\end{split}
\Ee
Since $\max \{ 0, t - \tb \} \leq s \leq t \leq T_0$, the Gronwall's inequality implies that
\Be \notag
\begin{split}
|\nabla_v X(s;t,x,v)| 
& \lesssim \exp \big( \Vert \nabla_x^2 (\phi_f + \phi_E) \Vert_{L^\infty_x}
\int^{t}_{s} (s_2-s) \dd s_2 \big)
\leq e^{\Vert \nabla_x^2 (\phi_f + \phi_E) \Vert_{L^\infty_x} \frac{(t-s)^2}{2} },
\\ |\nabla_v V(s;t,x,v)| 
& \lesssim \int^{t}_{s} \Vert \nabla_x^2 (\phi_f + \phi_E) \Vert_{L^\infty_x} |\nabla_x X(s_1;t,x,v)| \dd s_1 < \infty.
\end{split}
\Ee
This, together with \eqref{eq2:Xx first} and the Gronwall's inequality, implies that
\Be \notag
\begin{split} 
|\nabla_v X(s;t,x,v)| + |\nabla_v V(s;t,x,v)|
& \lesssim 1 + \int^{t}_{s} (s_2 - s + 1) \Vert \nabla_x^2 (\phi_f + \phi_E) \Vert_{L^\infty_x} |\nabla_v X(s_2;t,x,v)| \dd s_2
\\& \lesssim e^{\Vert \nabla_x^2 (\phi_f + \phi_E) \Vert_{L^\infty_x} \big( \frac{1}{2} (t-s)^2 + (t-s) \big)}.
\end{split}
\Ee
and thus we conclude \eqref{est:X_x first} and \eqref{est:V_x first}.

\smallskip

We now prove \eqref{est:X_x second}. 
The characteristics \eqref{characteristics} implies that
\Be \label{eq1:Xx second}
\begin{split}
\nabla_x X(s;t,x,v)
& = \nabla_x \big( x - \int^{t}_{s} \big( v - \int^{t}_{s_1} - \nabla_x (\phi_f + \phi_E) (X(s_2; t,x,v)) \dd s_2 \big) \dd s_1 \big)
\\& = 1 - \int^{t}_{s} \int^{t}_{s_1} \nabla_x^2 (\phi_f + \phi_E) (X(s_2; t,x,v)) \nabla_x X(s_2;t,x,v) \dd s_2 \dd s_1
\\& = 1 - 
\underbrace{
\int^{t}_{s} \int^{t}_{s_1} \nabla_x^2 (\phi_f + \phi_E) (X(s_2; t,x,v)) \nabla_x X(s_2;t,x,v) \dd s_2 \dd s_1
}_{\eqref{eq1:Xx second}^*}.
\end{split}
\Ee
Using \eqref{est:X_x first} and $0 \leq s \leq t$, we have
\Be \label{eq2:Xx second}
\begin{split}
| \eqref{eq1:Xx second}^* |
& \leq \int^{t}_{s} \int^{t}_{s_1} | \nabla_x^2 (\phi_f + \phi_E) (X(s_2; t,x,v)) | | \nabla_x X(s_2;t,x,v) | \dd s_2 \dd s_1
\\& \lesssim e^{\Vert \nabla_x^2 (\phi_f + \phi_E) \Vert_{L^\infty_x} \frac{(t)^2}{2} } \| \nabla_x^2 (\phi_f + \phi_E) \|_{L^\infty_x} (t)^2.
\end{split}
\Ee
Further, assume that $\| \nabla_x^2 (\phi_f + \phi_E) \|_{L^\infty_x} (t)^2 e^{t} \ll 1$.  
From \eqref{eq1:Xx second} and \eqref{eq2:Xx second}, we derive that
\Be \notag
\big| \nabla_x X(s;t,x,v) \big| \gtrsim 1 - e^{\Vert \nabla_x^2 (\phi_f + \phi_E) \Vert_{L^\infty_x} \frac{(t)^2}{2} } \| \nabla_x^2 (\phi_f + \phi_E) \|_{L^\infty_x} (t)^2 \gtrsim 1.
\Ee
\end{proof}

\begin{lemma} \label{lemma:deri_backward}

Fix some $t > 0$. Assume that $\tb (t,x,v) < t$. Then we have
\begin{align}
\frac{\partial \tb(t,x,v)}{\partial v_j} = & - \frac{\tb(t,x,v)}{ \vb \cdot n(\xb)} \big[ e_j + O(\Vert \phi_f + \phi_E \Vert_{C^2}) |\tb|^2 e^{\Vert \nabla_x^2 (\phi_f + \phi_E) \Vert_{L^\infty}\tb^2} \big] \cdot n(\xb), 
\label{eq:tb_v} \\
\frac{\partial \tb(t,x,v)}{\partial x_j} = & \frac{1}{ \vb \cdot n(\xb)} \big[ e_j + O(\Vert \phi_f + \phi_E \Vert_{C^2}) |\tb|^2 e^{\Vert \nabla_x^2 (\phi_f + \phi_E) \Vert_{L^\infty_x}\tb^2} \big] \cdot n(\xb),
\label{eq:tb_x} \\
\frac{\partial \vb (t,x,v)}{\partial v_j} = & e_j - \frac{\tb(t,x,v)}{ \vb \cdot n(\xb)} \big[ e_j + O(\Vert \phi_f + \phi_E \Vert_{C^2}) |\tb|^2 e^{\Vert \nabla_x^2 (\phi_f + \phi_E) \Vert_{L^\infty_x}\tb^2} \big] \cdot n(\xb) \Vert \nabla_x (\phi_f + \phi_E) \Vert_{L^\infty_x}
\notag \\
& - |\tb|^2 \Vert \nabla_x^2 (\phi_f + \phi_E) \Vert_{L^\infty_x} e^{\Vert \nabla_x^2 \phi_f \Vert_{L^\infty_x} \frac{|t - \tau |^2}{2}}, 
\label{eq:vb_v} \\
\frac{\partial \vb (t,x,v)}{\partial x_j} = & \frac{1}{ \vb \cdot n(\xb)} \big[ e_j + O(\Vert \phi_f + \phi_E \Vert_{C^2}) |\tb|^2 e^{\Vert \nabla_x^2 (\phi_f + \phi_E) \Vert_{L^\infty_x}\tb^2} \big] \cdot n(\xb) \Vert \nabla_x (\phi_f + \phi_E) \Vert_{L^\infty_x}
\notag \\
& - |\tb| \Vert \nabla_x^2 (\phi_f + \phi_E) \Vert_{L^\infty_x} e^{\Vert \nabla_x^2 (\phi_f + \phi_E) \Vert_{L^\infty_x} \frac{|t - \tau |^2}{2}}, 
\label{eq:vb_x} \\
\frac{\partial \xb (t,x,v)}{\partial v_j} = & e_j \tb - \frac{\tb(t,x,v)}{ \vb \cdot n(\xb)}\big[ e_j + O(\Vert \phi_f + \phi_E \Vert_{C^2}) |\tb|^2 e^{\Vert \nabla_x^2 (\phi_f + \phi_E) \Vert_{L^\infty_x}\tb^2} \big] \cdot n(\xb) \vb
\notag \\
& + |\tb|^3 \Vert \nabla_x^2 (\phi_f + \phi_E) \Vert_{L^\infty_x} e^{\Vert \nabla_x^2 (\phi_f + \phi_E) \Vert_{L^\infty_x} \frac{|t - \tau |^2}{2}}, 
\label{eq:xb_v} \\
\frac{\partial \xb (t,x,v)}{\partial x_j} = & e_j + \frac{1}{ \vb \cdot n(\xb)} \big[ e_j + O(\Vert \phi_f + \phi_E \Vert_{C^2}) |\tb|^2 e^{\Vert \nabla_x^2 (\phi_f + \phi_E) \Vert_{L^\infty_x}\tb^2} \big] \cdot n(\xb) \vb
\notag \\
& - |\tb|^2 \Vert \nabla_x^2 (\phi_f + \phi_E) \Vert_{L^\infty_x} e^{\Vert \nabla_x^2 (\phi_f + \phi_E) \Vert_{L^\infty_x} \frac{|t - \tau |^2}{2}}.
\label{eq:xb_x}
\end{align}
For $i=1,2$, $j=1,2,3$, 
\begin{align}
&\frac{\partial \mathbf{x}_{p^{1},i}^{1}}{\partial x_j}= \frac{\p \xb(t,x,v)}{\p x_j} \cdot \frac{\p_i \eta_{p^1}(\mathbf{x}_{p^1}^1)}{|\p_i \eta_{p^1}(\mathbf{x}_{p^1}^1)|^2}, 
\label{xi deri xbp} \\
& \frac{\partial \mathbf{x}^{1}_{p^{1},i}}{\partial v_j}= \frac{\p \xb(t,x,v)}{v_j} \cdot \frac{\p_i \eta_{p^1}(\mathbf{x}_{p^1}^1)}{|\p_i \eta_{p^1}(\mathbf{x}_{p^1}^1)|^2}.  
\label{vi deri xbp}
\end{align}
\end{lemma}

\begin{proof}

Denote that $t^1 = t - \tb(t,x,v)$.
From the assumption $\tb(t,x,v)<t$, then
$\nabla_{x, v} \tb$, $\nabla_{x, v} \vb$, and $\nabla_{x, v} \xb$ are well-defined.
Following the characteristics \eqref{characteristics}, we have
\Be \label{eq1:deri_backward}
\begin{split}
\lim\limits_{s \to t^1} V (s;t,x,v)
& = v + \int^{t}_{t^1} \nabla_{x} (\phi_f + \phi_E) (X(s; t,x,v)) \dd s,
\\ X (\tau;t,x,v)
& = x + v (\tau - t) - \int^{t}_{\tau} \int^{t}_{s} \nabla_x (\phi_f + \phi_E) (X(s_1; t,x,v)) \dd s_1 \dd s.
\end{split}
\Ee
This implies that
\Be \notag
\begin{split}
\lim\limits_{s \to t^1} \frac{\partial V (s;t,x,v)}{\partial v}
& = 1 + \frac{\partial \tb(x,v)}{\partial v} \nabla_{x} (\phi_f + \phi_E) (X(t^1; t,x,v))
+ \int^{t}_{t^1} \nabla^2_{x} (\phi_f + \phi_E) \frac{\partial X(s; t,x,v)}{\partial v} \dd s,
\\ \sup\limits_{\tau \leq s' \leq t} \Big| \frac{\partial X (s';t,x,v)}{\partial v} \Big|
& \leq | \tau - t | + \int^{t}_{\tau} |t-s| \Vert \nabla_x^2 (\phi_f + \phi_E) \Vert_{L^\infty_x} \sup\limits_{\tau \leq s' \leq t} \Big| \frac{\partial X (s';t,x,v)}{\partial v} \Big| \dd s.
\end{split}
\Ee
Using Gronwall's inequality, we derive that
\Be \label{eq3:deri_backward}
\sup\limits_{\tau \leq s' \leq t} \Big| \frac{\partial X (s';t,x,v)}{\partial v} \Big|
\leq | \tau - t | e^{\Vert \nabla_x^2 (\phi_f + \phi_E) \Vert_{L^\infty_x} \frac{|t - \tau |^2}{2}}.
\Ee
Similarly, \eqref{eq1:deri_backward} also implies that
\Be \notag
\begin{split}
\lim\limits_{s \to t^1} \frac{\partial V (s;t,x,v)}{\partial x}
& = \frac{\partial \tb(t,x,v)}{\partial x} \nabla_{x} (\phi_f + \phi_E) (X(t^1; t,x,v))
+ \int^{t}_{t^1} \nabla_x^2 (\phi_f + \phi_E) \frac{\partial X(s; t,x,v)}{\partial x} \dd s,
\\ \sup\limits_{\tau \leq s' \leq t} \Big| \frac{\partial X (s';t,x,v)}{\partial x} \Big|
& \leq 1 + \int^{t}_{\tau} |t-s| \Vert \nabla_x^2 (\phi_f + \phi_E) \Vert_{L^\infty_x} \sup\limits_{\tau \leq s' \leq t} \Big| \frac{\partial X (s';t,x,v)}{\partial x} \Big| \dd s.
\end{split}
\Ee
Using Gronwall's inequality, we derive that
\Be \label{eq5:deri_backward}
\sup\limits_{\tau \leq s' \leq t} \Big| \frac{\partial X (s';t,x,v)}{\partial x} \Big|
\leq e^{\Vert \nabla_x^2 (\phi_f + \phi_E) \Vert_{L^\infty_x} \frac{|t - \tau |^2}{2}}.
\Ee
Moreover, let $\tau = t^1 \leq t$ in the second equality in \eqref{eq1:deri_backward}, we obtain
\Be \label{eq4:deri_backward}
\begin{split}
X (t^1;t,x,v)
= x + v (t^1 - t) - \int^{t}_{t^1} \int^{t}_{s} \nabla_x (\phi_f + \phi_E) (X(s_1; t,x,v)) \dd s_1 \dd s.
\end{split}
\Ee

First, we show \eqref{eq:tb_v}.
Taking $\nabla_{v_j}$ on the above equality \eqref{eq4:deri_backward}, we get
\Be \label{eq1:tb_v}
\begin{split}
\nabla_{v_j} [X (t^1;t,x,v)]
= & - e_j \tb - \frac{\partial \tb}{\partial v_j} v
- \frac{\partial \tb}{\partial v_j} \int^{t}_{t^1} \nabla_x (\phi_f + \phi_E) (X(s_1; t,x,v)) \dd s_1
\\& - \int^{t}_{t^1} \int^{t}_{s} \nabla^2_{x} (\phi_f + \phi_E) \frac{\partial X (s_1;t,x,v)}{\partial v_j} \dd s_1 \dd s.
\end{split}
\Ee
Since $\nabla_{x, v} [X (t^1;t,x,v)] \cdot n(\xb) = 0$, this further implies that
\begin{equation} \notag
\begin{split}
\frac{\partial \tb(t,x,v)}{\partial v_j} ( \vb \cdot n(\xb) )
= n(\xb) \cdot \Big( - e_j \tb - \int^{t}_{t^1} \int^{t}_{s} \nabla^2_{x} (\phi_f + \phi_E) \frac{\partial X (s_1;t,x,v)}{\partial v_j} \dd s_1 \dd s \Big).
\end{split}
\end{equation}
Combining this with \eqref{eq3:deri_backward}, we conclude \eqref{eq:tb_v}.

\smallskip

Second, we show \eqref{eq:tb_x}.
Taking $\nabla_{x_j}$ in the equality \eqref{eq4:deri_backward}, we get
\Be \label{eq1:tb_x}
\begin{split}
\nabla_{x_j} [X (t^1;t,x,v)]
= & e_j - \frac{\partial \tb}{\partial x_j} v
- \frac{\partial \tb}{\partial x_j} \int^{t}_{t^1} \nabla_x (\phi_f + \phi_E) (X(s_1; t,x,v)) \dd s_1
\\& - \int^{t}_{t^1} \int^{t}_{s} \nabla^2_{x} (\phi_f + \phi_E) \frac{\partial X (s_1;t,x,v)}{\partial x_j} \dd s_1 \dd s.
\end{split}
\Ee
This further implies that
\begin{equation} \notag
\begin{split}
\frac{\partial \tb(t,x,v)}{\partial x_j} ( \vb \cdot n(\xb) )
= n(\xb) \cdot \Big( e_j - \int^{t}_{t^1} \int^{t}_{s} \nabla^2_{x} (\phi_f + \phi_E) \frac{\partial X (s_1;t,x,v)}{\partial x_j} \dd s_1 \dd s \Big).
\end{split}
\end{equation}
Together with \eqref{eq5:deri_backward}, this establishes \eqref{eq:tb_x}.

\smallskip

Third, we show \eqref{eq:vb_v}, \eqref{eq:vb_x}, \eqref{eq:xb_v} and \eqref{eq:xb_x}.
From \eqref{eq1:deri_backward}, we obtain
\Be \label{eq1:vb_v}
\vb (t,x,v)
= v + \int^{t}_{t^1} \nabla_{x} (\phi_f + \phi_E) (X(s; t,x,v)) \dd s.
\Ee
Taking $\partial_{v_j}$ and $\partial_{x_j}$ in the equality \eqref{eq1:vb_v}, we get
\Be \notag
\begin{split}
\frac{\partial \vb (t,x,v)}{\partial v_j}
& = e_j + \frac{\partial \tb}{\partial v_j} \nabla_{x} (\phi_f + \phi_E) (X(t^1; t,x,v))
- \int^{t}_{t^1} \nabla^2_{x} (\phi_f + \phi_E) \frac{\partial X (s;t,x,v)}{\partial v_j} \dd s,
\\ \frac{\partial \vb (t,x,v)}{\partial x_j}
& = \frac{\partial \tb}{\partial x_j} \nabla_x (\phi_f + \phi_E) (X(t^1; t,x,v))
- \int^{t}_{t^1} \nabla^2_{x} (\phi_f + \phi_E) \frac{\partial X (s;t,x,v)}{\partial x_j} \dd s.
\end{split}
\Ee
Combining \eqref{eq3:deri_backward} with \eqref{eq:tb_v}, and \eqref{eq5:deri_backward} with \eqref{eq:tb_x}, we obtain \eqref{eq:vb_v} and \eqref{eq:vb_x}.
From \eqref{eq1:deri_backward}, we obtain
\be \label{eq1:xb_v}
\xb (t,x,v)
= x + v \tb - \int^{t}_{t_1} \int^{t}_{s} \nabla_x (\phi_f + \phi_E) (X(s_1; t,x,v)) \dd s_1 \dd s.
\ee
Analogously, \eqref{eq:xb_v} and \eqref{eq:xb_x} are obtained by taking $\partial_{v_j}$ and $\partial_{x_j}$ in the equality \eqref{eq1:xb_v}, respectively.

\smallskip

Finally, \eqref{xi deri xbp} and \eqref{vi deri xbp} follow from
\begin{align*}
    &   \p_{x_j} \xb(t,x,v) = \p_{x_j} \eta_{p^1}(\mathbf{x}_{p^1}^1) = \p_{x_j} \mathbf{x}_{p^1,1}^1 \p_1 \eta_{p^1}(\mathbf{x}_{p^1}^1) + \p_{x_j} \mathbf{x}_{p^1,2}^1 \p_2 \eta_{p^1}(\mathbf{x}_{p^1}^1).
\end{align*}
Taking the dot product with $\p_i \eta_{p^1}(\mathbf{x}_{p^1}^1)$, we derive
\begin{align*}
    &   \p_{x_j} \xb(t,x,v) \cdot \p_i\eta_{p^1}(\mathbf{x}_{p^1}^1) = \p_{x_j} \mathbf{x}_{p^1,i}^1 |\p_i \eta_{p^1}(\mathbf{x}_{p^1}^1)|^2,
\end{align*}
and this concludes \eqref{xi deri xbp}. 
The proof of \eqref{vi deri xbp} follows analogously by replacing $\p_{x_j}$ with $\p_{v_j}$.
\end{proof}

\begin{lemma} \label{lemma:int_cov}

Assume that $X(s;t,x,v) \in \O$. Then the change of variable
\be \notag
(x,v) \mapsto (X(s;t,x,v), V(s;t,x,v))
\ee
has a Jacobian determinant
\begin{align*}
    & \det \Big( \frac{\p (X(s;t,x,v),V(s;t,x,v))}{\p(x,v)} \Big) = 1.
\end{align*}
Therefore, we conclude the following change of variable formula:
\be \notag
\int_{\mathbb{R}^3} \int_{\O} \mathbf{1}_{X(s;t,x,v)\in \O} |f(X(s;t,x,v),V(s;t,x,v))| \dd x \dd v 
\leq \int_{\mathbb{R}^3} \int_{\O} |f(x,v)| \dd x \dd v.
\ee
\end{lemma}

\begin{proof}

Since $\nabla_v ( \nabla_x (\phi_f + \phi_E) )  = 0$, the Liouville's formula implies that
\[
\det \Big( \frac{\p (X(s;t,x,v),V(s;t,x,v))}{\p(x,v)} \Big) = 1.
\]
\end{proof}



\begin{lemma} \label{lemma:est_tf}

Suppose the external field $\nabla_x \phi_E$ satisfies the assumptions in Theorem~\ref{thm:steady_wellpose}.
Further, assume that $\Vert \nabla_x^2 \phi_f \Vert_{L^\infty_x}
+ \Vert \nabla_x \phi_f \Vert_{L^\infty_{x}} < \frac{1}{20} C_E$.
Then there exists a constant $0< \delta \ll 1$ such that the following holds.

\begin{enumerate}
\item For $(x,v)\in \gamma_+,$ if $\tb(t,x,v) < \infty$ and $X(s;t,x,v)\in \O_\delta$ for all $t- \tb (t,x,v)\leq s\leq t$, then
\be \label{est_tb}
\tb(t,x,v) \lesssim |n(x)\cdot v|. 
\ee

\item For $(x,v) \in \gamma_-$, if $\tf(t,x,v)<\infty$ and $X(s;t,x,v)\in \O_\delta$ for all $t \leq s \leq t+\tf$, then
\be \label{est_tf}
\tf(t,x,v) \lesssim |n(x)\cdot v|.  
\ee
\end{enumerate}
Moreover, if $X (s;t,x,v) \in \O_\delta$ for $t_1\leq s\leq t_2$, then
\be \label{t2_t1_bdd}
|t_2-t_1| \lesssim_{\delta, C_E} 1.    
\ee
\end{lemma}

\begin{proof}

For simplicity, we write $(X(s), V(s)) := (X(s;t,x,v), V(s;t,x,v))$ throughout the proof.

First, we prove \eqref{est_tb}.
From \eqref{xi_def} and \eqref{xi_dist} in Section \ref{sec:basic_setting}, there exists a $C^2$ function $\xi(x)$ and a constant $0 < \delta \ll 1$ such that
\be \label{eq3:est_tf}
\xi(x) = -dist(x,\p\O)
\ \text{ when } \
dist(x,\p\O) < \delta
\ \text{ and } \
n(x) = \frac{\nabla_x \xi(x)}{|\nabla_x \xi(x)|}, 
\quad x\in \p\O.
\ee
Further, \eqref{xi_convex} shows that for any $y \in \mathbb{R}^3$,
\be \label{eq2:est_tf}
y^{\intercal} \cdot \nabla_x^2 \xi(x) \cdot y \gtrsim |y|^2.
\ee

Consider any $(x,v) \in \gamma_+$. For any $t-\tb(x,v) < s < t $, we compute that
\be \label{eq1:est_tf}
\begin{split}
\frac{\dd }{\dd s}(-\xi(X(s))) 
& = -V(s)\cdot \nabla_x \xi(X(s)),
\\ \frac{\dd ^2}{\dd s^2} (-\xi(X(s))) 
& = \nabla_x \xi(X(s)) \cdot \nabla_x (\phi_f + \phi_E) - V(s)^{\intercal} \cdot \nabla_x^2 \xi(X(s))\cdot V(s).
\end{split}
\ee
Since $\frac{\nabla_x\xi(x)}{|\nabla_x\xi(x)| }=n(x)$, the assumptions on $\nabla_x \phi_E$ and $\nabla_x \phi_f$ and the continuity of $\nabla_x \xi$ imply that for every $X(s)\in \O_\delta$,
\begin{equation} \notag
\begin{split}
\nabla_x \xi(X(s)) \cdot \nabla_x  \phi_E(X(s))
& \leq -\frac{C_E}{2} |\nabla_x \xi(x)|,
\\ \nabla_x \xi(X(s))\cdot \nabla_x (\phi_f+\phi_E)(X(s))
& \leq -\frac{C_E}{4} |\nabla_x \xi(x)|.
\end{split}
\end{equation}
This, together with \eqref{eq2:est_tf} and \eqref{eq1:est_tf}, shows that for every $X(s)\in \O_\delta$,
\be \label{eq4:est_tf}
\frac{\dd ^2}{\dd s^2} (-\xi(X(s)))\lesssim - C_E |\nabla_x \xi(x)|.
\ee
From \eqref{eq3:est_tf} and \eqref{eq1:est_tf}, we obtain that $\xi(X(t))=0$ and $\frac{\dd}{\dd s} \xi(X(t))=-|\nabla_x\xi(x)|n(x)\cdot v < 0$. With $\xi(X(t-\tb))=0$, this implies that
\be \notag
\tb (t,x,v) \lesssim \frac{1}{C_E} |n(x)\cdot v|,
\ee
and thus we conclude \eqref{est_tb}.
The proof of \eqref{est_tf} follows by the same argument.



\smallskip

Second, we prove \eqref{t2_t1_bdd}. We may assume that $X(t_1), X(t_2) \in \p\O_\delta$. Otherwise, \eqref{eq4:est_tf} implies that there exist $t_3 \leq t_1 \leq t_2 \leq t_4$ such that $X(t_3), X(t_4) \in \p\O_\delta$ and $X (s) \in \O_\delta$ for all $t_3 \leq s \leq t_4$. It then suffices to show that $|t_4 - t_3| \lesssim_{\delta, C_E} 1$, which directly implies the upper bound on $|t_2 - t_1|$.
We now consider three cases: (1) $X(t_1), X(t_2) \in \p\O$, (2) one of $X(t_1), X(t_2)$ lies in $\p\O$ and the other lies in $\p \O_\delta \setminus \p\O$, (3) $X(t_1), X(t_2) \in \p \O_\delta \setminus \p\O$.

\smallskip

\textit{Case 1: $X(t_1), X(t_2)\in \p\O$.}
From \eqref{xi_dist}, we have $\xi(X(t_1)) = \xi(X(t_2)) = 0$ and $-\delta \leq \xi(X(s)) \leq 0$ for all $t_1 \leq s \leq t_2$.
Then there exists $t^* \in [t_1,t_2]$ such that $\frac{\dd}{\dd s} \xi(X(t^*))=0$.
Combining with \eqref{eq4:est_tf}, we obtain 
\be \notag
\begin{split}
\delta \geq |\xi(X(t_1)) - \xi(X(t_*))| 
= \Big| \int_{t_1}^{t_*} \int_s^{t_*} \frac{\dd ^2}{\dd s^2}\xi(X(\tau))  \dd \tau \dd s \Big| 
\gtrsim \int_{t_1}^{t_*} \int_s^{t_*} C_E \dd \tau \dd s \gtrsim (t_*-t_1)^2 C_E.
\end{split}
\ee
This implies $|t^* - t_1| \lesssim \sqrt{\frac{\delta}{C_E}}$. Similarly, we obtain $|t_2 - t^*| \lesssim \sqrt{\frac{\delta}{C_E}}$. Consequently, we derive that
\[
|t_2 - t_1|  \leq |t^* - t_1| + |t_2 - t^*| \lesssim \sqrt{\frac{\delta}{C_E}}.
\]

\smallskip

\textit{Case 2: $X(t_1) \in \p\O$, $X(t_2) \in \p \O_\delta \setminus \p\O$ or $X(t_1) \in \p \O_\delta \setminus \p\O$, $X(t_2) \in \p\O$.}
It suffices to consider the case when $X(t_1) \in \p\O$ and $X(t_2) \in \p \O_\delta \setminus \p\O$, since the other case follows similarly.
Since $\xi(X(t_1))$ and $\xi(X(t_2)) = - \delta$, we compute that
\be \label{eq5:est_tf}
\delta = \xi(X(t_1))-\xi(X(t_2)) 
= \int^{t_1}_{t_2} \Big( \frac{\dd }{\dd s}\xi(X(s)) - \frac{\dd}{\dd s}\xi(X(t_2)) \Big) \dd s + \frac{\dd}{\dd s}\xi(X(t_2))(t_1-t_2).
\ee
Since $X(s)\in \O_\delta$ for all $t_1 \leq s \leq t_2$, it follows that $\frac{\dd }{\dd s}\xi(X(t_2)) \leq 0$. This, together with \eqref{eq5:est_tf}, implies
\be \notag
\delta \geq \int^{t_1}_{t_2} \Big( \frac{\dd }{\dd s}\xi(X(s)) - \frac{\dd}{\dd s}\xi(X(t_2)) \Big) \dd s = \int^{t_1}_{t_2} \int^s_{t_2} \frac{\dd^2}{\dd s^2} \xi(X(\tau)) \dd \tau \dd s \gtrsim \int^{t_1}_{t_2} \int^s_{t_2} C_E \dd \tau \dd s \gtrsim (t_2-t_1)^2 C_E.
\ee
Thus, we conclude that $|t_2-t_1| \lesssim \sqrt{\frac{\delta}{C_E}}$.

\smallskip

\textit{Case 3: $X(t_1), X(t_2) \in \p \O_\delta \setminus \p\O$.}
Since $X(s)\in \O_\delta$ for $t_1 \leq s \leq t_2$, it follows that $\frac{\dd }{\dd s}\xi(X(t_2)) \leq 0$.
Recall from \eqref{eq4:est_tf} that $\frac{\dd^2}{\dd s^2} (-\xi(X(s)))\lesssim -C_E$. This implies that $\frac{\dd }{\dd s}\xi(X(s)) \leq 0$ for all $t_1 \leq s \leq t_2$. Hence, we conclude that $X(t_1) \notin \p \O_\delta \setminus \p\O$.
Therefore, this case does not occur.
\end{proof}

The following corollary is a direct consequence of Lemma \ref{lemma:deri_backward} and Lemma \ref{lemma:est_tf}.

\begin{corollary} \label{cor:est_tf}

Assume the assumptions of Lemma \ref{lemma:est_tf} hold. There exists a constant $0 < \delta \ll 1$ such that if $X(s;t,x,v) \in \O_\delta$ for all $t - \tb (t, x,v) \leq s \leq t$, then
\be \notag
\begin{split}
\tb (t,x,v) 
& \lesssim_{\delta, C_E} \min\{ |n (\xb (t,x,v)) \cdot \vb (t,x,v) |, \ 1 \}.
\end{split}
\ee
In addition, we have
\be \notag
\begin{split}
|\nabla_v \tb (t,x,v)| \lesssim 1, 
\qquad |\nabla_v \xb (t,x,v)| \lesssim 1.
\end{split}
\ee
\end{corollary}

\begin{lemma} \label{lemma:v_variation}

Fix some $t > 0$. Assume that $\Vert \nabla_x (\phi_f + \phi_E) \Vert_{L^\infty_{x}} t \ll 1$.
For any $\max \{ 0, t - \tb \} \leq s \leq t$ and $0 < \theta \leq 1$,
\be \label{w_quotient}
\frac{e^{\theta |v|^2}}{e^{\theta |V (s;t,x,v)|^2}} 
\lesssim e^{o(1)(t-s)|v|} 
\lesssim e^{\frac{1}{2}\theta |v|^2}.
\ee
Similarly, for any $\max \{ 0, t - \tb \} \leq s \leq t$,
\be \label{nu_quotient}
\frac{ \nu (v)}{ \nu (V (s;t,x,v))} 
\lesssim 1.
\ee
Consequently, the following estimate holds:
\be \notag
\begin{split}
& e^{-\int^t_s \frac{\nu(V(\tau))}{2}\dd \tau} \frac{w_{\tilde{\theta}}(v)}{w_{\tilde{\theta}}(V (s;t,x,v))} \lesssim e^{-\int^t_s \frac{\nu(V(\tau))}{4} \dd \tau}.
\end{split}
\ee
\end{lemma}

\begin{proof}

From the characteristic, for any $\max \{ 0, t - \tb \} \leq s \leq t$,
\be \label{eq1:v_variation}
\big| |V (s;t,x,v)| - |v| \big| \leq \Vert \nabla_x (\phi_f + \phi_E) \Vert_{L^\infty_x} t.
\ee
This implies that
\be \notag
\begin{split}
\frac{e^{\theta |v|^2}}{e^{\theta |V (s;t,x,v)|^2}}
& \lesssim \exp \big\{ \theta \big( 2 \Vert \nabla_x (\phi_f + \phi_E) \Vert_{L^\infty_{x}} |v|(t-s) + \Vert \nabla_x (\phi_f + \phi_E) \Vert_{L^\infty_{x}}^2 (t-s)^2 \big) \big\}.
\end{split}
\ee
Using the assumption that $\Vert \nabla_x (\phi_f + \phi_E) \Vert_{L^\infty_x }t\ll 1$, we conclude \eqref{w_quotient}.
Similarly, for any $\max \{ 0, t - \tb \} \leq s \leq t$,
\be \notag
\frac{ \nu (v)}{ \nu (V(s;t,x,v))} 
\lesssim \frac{ \nu (v)}{ \nu (|v| - \Vert \nabla_x (\phi_f + \phi_E) \Vert_{L^\infty_x} t)} \lesssim 1,
\ee
and thus we conclude \eqref{nu_quotient}.
\end{proof}

\begin{lemma} \label{lemma:X_vv_Lp}

Assume that $\Vert \nabla_x^2 (\phi_f + \phi_E) \Vert_{L^\infty_x}
+ \Vert \nabla_x (\phi_f + \phi_E) \Vert_{L^\infty_{x}} \ll 1$.
Fix some $t > 0$.
Further, assume that  $\Vert \nabla_x (\phi_f + \phi_E) \Vert_{L^\infty_{x}} t\ll 1$.
For any $0 < \theta \leq 1$ and $\max \{ 0, t - \tb \} \leq s \leq t$,
\begin{align}
\big\| e^{- \theta |v|^2} \nabla^{2}_{v} X(s;t,x,v) \big\|_{L^p_{x, v}} 
& \lesssim |t-s|^2 e^{ \frac{3}{2} \Vert \nabla_x^2 (\phi_f + \phi_E) \Vert_{L^\infty_x} (t-s)^2 } \int^{t}_{s} \int^{t}_{s_1} 
\big\| \nabla_x^3 (\phi_f + \phi_E) \big\|_{L^p_{x}} \dd s_2 \dd s_1,
\label{est:X_vv first} \\
\big\| e^{- \theta |v|^2} \nabla_{x} \nabla_{v} X(s;t,x,v) \big\|_{L^p_{x, v}} 
& \lesssim |t-s| e^{ \frac{3}{2} \Vert \nabla_x^2 (\phi_f + \phi_E) \Vert_{L^\infty_x} (t-s)^2 } \int^{t}_{s} \int^{t}_{s_1} 
\big\| \nabla_x^3 (\phi_f + \phi_E) \big\|_{L^p_{x}} \dd s_2 \dd s_1.
\label{est:X_xv first}
\end{align}
\end{lemma}

\begin{proof}

First, we prove \eqref{est:X_vv first}.
Following the proof of Lemma \ref{lemma:deri_XV}, the characteristics of \eqref{characteristics} implies that
\Be \notag
\begin{split} 
\frac{\dd}{\dd s} | \nabla^2_v X(s;t,x,v)| 
& \lesssim |\nabla^2_v V(s;t,x,v)|,
\\ \frac{\dd}{\dd s} | \nabla^2_v V(s;t,x,v)| 
& \lesssim 
| \nabla_x^3 (\phi_f + \phi_E) (X(s;t,x,v))| |\nabla_v X(s;t,x,v)|^2 + \Vert \nabla_x^2 (\phi_f + \phi_E) \Vert_{L^\infty_x} |\nabla^2_v X(s;t,x,v)|.
\end{split}
\Ee
Since $\nabla^2_v X(t;t,x,v) = \nabla^2_v x = 0$, $\nabla^2_v V(t;t,x,v) = \nabla^2_v v = 0$. Then we obtain
\Be \label{eq1:X_vv_Lp}
\begin{split} 
|\nabla^2_v X(s;t,x,v)| 
& \lesssim \int^{t}_{s} |\nabla^2_v V(s_1;t,x,v)| \dd s_1
\\& \lesssim \int^{t}_{s} \int^{t}_{s_1} | \nabla_x^3 (\phi_f + \phi_E) (X(s_2;t,x,v))| |\nabla_v X(s_2;t,x,v)|^2 \dd s_2 \dd s_1 
\\& \quad + \int^{t}_{s} \int^{t}_{s_1} \Vert \nabla_x^2 (\phi_f + \phi_E) \Vert_{L^\infty_x} |\nabla^2_v X(s_2;t,x,v)| \dd s_2 \dd s_1.
\end{split}
\Ee
Using Fubini's theorem, we derive
\Be \notag
\begin{split}
\eqref{eq1:X_vv_Lp} 
& \lesssim \int^{t}_{s} \int^{t}_{s_1} | \nabla_x^3 (\phi_f + \phi_E) (X(s_2;t,x,v))| |\nabla_v X(s_2;t,x,v)|^2 \dd s_2 \dd s_1
\\& \quad + \int^{t}_{s} (s_2 - s) \Vert \nabla_x^2 (\phi_f + \phi_E) \Vert_{L^\infty_x} |\nabla^2_v X(s_2;t,x,v)| \dd s_2.
\end{split}
\Ee
Then the Gronwall's inequality implies that
\Be \notag
\begin{split}
& |\nabla^2_v X(s;t,x,v)| 
\\& \lesssim \exp \big( \Vert \nabla_x^2 (\phi_f + \phi_E) \Vert_{L^\infty_x}
\int^{t}_{s} (s_2-s) \dd s_2 \big) \int^{t}_{s} \int^{t}_{s_1} | \nabla_x^3 (\phi_f + \phi_E) (X(s_2;t,x,v))| |\nabla_v X(s_2;t,x,v)|^2 \dd s_2 \dd s_1 
\\& \leq e^{\Vert \nabla_x^2 (\phi_f + \phi_E) \Vert_{L^\infty_x} \frac{(t-s)^2}{2} } \int^{t}_{s} \int^{t}_{s_1} | \nabla_x^3 (\phi_f + \phi_E) (X(s_2;t,x,v))| |\nabla_v X(s_2;t,x,v)|^2 \dd s_2 \dd s_1.
\end{split}
\Ee
This, together with \eqref{est:X_v first}, implies that
\Be \label{eq3:X_vv_Lp}
\begin{split}
|\nabla^2_v X(s;t,x,v)| 
\lesssim |t-s|^2 e^{ \frac{3}{2} \Vert \nabla_x^2 (\phi_f + \phi_E) \Vert_{L^\infty_x} (t-s)^2 } \int^{t}_{s} \int^{t}_{s_1} | \nabla_x^3 (\phi_f + \phi_E) (X(s_2;t,x,v))| \dd s_2 \dd s_1.
\end{split}
\Ee
Using \eqref{w_quotient} in Lemma \ref{lemma:v_variation} and Minkowski's integral inequality, from \eqref{eq3:X_vv_Lp} we derive that
\Be \notag
\begin{split}
& \big\| e^{- \theta |v|^2} \nabla^2_v X(s;t,x,v) \big\|_{L^p_{x,v}} 
\\& \lesssim |t-s|^2 e^{ \frac{3}{2} \Vert \nabla_x^2 (\phi_f + \phi_E) \Vert_{L^\infty_x} (t-s)^2 } \int^{t}_{s} \int^{t}_{s_1} 
\big\| e^{- \frac{1}{2} \theta |V (s_2;t,x,v)|^2} \nabla_x^3 (\phi_f + \phi_E) (X(s_2;t,x,v)) \big\|_{L^p_{x,v}} \dd s_2 \dd s_1.
\end{split}
\Ee
Since $\max \{ 0, t - \tb \} \leq s \leq t$, the change of variable in Lemma \ref{lemma:int_cov} implies that
\Be \notag
\begin{split}
\big\| e^{- \theta |v|^2} \nabla^2_v X(s;t,x,v) \big\|_{L^p_{x,v}} 
\lesssim |t-s|^2 e^{ \frac{3}{2} \Vert \nabla_x^2 (\phi_f + \phi_E) \Vert_{L^\infty_x} (t-s)^2 } \int^{t}_{s} \int^{t}_{s_1} 
\big\| \nabla_x^3 (\phi_f + \phi_E) \big\|_{L^p_{x}} \dd s_2 \dd s_1.
\end{split}
\Ee

\smallskip

Second, we prove \eqref{est:X_xv first}.
Again from the proof of Lemma \ref{lemma:deri_XV}, the characteristics of \eqref{characteristics} implies that
\be \notag
\frac{\dd}{\dd s} | \nabla_x \nabla_v X(s;t,x,v)| 
\lesssim | \nabla_x \nabla_v V(s;t,x,v)|,
\ee
and thus,
\Be \notag
\begin{split} 
& \frac{\dd}{\dd s} | \nabla_x \nabla_v V(s;t,x,v)| 
\\& \lesssim | \nabla^3_x (\phi_f + \phi_E) | | \nabla_v X(s;t,x,v) | | \nabla_x X(s;t,x,v) | + | \nabla^2_x (\phi_f + \phi_E) | | \nabla_x \nabla_v X(s;t,x,v) |
\\& \leq | \nabla_x^3 (\phi_f + \phi_E) (X(s;t,x,v))| |\nabla_v X(s;t,x,v)| |\nabla_x X(s;t,x,v)| + \Vert \nabla_x^2 (\phi_f + \phi_E) \Vert_{L^\infty_x} | \nabla_x \nabla_v X(s;t,x,v) |.
\end{split}
\Ee
Since $\nabla_x \nabla_v X(t;t,x,v) = \nabla_x \nabla_v x = 0$, $\nabla_x \nabla_v V(t;t,x,v) = \nabla_x \nabla_v v = 0$. Then we obtain
\Be \label{eq1:X_xv_Lp}
\begin{split} 
| \nabla_x \nabla_v X (s;t,x,v)| 
& \lesssim \int^{t}_{s} | \nabla_x \nabla_v V(s_1;t,x,v)| \dd s_1
\\& \lesssim \int^{t}_{s} \int^{t}_{s_1} | \nabla_x^3 (\phi_f + \phi_E) (X(s_2;t,x,v))| |\nabla_v X(s;t,x,v)| |\nabla_x X(s;t,x,v)| \dd s_2 \dd s_1 
\\& \qquad + \int^{t}_{s} \int^{t}_{s_1} \Vert \nabla_x^2 (\phi_f + \phi_E) \Vert_{L^\infty_x} | \nabla_x \nabla_v  X(s_2;t,x,v)| \dd s_2 \dd s_1.
\end{split}
\Ee
Using Fubini's theorem, we derive
\Be \notag
\begin{split}
\eqref{eq1:X_xv_Lp} 
\lesssim & \int^{t}_{s} \int^{t}_{s_1} | \nabla_x^3 \phi_f (X(s_2;t,x,v))| |\nabla_v X(s;t,x,v)| |\nabla_x X(s;t,x,v)| \dd s_2 \dd s_1
\\& + \int^{t}_{s} (s_2 - s) \Vert \nabla_x^2 \phi_f \Vert_{L^\infty_x} | \nabla_x \nabla_v X(s_2;t,x,v)| \dd s_2.
\end{split}
\Ee
This, together with \eqref{est:X_v first} and the Gronwall's inequality, implies that
\Be \notag
\begin{split}
| \nabla_x \nabla_v X(s;t,x,v)| 
\lesssim |t-s| e^{ \frac{3}{2} \Vert \nabla_x^2 \phi_f\Vert_{L^\infty_x} (t-s)^2 } \int^{t}_{s} \int^{t}_{s_1} | \nabla_x^3 (\phi_f+\phi_E) (X(s_2;t,x,v))| \dd s_2 \dd s_1.
\end{split}
\Ee
The remaining estimate follows analogously to the first part, and we obtain \eqref{est:X_xv first}.
\end{proof}

\begin{lemma} \label{lemma:W3p}

For $2 < p < 3$, the following estimates for $\Vert \nabla^3_x \phi_f \Vert_{L^p_{x}}$ hold:

\begin{enumerate}
\item For the stationary problem \eqref{eqn:h},
\be \notag
\Vert \nabla^3_x \phi_h \Vert_{L^p_{x}} \lesssim \Vert \nabla_x h \Vert_{L^p_{x,v}} + \Vert w h \Vert_{L^\infty_{x,v}}.
\ee

\item For the dynamical problem \eqref{linear_f_dynamical},
\be \notag
\Vert \nabla^3_x \phi_f \Vert_{L^p_{x}} \lesssim \Vert \nabla_x (h + f) \Vert_{L^p_{x,v}} + \Vert w (h + f) \Vert_{L^\infty_{x,v}}.
\ee
\end{enumerate}
\end{lemma}

\begin{proof}

For the stationary problem \eqref{eqn:h}, applying the elliptic estimate to the equation for $\phi_h$ in \eqref{eqn:h} and H\"older's inequality, we obtain
\be \notag
\begin{split}
\Vert \nabla^3_x\phi_h \Vert_{L^p_x} 
& \lesssim \Big\Vert  \int_{\mathbb{R}^3} h \sqrt{\mu} \dd v  \Big\Vert_{L^p_x} + \Big\Vert  \int_{\mathbb{R}^3}  \nabla_x h \sqrt{\mu} \dd v  \Big\Vert_{L^p_x}
\\& \lesssim \Vert w h \Vert_{L^\infty_{x,v}} + \Vert \Vert \nabla_x h \Vert_{L^p_v} \Vert \sqrt{\mu}\Vert_{L^{p'}_v} \Vert_{L^p_x} \lesssim \Vert w h \Vert_{L^\infty_{x,v}} + \Vert \nabla_x h \Vert_{L^p_{x,v}}.
\end{split} 
\ee
We omit the dynamical problem \eqref{linear_f_dynamical} case, as it is similar to the stationary case.
\end{proof}

\begin{lemma} \label{lemma:phi_x_infinity}

For $p > 3$, the following $C^1$ estimates hold:

\begin{enumerate}
\item For the stationary problem \eqref{eqn:h},
\begin{align}
& \Vert \phi_h \Vert_{C^1_x}  
\lesssim  \Big\Vert \int_{\mathbb{R}^3} h \sqrt{\mu} \dd v \Big\Vert_{L^p_x} \lesssim \Vert w h \Vert_{L^\infty_{x,v}},
\label{phi_f_C^1} \\
& \Vert \phi_h \Vert_{C^1_x}  
\lesssim o(1) \Vert w h \Vert_{L^\infty_{x,v}} + \Vert h \Vert_{L^2_{x,v}}. 
\label{phi_f_C^1_2}
\end{align}

\item For the dynamical problem \eqref{linear_f_dynamical},
\begin{align}
\Vert  \nabla_x\phi_f\Vert_{L^\infty_x} & \lesssim \Vert wh\Vert_{L^\infty_{x,v}} + \Vert wf\Vert_{L^\infty_{x,v}}, 
\label{phif_bdd_dynamical} \\
\Vert \nabla_x (\phi_f-\phi_{h})\Vert_{L^\infty_x} & \lesssim \Vert wf\Vert_{L^\infty_{x,v}}. 
\label{phif_fs_bdd}
\end{align}
\end{enumerate}
\end{lemma}

\begin{proof}

For the stationary problem \eqref{eqn:h}, using Morrey's inequality and the assumption $p>3$, we have
\be \notag
\Vert \phi_h \Vert_{C^1_x} 
\lesssim \Vert \phi_h \Vert_{W^{2,p}_x} \lesssim \Big\Vert \int_{\mathbb{R}^3} h \sqrt{\mu} \dd v \Big\Vert_{L^p_x} 
\lesssim \Vert w h \Vert_{L^\infty_{x,v}}.
\ee
Furthermore, H\"older's inequality and Young's inequality imply that
\be \notag
\Vert \phi_h \Vert_{C^1_x} 
\lesssim \Big\Vert \int_{\mathbb{R}^3} h \sqrt{\mu} \dd v \Big\Vert_{L^p_x} \lesssim \Vert h \Vert_{L^p_x L^2_v} \lesssim \Vert w h \Vert_{L^\infty_{x,v}}^{\frac{p-2}{p}} \Vert h \Vert_{L^2_{x,v}}^{\frac{2}{p}}  \lesssim o(1)\Vert w h \Vert_{L^\infty_{x,v}} + \Vert h \Vert_{L^2_{x,v}}. 
\ee
We omit the dynamical problem \eqref{linear_f_dynamical} case, as it is similar to the stationary case.
\end{proof}


\subsection{
\texorpdfstring{Properties of the collision operators $\mathcal{L}$ and $\Gamma$}{Properties of the linearized collision operator L and the nonlinear collision operator Gamma}
}
\label{sec:collision_operator}

In this section, we introduce the properties of the linearized collision operator $\mathcal{L}$ and the nonlinear collision operator $\Gamma$ defined in \eqref{def:collision_operator}. It is known (see \cite{R}) that the linearized operator $\mathcal{L}$ admits the decomposition
\begin{equation} \notag
\mathcal{L} f = \nu(v) f - K f,
\end{equation}
where $\nu(v)$ denotes the collision frequency and $K$ denotes the integral operator.

\begin{lemma}[\cite{R}] \label{lemma:k_nu}

The integral operator $K$ is a compact operator on $L^2_v$, and is defined by
\[
Kf(x,v) = \int_{\mathbb{R}^3} \mathbf{k} (v,u) f(x,u) \, \dd u.
\]
Define $\mathbf{k}_{\varrho}(v,u):= e^{- \varrho |v-u|^2}/|v-u|$ for some $\varrho>0$. Then the kernel $\mathbf{k}(v,u)$ satisfies
\Be \notag
|\mathbf{k}(v,u)| \lesssim \mathbf{k}_\varrho (v,u).
\Ee
Moreover, the collision frequency $\nu(v)$ satisfies
\begin{equation} \notag
\nu(v) \geq \nu_0 \big( |v|+1 \big)
\text{ with } \nu_0 > 0,
\qquad
|\nabla_v \nu(v)|\lesssim 1.
\end{equation}
\end{lemma}

\begin{lemma}[\cite{CK, G}] \label{lemma:k_theta}

Let $\mathbf{k}(v,u)$ be the kernel of the integral operator $K$ introduced in Lemma~\ref{lemma:k_nu}. Recall the definition of $\mathbf{k}_{\varrho}(v,u)$ in Lemma~\ref{lemma:k_nu}. Then the following estimates hold.

\begin{enumerate}
\item  Let $0 \leq \theta < \frac{1}{4}$. There exists $\tilde{\varrho}>0$ such that
\begin{equation*} 
\mathbf{k}(v,u) \frac{\nu^n(v)e^{\theta |v|^2}}{\nu^n(u)e^{\theta |u|^2}} \lesssim \mathbf{k}_{\tilde{\varrho}}(v,u) = e^{- \tilde{\varrho} |v-u|^2}/ |v-u|.   
\end{equation*}
Moreover, 
\begin{equation} \label{k_theta}
\int_{\mathbb{R}^3} \mathbf{k}(v,u) \frac{\nu^n(v)e^{\theta |v|^2}}{\nu^n(u)e^{\theta |u|^2}} \dd u 
\lesssim \frac{1}{1+|v|},
\end{equation}
and for $1 < p <2$, we have
\be \label{k_theta_p'}
\int_{\mathbb{R}^3} \mathbf{k}^{p}(v,u) \frac{\nu^n(v)e^{\theta |v|^2}}{\nu^n(u)e^{\theta |u|^2}} \dd u \lesssim \frac{1}{1+|v|}. 
\ee

\item For $N\gg 1$, we have
\begin{equation} \label{k_N_upper_bdd}
 \mathbf{k}(v,u) \frac{\nu^n(v)e^{\theta |v|^2}}{\nu^n(u)e^{\theta |u|^2}} \mathbf{1}_{|v-u|> \frac{1}{N}} \leq C_N,
\end{equation}
and
\begin{equation} \notag
\int_{|u|>N \text{ or } |v-u|\leq \frac{1}{N}}  \mathbf{k}^{}(v,u) \frac{\nu^n(v)e^{\theta |v|^2}}{\nu^n(u)e^{\theta |u|^2}} \dd u \lesssim \frac{1}{N} \leq o(1).
\end{equation}

\item For some $\varrho > 0$, the derivative of $\mathbf{k}(v,u)$ satisfies
\be \notag
| \nabla_v \mathbf{k} (v,u)| 
\lesssim \frac{\langle v \rangle}{|v-u|} \mathbf{k}_\varrho (v,u),
\qquad
| \nabla_u \mathbf{k} (v,u)| 
\lesssim \frac{\langle u\rangle}{|v-u|} \mathbf{k}_\varrho (v,u). 
\ee
Moreover, there exists $\tilde{\varrho}>0$ such that
\begin{equation} \notag
| \nabla_v \mathbf{k}(v,u)| \frac{e^{\theta| v|^2}}{e^{\theta| u|^2}} \lesssim \frac{[1+|v|^2]}{| v-u|} \mathbf{k}_{\tilde{\varrho}}(v,u),
\qquad
| \nabla_u \mathbf{k}(v,u)| \frac{e^{\theta| v|^2}}{e^{\theta| u|^2}} \lesssim \frac{[1+|u|^2]}{| v-u|} \mathbf{k}_{\tilde{\varrho}}(v,u).
\end{equation}
\end{enumerate}
\end{lemma}

\begin{proof}

We only prove \eqref{k_theta_p'}, while the other estimates can be found in \cite{CK, G}.
It follows from \cite{R} that
\begin{align*}
\mathbf{k}^{p}(v,u) 
\leq C \big(|v-u|^{p} + |v-u|^{-p}\big) 
e^{-\frac{p}{8}|v-u|^2-\frac{p}{8} \frac{(|v|^2-|u|^2)^2}{|v-u|^2}},
\end{align*}
which contains the singular term $|v-u|^{-p}$.
Following the proof of Lemma 3 in \cite{G}, we set $\eta = v-u$ and decompose
\[
\eta_\parallel = \Big(\eta \cdot \frac{v}{|v|}\Big)\frac{v}{|v|}, 
\qquad 
\eta_\perp = \eta - \eta_\parallel.
\]
Then we compute that
\begin{align*}
\int_{\mathbb{R}^3} \mathbf{k}^{p}(v,u) \frac{\nu^n(v)e^{\theta|v|^2}}{\nu^n(u)e^{\theta|u|^2}} \,\dd u
&\lesssim \int_{\mathbb{R}^2} \Big(\frac{1}{|\eta_\perp|^{p}}+1 \Big)e^{-\frac{C}{4}|\eta|^2} 
\Big\{\int_{-\infty}^\infty e^{-C|v||\eta_\parallel|} \,\dd |\eta_\parallel| \Big\} \dd \eta_\perp \\
&\lesssim \frac{1}{1+|v|} \int_{\mathbb{R}^2} \Big(\frac{1}{|\eta_\perp|^{p}}+1 \Big) e^{-\frac{C}{4}|\eta_\perp|^2} 
\Big\{\int_{-\infty}^\infty e^{-C|y|} \,\dd y \Big\} \dd \eta_\perp \lesssim \frac{1}{1+|v|},
\end{align*}
where we use $p<2$ in the second line to ensure local integrability with respect to $\eta_\perp$.
\end{proof}

\begin{lemma}[\cite{R}] 
\label{lemma:gamma}

For the nonlinear collision operator $\Gamma$ in \eqref{def:collision_operator}, the following estimates hold.

\begin{enumerate}
\item We have the following $L^2$ and $L^\infty$ estimates:
\be \notag
\begin{split}
& \Vert \nu^{-1}w\Gamma(f,f)\Vert_{L^\infty_{x,v}} \lesssim \Vert wf\Vert_{L^\infty_{x,v}}^2, \\
& \Vert \nu^{-1/2}\Gamma(f,h)\Vert_{L^2_{x,v}} \lesssim \Vert wh\Vert_{L^\infty_{x,v}}\Vert f\Vert_{L^2_{x,\nu}}.
\end{split}
\ee

\item The $v$ derivative on $\Gamma(f,f)$ is computed as
\begin{align*}
    & \p_v \Gamma(f,f) = \Gamma(\p_v f,f) + \Gamma(f,\p_v f) + \Gamma_v(f,f), 
\end{align*}
where $\Gamma_v(f,f)$ is given via the Carleman representation:
\be \notag
\begin{split}
\Gamma_v(f,f) 
& = \int_{\mathbb{R}^3}\int_{\mathbb{S}^2} |u\cdot \omega| f(v+u_\perp)f(v+u_\parallel) \p_v \sqrt{\mu(v+u)} \dd \omega \dd u 
\\& \quad + \int_{\mathbb{R}^3}\int_{\mathbb{S}^2} |u\cdot \omega| f(v+u)f(v) \p_v \sqrt{\mu(v+u)} \dd \omega \dd u.
\end{split}
\ee
Moreover,
\begin{align*}
    & \Vert \nu^{-1}w\Gamma_v(f,f)\Vert_{L^\infty_{x,v}} \lesssim \Vert wf\Vert_{L^\infty_{x,v}}^2.
\end{align*}

\item We define $\Gamma_{\text{loss}}$ and $\Gamma_{\text{gain}}$ as follows:
\be \notag
\begin{split}
\Gamma_{\text{loss}} (f,g) & := \frac{Q_{\text{loss}}(\sqrt{\mu}f,\sqrt{\mu}g)}{\sqrt{\mu}}
\ \text{ with } \
Q_{\text{loss}}(F,G) := \iint_{\R^2\times \S^2} |(v-u)\cdot \omega| F(u)G(v) \dd \omega \dd u, \\
\Gamma_{\text{gain}} (f,g) & := \frac{Q_{\text{gain}}(\sqrt{\mu}f,\sqrt{\mu}g)}{\sqrt{\mu}}
\ \text{ with } \
Q_{\text{gain}} (F,G) := \iint_{\R^2\times \S^2} |(v-u)\cdot \omega| F(u')G(v') \dd \omega \dd u.
\end{split}
\ee
There exists a kernel $\mathbf{k}_1 (v,u)$ such that
\begin{align}
|\Gamma_{\text{loss}}( f,\p_{x,v} f)| 
& \lesssim \Vert wf\Vert_{L^\infty_{x,v}} \nu(v) |\p_{x,v} f(v)|,
\label{eq1:gamma} \\ 
|\Gamma_{\text{gain}}(f,\p_{x,v} f) + \Gamma(\p_{x,v} f,f)| & \lesssim \Vert wf\Vert_{L^\infty_{x,v}} \int_{\mathbb{R}^3} \mathbf{k}_{1}(v,u) |\p_{x,v} f(u)| \dd u.
\label{eq2:gamma}
\end{align}
Moreover, for $0 \leq \tilde{\theta} \ll \theta < \frac{1}{4}$, the kernel $\mathbf{k}_1(v,u)$ satisfies
\be \notag
\begin{split}
\mathbf{k}_1 (v,u) \frac{\nu^n(v) w_{\tilde{\theta} (v)}}{\nu^n(u) w_{\tilde{\theta}} (u)} 
& \lesssim e^{- \tilde{\varrho} |v-u|^2}/ |v-u|
\ \text{ for some } 
\tilde{\varrho} > 0,
\\ \int_{\mathbb{R}^3} \mathbf{k}_1 (v,u) \frac{w_{\tilde{\theta}}(v)}{w_{\tilde{\theta}}(u)} \dd u 
& \lesssim \frac{1}{1+|v|},
\end{split}
\ee
and for $1 < p <2$, we have
\be \notag
\int_{\mathbb{R}^3} \mathbf{k}^{p}_1 (v,u) \frac{\nu^n(v) w_{\tilde{\theta}}(v)}{\nu^n(u) w_{\tilde{\theta}}(u)} \dd u 
\lesssim \frac{1}{1+|v|}. 
\ee
We remark that these bounds are of the same type as those satisfied by $\mathbf{k}(v,u)$ in Lemma~\ref{lemma:k_theta}.
\end{enumerate}
\end{lemma}

\begin{proof}

We only prove \eqref{eq1:gamma} and \eqref{eq2:gamma}, while the other estimates can be found in \cite{R} or proved similarly to Lemma \ref{lemma:k_theta}.
For \eqref{eq1:gamma}, we have
\be \notag
\begin{split}
\Gamma_{\text{loss}}(f,\p_{x,v} f) 
& = \p_{x,v} f(v) \iint_{\mathbb{S}^2\times \mathbb{R}^3}  |(v-u)\cdot \omega|  \sqrt{\mu(u)}f(u) \dd u \dd \omega 
\\& \leq \Vert wf\Vert_{L^\infty_{x,v}} |\p_{x,v} f(v)| \int_{\mathbb{R}^3} |v-u| \sqrt{\mu(u)}w^{-1}(u) \dd u \lesssim \Vert wf\Vert_{L^\infty_{x,v}} |\p_{x,v} f(v)| \nu(v).
\end{split}
\ee
For \eqref{eq2:gamma}, we have
\begin{align*}
    & | \Gamma_{\text{gain}}(f,\p_{x,v} f) + \Gamma(\p_{x,v} f,f)| \lesssim \Vert wf\Vert_{L^\infty_{x,v}} \frac{Q_{\text{gain}}(w^{-1}\sqrt{\mu},\sqrt{\mu}\p_{x,v} f) + Q(\sqrt{\mu}\p_{x,v} f, w^{-1}\sqrt{\mu})}{\sqrt{\mu}}.
\end{align*}
By direct computation, the kernel $\mathbf{k}(v,u)$ in Lemma~\ref{lemma:k_theta} satisfies
\begin{align*}
   \frac{Q_{\text{gain}}(\mu,\sqrt{\mu}\p_{x,v} f)+ Q(\sqrt{\mu}\p_{x,v} f,\mu)}{\sqrt{\mu}} = \int_{\mathbb{R}^3} \mathbf{k}(v,u)\p_{x,v} f(u)\dd u. 
\end{align*}
Following the argument in \cite{R}, there exists another kernel $\mathbf{k}_1(v,u)$ such that
\be \notag
|\Gamma_{\text{gain}}(f,\p_{x,v} f) + \Gamma(\p_{x,v} f,f) | \lesssim \Vert wf \Vert_{L^\infty_{x,v}}\int_{\mathbb{R}^3} \mathbf{k}_{1}(v,u) |\p_{x,v} f(u)| \dd u.
\ee
Therefore, we conclude \eqref{eq2:gamma}.
\end{proof}

\subsection{Property of the kinetic weight}
\label{sec:kinetic_weight}

In this section, we study the properties of the kinetic weights $\alpha_h(x,v)$ in \eqref{alpha_weight_steady} for the stationary problem and $\alpha_f(t,x,v)$ in \eqref{alpha_weight_dyna} for the dynamical problem.

Throughout this section, we assume that the boundary condition for the external field $\phi_E$ in both \eqref{eqn:h} and \eqref{linear_f_dynamical}, and the self-consistent fields satisfy
\be \notag
\Vert \nabla_x \phi_h (x) \Vert_{L^\infty_x} < \frac{1}{20} C_E,
\qquad 
\Vert \nabla_x \phi_f (t, x) \Vert_{L^\infty_x} < \frac{1}{20} C_E.
\ee
All lemmas in this section apply to both kinetic weights $\alpha_h(x,v)$ and $\alpha_f(t,x,v)$\footnote{In later sections, we will specify whether we are considering $\alpha_h(x,v)$ in \eqref{alpha_weight_steady} or $\alpha_f(t,x,v)$ in \eqref{alpha_weight_dyna}.}. For simplicity, we slightly abuse notation (except in Lemma \ref{lemma:nabla2_phi_bdd} and Lemma \ref{lemma:phi_C2}, where the stationary and dynamical settings are treated separately) and write
\be \notag
\begin{split}
E(t,x) & := 
\begin{cases}
- \nabla_x \big( \phi_h(x)+ \phi_E(x) \big), & \text{ for stationary problem}, \\[5pt]
- \nabla_x \big( \phi_f(t,x) + \phi_E(x) \big), & \text{ for dynamical problem},
\end{cases}
\\[5pt] 
\alpha(t,x,v) & : = 
\begin{cases}
\alpha_h(x,v), & \text{ for stationary problem}, \\[5pt]
\alpha_f(t,x,v), & \text{ for dynamical problem}.
\end{cases} 
\end{split}
\ee
We remark that for the stationary problem, both $\alpha(t,x,v)=\alpha_h(x,v)$ and $E(t,x)= - \nabla_x \big( \phi_h(x)+ \phi_E(x) \big)$ are independent of time.

Recall from Lemma~\ref{lemma:dist_unique} that for any $x \in \O_{\delta}$ with $0 < \delta \ll 1$, there exists a unique $\tilde{x} \in \p \O$ satisfying that $dist (x, \tilde{x}) = dist (x, \p \O)$.
Then the above assumptions imply that
\be \label{sign_condition}
E (t, x) \cdot n(\tilde{x}) > \frac{C_E}{2} 
\ \text{ for } 
x \in \O_\delta.
\ee
Furthermore, from \eqref{alpha_weight_steady} and \eqref{alpha_weight_dyna}, and the choice $\delta' \ll \delta^{1/2}$, we have
\begin{align*}
\big[2(E(t,\tilde{x})\cdot \nabla_x \xi(\tilde{x}))\xi(x)\big]^{1/2}
\gtrsim C_E^{1/2}\delta^{1/2} > \delta'
\ \text{ for } 
(1-\varepsilon)\delta \leq \mathrm{dist}(x,\p\O) \leq \delta.
\end{align*}
Hence, $\alpha(t,x,v) = \delta'$ when $(1-\varepsilon) \delta \leq \mathrm{dist}(x,\p\O) \leq \delta$. Therefore, the kinetic weight $\alpha(t,x,v)$ is continuous across the interface $\{x \in \O \mid \mathrm{dist}(x,\p\O)=\delta\} \subset \p \O_\delta$.

\smallskip

We start with the velocity lemma for the kinetic weight as follows.

\begin{lemma}[Velocity Lemma]
\label{lemma:velocity}

There exists a constant $C_0 = C_0(C_E,\O) > 0$ such that for any $t - t_b \leq s \leq t$,
\be \notag
e^{-C_0}\alpha(s, X(s;t,x,v), V(s;t,x,v)) \leq \alpha(t,x,v) \leq e^{C_0} \alpha(s, X(s;t,x,v), V(s;t,x,v))
\ee
\end{lemma}

\begin{proof}

For simplicity, we write $(X(s), V(s)) := (X(s;t,x,v), V(s;t,x,v))$ throughout the proof.
From Lemma \ref{lemma:est_tf}, there exist $0< \delta \ll 1$ and $s \leq t_1 \leq t_2 \leq t$ such that 
\be \notag
X(t') \in 
\begin{cases}
\O \setminus \O_\delta, & t' \in [t_1, t_2], \\[5pt]
\O_\delta, & t' \in [s, t_1) \cup (t_2, t].
\end{cases}
\ee
If $X(t') \in \O_\delta$ for all $s \leq t' \leq t$, we set $t_1 = t_2$. Moreover, by \eqref{t2_t1_bdd} in Lemma~\ref{lemma:est_tf}, there exists a constant $C_{\phi_E,\delta} > 0$ such that
\be \label{eq1:velocity}
|t_1 - s| + |t - t_2| \leq C_{\phi_E,\delta}.
\ee
Then we compute
\be \notag
\frac{\alpha(X(s),V(s))}{\alpha(t,x,v)} 
= \frac{\alpha(X(s),V(s))}{\alpha(X(t_1),V(t_1))} \frac{\alpha(X(t_1),V(t_1))}{\alpha(X(t_2),V(t_2))} \frac{\alpha(X(t_2),V(t_2))}{\alpha(t,x,v)}
= \frac{\alpha(X(s),V(s))}{\alpha(X(t_1),V(t_1))} \frac{\alpha(X(t_2),V(t_2))}{\alpha(t,x,v)},
\ee
where we used that $\alpha(X(t'),V(t')) = \delta'$ for $t_1 \leq t' \leq t_2$.
Using Lemma 7 in \cite{cao2019regularity}, we have
\begin{align*}
e^{-C_{\phi_E,\delta}(|t_1-s|+|t-t_2|)}
\leq \frac{\alpha(X(s),V(s))}{\alpha(X(t_1),V(t_1))}
\frac{\alpha(X(t_2),V(t_2))}{\alpha(t,x,v)}
\leq e^{C_{\phi_E,\delta}(|t_1-s|+|t-t_2|)}.
\end{align*}
This, together with \eqref{eq1:velocity}, concludes the lemma.
\end{proof}


\begin{corollary}
\label{cor:est_x_v}

Let $0 < \delta \ll1$ be the constant in Lemma~\ref{lemma:est_tf}. Suppose that
$X(s;t,x,v)\in \O_\delta$ for all $t - \tb \leq s \leq t$. Then we have
\begin{align}
|\nabla_v X(s; t,x,v)| & \lesssim \alpha(t,x,v).
\notag
\end{align}
\end{corollary} 

\begin{proof}


We apply \eqref{est:X_v first} in Lemma \ref{lemma:deri_XV}, Corollary \ref{cor:est_tf}, and Lemma~\ref{lemma:velocity} to obtain
\be \notag
|\nabla_v X(s)| \lesssim |t-s| \lesssim \tb(x,v) \lesssim \min\{ |n(\xb) \cdot \vb|, \ 1 \} \lesssim \alpha(t-\tb,\xb,\vb) \lesssim \alpha(t,x,v).
\ee
\end{proof}


Using the velocity lemma, the following lemma shows that the kinetic weight $\alpha(t,x,v)$ compensates for the singularity term $|n(\xb) \cdot \vb |^{-1}$ arising in the characteristics from Lemma \ref{lemma:deri_backward}.

\begin{lemma} \label{lemma:weight_singularity}

Assume that $t_b < \infty$. Then
\be \notag
\alpha(t,x,v) \lesssim |n(\xb(t,x,v))\cdot \vb(t,x,v)|.
\ee
Consequently,
\be \notag
\frac{1}{|n(\xb(t,x,v))\cdot \vb(t,x,v)|} \lesssim \frac{1}{\alpha(t,x,v)}.
\ee
\end{lemma}

\begin{proof}


From Lemma~\ref{lemma:velocity}, we have
\be \notag
\alpha(t,x,v) \lesssim \alpha(t-t_b,\xb,\vb).
\ee
Using the definition of $\alpha(x,v)$ in \eqref{alpha_weight_steady} and \eqref{alpha_weight_dyna}, and the fact that $\xi(\xb)=0$, we obtain
\be \notag
\alpha(t - \tb,\xb,\vb)
\lesssim \min\{ |n(\xb)\cdot \vb|,  \ \delta' \}
\lesssim |n(\xb)\cdot \vb|,
\ee
and thus conclude the lemma.
\end{proof}

\begin{lemma} \label{lemma:integrate_nv}

Let $\delta' > 0$ be the constant defined in \eqref{alpha_weight_steady}. For any $c>0$ and $2 < p < 3$, we have
\be \notag
\iint_{\O\times \mathbb{R}^3} \mathbf{1}_{\tb<\infty} \frac{e^{-\nu_0 \tb/2}e^{-c|\vb|^2}}{|n(\xb)\cdot \vb|^p} \,\dd x \dd v 
\lesssim (\delta')^{-p+2-\e} \lesssim 1,
\ee
where $0<\e\ll1$.
\end{lemma}

\begin{proof}

To prove the lemma, we use the following change of variables:
\be \label{eq1:integrate_nv}
\iint_{\O\times \mathbb{R}^3} \mathbf{1}_{\tb<\infty} |f(t,x,v)| \dd x \dd v = \int_{\gamma_-} \int_0^{\tf} |f(t,X(t+s;t,x,v), V (t+s;t,x,v))| |n(x)\cdot v| \dd s \dd v \dd S_x.
\ee
Here $0 < s < \tf (t,x,v)$, and thus $X(t+s;t,x,v)$ represents a forward characteristic. We refer to \cite{CKL} for the proof of this change of variables.

Note that for any $(x,v) \in \gamma_-$ and $0 < s < \tf(t,x,v)$,
\begin{equation} \notag
\begin{split}
\tb (t+s, X(t+s;t,x,v), V(t+s;t,x,v)) & = s, 
\\ \xb (t+s, X(t+s;t,x,v), V(t+s;t,x,v)) & = x,
\\ \vb (t+s, X(t+s;t,x,v), V(t+s;t,x,v)) & = v.
\end{split}
\end{equation}
This, together with the change of variables \eqref{eq1:integrate_nv}, implies that
\begin{equation} \label{eq2:integrate_nv}
\begin{split}
& \iint_{\O\times \mathbb{R}^3} \mathbf{1}_{\tb<\infty}\frac{e^{- \nu_0 \tb/2}e^{- c |\vb|^2}}{|n(\xb)\cdot \vb|^p} \dd x \dd v 
\\& = \int_{\gamma_-} \int_0^{\tf (t,x,v)} \frac{e^{- \nu_0 s/2}e^{-c |v|^2}}{|n(x)\cdot v|^p} |n(x)\cdot v| \dd s \dd v \dd S_x  
= \int_{\gamma_-} \int_0^{\tf(t,x,v)} \frac{e^{- \nu_0 s/2}e^{-c |v|^2}}{|n(x)\cdot v|^{p-1}} \dd s \dd v \dd S_x.
\end{split}
\end{equation}
For the characteristic $X(t+s;t,x,v)$ with $0 < s < \tf (t,x,v)$, there exists $0 < t_* \leq \tf (t,x,v)$ such that
\[
X (t+s;t,x,v) \in
\begin{cases}
\O_\delta, & s \in [0, t_*), \\[5pt]
\p \O_\delta \setminus \p \O, & s = t_*.
\end{cases}
\]
If $t_* = \tf (t,x,v)$, then \eqref{est_tf} in Lemma \ref{lemma:est_tf} implies $\tf (t,x,v) \lesssim |n(x) \cdot v|$.
Otherwise, $t_* < \tf (t,x,v)$. From \eqref{alpha_weight_steady} and Lemma \ref{lemma:velocity}, we obtain
\be \notag
|n(x)\cdot v| \gtrsim \alpha (t,x,v) \gtrsim \alpha(t+t_*, X(t+t_*;t,x,v), V(t+t_*;t,x,v)) = \delta'.
\ee
Combining the two cases and using the assumption $2<p<3$, we derive
\be \notag
\begin{split}
&  \int_{\gamma_-} \int^{\tf (t,x,v)}_0 \frac{e^{-\nu_0 s/2}e^{-c|v|^2}}{|n(x)\cdot v|^{p-1}} \dd s \dd v \dd S_x 
\\& \lesssim \int_{\gamma_-} \frac{e^{-c|v|^2}}{|n(x)\cdot v|^{p-2}} \dd v \dd S_x  +  \int_{\gamma_-} \int_0^{\tf (t,x,v)} \frac{e^{-\nu_0 s/2}e^{-c|v|^2}}{|n(x)\cdot v|^{1-\e} (\delta')^{p-2+\e}} \dd s \dd v \dd S_x \lesssim \frac{1}{(\delta')^{p-2+\e}}. 
\end{split}
\ee
This, together with \eqref{eq2:integrate_nv}, completes the proof.
\end{proof}


\begin{lemma} \label{lemma:alpha_int}

For any $c>0$ and $q < 1$, we have
\be \label{eq:alpha_int}
\sup_{x \in \O} \int_{\mathbb{R}^3} \frac{e^{-c|v|^2}}{\alpha^q(t,x,v)} \,\dd v \lesssim 1.
\ee
\end{lemma}

\begin{proof}

If $x \in \O_\delta \setminus \O$, then from the definitions in \eqref{alpha_weight_steady} and \eqref{alpha_weight_dyna}, we have $\alpha^q(t,x,v)=\delta'$, and thus \eqref{eq:alpha_int} holds.
Otherwise, $x \in \O_\delta$. From the definition of $\alpha$, we have
\be \notag
\int_{\mathbb{R}^3} \frac{e^{- c |v|^2}}{\alpha^q(t,x,v)} \dd v \lesssim \int_{\mathbb{R}^3} \frac{e^{- c |v|^2}}{|v\cdot \nabla_x \xi|^q} \dd v
\ \text{ with } \
|\nabla_x \xi| \neq 0. 
\ee
Let $\big(\frac{\nabla_x \xi}{|\nabla_x \xi|}, \tau_1, \tau_2 \big)$ be an orthonormal basis of $\R^3$. We decompose 
\[
v = v_n \frac{\nabla_x \xi}{|\nabla_x \xi|} + v_{\tau_1}\tau_1 + v_{\tau_2}\tau_2
\ \text{ with } \
v_n = v\cdot \frac{\nabla_x \xi}{|\nabla_x \xi|}, \
v_{\tau_1} = v \cdot \tau_1, \
v_{\tau_2} = v \cdot \tau_2.
\]
Then we derive that
\begin{align*}
\int_{\mathbb{R}^3} \frac{e^{-c|v|^2}}{\alpha^q(t,x,v)} \dd v
&\lesssim \int_{\mathbb{R}^2} e^{-c \big( v_{\tau_1}^2 + v_{\tau_2}^2 \big) } \dd v_\tau
\int_{\mathbb{R}} \frac{e^{-c|v_n|^2}}{|v_n|^q} \dd v_n
\lesssim 1.
\end{align*}
Combining the two cases, we conclude the proof of the lemma.
\end{proof}


The following lemma includes two nonlocal-to-local estimates from \cite{cao2019regularity} and one new estimate proved here.

\begin{lemma}[Lemma 11 in \cite{cao2019regularity}]
\label{lemma:nonlocal_to_local}

Recall the collision frequency $\nu(v)$ in Lemma \ref{lemma:k_nu}. Let $\tilde{\nu}(x,v)$ satisfy
\[
\tilde{\nu} (x, v) > \frac{1}{8} \nu (v).
\]
For any $c>0$ and $0 < \varepsilon \ll 1$, we have
\be \notag
\begin{split}
\int^t_{\max\{0,t-\tb\}}  e^{-\int^t_s \tilde{\nu}(X(\tau;t,x,v), V(\tau;t,x,v)) \dd \tau} \int_{\mathbb{R}^3} \frac{e^{- c |V(s;t,x,v)-u|^2}}{|V(s;t,x,v)-u|} \frac{1}{\alpha^2(s,X(s;t,x,v),u)} \dd u \dd s
& \lesssim \frac{1}{\alpha(t,x,v)},
\\ \int^t_{t - \varepsilon}  e^{-\int^t_s \tilde{\nu}(X(\tau;t,x,v),V(\tau;t,x,v))\dd \tau} \int_{\mathbb{R}^3} \frac{e^{- c |V(s;t,x,v)-u|^2}}{|V(s;t,x,v)-u|} \frac{1}{\alpha^2(s,X(s;t,x,v),u)} \dd u \dd s 
& \lesssim \frac{\varepsilon}{\alpha(t,x,v)} + \varepsilon.
\end{split}
\ee
Moreover, for some sufficiently large $N\gg1$, we have
\be \notag 
\int^t_{\max\{0,t-\tb\}}  e^{-\int^t_s \tilde{\nu}(X(\tau;t,x,v),V(\tau;t,x,v)) \dd \tau} \int_{|V(s)-u| < \frac{1}{N}} \frac{e^{- c |V(s;t,x,v)-u|^2}}{|V(s;t,x,v)-u|} \frac{1}{\alpha(s,X(s;t,x,v),u)} \dd u \dd s  \lesssim \frac{o(1)}{\alpha(t,x,v)}.
\ee
\end{lemma}

\begin{proof}

We only prove the last estimate, while the other estimates can be found in \cite{cao2019regularity}. 
Using H\"older's inequality in $\dd u$, we have
\be \label{eq1:nonlocal_to_local}
\begin{split}
&  \int^t_{\max\{0,t-\tb\}}  e^{-\int^t_s \tilde{\nu}(X(\tau;t,x,v), V(\tau;t,x,v))\dd \tau} \int_{|V(s;t,x,v)-u|<\frac{1}{N}} \frac{e^{-C|V(s;t,x,v)-u|^2}}{|V(s;t,x,v)-u|} \frac{1}{\alpha(s,X(s;t,x,v),u)} \dd u \dd s \\
& \lesssim \int^t_{\max\{0,t-\tb\}} e^{-\int^t_s \tilde{\nu}(X(\tau;t,x,v),V(\tau;t,x,v)) \dd \tau} \Big( \int_{\mathbb{R}^3} \frac{e^{-C|V(s;t,x,v)-u|^2}}{|V(s;t,x,v)-u|} \frac{1}{\alpha^2(s,X(s;t,x,v),u)} \dd u \Big)^{1/2} \\
& \qquad \times \Big(\int_{|V(s;t,x,v)-u|<\frac{1}{N}} \frac{e^{-C|V(s;t,x,v)-u|^2}}{|V(s;t,x,v)-u|} \dd u\Big)^{1/2} \dd s.
\end{split}
\ee
Since $N \gg 1$ and $\tilde{\nu} (x, v) > \frac{1}{8} \nu (v) > \frac{\nu_0}{8}$, we apply H\"older's inequality in $\dd s$ to obtain
\begin{align*}
\eqref{eq1:nonlocal_to_local}
& \lesssim o(1) \Big( \int^t_{\max\{0,t-\tb\}}  e^{-\int^t_s \frac{\tilde{\nu}(X(\tau;t,x,v),V(\tau;t,x,v))}{2}\dd \tau} \int_{\mathbb{R}^3} \frac{e^{-C|V(s;t,x,v)-u|^2}}{|V(s;t,x,v)-u|} \frac{1}{\alpha^2(s,X(s;t,x,v),u)} \dd u \dd s \Big)^{1/2} \\
& \qquad \times \Big(\int^t_0 e^{-\int_s^t \frac{\tilde{\nu}(X(\tau;t,x,v),V(\tau;t,x,v))}{2} \dd \tau} \dd s \Big)^{1/2} \\
& \lesssim o(1) \frac{1}{\alpha^{1/2}(t,x,v)} \lesssim o(1)\frac{1}{\alpha(t,x,v)},
\end{align*}
where the last inequality follows from the fact that $\alpha(t,x,v) \leq \delta' \ll 1$.
\end{proof}

\begin{lemma} \label{lemma:nabla2_phi_bdd}

For $2 < p < 3 < q$, the following estimates for $\Vert \nabla_x^2 \phi_f \Vert_{C^{0,1-\frac{3}{q}}_x}$ hold:

\begin{enumerate}
\item For the stationary problem \eqref{eqn:h},
\be \notag
\Vert \nabla_x^2 \phi_h \Vert_{C^{0,1-\frac{3}{q}}_x} \lesssim \Vert \nabla_x h \Vert_{L^p_{x,v}} + \Vert \alpha_h \nabla_x h \Vert_{L^\infty_{x,v}} + \Vert w h \Vert_{L^\infty_{x,v}}.
\ee

\item For the dynamical problem \eqref{linear_f_dynamical},
\be \notag
\Vert \nabla_x^2 \phi_f \Vert_{C^{0,1-\frac{3}{q}}_x} \lesssim \Vert \nabla_x (h+f) \Vert_{L^p_{x,v}} + \Vert \alpha_f \nabla_x (h+f) \Vert_{L^\infty_{x,v}} + \Vert wh\Vert_{L^\infty_{x,v}} + \Vert wf\Vert_{L^\infty_{x,v}}.
\ee
\end{enumerate}
\end{lemma}

\begin{proof}

For the stationary problem \eqref{eqn:h}, using the Schauder estimates and Morrey's inequality, together with \eqref{phi_f_C^1} in Lemma~\ref{lemma:phi_x_infinity}, we derive
\be \label{eq1:nabla2_phi_bdd}
\Vert \phi_h \Vert_{C^{2,1-\frac{3}{q}}_x} \lesssim \Big\Vert  \int_{\mathbb{R}^3} h \sqrt{\mu} \dd v\Big\Vert_{C^{0,1-\frac{3}{q}}_x} \lesssim \Big\Vert \int_{\mathbb{R}^3} h \sqrt{\mu} \dd v \Big\Vert_{W^{1,q}_x}
\lesssim_{\Omega} \Vert w h \Vert_{L^\infty_{x,v}} + \Big\Vert \int_{\mathbb{R}^3} \nabla_x h \sqrt{\mu} \dd v \Big\Vert_{L^q_x}.
\ee
Moreover, applying H\"older's inequality with $p < 3 < q$, we obtain
\be \label{eq2:nabla2_phi_bdd}
\begin{split}
\int_{\O}\Big| \int_{\mathbb{R}^3} | \nabla_x h | \sqrt{\mu } \dd v \Big|^q \dd x 
& \lesssim \int_{\O}\Big| \int_{\mathbb{R}^3} |\nabla_x h|^{\frac{p}{q}} |\nabla_x h|^{\frac{q-p}{q}} \sqrt{\mu } \dd v \Big|^q \dd x 
\\& \lesssim \int_{\O} \Big( \int_{\mathbb{R}^3}  |\nabla_x h|^p \dd v  \Big)^{\frac{q}{q}} \Big(\int_{\mathbb{R}^3} |\nabla_x h|^{\frac{q-p}{q-1}} \mu^{1/2} \dd v \Big)^{\frac{(q-1)q}{q}} \dd x \\& \lesssim \Vert \nabla_x h \Vert_{L^p_{x,v}}^p \Vert \alpha_h \nabla_x h \Vert_{L^\infty_{x,v}}^{q-p} \sup_x\Big( \int_{\mathbb{R}^3} \frac{\mu^{1/2}(v)}{\big(\alpha_h (x,v) \big)^{\frac{q-p}{q-1}}} \dd v \Big)^{q-1} 
\\& \lesssim \Vert \nabla_x h \Vert_{L^p_{x,v}}^p \Vert \alpha_h \nabla_x h \Vert_{L^\infty_{x,v}}^{q-p}  \lesssim \Vert \nabla_x h \Vert_{L^p_{x,v}}^q + \Vert \alpha_h \nabla_x h \Vert_{L^\infty_{x,v}}^q,
\end{split}
\ee
where we apply Lemma \ref{lemma:alpha_int} with $\frac{q-p}{q-1} < 1$ in the last line.
Combining \eqref{eq1:nabla2_phi_bdd} and \eqref{eq2:nabla2_phi_bdd}, we conclude the stationary case.
We omit the dynamical problem \eqref{linear_f_dynamical} case, as it is similar to the stationary case.
\end{proof}

We now use an interpolation argument, together with Lemmas \ref{lemma:phi_x_infinity} and \ref{lemma:nabla2_phi_bdd}, to obtain $C^2$ control of $\phi_f$.

\begin{lemma} \label{lemma:phi_C2}

For $2 < p < 3$ and $0 < \beta < \frac{1}{10}$, the following $C^2$ estimate holds:
\be \notag
\Vert \phi_f \Vert_{C_x^2} 
\lesssim o(1) \Vert \phi_f \Vert_{C^{2, \beta}_x} + \Vert \phi_f \Vert_{C^{1, 1- 10 \beta}_x}.
\ee
In particular,
\begin{enumerate}
\item For the stationary problem \eqref{eqn:h},
\be \notag
\Vert \phi_h \Vert_{C_x^2} 
\lesssim o(1) \big( \Vert \nabla_x h \Vert_{L^p_{x,v}} + \Vert \alpha_h \nabla_x h \Vert_{L^\infty_{x,v}} \big) + \Vert w h \Vert_{L^\infty_{x,v}}.
\ee

\item For the dynamical problem \eqref{linear_f_dynamical},
\be \notag
\Vert \phi_f \Vert_{C_x^2} 
\lesssim o(1) \big( \Vert \nabla_x (h+f) \Vert_{L^p_{x,v}} + \Vert \alpha_f \nabla_x (h+f) \Vert_{L^\infty_{x,v}} \big) + \Vert wh\Vert_{L^\infty_{x,v}} + \Vert wf\Vert_{L^\infty_{x,v}}.
\ee
\end{enumerate}
\end{lemma}

\begin{proof}

From Lemmas \ref{lemma:phi_x_infinity} and \ref{lemma:nabla2_phi_bdd}, it suffices to show that there exists a constant $0< \e \ll 1$ such that
\begin{equation} \label{eq1:phi_C2}
\Vert \nabla^2_x \phi_f \Vert_{L^{\infty}_{x}} 
\lesssim \e \Vert \phi_f \Vert_{C^{2, \beta}_x} + \frac{1}{\e} \Vert \phi_f \Vert_{C^{1, 1 - 10 \beta}_x}.
\end{equation}
Let $\tilde{\Omega}$ be an open bounded subset of $\R^3$ that contains the closure of $\Omega$. Suppose $\phi_f (x) \in C^{2, \beta}_x (\Omega)$. From the standard extension theorem, there exists a function $\tilde{\phi}_f (x) \in C^{2, \beta}_x (\tilde{\Omega})$ such that
\begin{equation} \notag
\begin{split}
\tilde{\phi}_f (x) = 0 \ \text{ for every $x \in \R^3 \setminus \tilde{\Omega}$}
\ \text{ and } \ 
\tilde{\phi}_f (x) = \phi_f (x) \ \text{ for every $x \in \Omega$}.
\end{split}
\end{equation}
Moreover, it satisfies that
\begin{equation} \notag
\begin{split}
\Vert \tilde{\phi}_f \Vert_{C^{1, 1 - 10 \beta}_x (\tilde{\Omega})} \leq \Vert \phi_f \Vert_{C^{1, 1 - 10 \beta}_x (\Omega)}
\ \text{ and } \  
\Vert \tilde{\phi}_f \Vert_{C^{2, \beta}_x (\tilde{\Omega})} 
\leq \Vert \phi_f \Vert_{C^{2, \beta}_x (\Omega)}.
\end{split}
\end{equation}
Consider two arbitrary points $x, y \in \R^3$, we compute that for any $0 \leq s \leq 1$,
\begin{equation} \notag
\begin{split}
& \ \ \ \ \big( \frac{y - x}{| y - x |} \cdot \nabla \big) \nabla \tilde{\phi}_f (x)
\\& = \big( \frac{y - x}{| y - x |} \cdot \nabla \big) \nabla \tilde{\phi}_f ( (1-s)x + sy) - \big( \frac{y - x}{| y - x |} \cdot \nabla \big) \big( \nabla \tilde{\phi}_f ( (1-s)x + sy) - \nabla \tilde{\phi}_f ( x) \big)
\\& = \big( \frac{y - x}{| y - x |} \cdot \nabla \big) \nabla \tilde{\phi}_f ( (1-s)x + sy) - | (1-s)x + sy - x|^{\beta} \frac{ \big( \frac{y - x}{| y - x |} \cdot \nabla \big) \big( \nabla \tilde{\phi}_f ( (1-s)x + sy) - \nabla \tilde{\phi}_f ( x) \big) }{| (1-s)x + sy - x|^{\beta}}
\\& = \big( \frac{y - x}{| y - x |} \cdot \nabla \big) \nabla \tilde{\phi}_f ( (1-s)x + sy) - s^{\beta} | y - x|^{\beta}
\frac{ \big( \frac{y - x}{| y - x |} \cdot \nabla \big) \big( \nabla \tilde{\phi}_f ( (1-s)x + sy) - \nabla \tilde{\phi}_f ( x) \big) }{| (1-s)x + sy - x|^{\beta}}.
\end{split}
\end{equation}
Taking an integration on $s \in [0, 1]$ on the above, we get
\begin{equation} \notag
\begin{split}
& \int^{1}_{0} \big( \frac{y - x}{| y - x |} \cdot \nabla \big) \nabla \tilde{\phi}_f (x) \dd s
\\& = \int^{1}_{0} \big( \frac{y - x}{| y - x |} \cdot \nabla \big) \nabla \tilde{\phi}_f ( (1-s)x + sy) \dd s
- \int^{1}_{0} s^{\beta} | y - x|^{\beta} \frac{ \big( \frac{y - x}{| y - x |} \cdot \nabla \big) \big( \nabla \tilde{\phi}_f ( (1-s)x + sy) - \nabla \tilde{\phi}_f ( x) \big) }{| (1-s)x + sy - x|^{\beta}} \dd s.
\end{split}
\end{equation}
This further implies that
\begin{equation} \label{eq2:phi_C2}
\begin{split}
& \ \ \ \ \Big| \big( \frac{y - x}{| y - x |} \cdot \nabla \big) \nabla \tilde{\phi}_f (x) \Big|
\\& \leq \Big| \int^{1}_{0} \big( \frac{y - x}{| y - x |} \cdot \nabla \big) \nabla \tilde{\phi}_f ( (1-s)x + sy) \dd s \Big|
+ | y - x|^{\beta} \Vert \tilde{\phi}_f \Vert_{C^{2, \beta}_x (\tilde{\Omega})} \int^{1}_{0} s^{\beta} \dd s
\\& \leq \Big| \frac{1}{| y - x |} \int^{1}_{0} \big( (y - x) \cdot \nabla \big) \nabla \tilde{\phi}_f ( (1-s)x + sy) \dd s \Big|
+ \frac{1}{1 + \beta} | y - x|^{\beta} \Vert \tilde{\phi}_f \Vert_{C^{2, \beta}_x (\tilde{\Omega})}
\\& = \Big| \frac{1}{| y - x |} \big( \nabla \tilde{\phi}_f (y) - \nabla \tilde{\phi}_f (x) \big) \Big|
+ \frac{1}{1 + \beta} | y - x|^{\beta} \Vert \tilde{\phi}_f \Vert_{C^{2, \beta}_x (\tilde{\Omega})}
\\& \leq \frac{1}{| y - x |^{10 \beta}} \Vert \tilde{\phi}_f \Vert_{C^{1, 1 - 10 \beta}_x (\tilde{\Omega})} + \frac{1}{1 + \beta} | y - x|^{\beta} \Vert \tilde{\phi}_f \Vert_{C^{2, \beta}_x (\tilde{\Omega})}.
\end{split}
\end{equation}
Let $0< \e \ll 1$. We consider all $y$ such that $| y - x | = \e^{1 / \beta}$. From the directional derivative, we obtain
\[
\Big| \big( \frac{y - x}{| y - x |} \cdot \nabla \big) \nabla \tilde{\phi}_f (x) \Big| 
\lesssim \Vert \nabla^2  \phi_f \Vert_{L^{\infty}_{x}} 
\ \text{ for all $x, y \in \Omega$}.
\]
Combining this with \eqref{eq2:phi_C2}, we conclude \eqref{eq1:phi_C2}.
\end{proof}

\section{
\texorpdfstring{A priori $L^2-L^\infty$ estimate}{L2-Linfty estimate}
}
\label{sec:L2Linfty_estimate}

In this section, we establish a priori $L^2$ and $L^\infty$ estimates for the stationary problem \eqref{eqn:h} in Proposition~\ref{prop:f_L2} and Proposition~\ref{prop:wf_Linfty}, respectively. 
These estimates will be used in the well-posedness arguments in Section~\ref{sec:stationary_uniqueness} and Section~\ref{sec:existence}.

Throughout this section, when applying the lemmas from Section \ref{sec:prelim}, we replace $f,\phi_f$ by the stationary solution $h,\phi_h$ defined in \eqref{eqn:h}.
We also recall from \cite{R} that the kernel of the linearized collision operator $\mathcal{L}$ is given by
\[
\ker (\mathcal{L})
=
\spn \left\{
\sqrt{\mu}, \ v\sqrt{\mu}, \ \frac{|v|^2-3}{2}\sqrt{\mu}
\right\}.
\]
Denote by $\mathbf{P}$ the projection onto $\ker (\mathcal{L})$, so that $\mathbf{P} \mathcal{L} = \mathcal{L}  \mathbf{P} = 0$. More precisely, for any $h \in L^2(\Omega \times \mathbb{R}^3)$,
\be \label{def:Ph}
\mathbf{P} h
= a (x) \sqrt{\mu} + (\mathbf{b} (x) \cdot v)\sqrt{\mu} + c (x) \frac{|v|^2-3}{2}\sqrt{\mu},
\ee
where
\be \notag
a (x) = \int_{\mathbb{R}^3}
h\sqrt{\mu} \dd v,
\qquad
\mathbf{b} (x) = \int_{\mathbb{R}^3}
v h\sqrt{\mu} \dd v,
\qquad
c (x) = \int_{\mathbb{R}^3}
\frac{|v|^2-3}{2}
h\sqrt{\mu} \dd v.
\ee
Moreover, $\mathbf{P} h$ satisfies that
\[
\|\mathbf{P}h\|_{L^2_{x,v}}^2 = \|a\|_{L^2_x}^2 + \|\mathbf{b}\|_{L^2_x}^2 + \|c\|_{L^2_x}^2 \sim \|\mathbf{P}h\|_{L^2_{x,\nu}}^2.
\]
Accordingly, applying the mass, momentum, and energy conservation laws to an equation corresponds to taking the velocity integration against $\sqrt{\mu}$, $v\sqrt{\mu}$, and $\frac{|v|^2-3}{2}\sqrt{\mu}$, respectively.

\smallskip

We first show an energy estimate on the microscopic component $(\mathbf{I}-\mathbf{P})h$ and the boundary term $| h|_{L^2_{\gamma_+}}^2$.

\begin{lemma} \label{lemma:l2_energy}

Let $(h, \phi_h)$ be a solution to \eqref{eqn:h}. Then the following $L^2$ energy estimate holds:
\be \notag
| h|_{L^2_{\gamma_+}}^2 + \Vert (\mathbf{I}-\mathbf{P})h \Vert_{L^2_{x,\nu}}^2 
\lesssim |f_b|_{L^2_{\gamma_-}}^2 + \big( \Vert w h \Vert_{L^\infty_{x,v}} + \Vert \nabla_x \phi_E\Vert_{L^\infty_x} \big) \Vert h\Vert_{L^2_{x,\nu}}^2 + \Vert \nu^{-1/2}\Gamma(h,h)\Vert_{L^2_{x,v}}^2.
\ee
\end{lemma}

\begin{proof}

By multiplying the first equation in \eqref{eqn:h} by $h$, we obtain the following $L^2$ energy estimate:
\begin{align}
& |h|_{L^2_{\gamma_+}}^2 + \Vert e^{-\phi_E/2}(\mathbf{I}-\mathbf{P})h\Vert_{L^2_{x,\nu}}^2 
\lesssim  |f_b|_{L^2_{\gamma_-}}^2 + o(1)\Vert e^{-\phi_E/2}(\mathbf{I}-\mathbf{P})h\Vert_{L^2_{x,\nu}}^2 + \Vert \nu^{-1/2}e^{\phi_E/4}\Gamma(h,h)\Vert_{L^2_{x,v}}^2 
\label{eq1:l2_energy} \\
& \qquad + \Vert  \nabla_x \phi_h\Vert_{L^\infty_x} \Vert \nu^{1/2} h\Vert_{L^2_{x,v}}^2+ \Big|\iint_{\O\times \mathbb{R}^3}  (v\cdot \nabla_x \phi_h ) e^{-\phi_E/2}\sqrt{\mu} h \dd x \dd v  \Big|.
\label{eq2:l2_energy}
\end{align}

From \eqref{phi_f_C^1}, we control the first term in \eqref{eq2:l2_energy} by
\be \notag
\Vert \nabla_x \phi_h\Vert_{L^\infty_x} \Vert \nu^{1/2} h\Vert_{L^2_{x,v}}^2 \lesssim \Vert wh\Vert_{L^\infty_{x,v}}  \Vert h\Vert_{L^2_{x,\nu}}^2.
\ee
For the second term in \eqref{eq2:l2_energy}, we compute
\be \label{eq3:l2_energy}
\begin{split}
& \iint_{\O\times \mathbb{R}^3} (v\cdot \nabla_x \phi_h)e^{-\phi_E/2} \sqrt{\mu} h \dd x \dd v \\
& = -\iint_{\O\times \mathbb{R}^3} \phi_h\sqrt{\mu} e^{-\phi_E/2} (v\cdot \nabla_x)h \dd x \dd v + \iint_{\O\times \mathbb{R}^3} \phi_h \sqrt{\mu} \frac{v \cdot \nabla_x \phi_E}{2} e^{-\phi_E/2}h \dd x \dd v,
\end{split}
\ee
where the boundary term vanishes due to the zero Dirichlet boundary condition for $\phi_h$. Since $h$ solves \eqref{eqn:h}, we obtain
\begin{align*}
& \iint_{\O\times \mathbb{R}^3} \phi_h\sqrt{\mu} e^{-\phi_E/2} (v\cdot \nabla_x)h \dd x \dd v \\
& = \iint_{\O\times \mathbb{R}^3}  \phi_h\sqrt{\mu}e^{-\phi_E/2}(\nabla_x (\phi_h+\phi_E)\cdot \nabla_{v}h)    \dd x \dd v - \iint_{\O\times \mathbb{R}^3} \phi_h \sqrt{\mu}e^{-\phi_E/2} \frac{v\cdot \nabla_x  \phi_h}{2}h  \dd x \dd v \\
& = \iint_{\O\times \mathbb{R}^3} \phi_he^{-\phi_E/2}\nabla_x \phi_h  \cdot \nabla_{v}(\sqrt{\mu}h)   \dd x \dd v + \iint_{\O\times \mathbb{R}^3} \phi_h \sqrt{\mu} e^{-\phi_E/2} \nabla_x \phi_E \cdot \nabla_{v} h \dd x \dd v  \\
& = \iint_{\O\times \mathbb{R}^3} \phi_h \sqrt{\mu} e^{-\phi_E/2} \nabla_x \phi_E \cdot \nabla_{v} h \dd x \dd v = -\iint_{\O\times \mathbb{R}^3} \phi_h \sqrt{\mu}e^{-\phi_E/2} \frac{v\cdot \nabla_x \phi_E}{2}h\dd x\dd v.
\end{align*}
Therefore, the right-hand side of \eqref{eq3:l2_energy} can be combined and bounded by
\be \label{eq4:l2_energy}
\Big|  \iint_{\O\times \mathbb{R}^3}  \phi_h \sqrt{\mu}v \cdot \nabla_x \phi_E e^{-\phi_E/2}h\dd x \dd v\Big| \lesssim \Vert \nabla_x \phi_E\Vert_{L^\infty_x} \Vert \phi_h\Vert_{L^2_x}\Vert h\Vert_{L^2_{x,v}} \lesssim \Vert \nabla_x \phi_E\Vert_{L^\infty_x}\Vert h\Vert_{L^2_{x,v}}^2.
\ee
Since $\|\nabla_x \phi_E\|_{L^\infty_x} \ll 1$, it follows that $|e^{-\phi_E/2}-1| + |e^{-\phi_E}-1| \ll 1$.
Combining this with \eqref{eq1:l2_energy}-\eqref{eq4:l2_energy}, we conclude the lemma.
\end{proof}

Next, we control the macroscopic component $\mathbf{P} h$ in the following lemma:

\begin{lemma} \label{lemma:macro_l2}

Let $(h, \phi_h)$ be a solution to \eqref{eqn:h}. Then the following estimate for $\Vert \mathbf{P}h \Vert_{L^2_{x,v}}$ holds:
\begin{align*}
& \Vert \mathbf{P} h \Vert_{L^2_{x,v}}^2 + \Vert \nabla_x  \phi_h \Vert_{L^2_x}^2 \\
& \lesssim \Vert (\mathbf{I}-\mathbf{P})h\Vert_{L^2_{x,\nu}}^2 + |h|_{L^2_{\gamma_+}}^2 + |f_b|_{L^2_{\gamma_-}}^2 + \Vert \nu^{-1/2}\Gamma(h,h)\Vert_{L^2_{x,v}}^2 + \big( \Vert wh\Vert_{L^\infty_{x,v}} + \Vert \nabla_x \phi_E \Vert_{L^\infty_x} \big) \Vert h \Vert_{L^2_{x,v}}^2.
\end{align*}
\end{lemma}

\begin{proof}

Given a test function $\psi \in L^2(\Omega \times \mathbb{R}^3) \cap L^2(\partial \Omega \times \mathbb{R}^3; d\gamma)$ and $v\cdot \nabla_x \psi - \nabla_x(\phi_h + \phi_E)\cdot \nabla_v \psi \in L^2(\Omega \times \mathbb{R}^3)$, the weak formulation of \eqref{eqn:h} is given by
\be \label{weak_formulation}
\begin{split}
& - \underbrace{\iint_{\O\times \mathbb{R}^3} \mathbf{P}h (v\cdot \nabla_x \psi) \dd x \dd v}_{\eqref{weak_formulation}_1}  - \underbrace{\iint_{\O\times \mathbb{R}^3} (\mathbf{I}-\mathbf{P}) h ( v \cdot \nabla_x \psi) \dd x \dd v}_{\eqref{weak_formulation}_2}  +\underbrace{\int_{\gamma} \psi h \dd \gamma}_{\eqref{weak_formulation}_3}
\\& \qquad + \underbrace{\iint_{\O\times \mathbb{R}^3} \sqrt{\mu} h \nabla_x \phi_h \cdot  \nabla_{v} \big[\frac{1}{\sqrt{\mu}}\psi \big] \dd x \dd v}_{\eqref{weak_formulation}_4} + \underbrace{\iint_{\O\times \mathbb{R}^3} h \nabla_x \phi_E \cdot \nabla_{v} \psi \dd x \dd v}_{\eqref{weak_formulation}_5} + \underbrace{\iint_{\O\times \mathbb{R}^3} e^{-\phi_E}\mathcal{L}h \psi \dd x \dd v}_{\eqref{weak_formulation}_6}  
\\& = \underbrace{\iint_{\O\times \mathbb{R}^3} -\psi (v \cdot \nabla_x  \phi_h)e^{-\phi_E/2} \sqrt{\mu} \dd x \dd v}_{\eqref{weak_formulation}_7} + \underbrace{\iint_{\O\times \mathbb{R}^3} \psi e^{-\phi_E/2} \Gamma(h,h) \dd x \dd v}_{\eqref{weak_formulation}_8}.  
\end{split}
\ee
Recall from \eqref{def:Ph} that $\mathbf{P} h
= a (x) \sqrt{\mu} + (\mathbf{b} (x) \cdot v)\sqrt{\mu} + c (x) \frac{|v|^2-3}{2}\sqrt{\mu}$.
To estimate the macroscopic coefficients $a$, $\mathbf{b}$, and $c$, we introduce suitable test functions $\psi_c$, $\psi_1$, $\psi_2$, $\psi_3$, and $\psi_a$.

\smallskip

(1) \textit{Estimate for $c$.}
We choose the test function $\psi_c$ as follows:
\begin{align*}
\begin{cases}
& \psi_c = (|v|^2-5)\sqrt{\mu} (v \cdot \nabla_x \phi_c) \perp \ker \mathcal{L}, \\[5pt]
& - \Delta_x \phi_c = c(x), \quad \phi_c |_{\p\O} = 0.
\end{cases}
\end{align*}
Using the elliptic regularity and the trace theorem, we obtain
\begin{align*}
& |\phi_c|_{H^1(\p\O)} \lesssim \Vert  \phi_c \Vert_{H^2_x} \lesssim \Vert c\Vert_{L^2_x}.
\end{align*}
We now compute the terms in \eqref{weak_formulation}:
\begin{align*}
\eqref{weak_formulation}_1 & = -\iint_{\O\times \mathbb{R}^3} c\frac{|v|^2-3}{2} \mu (|v|^2-5) |v|^2 \Delta_x \phi_c \dd v \dd x 
= -5\int_{\O} c\Delta_x \phi_c \dd x = 5\Vert c\Vert_{L^2_x}^2, \\
|\eqref{weak_formulation}_2| & \lesssim \Vert (\mathbf{I}-\mathbf{P})h\Vert_{L^2_x} \Vert \phi_c \Vert_{H^2_x} \lesssim \Vert (\mathbf{I}-\mathbf{P})h\Vert_{L^2_x}\Vert c\Vert_{L^2_x}, \\
|\eqref{weak_formulation}_3| & = \Big|\int_{\gamma_+} (|v|^2-5)\sqrt{\mu}(v\cdot \nabla_x \phi_c) (n(x)\cdot v) h \dd v \dd S_x + \int_{\gamma_-} (|v|^2-5)\sqrt{\mu}(v\cdot \nabla_x \phi_c) (n(x)\cdot v) f_b \dd v \dd S_x \Big| \\
& \lesssim \big( |h|_{L^2_{\gamma_+}} + |f_b|_{L^2_{\gamma_-}} \big) | \nabla_x \phi_c |_{L^2(\p\O)} \lesssim \big( |h|_{L^2_{\gamma_+}} + |f_b|_{L^2_{\gamma_-}} \big) \Vert c\Vert_{L^2_x}, \\
|\eqref{weak_formulation}_4| & \lesssim \Vert \nabla_x  \phi_h \Vert_{L^\infty_x} \Vert h\Vert_{L^2_{x,v}} \Vert \nabla_x \phi_c\Vert_{L^2_x} \lesssim \Vert wh\Vert_{L^\infty_{x,v}}  \Vert h\Vert_{L^2_{x,v}} \Vert c\Vert_{L^2_x}  \lesssim \Vert wh\Vert_{L^\infty_{x,v}}  \Vert h\Vert_{L^2_{x,v}}^2, \\
|\eqref{weak_formulation}_5| & \lesssim \Vert \nabla_x \phi_E\Vert_{L^\infty_x} \Vert h\Vert_{L^2_{x,v}}\Vert \nabla_x \phi_c\Vert_{L^2_x} \lesssim \Vert \nabla_x \phi_E\Vert_{L^\infty_x} \Vert h\Vert_{L^2_{x,v}}^2, \\
|\eqref{weak_formulation}_6| & \lesssim \Vert e^{-\phi_E}(\mathbf{I}-\mathbf{P})h\Vert_{L^2_{x,v}} \Vert \nabla_x \phi_c\Vert_{L^2_x} \lesssim \Vert e^{-\phi_E}(\mathbf{I}-\mathbf{P})h\Vert_{L^2_{x,v}} \Vert c\Vert_{L^2_x}, \\
|\eqref{weak_formulation}_8| & \lesssim \Vert \nabla_x \phi_c\Vert_{L^2_x} \Vert \nu^{-1/2}e^{-\phi_E/2}\Gamma(h,h)\Vert_{L^2_{x,v}} \lesssim \Vert c\Vert_{L^2_x} \Vert \nu^{-1/2}e^{-\phi_E/2}\Gamma(h,h)\Vert_{L^2_{x,v}}.
\end{align*}
Since $\psi_c \perp \ker \mathcal{L}$, we further derive that
\[
\eqref{weak_formulation}_7 =0.
\]
Combining with $|e^{-\phi_E} - 1| \ll 1$, we obtain the estimate for $c$ as follows:
\begin{align*}
    & \Vert c\Vert_{L^2_{x}}^2 \lesssim \Vert (\mathbf{I}-\mathbf{P})h\Vert_{L^2_{x,v}}^2 + \big( \Vert wh\Vert_{L^\infty_{x,v}} + \Vert \nabla_x \phi_E\Vert_{L^\infty_{x}} \big) \Vert h\Vert_{L^2_{x,v}}^2 + |h|_{L^2_{\gamma_+}}^2 + |f_b|_{L^2_{\gamma_-}}^2 + \Vert \nu^{-1/2}\Gamma(h,h)\Vert_{L^2_{x,v}}^2. 
\end{align*}

\smallskip

(2) \textit{Estimate for $\mathbf{b} = (b_1, b_2, b_3)$.} 
To estimate $b_1$, we introduce the test function $\psi_1$ as follows:
\begin{align*}
\begin{cases}
& \psi_1 =  \frac{3}{2}\Big(|v_1|^2 - \frac{|v|^2}{3} \Big)\sqrt{\mu} \p_{x_1} \phi_1 + v_1v_2\sqrt{\mu} \p_{x_2}\phi_1 + v_1v_3 \sqrt{\mu} \p_{x_3}\phi_1  \perp \ker \mathcal{L}, \\[5pt]
& - \Delta_x \phi_1 - \p_{x_1}^2 \phi_1  = b_1, \quad \phi_1|_{\p\O} = 0.
\end{cases}
\end{align*}
From the elliptic regularity and the trace theorem, we obtain
\begin{align*}
|\nabla_x \phi_1|_{L^2(\p\O)} \lesssim \Vert \phi_1\Vert_{H^2_x} \lesssim \Vert b_1 \Vert_{L^2_x}.
\end{align*}
This, together with $\psi_1 \perp \ker \mathcal{L}$, allows us to compute the terms in \eqref{weak_formulation}:
\begin{align*}
\eqref{weak_formulation}_1 & = -\int_{\O} b_1 (\Delta_x \phi_1 + \p_{x_1}^2 \phi_1 )\dd x= \Vert b_1\Vert_{L^2_x}, \\
| \eqref{weak_formulation}_2| & \lesssim \Vert (\mathbf{I}-\mathbf{P})h\Vert_{L^2_{x,v}} \Vert \phi_1 \Vert_{H^2_x} \lesssim \Vert (\mathbf{I}-\mathbf{P})h\Vert_{L^2_{x,v}}\Vert b_1\Vert_{L^2_x}, \\
|\eqref{weak_formulation}_3| & = \Big|\int_{\gamma_+}  \psi h (n(x)\cdot v) \dd v \dd S_x + \int_{\gamma_-} \psi f_b (n(x)\cdot v) \dd v \dd S_x \Big| \\
& \lesssim \big( |h|_{L^2_{\gamma_+}} + |f_b|_{L^2_{\gamma_-}} \big) |\nabla_x \phi_1|_{L^2(\p\O)} \lesssim \big( |h|_{L^2_{\gamma_+}} + |f_b|_{L^2_{\gamma_-}} \big) \Vert b_1\Vert_{L^2_x}, \\
|\eqref{weak_formulation}_4| & \lesssim \Vert \nabla_x  \phi_h \Vert_{L^\infty_x}  \Vert h\Vert_{L^2_{x,v}} \Vert \nabla_x \phi_1\Vert_{L^2_x} \lesssim \Vert wh\Vert_{L^\infty_{x,v}}  \Vert h\Vert_{L^2_{x,v}} \Vert b_1\Vert_{L^2_x} \lesssim \Vert wh\Vert_{L^\infty_{x,v}}\Vert h\Vert_{L^2_{x,v}}^2, \\
|\eqref{weak_formulation}_5| & \lesssim \Vert \nabla_x \phi_E\Vert_{L^\infty_x} \Vert h\Vert_{L^2_{x,v}}\Vert \nabla_x \phi_1\Vert_{L^2_x} \lesssim \Vert \nabla_x \phi_E\Vert_{L^\infty_x} \Vert h\Vert_{L^2_{x,v}}^2, \\
|\eqref{weak_formulation}_6| & \lesssim \Vert e^{-\phi_E}(\mathbf{I}-\mathbf{P})h\Vert_{L^2_{x,v}} \Vert \nabla_x \phi_1\Vert_{L^2_x} \lesssim \Vert e^{-\phi_E}(\mathbf{I}-\mathbf{P})h\Vert_{L^2_{x,v}} \Vert b_1\Vert_{L^2_x}, \\
|\eqref{weak_formulation}_8| & \lesssim \Vert \nabla_x \phi_1\Vert_{L^2_x} \Vert \nu^{-1/2}\Gamma(h,h)\Vert_{L^2_{x,v}} \lesssim \Vert b_1\Vert_{L^2_x} \Vert \nu^{-1/2}\Gamma(h,h)\Vert_{L^2_{x,v}},
\end{align*}
and $\eqref{weak_formulation}_7 =0$. Therefore, we derive the estimate for $b_1$ as follows:
\[
\Vert b_1 \Vert_{L^2_x}^2 \lesssim \Vert (\mathbf{I}-\mathbf{P})h\Vert_{L^2_{x,v}}^2 + \big( \Vert wh\Vert_{L^\infty_{x,v}} + \Vert \nabla_x \phi_E\Vert_{L^\infty_{x}} \big) \Vert h\Vert_{L^2_{x,v}}^2 + |h|_{L^2_{\gamma_+}}^2 + |f_b|_{L^2_{\gamma_-}}^2 + \Vert \nu^{-1/2}\Gamma(h,h)\Vert_{L^2_{x,v}}^2.
\]

For the estimates of $b_2$ and $b_3$, we use analogous test functions $\psi_2$ and $\psi_3$ as follows:
\begin{align*}
& \begin{cases}
    \psi_2 = v_1v_2 \sqrt{\mu} \p_{x_1} \phi_2 + \frac{3}{2}\Big(|v_2|^2 - \frac{|v|^2}{3} \Big)\sqrt{\mu} \p_{x_2}\phi_2 + v_2v_3 \p_{x_3} \phi_2 \perp \ker\mathcal{L}, \\[5pt]
    - \Delta_x \phi_2 - \p_{x_2}^2 \phi_2 = b_2, \quad \phi_2|_{\p\O} = 0, 
\end{cases} 
\\& \begin{cases}
    \psi_3 = v_1 v_3 \sqrt{\mu} \p_{x_1} \phi_3 + v_2v_3\sqrt{\mu}\p_{x_2}\phi_3  + \frac{3}{2}\Big(|v_3|^2 - \frac{|v|^2}{3} \Big)\sqrt{\mu} \p_{x_3}\phi_3\perp \ker \mathcal{L}, \\[5pt]
    -\Delta_x \phi_3 - \p_{x_3}^2 \phi_3 = b_3, \quad \phi_3|_{\p\O} = 0.
\end{cases}
\end{align*} 
Analogously, we obtain estimates for $b_2$ and $b_3$, and hence the following bound for $\mathbf{b}$:
\be \notag
\Vert \mathbf{b}\Vert_{L^2_x}^2 \lesssim \Vert (\mathbf{I}-\mathbf{P})h\Vert_{L^2_{x,v}}^2 + \big( \Vert w h \Vert_{L^\infty_{x,v}} + \Vert \nabla_x \phi_E\Vert_{L^\infty_{x}} \big) \Vert h\Vert_{L^2_{x,v}}^2 + |h|_{L^2_{\gamma_+}}^2 + |f_b|_{L^2_{\gamma_-}}^2 + \Vert \nu^{-1/2}\Gamma(h,h)\Vert_{L^2_{x,v}}^2.
\ee

\smallskip

(3) \textit{Estimate for $a$}. 
We choose the test function $\psi_a$ as follows:
\be \label{eq1:macro_l2}
\begin{cases}
& \psi_a = -(|v|^2-10)\sqrt{\mu}(v \cdot \nabla_x  \phi_a), \\[5pt]
& - \Delta_x \phi_a = a(x), \quad \phi_a |_{\p\O} = 0.
\end{cases}
\ee
Similarly, we have $|\nabla_x \phi_a|_{L^2_{\p\O}} \lesssim \Vert \phi_a\Vert_{H^2_x} \lesssim \Vert a\Vert_{L^2_x}$. We now estimate the terms in \eqref{weak_formulation}:
\begin{align*}
\eqref{weak_formulation}_1 & = -\iint_{\O\times \mathbb{R}^3} a \mu (|v|^2-10) |v|^2 \Delta_x \phi_a\dd v = -5\int_{\O} a\Delta_x \phi_a \dd x = 5\Vert a\Vert_{L^2_x}^2, \\
| \eqref{weak_formulation}_2 | & \lesssim \Vert (\mathbf{I}-\mathbf{P})h\Vert_{L^2_x} \Vert \phi_a \Vert_{H^2_x} \lesssim \Vert (\mathbf{I}-\mathbf{P})h\Vert_{L^2_{x,v}}\Vert a\Vert_{L^2_x}, \\
|\eqref{weak_formulation}_3| & = \Big|\int_{\gamma_+}  (|v|^2-10)\sqrt{\mu}(v\cdot \nabla_x \phi_a)  h (n(x)\cdot v)  \dd v \dd S_x + \int_{\gamma_-}  (|v|^2-10)\sqrt{\mu}(v\cdot \nabla_x \phi_a)f_b (n(x)\cdot v)  \dd v \dd S_x  \Big| \\
& \lesssim \big( |h|_{L^2_{\gamma_+}} + |f_b|_{L^2_{\gamma_-}} \big) | \nabla_x \phi_a|_{H^1(\p\O)} \lesssim \big( |h|_{L^2_{\gamma_+}} + |f_b|_{L^2_{\gamma_-}} \big) \Vert a\Vert_{L^2_x}, \\
|\eqref{weak_formulation}_4| & \lesssim \Vert \nabla_x  \phi_h+\nabla_x \phi_E\Vert_{L^\infty_x}  \Vert h\Vert_{L^2_x} \Vert \nabla_x \phi_a \Vert_{L^2_x} 
\lesssim \big( \Vert wh\Vert_{L^\infty_{x,v}} + \Vert \nabla_x \phi_E\Vert_{L^\infty_x} \big) \Vert h \Vert_{L^2_{x,v}}^2, \\
|\eqref{weak_formulation}_5| & \lesssim \Vert \nabla_x \phi_E\Vert_{L^\infty_x} \Vert h\Vert_{L^2_{x,v}}\Vert \nabla_x \phi_a\Vert_{L^2_x} \lesssim \Vert \nabla_x \phi_E\Vert_{L^\infty_x} \Vert h\Vert_{L^2_{x,v}}^2, \\
|\eqref{weak_formulation}_6| & \lesssim \Vert e^{-\phi_E}(\mathbf{I}-\mathbf{P})h\Vert_{L^2_{x,v}} \Vert \nabla_x \phi_a\Vert_{L^2_x} \lesssim \Vert e^{-\phi_E}(\mathbf{I}-\mathbf{P})h\Vert_{L^2_{x,v}} \Vert a\Vert_{L^2_x}, \\
| \eqref{weak_formulation}_8 | & \lesssim \Vert \nabla_x \phi_a\Vert_{L^2_x} \Vert \nu^{-1/2}\Gamma(h,h)\Vert_{L^2_{x,v}} \lesssim \Vert a\Vert_{L^2_x} \Vert \nu^{-1/2}\Gamma(h,h)\Vert_{L^2_{x,v}}.
\end{align*}

From \eqref{eq1:macro_l2}, the difference of $\phi_h - \phi_a$ satisfies
\be \notag
- \Delta_x ( \phi_h-\phi_a)  = (e^{-\phi_E/2}-1) a(x), \quad \phi_h-\phi_a |_{\p\O} = 0.
\ee
The elliptic regularity and $|e^{-\phi_E} - 1| \ll 1$ imply that
\begin{align*}
& \Vert  \phi_h-\phi_a\Vert_{H^2_x} \lesssim o(1)\Vert a\Vert_{L^2_{x,v}}.
\end{align*}
For $\eqref{weak_formulation}_7$, we compute that
\begin{align*}
\eqref{weak_formulation}_7 & = \iint_{\O\times \mathbb{R}^3} (v \cdot \nabla_x  \phi_h) (|v|^2-5-5)( v \cdot \nabla_x \phi_a) e^{-\phi_E/2} \mu  \dd x \dd v \\
& = -5\iint_{\O\times \mathbb{R}^3}  (v \cdot \nabla_x  \phi_h )  (v \cdot \nabla_x \phi_a)e^{-\phi_E/2}\mu(v) \dd x \dd v \\
& = -5 \int_{\O} \nabla_x  \phi_h   \cdot \nabla_x \phi_a e^{-\phi_E/2} \dd x  = -5 \Vert e^{-\phi_E/4}\nabla_x   \phi_h\Vert_{L^2_x}^2 -5\int_\O e^{-\phi_E/2}\nabla_x  \phi_h \cdot (\nabla_x \phi_a - \nabla_x  \phi_h) \dd x,
\end{align*}
and the last term is bounded by
\begin{align*}
    &   5\Big|\int_\O \nabla_x  \phi_h \cdot (\nabla_x \phi_a - \nabla_x  \phi_h) \dd x \Big| \lesssim \Vert \nabla_x  \phi_h\Vert_{L^2_x} \Vert \nabla_x \phi_a -\nabla_x  \phi_h\Vert_{L^2_x} \lesssim o(1) \Vert a \Vert_{L^2_x} \Vert  \nabla_x \phi_h \Vert_{L^2_x}.
\end{align*}
Therefore, we derive the estimate for $a$ as follows:
\begin{align*}
\Vert a \Vert_{L^2_{x}}^2 + \Vert \nabla_x  \phi_h \Vert_{L^2_x}^2 & \lesssim \Vert (\mathbf{I}-\mathbf{P})h\Vert_{L^2_{x,v}}^2 + \big( \Vert wh\Vert_{L^\infty_{x,v}} + \Vert \nabla_x \phi_E\Vert_{L^\infty_{x}} \big) \Vert h\Vert_{L^2_{x,v}}^2 \\
& \quad + |h|_{L^2_{\gamma_+}}^2 + |f_b|_{L^2_{\gamma_-}}^2 + \Vert \nu^{-1/2}\Gamma(h,h)\Vert_{L^2_{x,v}}^2. 
\end{align*}
Collecting the estimates for $a$, $\mathbf{b}$, and $c$, we conclude the lemma.
\end{proof}

Combining Lemma \ref{lemma:l2_energy} and Lemma \ref{lemma:macro_l2}, we obtain the following $L^2$ estimate.

\begin{proposition} \label{prop:f_L2}

Let $(h, \phi_h)$ be a solution to \eqref{eqn:h} satisfying $\|h\|_{L^2_{x,\nu}} + |h|_{L^2_{\gamma_+}}^2 < \infty$. Then we have
\be \notag
\Vert h \Vert_{L^2_{x,\nu}} + |h|_{L^2_{\gamma_+}} 
\lesssim \Vert \nu^{-1/2}\Gamma(h,h)\Vert_{L^2_{x,v}} + \big( \Vert w h \Vert_{L^\infty_{x,v}}^{1/2} + \Vert \nabla_x \phi_E\Vert_{L^\infty_x}^{1/2} \big) \Vert h \Vert_{L^2_{x,\nu}} + |f_b|_{L^2_{\gamma_-}}.
\ee
\end{proposition}


Finally, we apply the method of characteristics to derive a weighted $L^\infty_{x,v}$ estimate for $h$.

\begin{proposition} \label{prop:wf_Linfty}

Let $h$ be a solution to \eqref{eqn:h}. Assume that
\be \notag
\Vert wh \Vert_{L^\infty_{x,v}} + \Vert \nabla_x (\phi_h+\phi_E) \Vert_{L^\infty_x}+ \Vert \nabla_x^2 (\phi_h+\phi_E) \Vert_{L^\infty_{x}} \ll 1.
\ee
Then the following weighted $L^\infty$ estimate holds:
\be \notag
\Vert wh \Vert_{L^\infty_{x,v}} 
\lesssim | wf_b|_{L^\infty_{\p\O,v}}.
\ee
\end{proposition}

\begin{proof}

For simplicity, we write $(X(s), V(s)) := (X(s;t,x,v), V(s;t,x,v))$ throughout the proof.
We rewrite the first equation in \eqref{eqn:h} as
\be \label{eq2:wf_Linfty}
\begin{split}
& v \cdot \nabla_x h -\nabla_x ( \phi_h+\phi_E) \cdot \nabla_{v} h + \Big( e^{-\phi_E}\nu(v) + \frac{v\cdot \nabla_x  \phi_h}{2} \Big) h \\
& =e^{-\phi_E} Kh - (v\cdot \nabla_x \phi_h)e^{-\phi_E/2}\sqrt{\mu} + e^{-\phi_E/2}\Gamma(h,h).
\end{split}
\ee
Since $\|\nabla_x \phi_h\|_{L^\infty_x} \ll 1$ and $|e^{-\phi_E}-1| \ll 1$, the damping factor admits the lower bound
\be \label{eq3:wf_Linfty}
e^{-\phi_E}\nu(v) + \frac{v\cdot \nabla_x  \phi_h}{2} \geq (1-o(1))\nu(v) - o(1)|v| \geq \frac{\nu(v)}{2}.
\ee
Moreover, since $\|\nabla_x (\phi_h+\phi_E)\|_{L^\infty_x} \ll 1$ and $\nu(v)\geq \nu_0 >0$, we may choose $t>0$ such that
\be \label{eq4:wf_Linfty}
\Vert \nabla_x (\phi_f + \phi_E) \Vert_{L^\infty_{x}} t +
\| \nabla_x^2 (\phi_h + \phi_E) \|_{L^\infty_x} (t)^2 e^{t} \ll 1
\ \text{ and } \
e^{-\frac{\nu_0}{4}t} \ll 1.
\ee
Applying the method of characteristics, $h$ satisfies
\begin{align*}
& | w(v)h(x,v) | \leq \mathbf{1}_{\tb(x,v) \geq t} e^{-\int^t_0 \frac{\nu(V(s))}{2}\dd s} \frac{w(v)}{w(V(0))} w(V(0))|h(X(0),V(0))| \\
& \quad + \mathbf{1}_{\tb(x,v)<t} e^{-\int^t_{t-\tb(x,v)} \frac{\nu(V(s))}{2}\dd s} \frac{w(v)}{w(\vb)} w(\vb)|f_b(\xb,\vb)| \\
& \quad + \int^t_{\max\{0,t-\tb\}} e^{-\int^t_s \frac{\nu(V(\tau))}{2}\dd \tau} \frac{w(v)}{w(V(s))} \int_{\mathbb{R}^3} \dd u e^{-\phi_E(X(s))}\mathbf{k}(V(s),u) \frac{w(V(s))}{w(u)} w(u) |h(X(s),u)| \dd s \\
& \quad + \int^t_{\max\{0,t-\tb\}} e^{-\int^t_s \frac{\nu(V(\tau))}{2}\dd \tau} \frac{w(v)}{w(V(s))} | V(s) \cdot \nabla_x \phi_h(X(s)) | e^{-\phi_E(X(s))/2}  w(V(s))\sqrt{\mu(V(s))} \dd s \\
& \quad + \int^t_{\max\{0,t-\tb\}} e^{-\int^t_s \frac{\nu(V(\tau))}{2}\dd \tau}  \frac{w(v)}{w(V(s))}w(V(s))e^{-\phi_E(X(s))/2} |\Gamma(h,h)|(X(s),V(s)) \dd s. 
\end{align*}
Using Lemma \ref{lemma:v_variation} and $|e^{-\phi_E}-1| \ll 1$, we further simplify the characteristic formula as follows:
\begin{align}
& |w(v)h(x,v)| \leq \mathbf{1}_{\tb(x,v)\geq t} e^{-\int^t_0 \frac{\nu(V(s))}{4}\dd s}  w(V(0)) |h(X(0),V(0))| 
\label{chara:0} \\
& \quad + \mathbf{1}_{\tb(x,v)<t} e^{-\int^t_{t-\tb(x,v)} \frac{\nu(V(s))}{4}\dd s}  w(\vb) |f_b(\xb,\vb)| 
\label{chara:bdr} \\
& \quad + \int^t_{\max\{0,t-\tb\}} e^{-\int^t_s \frac{\nu(V(\tau))}{4}\dd \tau}  \int_{\mathbb{R}^3} \dd u\mathbf{k}(V(s),u) \frac{w(V(s))}{w(u)}  w(u) |h(X(s),u)| \dd s 
\label{chara:K} \\
& \quad + \int^t_{\max\{0,t-\tb\}} e^{-\int^t_s \frac{\nu(V(\tau))}{4}\dd \tau}    |V(s) \cdot \nabla_x \phi_h(X(s))| w(V(s))\sqrt{\mu(V(s))} \dd s 
\label{chara:E} \\
& \quad + \int^t_{\max\{0,t-\tb\}} e^{-\int^t_s \frac{\nu(V(\tau))}{4}\dd \tau}  w(V(s))|\Gamma(h, h)|(X(s),V(s)) \dd s. \label{chara:Gamma}
\end{align}

For \eqref{chara:0} and \eqref{chara:bdr}, we use \eqref{nu_quotient} and \eqref{eq4:wf_Linfty} to get
\begin{align}
|\eqref{chara:0}| & \lesssim \Vert wh\Vert_{L^\infty_{x,v}} e^{-\frac{\nu_0 t}{4}} \lesssim o(1)\Vert wh\Vert_{L^\infty_{x,v}}. \label{chara:0_bdd} \\
|\eqref{chara:bdr}| & \lesssim | wf_b|_{L^\infty_{\p\O,v}}. \label{chara:bdr_bdd}
\end{align}

For \eqref{chara:E} and \eqref{chara:Gamma}, we further apply \eqref{phi_f_C^1_2} and Lemma \ref{lemma:gamma} to obtain
\begin{align}
|\eqref{chara:E}| & \lesssim \Vert \nabla_x  \phi_h\Vert_{L^\infty_x} \int^t_{\max\{0,t-\tb\}} e^{-\frac{\nu_0(t-s)}{4}} \dd s \lesssim o(1)\Vert wh\Vert_{L^\infty_{x,v}} + \Vert h\Vert_{L^2_{x,v}},
\label{chara:E_bdd} \\
|\eqref{chara:Gamma}| & \lesssim \Vert \nu^{-1}w\Gamma(h,h)\Vert_{L^\infty_{x,v}} \int^t_{\max\{0,t-\tb\}} e^{-\nu(v)(t-s)/4} \nu^{-1}(v) \dd s \lesssim \Vert wh\Vert_{L^\infty_{x,v}}^2.   
\label{chara:Gamma_bdd}
\end{align}

For \eqref{chara:K}, we further apply the method of characteristics to $h(X(s), V(s))$. In particular, we denote
\begin{align*}
V(s',s) := V(s';s,X(s),u), \qquad
X(s',s) := X(s';s,X(s),u).
\end{align*}
Thus, we obtain that
\begin{align}
& \eqref{chara:K} \leq \int^t_{\max\{0,t-\tb\}} e^{-\int^t_s \frac{\nu(V(\tau))}{4}\dd \tau} \int_{\mathbb{R}^3} \dd u\mathbf{k}(V(s),u) \frac{w(V(s))}{w(u)}  
\notag \\
& \quad \times \Big[ \mathbf{1}_{\tb(X(s),u)\geq s} e^{-\int^s_0 \frac{\nu(V(s',s))}{4} \dd s'} w(V(0,s))h(X(0,s),V(0,s)) \dd s  \label{K_0} \\
& \quad + \mathbf{1}_{\tb(X(s),u)<s}  e^{-\int^s_{s-\tb(X(s),u)} \frac{\nu(V(s',s))}{4} \dd s'} w(\vb(X(s),u)) h(\xb(X(s),u),\vb(X(s),u))  
\label{K_bdr} \\
& \quad + \int^s_{\max\{0,s-\tb(X(s),u)\}} e^{-\int^s_{s'} \frac{\nu(V(\tau',s))}{4} \dd \tau'} \int_{\mathbb{R}^3} \dd u' \mathbf{k}(V(s',s),u') \frac{w(V(s',s))}{w(u')}w(u')h(X(s',s),u') \dd s'  
\label{K_K} \\
& \quad + \int^s_{\max\{0,s-\tb(X(s),u)\}} e^{-\int^s_{s'} \frac{\nu(V(\tau',s))}{4} \dd \tau'} V(s',s)\cdot \nabla_x \phi_h(X(s',s))   w(V(s',s))\sqrt{\mu(V(s',s))}\dd s' 
\label{K_E} \\
& \quad + \int^s_{\max\{0,s-\tb(X(s),u)\}} e^{-\int^s_{s'} \frac{\nu(V(\tau',s))}{4} \dd \tau'} w(V(s',s))\Gamma(h,h)(X(s',s),V(s',s))\dd s'\Big]. 
\label{K_Gamma}
\end{align}

Applying \eqref{k_theta} to control the $\dd u$ integration, we bound the terms \eqref{K_0}, \eqref{K_bdr}, \eqref{K_E}, and \eqref{K_Gamma} by estimates analogous to \eqref{chara:0_bdd}, \eqref{chara:bdr_bdd}, \eqref{chara:E_bdd}, and \eqref{chara:Gamma_bdd}, so that
\be \label{K_rest_bdd}
|\eqref{K_0}| + |\eqref{K_bdr}| + |\eqref{K_E}| + |\eqref{K_Gamma}| 
\lesssim o(1)\|wh\|_{L^\infty_{x,v}}
+ \|h\|_{L^2_{x,v}}
+ \|wf_b\|_{L^\infty_{\partial\Omega,v}}
+ \|wh\|_{L^\infty_{x,v}}^2.
\ee

Finally, to compute \eqref{K_K}, we pick $\delta$ and $N$ such that
\[
0 < \delta \ll 1 \ll N,
\]
and we consider four cases: (1) $s-s'<\delta$, \ (2) $|V(s',s) - u'|< \frac{1}{N} \text{ or } |u'|>N$, \ (3) $|V(s)-u|<\frac{1}{N} \text{ or } |u|>N$, \ (4) $s' \leq s-\delta$, $|V(s',s)-u'| \geq \frac{1}{N}$, $|u'| \leq N$, and $|u-V(s)| \geq \frac{1}{N}$, $|u| \leq N$.

\smallskip

\textit{Case 1: $s-s'<\delta$.}
Using Lemma \ref{lemma:k_theta}, we have
\begin{align*}
& |\eqref{K_K} \mathbf{1}_{s-s'<\delta}|  \lesssim    \int^t_{\max\{0,t-\tb\}}\dd s e^{-\int^t_s \frac{\nu(V(\tau))}{4}\dd \tau} \int_{\mathbb{R}^3} \dd u \mathbf{k}(V(s),u) \frac{w(V(s))}{w(u)} \\
& \quad \times \int^s_{s-\delta} \dd s' e^{-\int^s_{s'} \frac{\nu(V(\tau',s))}{4} \dd \tau'}\int_{\mathbb{R}^3} \dd u' \mathbf{k}(V(s',s),u') \frac{w(V(s',s))}{w(u')} \Vert wh\Vert_{L^\infty_{x,v}} \\
& \quad \lesssim \Vert wh\Vert_{L^\infty_{x,v}}\int^t_{0} \dd s e^{-(t-s)/4} \int_{\mathbb{R}^3} \dd u \mathbf{k}(V(s),u)\frac{w(V(s))}{w(u)} \int^s_{s-\delta} \dd s' e^{-(s-s')/4} \int_{\mathbb{R}^3} \dd u' \mathbf{k}(V(s',s),u')\frac{w(V(s',s))}{w(u')} \\
& \quad \lesssim \delta \Vert wh\Vert_{L^\infty_{x,v}}.
\end{align*}

\smallskip

\textit{Case 2: $|V(s',s) - u'|< \frac{1}{N} \text{ or } |u'|>N $.}
Similar to Case 1, we apply Lemma \ref{lemma:k_theta} to obtain
\be \notag
|\eqref{K_K} \mathbf{1}_{|V(s',s) - u'|< \frac{1}{N} \text{ or } |u'|>N }| 
\lesssim o(1)\Vert wh\Vert_{L^\infty_{x,v}}.
\ee

\smallskip

\textit{Case 3: $|V(s)-u|<\frac{1}{N} \text{ or }|u|>N$.}
By Lemma \ref{lemma:k_theta} and analogous computations, we obtain
\be \notag
|\eqref{K_K} \mathbf{1}_{|V(s)-u|<\frac{1}{N} \text{ or }|u|>N}| \lesssim o(1)\Vert wh\Vert_{L^\infty_{x,v}}.
\ee

\smallskip

\textit{Case 4: $s' \leq s-\delta$, $|V(s',s)-u'| \geq \frac{1}{N}$, $|u'| \leq N$, and $|u-V(s)| \geq \frac{1}{N}$, $|u| \leq N$.}
In this case, using \eqref{k_N_upper_bdd}, we obtain
\be \notag
\mathbf{k}(V(s),u)\frac{w(V(s))}{w(u)}\mathbf{k}(V(s,s'),u')\frac{w(V(s',s))}{w(u')}w(u')\leq C_N.
\ee
We now apply the change of variables
\be \notag
u \in \mathbb{R}^3 \mapsto y = X(s';s,X(s),u) \in \O.
\ee
By Lemma \ref{lemma:deri_XV} and \eqref{eq4:wf_Linfty}, the Jacobian determinant satisfies
\begin{align*}
    &  \Big| \det \Big(\frac{\p X(s';s,X(s),u)}{\p u} \Big) \Big| \gtrsim |s-s'|^3  \gtrsim \delta^3. 
\end{align*}
Then we apply H\"older's inequality to obtain
\be \label{bootstrap}
\begin{split}
& |\eqref{K_K}\mathbf{1}_{|u-V(s)|>\frac{1}{N}, \ |u|\leq N, s'<s-\delta \text{ and }|u'|<N, \  |V(s',s)-u'|>\frac{1}{N}}| 
\\& \lesssim \int^t_{\max\{0,t-\tb\}} \dd s e^{-(t-s)/2}  \int^{s-\delta}_{0} \dd s' \frac{1}{\delta^3} e^{-(s-s')/2} \int_{\O} \dd y \int_{|u'|<N} |h(y,u')|\dd u' \dd s'  \lesssim_{\delta,N} \Vert h\Vert_{L^2_{x,v}}. 
\end{split}
\ee

Collecting the estimates from all four cases, we obtain the bound for \eqref{K_K} as follows:
\be \notag
|\eqref{K_K}| \lesssim o(1)\Vert wh\Vert_{L^\infty_{x,v}} + \Vert h\Vert_{L^2_{x,v}}.
\ee
This, together with \eqref{K_rest_bdd}, implies that
\be \notag
|\eqref{chara:K}|\lesssim o(1)\Vert wh\Vert_{L^\infty_{x,v}} + \Vert h\Vert_{L^2_{x,v}} + |wf_b|_{L^\infty_{\p\O,v}} + \Vert wh \Vert_{L^\infty_{x,v}}^2. 
\ee
Further, combining the above estimate with \eqref{chara:0_bdd}, \eqref{chara:bdr_bdd}, \eqref{chara:E_bdd}, and \eqref{chara:Gamma_bdd}, we obtain
\be \label{eq1:wf_Linfty}
\Vert wh\Vert_{L^\infty_{x,v}} \lesssim  \Vert h\Vert_{L^2_{x,v}} + |wf_b|_{L^\infty_{\p\O,v}} + \Vert wh\Vert_{L^\infty_{x,v}}^2.
\ee
On the other hand, using the assumption $\|wh\|_{L^\infty_{x,v}} + \|\nabla_x (\phi_h+\phi_E)\|_{L^\infty_x} \ll 1$, together with Lemma \ref{lemma:gamma} and Proposition \ref{prop:f_L2}, we derive
\be \notag
\Vert h \Vert_{L^2_{x,\nu}}
\lesssim \Vert \nu^{-1/2}\Gamma(h,h)\Vert_{L^2_{x,v}} + |f_b|_{L^2_{\gamma_-}} \lesssim \Vert wh\Vert_{L^\infty_{x,v}} \Vert h \Vert_{L^2_{x,\nu}} + | wf_b|_{L^\infty_{\p\O,v}}.
\ee
This further implies that
\be \notag
\Vert h \Vert_{L^2_{x,\nu}}
\lesssim | wf_b|_{L^\infty_{\p\O,v}}.
\ee
Combining this with \eqref{eq1:wf_Linfty}, we conclude the proof of the proposition.
\end{proof}



\section{
\texorpdfstring{A priori $W^{1,p}$ estimate for $p<3$}{W1p estimate for p<3}
}
\label{sec:w1p_estimate}

In this section, we establish a priori weighted $W^{1,p}$ estimate for $p<3$ in Proposition \ref{prop:weight_W1p}. This is one of the key ingredients in the well-posedness analysis in Sections~\ref{sec:stationary_uniqueness} and \ref{sec:existence}.
Another key set of $C^1$ estimates will be derived in Section~\ref{sec:c1_estimate}.

Throughout this section, when applying the lemmas from Section \ref{sec:prelim}, we replace $f,\phi_f$ by the stationary solution $h,\phi_h$ defined in \eqref{eqn:h}, and for the kinetic weight lemmas from Section~\ref{sec:kinetic_weight}, we replace $\alpha (t,x,v)$ by the stationary weight $\alpha_h (x,v)$ defined in \eqref{alpha_weight_steady}. 

\begin{proposition} \label{prop:weight_W1p}

Let $h$ be a solution to \eqref{eqn:h}, and let $\alpha_h (x,v)$ be the kinetic weight in \eqref{alpha_weight_steady}. Assume that
\be \label{ap_assumption_steady}
\begin{split}
& \Vert w_{\tilde{\theta}} \p_{x,v}h\Vert_{L^p_{x,v}} + \Vert w_{\tilde{\theta}} \alpha_h \nabla_x h\Vert_{L^\infty_{x,v}} < \infty,
\\& \Vert \nabla_x \phi_h\Vert_{L^\infty_x}+ \Vert \nabla_x^2 \phi_h\Vert_{L^\infty_{x}} + \Vert wh\Vert_{L^\infty_{x,v}} + |wf_b|_{L^\infty_{\p\O,v}} + |w\p_{\mathbf{x}_p,v}f_b|_{L^\infty_{\p\O,v}} \ll 1.     
\end{split}
\ee
Then for $2<p<3$, the following weighted $W^{1,p}$ estimate holds:
\be \label{apriori_W1p}
\Vert w_{\tilde{\theta}}\p_{x,v} h\Vert_{L^p_{x,v}} \lesssim o(1)\Vert \alpha_h \nabla_x h\Vert_{L^\infty_{x,v}} + \Vert w h \Vert_{L^\infty_{x,v}} + | w f_b|_{L^\infty_{\p\O,v}} + | w \p_{\mathbf{x}_p,v} f_b|_{L^\infty_{\p\O,v}}. 
\ee
\end{proposition}

\begin{proof}


For simplicity, we write $(X(s), V(s)) := (X(s;t,x,v), V(s;t,x,v))$ throughout the proof.
Using the assumption \eqref{ap_assumption_steady} and following \eqref{eq4:wf_Linfty}, we may choose $t>0$ such that
\be \label{eq1:weight_W1p}
\Vert \nabla_x (\phi_f + \phi_E) \Vert_{L^\infty_{x}} t +
\| \nabla_x^2 (\phi_h + \phi_E) \|_{L^\infty_x} (t)^2 e^{t} \ll 1
\ \text{ and } \
e^{-\frac{\nu_0}{4}t} \ll 1.
\ee
For any $\max \{ 0, t - \tb \} \leq s \leq t$, Lemma \ref{lemma:deri_XV} implies that
\be \label{deri_XV_bdd}
|\p_{x,v}X(s)| + |\p_{x,v}V(s)| \lesssim e^{o(1)(t-s)}. 
\ee
Analogous to \eqref{eq2:wf_Linfty} and \eqref{eq3:wf_Linfty} in the proof of Proposition \ref{prop:wf_Linfty}, we define the damping factor $\tilde{\nu}(X(s),V(s))$ by
\be \notag
\tilde{\nu}(X(s),V(s)) := e^{-\phi_E(X(s))} \nu(V(s)) - \frac{V(s) \cdot \nabla_x \phi_h(X(s))}{2}.
\ee
The assumption \eqref{ap_assumption_steady} and Lemma \ref{lemma:k_nu} imply that
\be \notag
\tilde{\nu}(X(s),V(s))> \frac{\nu(V(s))}{2}> \frac{\nu_0}{2}
\ \text{ for }
\max \{ 0, t - \tb \} \leq s \leq t.
\ee 
Moreover, using Lemma \ref{lemma:k_nu}, Lemma \ref{lemma:phi_C2}, and \eqref{deri_XV_bdd}, we bound $|\nabla_x \tilde{\nu}(X(s),V(s))|$ by
\be \label{nabla_nu}
\begin{split}
|\nabla_x \tilde{\nu}(X(s),V(s))| 
& \lesssim | \nabla_x \phi_E(X(s))| |\nabla_x X(s)| |\nu(V(s))| + |\nabla_v \nu(V(s))| |\nabla_x V(s)|  
\\& \qquad + |\nabla_x^2 \phi_h (X(s))| |\nabla_x X(s)||V(s)|+ |\nabla_x \phi_h(X(s)) | |\nabla_x V(s)| 
\\& \lesssim e^{o(1)(t-s)}(1+|v|) \big( o(1)(\Vert \nabla_x h\Vert_{L^p_{x,v}} +\Vert \alpha_h \nabla_x h \Vert_{L^\infty_{x,v}}) + \Vert wh\Vert_{L^\infty_{x,v}} + 1 \big).
\end{split}
\ee

Analogous to the proof of Proposition \ref{prop:wf_Linfty}, applying the method of characteristics to \eqref{eqn:h}, we get
\be \label{h_piecewise}
\begin{split}
h(x,v) = & \mathbf{1}_{\tb<t}e^{-\int^t_{t-\tb} \tilde{\nu}(X(s),V(s)) \dd s}f_b(\xb,\vb)  
\\& + \mathbf{1}_{\tb\geq t} e^{-\int^t_{0}\tilde{\nu}(X(s),V(s))\dd s} h(X(0),V(0)) 
\\&  + \int^t_{\max\{t-\tb,0\}} e^{-\int^t_{s} \tilde{\nu}(X(\tau),V(\tau)) \dd \tau} \int_{\mathbb{R}^3} \dd u e^{-\phi_E(X(s))}\mathbf{k}(V(s),u) h(X(s),u) \dd s 
\\& +  \int^t_{\max\{t-\tb,0\}} e^{-\int^t_{s} \tilde{\nu}(X(\tau),V(\tau)) \dd \tau} e^{-\frac{\phi_E(X(s))}{2}}\Gamma(h,h)(X(s),V(s)) \dd s 
\\& +  \int^t_{\max\{t-\tb,0\}} e^{-\int^t_{s} \tilde{\nu}(X(\tau),V(\tau)) \dd \tau} V(s) \cdot \nabla_x \phi_h(X(s)) e^{-\frac{\phi_E(X(s))}{2}} \sqrt{\mu(V(s))} \dd s. 
\end{split}
\ee
Taking the derivative of \eqref{h_piecewise} and multiplying both sides by $w_{\tilde{\theta}}(v)$, we further obtain
{\small
\begin{align}
& w_{\tilde{\theta}} (v) \p_{x,v} h(x,v) 
\notag \\
= & \mathbf{1}_{\tb>t} e^{-\int^t_{0}\tilde{\nu}(X(s),V(s))\dd s} \frac{w_{\tilde{\theta}}(v)}{w_{\tilde{\theta}}(V(0))}w_{\tilde{\theta}}(V(0))[ \p_{x,v} X(0) \p_{x}h(X(0),V(0)) + \p_{x,v} V(0)\cdot \nabla_{v}h(X(0),V(0)) ] 
\label{rchara:nabla_0} \\
& + \mathbf{1}_{\tb > t} e^{-\int_0^t \tilde{\nu}(X(s),V(s))\dd s} \int^t_0 \p_{x,v}[\tilde{\nu}(X(s),V(s))] \dd s \frac{w_{\tilde{\theta}}(v)}{w_{\tilde{\theta}}(V(0))}w_{\tilde{\theta}}(V(0)) h(X(0),V(0)) \label{rchara:nabla_nu_1} \\
& + \mathbf{1}_{\tb<t} e^{-\int^t_{t-\tb} \tilde{\nu}(X(s),V(s)) \dd s  } \int_{t-\tb}^t \p_{x,v} [\tilde{\nu}(X(s),V(s))] \dd s  \frac{w_{\tilde{\theta}}(v)}{w_{\tilde{\theta}}(\vb)}w_{\tilde{\theta}}(\vb) f_b(\xb,\vb)
\label{rchara:nabla_nu_2} \\
& + \mathbf{1}_{\tb<t}\p_{x,v} \tb \tilde{\nu}(\xb,\vb) e^{-\int^t_{t-\tb} \tilde{\nu}(X(s),V(s)) \dd s}\frac{w_{\tilde{\theta}}(v)}{w_{\tilde{\theta}}(\vb)}w_{\tilde{\theta}}(\vb)f_b(\xb,\vb) \label{rchara:nabla_tb} \\
& + \mathbf{1}_{\tb<t} e^{-\int^t_{t-\tb} \tilde{\nu}(X(s),V(s)) \dd s} \frac{w_{\tilde{\theta}}(v)}{w_{\tilde{\theta}}(\vb)}w_{\tilde{\theta}}(\vb) \Big[  \sum_{i=1}^2 \nabla_{x,v} \mathbf{x}_{p^1,i}^1 \p_{\mathbf{x}_{p^1,i}^1}f_b(\xb,\vb)+   \p_{x,v} \vb \cdot \nabla_v f_b(\xb,\vb)  \Big]  
\label{rchara:bdr} \\
& + \int^t_{\max\{t-\tb,0\}} e^{-\int^t_{s}  \tilde{\nu}(X(\tau),V(\tau)) \dd \tau} \frac{w_{\tilde{\theta}}(v)}{w_{\tilde{\theta}}(V(s))}\int_{\mathbb{R}^3} \dd u e^{-\phi_E(X(s))} \mathbf{k}(V(s),u) w_{\tilde{\theta}}(V(s)) \p_{x,v} [h(X(s),u)] \dd s 
\label{rchara:K_nabla} \\
& + \int^t_{\max\{t-\tb,0\}} e^{-\int^t_s \tilde{\nu}(X(\tau),V(\tau)) \dd \tau} \frac{w_{\tilde{\theta}}(v)}{w_{\tilde{\theta}}(V(s))}\int_{\mathbb{R}^3} \dd u e^{-\phi_E(X(s))}w_{\tilde{\theta}}(V(s))\p_{x,v} V(s) \cdot \nabla_v \mathbf{k}(V(s),u) h(X(s),u) \dd s  
\label{rchara:nabla_K} \\
& + \int^t_{\max\{t-\tb,0\}} e^{-\int^t_s \tilde{\nu}(X(\tau),V(\tau)) \dd \tau} \frac{w_{\tilde{\theta}}(v)}{w_{\tilde{\theta}}(V(s))} \int_{\mathbb{R}^3} \dd u \nabla_{x}\phi_E(X(s)) \p_{x,v} X(s) 
\notag \\
& \qquad \times
e^{-\phi_E(X(s))}w_{\tilde{\theta}}(V(s))\mathbf{k}(V(s),u)  h(X(s),u) \dd s  
\label{rchara:K_nabla_phi_E} \\
& + \mathbf{1}_{\tb<t} \p_{x,v} \tb e^{-\int^t_{t-\tb} \tilde{\nu}(X(\tau),V(\tau)) \dd \tau} \frac{w_{\tilde{\theta}}(v)}{w_{\tilde{\theta}}(\vb)}w_{\tilde{\theta}}(\vb) \int_{\mathbb{R}^3} \dd u e^{-\phi_E(\xb)}\mathbf{k}(\vb,u) f_b(\xb,u)   \label{rchara:nabla_int} \\
& + \int^t_{\max\{t-\tb,0\}} e^{-\int_s^t \tilde{\nu}(X(\tau),V(\tau))\dd \tau} \int_s^t \p_{x,v} [\tilde{\nu}(X(\tau),V(\tau))] \dd \tau
\notag \\
& \qquad \times \frac{w_{\tilde{\theta}}(v)}{w_{\tilde{\theta}}(V(s))} \int_{\mathbb{R}^3} \dd u e^{-\phi_E(X(s))} w_{\tilde{\theta}}(V(s))\mathbf{k}(V(s),u) h(X(s),V(s)) \dd s \label{rchara:nabla_nu_K} \\
& + \mathbf{1}_{\tb<t} \p_{x,v} \tb e^{-\int^t_{t-\tb}\tilde{\nu}(X(\tau),V(\tau)) \dd \tau} \frac{w_{\tilde{\theta}}(v)}{w_{\tilde{\theta}}(\vb)}w_{\tilde{\theta}}(\vb) e^{-\phi_E(\xb)/2} \Gamma(f_b,f_b)(\xb,\vb) 
\label{rchara:nabla_tb_gamma} \\
& + \int^t_{\max\{t-\tb,0\}} e^{-\int_s^t \tilde{\nu}(X(\tau),V(\tau))\dd \tau} \int_s^t \dd \tau  \p_{x,v} [\tilde{\nu}(X(\tau),V(\tau))] \frac{w_{\tilde{\theta}}(v)}{w_{\tilde{\theta}}(V(s))} w_{\tilde{\theta}}(V(s))e^{-\frac{\phi_E(X(s))}{2}}\Gamma(h,h)(X(s),V(s)) \dd s 
\label{rchara:nabla_nu_gamma} \\
& + \int^t_{\max\{t-\tb,0\}} e^{-\int_s^t \tilde{\nu}(X(\tau),V(\tau))\dd \tau} \frac{w_{\tilde{\theta}}(v)}{w_{\tilde{\theta}}(V(s))} w_{\tilde{\theta}}(V(s))e^{-\frac{\phi_E(X(s))}{2}} \p_{x,v} [\Gamma(h,h)(X(s),V(s))] \dd s
\label{rchara:nabla_gamma} \\
& + \int^t_{\max\{t-\tb,0\}} e^{-\int_s^t \tilde{\nu}(X(\tau),V(\tau))\dd \tau} \frac{w_{\tilde{\theta}}(v)}{w_{\tilde{\theta}}(V(s))} w_{\tilde{\theta}}(V(s))\frac{\nabla_x \phi_E(X(s))\p_{x,v}X(s)}{2} e^{-\frac{\phi_E(X(s))}{2}}\Gamma(h,h)(X(s),V(s)) \dd s
\label{rchara:gamma_nabla_phi_E} \\
& + \int^t_{\max\{t-\tb,0\}} e^{-\int^t_{s} \tilde{\nu}(X(\tau),V(\tau)) \dd \tau}\frac{w_{\tilde{\theta}}(v)}{w_{\tilde{\theta}}(V(s))} w_{\tilde{\theta}}(V(s))  V(s) \cdot \p_{x,v}[  \nabla_x \phi_h(X(s))  ] e^{-\frac{\phi_E(X(s))}{2}}\sqrt{\mu(V(s))} \dd s 
\label{rchara:nabla_E} \\
& + \mathbf{1}_{\tb<t} \p_{x,v} \tb e^{-\int^t_{t-\tb} \tilde{\nu}(X(\tau),V(\tau)) \dd \tau} \frac{w_{\tilde{\theta}}(v)}{w_{\tilde{\theta}}(\vb)} w_{\tilde{\theta}}(\vb) \vb\cdot \nabla_x \phi_h (\xb)  e^{-\frac{\phi_E(\xb)}{2}}\sqrt{\mu(\vb)} \label{rchara:nabla_tb_phi} \\
& + \int^t_{\max\{t-\tb,0\}} e^{-\int_s^t \tilde{\nu}(X(\tau),V(\tau))\dd \tau} \int_s^t \p_{x,v} [\tilde{\nu}(X(\tau),V(\tau))] \dd \tau 
\notag \\
& \qquad \times
\frac{w_{\tilde{\theta}}(v)}{w_{\tilde{\theta}}(V(s))} w_{\tilde{\theta}}(V(s)) V(s) \cdot \nabla_x \phi_h(X(s))    e^{-\frac{\phi_E(X(s))}{2}}\sqrt{\mu(V(s))}  \dd s \label{rchara:nabla_nu_phi} \\
& + \int^t_{\max\{t-\tb,0\}} e^{-\int^t_s \tilde{\nu}(X(\tau),V(\tau)) \dd \tau} \frac{w_{\tilde{\theta}}(v)}{w_{\tilde{\theta}}(V(s))} w_{\tilde{\theta}}(V(s)) \p_{x,v} V(s)\cdot\nabla_x \phi_h(X(s))    e^{-\frac{\phi_E(X(s))}{2}}\sqrt{\mu(V(s))} \dd s 
\label{rchara:nabla_V_phi} \\
& + \int^t_{\max\{t-\tb,0\}} e^{-\int^t_s \tilde{\nu}(X(\tau),V(\tau)) \dd \tau} \frac{w_{\tilde{\theta}}(v)}{w_{\tilde{\theta}}(V(s))} w_{\tilde{\theta}}(V(s))   V(s)\cdot\nabla_x \phi_h(X(s))   \p_{x,v} V(s)\cdot V(s)e^{-\frac{\phi_E(X(s))}{2}}\sqrt{\mu(V(s))} \dd s   \label{rchara:nabla_mu_phi} \\
& + \int^t_{\max\{t-\tb,0\}} e^{-\int^t_s \tilde{\nu}(X(\tau),V(\tau)) \dd \tau} \frac{w_{\tilde{\theta}}(v)}{w_{\tilde{\theta}}(V(s))} w_{\tilde{\theta}}(V(s)) V(s) \cdot \nabla_x \phi_h(X(s)) 
\notag \\
& \qquad \times \frac{\nabla_x \phi_E(X(s))\p_{x,v}X(s)}{2} e^{-\frac{\phi_E(X(s))}{2}}\sqrt{\mu(V(s))} \dd s. \label{rchara:nabla_phi_E_phi}
\end{align}
}
We note that such a piecewise derivative formula gives a weak derivative of \eqref{h_piecewise}; see \cite{GKTT, CK}.
Using \eqref{eq1:weight_W1p} and Lemma \ref{lemma:v_variation}, we bound the weighted exponential factors appearing on the right-hand side as follows. For any $\max \{ 0, t - \tb \} \leq s \leq t$,
\be \label{eq2:weight_W1p}
\begin{split}
e^{-\int^t_{s}\tilde{\nu}(X(\tau),V(\tau))\dd \tau} \frac{w_{\tilde{\theta}}(v)}{w_{\tilde{\theta}}(V(s))} & \lesssim e^{-\int^t_{s} \frac{\tilde{\nu}(X(\tau),V(\tau))}{2}\dd \tau}.
\end{split}
\ee

We now split the terms \eqref{rchara:nabla_0}–\eqref{rchara:nabla_phi_E_phi} on the right-hand side into four parts, denoted by $\mathcal{A}, \mathcal{B}, \mathcal{C}, \mathcal{D}$, as follows:
\be \label{ABCD_notation}
\begin{split}
\mathcal{A}(x,v) & = \eqref{rchara:nabla_nu_1} + \eqref{rchara:nabla_nu_2} + \eqref{rchara:nabla_K}+ \eqref{rchara:K_nabla_phi_E} + \eqref{rchara:nabla_nu_K}  + \eqref{rchara:nabla_nu_gamma} + \eqref{rchara:gamma_nabla_phi_E} 
\\& \qquad + \eqref{rchara:nabla_E}  + \eqref{rchara:nabla_nu_phi}  + \eqref{rchara:nabla_V_phi} + \eqref{rchara:nabla_mu_phi} + \eqref{rchara:nabla_phi_E_phi}, 
\\ \mathcal{B}(x,v) & = \eqref{rchara:nabla_0} + \eqref{rchara:nabla_gamma}, 
\\ \mathcal{C}(x,v) & = \eqref{rchara:nabla_tb} + \eqref{rchara:bdr} + \eqref{rchara:nabla_int} + \eqref{rchara:nabla_tb_gamma} + \eqref{rchara:nabla_tb_phi}, 
\\ \mathcal{D}(x,v) & = \eqref{rchara:K_nabla}. 
\end{split}
\ee

First, we estimate $\mathcal{A}$, which corresponds to the terms without singularity. All terms in $\mathcal{A}$ contain exponential weights in $v$. Using Lemma \ref{lemma:v_variation}, we obtain
\begin{align*}
w_{\tilde{\theta}}(V(0))|h(X(0),V(0))| + w_{\tilde{\theta}}(\vb) |f_b(\xb,\vb)| & \lesssim \big( \Vert wh\Vert_{L^\infty_{x,v}} + | wf_b|_{L^\infty_{\p\O,v}} \big) w^{-\frac{1}{2}}(v), \\
w_{\tilde{\theta}}(V(s))\sqrt{\mu(V(s))} + w_{\tilde{\theta}}(\vb) \sqrt{\mu(\vb)} & \lesssim  \mu^{-1/4}(v). 
\end{align*}
From Lemma \ref{lemma:k_theta} and Lemma \ref{lemma:gamma}, we have
\begin{align*}
& w_{\tilde{\theta}}(V(s)) \int_{\mathbb{R}^3} [\mathbf{k}(V(s),u)+ \nabla_v \mathbf{k}(V(s),u)] h(X(s),u) \dd u \\
& \lesssim \Vert wh\Vert_{L^\infty_{x,v}} w_{\tilde{\theta}}(V(s)) w^{-1}(V(s)) \int_{\mathbb{R}^3} \mathbf{k}(V(s),u)\frac{w(V(s))}{w(u)} [1+\frac{1}{|V(s)-u|}] \dd u \lesssim \Vert wh\Vert_{L^\infty_{x,v}}w^{-\frac{1}{2}}(v), \\
& w_{\tilde{\theta}}(V(s)) |\Gamma(h,h)(X(s),V(s))| \\
& < w_{\tilde{\theta}}(V(s))w^{-\frac{1}{2}}(V(s)) |\nu^{-1}(V(s))w(V(s))\Gamma(h,h)(X(s),V(s))| 
\lesssim w^{-\frac{1}{4}}(v)\Vert wh\Vert_{L^\infty_{x,v}}^2. 
\end{align*}
Moreover, using \eqref{deri_XV_bdd} and Lemma \ref{lemma:phi_C2}, we bound the term $\partial_{x,v}[\nabla_x \phi_h(X(s))]$ in \eqref{rchara:nabla_E} by
\begin{align*}
    &   |\p_{x,v} [\nabla_x \phi_h (X(s))]| = |\nabla_x^2 \phi_h (X(s)) \p_{x,v}X(s)|\lesssim [o(1)(\Vert \nabla_x  h\Vert_{L^p_{x,v}}+ \Vert \alpha_h \nabla_x h\Vert_{L^\infty_{x,v}}) + \Vert wh\Vert_{L^\infty_{x,v}}] e^{o(1)(t-s)}.
\end{align*}
Combining these estimates with Lemma \ref{lemma:phi_x_infinity}, we conclude that
\begin{align}
|\mathcal{A}(x,v)| 
& \lesssim w^{-\frac{1}{4}}(v)(1+\Vert \nabla_x \phi_E\Vert_{L^\infty_{x}}) (o(1)(\Vert \nabla_x  h\Vert_{L^p_{x,v}}+ \Vert \alpha_h \nabla_x h\Vert_{L^\infty_{x,v}}) + \Vert wh\Vert_{L^\infty_{x,v}} + | wf_b|_{L^\infty_{\p\O,v}} )
\notag \\
& \lesssim w^{-\frac{1}{4}}(v) \big( o(1)(\Vert w_{\tilde{\theta}}(v)\p_{x,v}h\Vert_{L^p_{x,v}} + \Vert \alpha_h \nabla_x h\Vert_{L^\infty_{x,v}}) + \Vert wh\Vert_{L^\infty_{x,v}}+ | wf_b|_{L^\infty_{\p\O,v}} \big), 
\label{A_bdd_infty} \\
\Vert \mathcal{A}\Vert_{L^p_{x,v}} 
& \lesssim o(1) (\Vert w_{\tilde{\theta}}(v)  \p_{x,v}h\Vert_{L^p_{x,v}} + \Vert \alpha_h \nabla_x h\Vert_{L^\infty_{x,v}}) + \Vert wh\Vert_{L^\infty_{x,v}}+ | wf_b|_{L^\infty_{\p\O,v}}. \label{A_bdd_lp}
\end{align}

Second, we estimate $\mathcal{B}$, which corresponds to the terms involving derivatives of $h$. 
For \eqref{rchara:nabla_0}, we use the change of variables in Lemma \ref{lemma:int_cov}, \eqref{eq1:weight_W1p}, and \eqref{eq2:weight_W1p} to compute
\be \label{nabla_0_compute}
\begin{split}
\Vert \eqref{rchara:nabla_0} \Vert_{L^p_{x,v}} 
& \lesssim \Vert  e^{-\frac{\nu_0 t}{4}} \mathbf{1}_{X(0)\in \O} w_{\tilde{\theta}}(V(0))\p_{x,v} h(X(0),V(0)) \Vert_{L^p_{x,v}}  
\\& \lesssim  e^{-\nu_0 t/4}\Vert w_{\tilde{\theta}} \p_{x,v} h\Vert_{L^p_{x,v}} \lesssim o(1) \Vert w_{\tilde{\theta}}\p_{x,v}h\Vert_{L^p_{x,v}}. 
\end{split}
\ee
For \eqref{rchara:nabla_gamma}, we analyze the term $\p_{x,v} \Gamma(h, h)$. By Lemma \ref{lemma:gamma}, we decompose it as
\be \notag
\p_{x,v} \Gamma(h, h) = \Gamma_{\text{gain}}(h,\p_{x,v}h) - \Gamma_{\text{loss}}(h,\p_{x,v}h) + \Gamma(\p_{x,v}h,h) + \Gamma_{x,v} (h, h).
\ee
Using Lemma \ref{lemma:v_variation} and Lemma \ref{lemma:gamma}, the contribution of $\Gamma_{\text{loss}}(h,\p_{x,v}h)$ in \eqref{rchara:nabla_gamma} is bounded by
\begin{align*}
& \Big\Vert    \int^t_{\max\{t-\tb,0\}} e^{-\frac{\nu(V(s))(t-s)}{4}} \nu(V(s)) w_{\tilde{\theta}}(V(s)) e^{-\frac{\phi_E (X(s))}{2}}\p_{x,v}h(X(s),V(s)) \dd s \Vert wh\Vert_{L^\infty_{x,v}}    \Big\Vert_{L^p_{x,v}} \\
& \lesssim \Vert wh\Vert_{L^\infty_{x,v}} \Big\Vert \int^t_0 e^{-(\frac{1}{p'}+\frac{1}{p})\frac{\nu(V(s))(t-s)}{4}} \nu^{\frac{1}{p'}+\frac{1}{p}}(V(s)) w_{\tilde{\theta}}(V(s))\p_{x,v}h(X(s),V(s))\dd s\Big\Vert_{L^p_{x,v}} \\
& \lesssim \Vert wh\Vert_{L^\infty_{x,v}} \Big\Vert    \Big( \int^t_0 e^{-\frac{\nu(V(s))(t-s)}{4}} |\nu(V(s))| |w_{\tilde{\theta}}(V(s))\p_{x,v}h(X(s),V(s))|^p \dd s\Big)^{1/p} \Big\Vert_{L^p_{x,v}} \\
& \lesssim \Vert wh\Vert_{L^\infty_{x,v}}\Big(\int^t_0 \int_{\O\times \mathbb{R}^3} e^{-\frac{\nu(v)(t-s)}{4}} \nu(v) | w_{\tilde{\theta}}(v) \p_{x,v} h(x,v)|^p \dd x \dd v \dd s\Big)^{1/p} 
\lesssim o(1)\Vert w_{\tilde{\theta}}\p_{x,v}h\Vert_{L^p_{x,v}},
\end{align*}
where in the third line we apply H\"older's inequality in $\dd s$, and in the last line we use the change of variables $(X(s), V(s)) \mapsto (x,v)$ in Lemma \ref{lemma:int_cov}.
Using Lemma \ref{lemma:v_variation} and Lemma \ref{lemma:gamma}, the contribution of $\Gamma_{\text{gain}}(h,\p_{x,v}h) + \Gamma(\p_{x,v}h,h)$ in \eqref{rchara:nabla_gamma} is bounded by
\begin{align}
& \Vert wh\Vert_{L^\infty_{x,v}} \Big\Vert \int_{\max\{0,t-\tb\}}^t e^{-\frac{\nu(V(s))(t-s)}{2}} \int_{\mathbb{R}^3} \frac{w_{\tilde{\theta}}(V(s))}{w_{\tilde{\theta}}(u)} e^{-\frac{\phi_E(X(s))}{2}}\mathbf{k}_{1}(V(s),u) w_{\tilde{\theta}}(u)\p_{x,v}h(X(s),u) \dd u \dd s \Big\Vert_{L^p_{x,v}} 
\notag \\
& \lesssim \Vert wh\Vert_{L^\infty_{x,v}} \Big\Vert  \int^t_{\max\{0,t-\tb\}} e^{-\frac{\nu(V(s))(t-s)}{4}} \Vert w_{\tilde{\theta}}(u) \p_{x,v}h(X(s),u)\Vert_{L^p_u}  \Vert \mathbf{k}_1(V(s),u)\frac{w_{\tilde{\theta}}(V(s))}{w_{\tilde{\theta}}(u)} \Vert_{L^{p'}_u} \dd s \Big\Vert_{L^p_{x,v}}  \notag \\
& \lesssim \Vert wh\Vert_{L^\infty_{x,v}} \Big\Vert \int^t_{\max\{0,t-\tb\}} e^{-(\frac{1}{p'}+\frac{1}{p})\frac{\nu(V(s))(t-s)}{4}} \nu^{-1/p'}(V(s)) \Vert w_{\tilde{\theta}}(u) \p_{x,v}h(X(s),u) \Vert_{L^p_u}  \dd s  \Big\Vert_{L^p_{x,v}} 
\notag \\
& \lesssim \Vert wh\Vert_{L^\infty_{x,v}} \Big\Vert  \nu^{-1/p'}(v) \nu^{-1/p'}(v) \Big(\int^t_{0} e^{-\frac{\nu(V(s))(t-s)}{4}} \Vert w_{\tilde{\theta}}(u)\p_{x,v}h(X(s),u)\Vert_{L^p_u}^p \dd s \Big)^{1/p}\Big\Vert_{L^p_{x,v}}, 
\label{eq3:weight_W1p} 
\end{align}
where in the second line we apply H\"older's inequality in $\dd u$, and in the last line we apply H\"older's inequality in $\dd s$.
Using Lemma \ref{lemma:v_variation} again, we obtain that
\begin{align}
\eqref{eq3:weight_W1p} 
& \lesssim \Vert wh\Vert_{L^\infty_{x,v}} \Big( \int_0^t \int_{x,v,u} \nu^{-2p/p'}(V(s))e^{-\frac{\nu(V(s))(t-s)}{4}} |w_{\tilde{\theta}}(u)\p_{x,v}h(X(s),u)|^p \dd u \dd x\dd v \dd s  \Big)^{1/p} 
\notag \\
& \lesssim \Vert wh\Vert_{L^\infty_{x,v}} \Big( \int_0^t \int_{x,v,u} \nu^{-2p/p'}(v)e^{-\frac{\nu(v)(t-s)}{4}} |w_{\tilde{\theta}}(u)\p_{x,v}h(x,u)|^p \dd u\dd x\dd v \dd s  \Big)^{1/p}  
\notag \\
& \lesssim \Vert wh\Vert_{L^\infty_{x,v}} \Big( \int_{x,v,u}    \nu^{-2p/p'-1}(v) |w_{\tilde{\theta}}(u)\p_{x,v}h(x,u)|^p   \dd u \dd x \dd v \dd s \Big)^{1/p}  
\notag \\
& \lesssim \Vert wh\Vert_{L^\infty_{x,v}} \Vert w_{\tilde{\theta}}\p_{x,v}h\Vert_{L^p_{x,v}},
\label{gamma_gain_est} 
\end{align}
where in the second line we use the change of variables $(X(s), V(s)) \mapsto (x,v)$ in Lemma \ref{lemma:int_cov}, and in the last line the $\dd v$ integration is bounded since $2<p<3$, which implies $2p/p'+1>3$.
We remark that we do not directly apply Minkowski's inequality in $L^p_{x,v}$ at the beginning, since sufficient decay in $\nu^{-2p/p'-1}(v)$ is needed to control the $\dd v$ integral. Such decay can only be extracted by applying H\"older's inequality successively.

From Lemma \ref{lemma:gamma}, the contribution of $\Gamma_v(h,h)$ in \eqref{rchara:nabla_gamma} is controlled by $\Vert wh\Vert_{L^\infty_{x,v}}^2$ via the gain from the weight $w^{-\frac{1}{2}}(v)$, while $\Gamma_x (h,h) = 0$ follows from the Carleman representation.
Combining all estimates for \eqref{rchara:nabla_gamma} with \eqref{nabla_0_compute}, we conclude that
\begin{align}
    & \Vert \eqref{rchara:nabla_gamma}\Vert_{L^p_{x,v}}\lesssim o(1)\Vert w_{\tilde{\theta}} \p_{x,v} h\Vert_{L^p_{x,v}} + \Vert wh\Vert_{L^\infty_{x,v}}, \label{B_gamma_est} \\
    &    \Vert \mathcal{B}\Vert_{L^p_{x,v}} \lesssim  o(1)\Vert w_{\tilde{\theta}} \p_{x,v} h\Vert_{L^p_{x,v}} + \Vert wh\Vert_{L^\infty_{x,v}}. \label{B_bdd_lp}
\end{align}

Third, we estimate $\mathcal{C}$, which corresponds to the terms containing the exponential weight $w_{\tilde{\theta}}(\vb)$ and the singular term $\frac{1}{|n(\xb)\cdot \vb|}$ arising from the derivatives of the backward exit time, position, and velocity in Lemma \ref{lemma:deri_backward}.
For \eqref{rchara:bdr}, we control the boundary derivative terms by
\begin{align*}
    &   |\p_{\mathbf{x}_p,v} f_b(\xb,\vb)| \lesssim w^{-1}(\vb) | w\p_{\mathbf{x}_p,v} f_b|_{L^\infty_{\p\O,v}}    .
\end{align*}
For \eqref{rchara:nabla_int} and \eqref{rchara:nabla_tb_gamma}, using Lemma \ref{lemma:k_theta} and Lemma \ref{lemma:gamma}, we derive that
\begin{align*}
\int_{\mathbb{R}^3} \dd u \mathbf{k}(\vb,u) f_b(\xb,u) 
& \lesssim | wf_b|_{L^\infty_{\p\O,v}} w^{-1}(\vb), \\
\big| \Gamma(f_b,f_b)(\xb,\vb) \big| 
& \lesssim | wf_b|_{L^\infty_{\p\O,v}}^2 w^{-\frac{1}{2}}(\vb).
\end{align*}
For the remaining terms in $\mathcal{C}$, applying Lemma \ref{lemma:deri_backward}, Lemma \ref{lemma:integrate_nv}, and \eqref{eq2:weight_W1p}, we obtain that
\begin{align}
|\mathcal{C}(x,v)| 
& \lesssim \big[\Vert wh\Vert_{L^\infty_{x,v}} + | wf_b|_{L^\infty_{\p\O,v}} +  |w\p_{\mathbf{x}_{p},v}f_b|_{L^\infty_{\p\O,v}} \big]   \frac{e^{-\nu_0 \tb /4}}{|n(\xb)\cdot \vb|} w^{-\frac{1}{2}}(\vb), \label{c_bdd_infty} \\
\Vert \mathcal{C} \Vert_{L^p_{x,v}} 
& \lesssim ( \Vert wh\Vert_{L^\infty_{x,v}}+| wf_b|_{L^\infty_{\p\O,v}}  + |w\p_{\mathbf{x}_{p},v}f_b|_{L^\infty_{\p\O,v}})  \Big\Vert e^{-\nu_0\tb/4} \frac{1}{|n(\xb)\cdot \vb|} w^{-\frac{1}{2}}(\vb) \Big\Vert_{L^p_{x,v}} 
\notag \\
& \lesssim \Vert wh\Vert_{L^\infty_{x,v}} + | wf_b|_{L^\infty_{\p\O,v}}+ |w\p_{\mathbf{x}_{p},v}f_b|_{L^\infty_{\p\O,v}}.\label{c_bdd_lp}
\end{align}


Finally, we estimate $\mathcal{D} = \eqref{rchara:K_nabla}$, which contains the most challenging term $\partial_{x,v}[h(X(s),u)]$. We begin by rewriting it as
\be \label{k_nabla_express}
\begin{split}
\eqref{rchara:K_nabla}
& = \int_{\max\{0,t-\tb\}}^t e^{-\int^t_s \tilde{\nu}(X(\tau),V(\tau))\dd \tau} \frac{w_{\tilde{\theta}}(v)}{w_{\tilde{\theta}}(V(s))}
\\& \qquad \times \int_{\mathbb{R}^3}  e^{-\phi_E(X(s))}\mathbf{k}(V(s),u) \frac{w_{\tilde{\theta}}(V(s))}{w_{\tilde{\theta}}(u)} \p_{x,v} X(s) w_{\tilde{\theta}}(u)\p_x h(X(s),u) \dd u \dd s.  
\end{split}
\ee
We expand $w_{\tilde{\theta}}(u)\partial_x h(X(s),u)$ in \eqref{k_nabla_express} using the characteristic formula \eqref{rchara:nabla_0}-\eqref{rchara:nabla_phi_E_phi}, with $(t,x,v)$ replaced by $(s,X(s),u)$. In particular, we denote
\be \notag
V(s',s) := V(s';s,X(s),u), \qquad
X(s',s) := X(s';s,X(s),u).
\ee
Analogous to \eqref{ABCD_notation}, the resulting expression can be decomposed into four parts. For simplicity, we denote them by $\mathcal{A} (X(s),u), \mathcal{B} (X(s),u), \mathcal{C} (X(s),u), \mathcal{D} (X(s),u)$, with a slight abuse of notation. Further, we decompose \eqref{k_nabla_express} as
\[
\eqref{k_nabla_express}
=
\eqref{k_nabla_express}_{\mathcal A} + \eqref{k_nabla_express}_{\mathcal B} + \eqref{k_nabla_express}_{\mathcal C} + \eqref{k_nabla_express}_{\mathcal D}.
\]

For $\eqref{k_nabla_express}_{\mathcal A}$, we follow \eqref{A_bdd_infty} to obtain
\begin{align}
| \eqref{k_nabla_express}_{\mathcal A} | 
& \leq \int_{\max\{0,t-\tb\}}^t e^{-\int^t_s \frac{\tilde{\nu}(X(\tau),V(\tau))}{4}\dd \tau} \int_{\mathbb{R}^3}  e^{-\phi_E(X(s))}\mathbf{k}(V(s),u) \frac{w_{\tilde{\theta}}(V(s))}{w_{\tilde{\theta}}(u)} \p_{x,v} X(s) |\mathcal{A}(X(s),u)|   \dd u \dd s 
\notag \\
& \lesssim [ o(1)(\Vert  w_{\tilde{\theta}}\p_{x,v} h \Vert_{L^p_{x,v}} + \Vert \alpha_h \nabla_x h\Vert_{L^\infty_{x,v}}) + \Vert wh\Vert_{L^\infty_{x,v}} +  | wf_b|_{L^\infty_{\p\O,v}}]  
\notag \\
& \qquad \times \int^t_{t-\tb} w^{-\frac{1}{4}}(V(s))e^{-\int^t_s \frac{\tilde{\nu}(X(\tau),V(\tau))}{4}\dd \tau} \int_{\mathbb{R}^3}  \mathbf{k}(V(s),u)\frac{w_{\tilde{\theta}}(V(s))}{w_{\tilde{\theta}}(u)}\frac{w^{\frac{1}{4}}(V(s))}{w^{\frac{1}{4}}(u)}   |\p_{x,v} X(s)|  \dd u \dd s 
\notag \\
& \lesssim [ o(1)(\Vert  w_{\tilde{\theta}}\p_{x,v} h \Vert_{L^p_{x,v}} + \Vert \alpha_h \nabla_x h\Vert_{L^\infty_{x,v}}) + \Vert wh\Vert_{L^\infty_{x,v}} +  | wf_b|_{L^\infty_{\p\O,v}} ] w^{-\frac{1}{8}}(v),
\label{k_A_bdd}
\end{align}
where in the last line we apply Lemma \ref{lemma:k_theta} and Lemma \ref{lemma:deri_XV}. This further implies that
\be \label{k_A_bdd_2}
\Vert \eqref{k_nabla_express}_{\mathcal A} \Vert_{L^p_{x,v}}  \lesssim o(1)\Vert w_{\tilde{\theta}} \p_{x,v} h \Vert_{L^p_{x,v}} + \Vert w h \Vert_{L^\infty_{x,v}} + | w f_b|_{L^\infty_{\p\O,v}}.
\ee

For $\eqref{k_nabla_express}_{\mathcal B}$, recall from $\mathcal{B}$ in \eqref{ABCD_notation} that $\mathcal{B} = \eqref{rchara:nabla_0} + \eqref{rchara:nabla_gamma}$.
Accordingly, we decompose $\eqref{k_nabla_express}_{\mathcal B}$ as
\be \notag
\eqref{k_nabla_express}_{\mathcal B} = \eqref{k_nabla_express}_{\eqref{rchara:nabla_0}} + \eqref{k_nabla_express}_{\eqref{rchara:nabla_gamma}}.
\ee
For the first component $\eqref{k_nabla_express}_{\eqref{rchara:nabla_0}}$, we estimate it by
\begin{align}
\Vert \eqref{k_nabla_express}_{\eqref{rchara:nabla_0}} \Vert_{L^p_{x,v}}
& \lesssim \Big\Vert \int_{\max\{t-\tb,0\}}^{t} \dd s e^{-\frac{\nu(v)(t-s)}{4}} \int_{\mathbb{R}^3} \mathbf{k}(v,u) \frac{w_{\tilde{\theta}}(v)}{w_{\tilde{\theta}}(u)} e^{-\frac{\nu(u)s}{4}} 
\notag \\
& \qquad \times w_{\tilde{\theta}}(u)[|\p_x h(X(0;s,x,u),V(0;s,x,u))|+  |\nabla_v h(X(0;s,x,u),V(0;s,x,u))|] \dd u \Big\Vert_{L^p_{x,v}} 
\notag \\
& \lesssim e^{-\nu_0 t/4}  \Big\Vert \nu^{-1}(v) \Vert \mathbf{k}(v,u)\frac{w_{\tilde{\theta}}(v)}{w_{\tilde{\theta}}(u)} \Vert_{L^{p'}_u}  \Vert w_{\tilde{\theta}}(u)\p_{x,v}h(X(0;s,x,u),V(0;s,x,u) )\Vert_{L^p_u} \Big\Vert_{L^p_{x,v}} 
\notag \\
& \lesssim e^{-\nu_0 t/4}  \Big\Vert \nu^{-1-1/p'}(v) \Vert w_{\tilde{\theta}}(u) \p_{x,v}h(X(0;s,x,u),V(0;s,x,u) )\Vert_{L^p_u} \Big\Vert_{L^p_{x,v}} 
\notag \\
& \lesssim e^{-\nu_0 t/4}  \Vert w_{\tilde{\theta}}\p_{x,v}h\Vert_{L^p_{x,v}} \lesssim o(1)\Vert w_{\tilde{\theta}}\p_{x,v}h\Vert_{L^p_{x,v}}, 
\label{k_B1_bdd}
\end{align}
where in the last line we use the change of variables $(X(0;s,x,u), V(0;s,x,u)) \mapsto (x,u)$ in Lemma \ref{lemma:int_cov}, and the $\dd v$ integration is bounded since $2<p<3$, which implies $p/p' + p > 3$.
For the second component $\eqref{k_nabla_express}_{\eqref{rchara:nabla_gamma}}$, we apply \eqref{B_gamma_est} to obtain
\begin{align*}
\Vert \eqref{k_nabla_express}_{\eqref{rchara:nabla_gamma}} \Vert_{L^p_{x,v}}
& \lesssim \Big\Vert \int_{\max\{t-\tb,0\}}^{t} \dd s e^{-\frac{\nu(v)(t-s)}{4}} \int_{\mathbb{R}^3}\dd u \mathbf{k}(v,u) \frac{w_{\tilde{\theta}}(v)}{w_{\tilde{\theta}}(u)} \eqref{rchara:nabla_gamma}(x,u) \Big\Vert_{L^p_{x,v}} \\ 
& \lesssim \Big\Vert \nu^{-1}(v)  \Vert \mathbf{k}(v,u)\frac{w_{\tilde{\theta}}(v)}{w_{\tilde{\theta}}(u)}\Vert_{L^{p'}_u} \Vert \eqref{rchara:nabla_gamma}(x,u) \Vert_{L^p_u} \Big\Vert_{L^p_{x,v}} \\
& \lesssim \Vert \nu^{-1-1/p'}(v)\Vert \eqref{rchara:nabla_gamma}(x,u) \Vert_{L^p_{x,u}}\Vert_{L^p_v} \lesssim \Vert \eqref{rchara:nabla_gamma}(x,u) \Vert_{L^p_{x,u}} \lesssim o(1)\Vert w_{\tilde{\theta}} \p_{x,v} h\Vert_{L^p_{x,v}} + \Vert wh\Vert_{L^\infty_{x,v}},
\end{align*}
where we apply the same method as in the estimate for $\eqref{k_nabla_express}_{\eqref{rchara:nabla_0}}$.
This, together with \eqref{k_B1_bdd}, shows that
\be \label{k_B_bdd}
\Vert \eqref{k_nabla_express}_{\mathcal{B}}\Vert_{L^p_{x,v}} \lesssim o(1)\Vert w_{\tilde{\theta}} \p_{x,v} h\Vert_{L^p_{x,v}} + \Vert wh\Vert_{L^\infty_{x,v}}. 
\ee

For $\eqref{k_nabla_express}_{\mathcal C}$, we follow \eqref{eq2:weight_W1p} and \eqref{c_bdd_infty} to obtain
\be \notag
\begin{split}
& \Vert  \eqref{k_nabla_express}_{\mathcal{C}} \Vert_{L^p_{x,v}} 
\lesssim \big[ \Vert wh\Vert_{L^\infty_{x,v}} + | wf_b|_{L^\infty_{\p\O,v}} +  |w\p_{\mathbf{x}_{p},v}f_b|_{L^\infty_{\p\O,v}} \big] 
\Big\Vert \int_{\max\{0,t-\tb\}}^t e^{-\int^t_s \tilde{\nu}(X(\tau),V(\tau))\dd \tau} \frac{w_{\tilde{\theta}}(v)}{w_{\tilde{\theta}}(V(s))}
\\& \qquad \times \int_{\mathbb{R}^3} \mathbf{1}_{\tb(X(s),u)<s} \mathbf{k}(V(s),u) \frac{w_{\tilde{\theta}}(V(s))}{w_{\tilde{\theta}}(u)} \p_{x,v} X(s) \frac{e^{-\nu_0 \tb(\xb(X(s),u)) /4}}{|n(\xb(X(s),u))\cdot \vb(X(s),u)|} w^{-\frac{1}{2}}(\vb) \dd u \dd s \Big\Vert_{L^p_{x,v}}
\\& \lesssim [\Vert wh\Vert_{L^\infty_{x,v}}+ | wf_b|_{L^\infty_{\p\O,v}}+ |w\p_{\mathbf{x}_p,v}f_b|_{L^\infty_{\p\O,v}}] \Big\Vert \int^t_0 \dd s w_{\tilde{\theta}}(V(s))w^{-\frac{1}{8}}(V(s))e^{-\frac{\nu_0(t-s)}{4}}  
\\& \qquad \times \int_{\mathbb{R}^3} \mathbf{1}_{\tb(X(s),u)<s}\mathbf{k}(V(s),u)\frac{w^{\frac{1}{8}}(V(s))}{w^{\frac{1}{8}}(u)}\frac{w^{\frac{1}{8}}(u)}{w^{\frac{1}{4}}(\vb(X(s),u))}   \frac{w^{-\frac{1}{4}}(\vb(X(s),u))e^{-\frac{\nu_0\tb(X(s),u)}{4}}}{n(\xb(X(s),u))\cdot \vb(X(s),u)}  \dd u \Big\Vert_{L^p_{x,v}}. 
\end{split}
\ee
Applying Minkowski's inequality in $L^p_{x,v}$ and the H\"older's inequality in $\dd u$, we further obtain
\begin{align}
& \Vert  \eqref{k_nabla_express}_{\mathcal{C}} \Vert_{L^p_{x,v}} 
\lesssim [\Vert wh\Vert_{L^\infty_{x,v}}+ | wf_b|_{L^\infty_{\p\O,v}}+ |w\p_{\mathbf{x}_p,v}f_b|_{L^\infty_{\p\O,v}}] 
\notag \\
& \qquad \times \int_0^t \dd s e^{-\nu_0(t-s)} \Big\Vert w^{-\frac{1}{16}}(V(s))\frac{w^{-\frac{1}{4}}(\vb(X(s),u))e^{-\frac{\nu_0\tb(X(s),u)}{4}}}{n(\xb(X(s),u))\cdot \vb(X(s),u)}  \Big\Vert_{L^p_{x,v,u}} 
\notag \\
& \lesssim [\Vert wh\Vert_{L^\infty_{x,v}}+ | wf_b|_{L^\infty_{\p\O,v}}+ |w\p_{\mathbf{x}_p,v}f_b|_{L^\infty_{\p\O,v}}] \int_0^t \dd s e^{-\nu_0(t-s)} \Big\Vert w^{-\frac{1}{16}}(v)\frac{w^{-\frac{1}{4}}(\vb(x,u))e^{-\frac{\nu_0\tb(x,u)}{4}}}{n(\xb(x,u))\cdot \vb(x,u)}  \Big\Vert_{L^p_{x,v,u}} 
\notag \\
& \lesssim [\Vert wh\Vert_{L^\infty_{x,v}}+ | wf_b|_{L^\infty_{\p\O,v}}+ |w\p_{\mathbf{x}_p,v}f_b|_{L^\infty_{\p\O,v}}] \Big\Vert \frac{w^{-\frac{1}{4}}(\vb(x,u))e^{-\frac{\nu_0\tb(x,u)}{4}}}{n(\xb(x,u))\cdot \vb(x,u)} \Big\Vert_{L^p_{x,u}}  
\notag \\
& \lesssim \Vert wh\Vert_{L^\infty_{x,v}} + | wf_b|_{L^\infty_{\p\O,v}} + |w\p_{\mathbf{x}_p,v}f_b|_{L^\infty_{\p\O,v}},
\label{minkowski_simplify}
\end{align}
where in the third line we use the change of variables $(X(s), V(s)) \mapsto (x,v)$ in Lemma \ref{lemma:int_cov}, and in the last line we apply Lemma \ref{lemma:integrate_nv} with $p<3$.
We remark that, with the extra exponential weight $w^{-\frac{1}{16}}(v)$, the computation can be simplified (compared to \eqref{gamma_gain_est}) by directly applying Minkowski's inequality.

For $\eqref{k_nabla_express}_{\mathcal D}$, which is the most delicate part, we write out the full expression:
\be \label{kk} 
\begin{split}
\eqref{k_nabla_express}_{\mathcal D} 
& =  \int^t_{t-\tb} \dd s e^{-\int^t_s \tilde{\nu}(X(\tau),V(\tau))\dd \tau} \frac{w_{\tilde{\theta}}(v)}{w_{\tilde{\theta}}(V(s))}\int_{\mathbb{R}^3} \dd u e^{-\phi_E(X(s))}\mathbf{k}(V(s),u)\frac{w_{\tilde{\theta}}(V(s))}{w_{\tilde{\theta}}(u)} \p_{x,v} X(s)  
\\& \qquad \times \int^s_{s-\tb(X(s),u)} \dd s'  e^{-\int^s_{s'} \tilde{\nu}(X(\tau',s),V(\tau',s))\dd \tau'}  \frac{w_{\tilde{\theta}}(u)}{w_{\tilde{\theta}}(V(s',s))} 
\\& \qquad \times \int_{\mathbb{R}^3} \dd u'e^{-\phi_E(X(s',s))}\mathbf{k}(V(s',s),u')\frac{w_{\tilde{\theta}}(V(s',s))}{w_{\tilde{\theta}}(u')} \p_{x,u} X(s',s) w_{\tilde{\theta}}(u') \nabla_x h(X(s',s),u').  
\end{split}
\ee
To compute \eqref{kk}, we pick $\delta$ and $N$ such that
\[
0 < \delta \ll 1 \ll N,
\]
and we consider three cases: (1) $s-s'<\delta$, \ (2) $|u|<N^{-1}$, \ (3) $s' \leq s-\delta$ and $|u| \geq N^{-1}$.

\smallskip

\textit{Case 1: $s-s'<\delta$.}
Using \eqref{deri_XV_bdd} and \eqref{eq2:weight_W1p}, we have
\begin{align*}
& \Vert  \eqref{kk}\mathbf{1}_{s-s'<\delta}\Vert_{L^p_{x,v}}   \lesssim \Big\Vert \int^t_{t-\tb} \dd s e^{-\frac{\nu(V(s))(t-s)}{4}} \int_{\mathbb{R}^3} \dd u \mathbf{k}(V(s),u) \frac{w_{\tilde{\theta}}(V(s))}{w_{\tilde{\theta}}(u)} \\
& \qquad \times \int^s_{s-\delta} \dd s ' e^{-\frac{\nu(V(s',s))(s-s')}{4}} \int_{\mathbb{R}^3} \dd u' \mathbf{k}(V(s',s),u')\frac{w_{\tilde{\theta}}(V(s',s))}{w_{\tilde{\theta}}(u')} w_{\tilde{\theta}}(u')\nabla_x h(X(s',s),u')   \Big\Vert_{L^p_{x,v}}.
\end{align*}
Similar to \eqref{gamma_gain_est}, we need sufficient decay in the velocity weights to control the $\dd v$ and $\dd u$ integrations. Since $2 < p < 3$, there exists a sufficiently small constant $\varepsilon > 0$ such that
\[
\frac{2 p}{p'} + 1 > (2-\varepsilon)\frac{p}{p'} + 1 > 3.
\]
This implies the integrability of $\nu^{-2p/p'-1}(v)$ and $\nu^{-(2-\varepsilon)p/p'-1}(u)$ with respect to $\dd v$ and $\dd u$, respectively.
Thus, we consecutively apply H\"older's inequality in the order $\dd u', \dd s', \dd u, \dd s$, and compute
\begin{align*}
& \Vert \eqref{kk}\mathbf{1}_{s-s'<\delta} \Vert_{L^p_{x,v}}   
\lesssim \Big\Vert \int^t_0 \dd s e^{-\frac{\nu(V(s))(t-s)}{4}} \int_{\mathbb{R}^3} \dd u \mathbf{k}(V(s),u)\frac{w_{\tilde{\theta}}(V(s))}{w_{\tilde{\theta}}(u)}\int^s_{s-\delta} \dd s' e^{-(\frac{1}{p'}+\frac{1}{p})\frac{\nu(V(s',s))(s-s')}{4}} \\
& \qquad \times \Big\Vert \mathbf{k}(V(s),u)\frac{w_{\tilde{\theta}}(V(s))}{w_{\tilde{\theta}}(u)} \Big\Vert_{L^{p'}_{u'}}\Vert w_{\tilde{\theta}}(u')\nabla_x h(X(s',s),u')\Vert_{L^p_{u'}}\Big\Vert_{L^p_{x,v}} \\
& \lesssim \Big\Vert \int_0^t \dd s e^{-\frac{\nu(V(s))(t-s)}{4}} \int_{\mathbb{R}^3} \dd u \mathbf{k}(V(s),u) \frac{w_{\tilde{\theta}}(V(s))}{w_{\tilde{\theta}}(u)}\delta^{\varepsilon/p'}\nu^{-(2-\varepsilon)/p'}(V(s',s)) \\
& \qquad \times \Big(\int^s_{s-\delta} \dd s' e^{-\frac{\nu(V(s',s))(s-s')}{4}} \Vert w_{\tilde{\theta}}(u')\nabla_x h(X(s',s),u)\Vert_{L^p_{u'}}^p  \Big)^{1/p} \Big\Vert_{L^p_{x,v}} \\
& \lesssim  o(1)\Big\Vert \int_0^t \dd s e^{-\frac{\nu(V(s))(t-s)}{4}} \Big\Vert \mathbf{k}(V(s),u)\frac{w_{\tilde{\theta}}(V(s))}{w_{\tilde{\theta}}(u)} \Big\Vert_{L^{p'}_u} \\
& \qquad \times \Big(\int_u\int^s_{0}\dd s' \nu^{-(2 -\varepsilon)p/p'}(V(s',s))e^{-p\frac{\nu(V(s',s))(s-s')}{4}}\Vert w_{\tilde{\theta}}(u') \nabla_x h(X(s',s),u')\Vert^p_{L^p_{u'}} \Big)^{1/p}  \Big\Vert_{L^p_{x,v}} \\
& \lesssim o(1)\Big\Vert \int_0^t \dd s e^{-(\frac{1}{p'}+\frac{1}{p})\frac{\nu(V(s))(t-s)}{4}} \nu^{-1/p'}(V(s)) \\
& \qquad \times \Big(\int_u\int^s_{0}\dd s' \nu^{-(2 -\varepsilon)p/p'}(V(s',s))e^{-\frac{\nu(V(s',s))(s-s')}{4}}\Vert w_{\tilde{\theta}}(u')\nabla_x h(X(s',s),u')\Vert^p_{L^p_{u'}} \Big)^{1/p} \Big\Vert_{L^p_{x,v}} \\
& \lesssim   o(1)\Big\Vert   \nu^{-2/p'}(V(s))\Big(\int_0^t\int_u\int_0^s \dd s\dd s' \nu^{-(2 -\varepsilon)p/p'}(V(s',s))e^{-\frac{\nu(V(s))(t-s)}{4}}e^{-\frac{\nu(V(s',s))(s-s')}{4}} \\
& \qquad \times \Vert w_{\tilde{\theta}}(u')\nabla_x h(X(s',s),u')\Vert^p_{L^p_{u'}} \Big)^{1/p}\Big\Vert_{L^p_{x,v}} \\
& \lesssim o(1)\Big(\iint_{x,v,u,u'}\int_0^t\int^s_0 \dd s\dd s' \nu^{-2p/p'}(V(s))\nu^{-(2 -\varepsilon)p/p'}(V(s',s)) e^{-\frac{\nu(V(s))(t-s)}{4}} e^{-\frac{\nu(V(s',s))(s-s')}{4}} \\
& \qquad \times |w_{\tilde{\theta}}(u') \nabla_x h(X(s',s),u')|^p \Big)^{1/p}. 
\end{align*}
We now apply the change of variables $(X(s),V(s)) \mapsto (x,v)$, so that 
\[
(X(s',s), V(s',s)) \to (X(s';s,x,u),V(s';s,x,u)).
\]
We then apply a further change of variables $(X(s';s,x,u), V(s';s,x,u)) \mapsto (x,u)$ and obtain
\be \label{kk_a}
\begin{split}
& \Vert \eqref{kk}\mathbf{1}_{s-s'<\delta}\Vert_{L^p_{x,v}}   
\lesssim o(1)\Big(\iint_{x,v,u,u'}\int_0^t\int^s_0 \dd s\dd s' \nu^{-2p/p'}(v)\nu^{-(2 -\varepsilon)p/p'}(V(s';s,x,u)) e^{-\frac{\nu(v)(t-s)}{4}} 
\\& \qquad \times e^{-\frac{\nu(V(s';s,x,u))(s-s')}{4}}  |w_{\tilde{\theta}}(u')\nabla_xh(X(s';s,x,u),u')|^p \Big)^{1/p} 
\\& \lesssim o(1)\Big(\iint_{x,v,u,u'}\int_0^t\int^s_0 \dd s\dd s' \nu^{-2p/p'}(v)\nu^{-(2 -\varepsilon)p/p'}(u) e^{-\frac{\nu(v)(t-s)}{4}} e^{-\frac{\nu(u)(s-s')}{2}} |w_{\tilde{\theta}}(u')\nabla_xh(x,u')|^p \Big)^{1/p} 
\\& \lesssim o(1)\Big( \iint_{x,v,u,u'} \nu^{-2p/p'-1}(v) \nu^{-(2 -\varepsilon)p/p'-1}(u) |w_{\tilde{\theta}}(u')\nabla_x h(x,u')|^p  \Big)^{1/p} \lesssim o(1)\Vert w_{\tilde{\theta}}\nabla_x h\Vert_{L^p_{x,v}}. 
\end{split}
\ee
where in the last line the $\dd v$ and $\dd u$ integration are bounded due to $2 p/p'+1 > (2 - \varepsilon) p/p'+1 >3$.

\smallskip

\textit{Case 2: $|u|<N^{-1}$.}
We apply the same computation as in the case $s-s'<\delta$ and derive
\be \label{kk_b}
\Vert \eqref{kk}\mathbf{1}_{|u|<N^{-1}} \Vert_{L^p_{x,v}} \lesssim o(1)\Vert  w_{\tilde{\theta}}(v)\nabla_x h\Vert_{L^p_{x,v}}. 
\ee
Here, the $o(1)$ term arises from H\"older's inequality in $\dd u$ and Lemma \ref{lemma:k_theta}, namely,
\be \notag
\Big\Vert \mathbf{1}_{|u|<N^{-1}} \mathbf{k}(V(s),u)\frac{w_{\tilde{\theta}}(V(s))}{w_{\tilde{\theta}}(u)} \Big\Vert_{L^{p'}_u} \lesssim o(1)\nu^{-1/p'}(V(s)).
\ee


\smallskip

\textit{Case 3: $s' \leq s-\delta$ and $|u| \geq N^{-1}$.}
We consider the following change of variables:
\be \label{eq4:weight_W1p}
\nabla_x h(X(s',s),u') = \nabla_{u}[h(X(s',s),u')]  (\p_{u} X(s',s))^{-1}.
\ee
Since $s' \leq s-\delta$, \eqref{est:X_v second} in Lemma \ref{lemma:deri_XV} implies that
\be \label{nablau_X_lowerbdd}
\big|(\nabla_{u} X(s',s))^{-1}\big| \lesssim |s-s'|^{-1} < \delta^{-1}.    
\ee
Using the change of variables \eqref{eq4:weight_W1p} and canceling all the factors involving $w_{\tilde{\theta}}$ in \eqref{kk}, we get
\begin{align*}
& \eqref{kk} \mathbf{1}_{s-s'>\delta \text{ and }|u|>N^{-1}} \\
&=\int^t_{t-\tb} \dd s e^{-\int^t_s \tilde{\nu}(X(\tau),V(\tau))\dd \tau} w_{\tilde{\theta}}(v)\int_{|u|>N^{-1}} \dd u e^{-\phi_E(X(s))} \mathbf{k}(V(s),u)\p_{x,v} X(s)  \int^{s-\delta}_{s-\tb(X(s),u)} \dd s' \\
& \qquad \times e^{-\int^s_{s'} \tilde{\nu}(X(\tau',s),V(\tau',s))\dd \tau'}  \int_{\mathbb{R}^3} \dd u' e^{-\phi_E(X(s',s))} \mathbf{k}(V(s',s),u') \\ 
& \qquad \times \p_{x,u} X(s',s) \nabla_{u} [h(X(s',s),u')] (\nabla_{u} X(s',s))^{-1}.
\end{align*}
Applying integration by parts in $u$, we further obtain
\begin{align}
& |\eqref{kk}\mathbf{1}_{s-s'>\delta \text{ and }|u|>N^{-1}} | 
\notag \\
& \leq  \int^t_{t-\tb} \dd s w_{\tilde{\theta}}(v)e^{-\int^t_s \tilde{\nu}(X(\tau),V(\tau))\dd \tau} \int_{\mathbb{R}^3} \dd u  |\p_{x,v} X(s)|  \int^{s-\delta}_{s-\tb(X(s),u)} \dd s' e^{-\int^s_{s'} \tilde{\nu}(X(\tau',s),V(\tau',s))\dd \tau'}
\notag \\
& \qquad \times \int_{\mathbb{R}^3} \dd u' \nabla_{u} [\mathbf{k}(V(s),u)\mathbf{k}(V(s',s),u') ]|\p_{x,u} X(s',s)  h(X(s',s),u') (\nabla_{u} X(s',s))^{-1}| 
\label{kk_ibp_1} \\ 
& + \int^t_{t-\tb} \dd s w_{\tilde{\theta}}(v) e^{-\int^t_s \tilde{\nu}(X(\tau),V(\tau))\dd \tau} \int_{\mathbb{R}^3} \dd u  \mathbf{k}(V(s),u) |\p_{x,v} X(s)|  \int^{s-\delta}_{s-\tb(X(s),u)} \dd s' e^{-\int^s_{s'} \tilde{\nu}(X(\tau',s),V(\tau',s))\dd \tau'} \notag \\
& \qquad \times \int_{\mathbb{R}^3} \dd u'  \mathbf{k}(V(s',s),u')   |h(X(s',s),u')| |\nabla_{u} [\p_{x,u} X(s',s)(\p_{u} X(s',s))^{-1}]| 
\label{kk_ibp_2} \\
& + \int^t_{t-\tb} \dd s w_{\tilde{\theta}}(v) e^{-\int^t_s \tilde{\nu}(X(\tau),V(\tau))\dd \tau} \int_{|u|>N^{-1}} \dd u  |\p_{x,v} X(s)|  \mathbf{k}(V(s),u) \nabla_{u} \tb(X(s),u)  
\notag \\
& \qquad \times e^{-\int^s_{s-\tb(X(s),u)} \tilde{\nu}(X(\tau',s),V(\tau',s))\dd \tau'} \int_{\mathbb{R}^3} \dd u' \mathbf{k}(\vb(X(s),u),u') |\p_x X(s-\tb(X(s),u),s)| 
\notag \\
& \qquad \times | f_b(X(s-\tb(X(s),u),s),u') | |(\nabla_{u} X(s-\tb(X(s),u),s))^{-1}| 
\label{kk_ibp_3} \\
& + \int^t_{t-\tb} \dd s w_{\tilde{\theta}}(v) e^{-\int^t_s \tilde{\nu}(X(\tau),V(\tau))\dd \tau} \int_{\mathbb{R}^3} \dd u  \mathbf{k}(V(s),u) |\p_{x,v} X(s)|  \int^{s-\delta}_{s-\tb(X(s),u)} \dd s' e^{-\int^s_{s'} \tilde{\nu}(X(\tau',s),V(\tau',s))\dd \tau'} \notag \\
& \qquad \times \int_{\mathbb{R}^3} \dd u' |\nabla_x \phi_E(X(s',s)) \nabla_{u}X(s',s)| \mathbf{k}(V(s',s),u')    
\notag \\
& \qquad \times |h(X(s',s),u')| |\p_{x,u} X(s',s)  h(X(s',s),u') (\nabla_{u} X(s',s))^{-1}| \label{kk_ibp_4} \\
& + \int^t_{t-\tb} \dd s w_{\tilde{\theta}}(v) e^{-\int^t_s \tilde{\nu}(X(\tau),V(\tau))\dd \tau} \int_{|u|=N^{-1}} \dd u  \mathbf{k}(V(s),u) | \p_{x,v} X(s) | \int^{s-\delta}_{s-\tb(X(s),u)} \dd s' e^{-\int^s_{s'} \tilde{\nu}(X(\tau',s),V(\tau',s))\dd \tau'} \notag \\
& \qquad \times \int_{\mathbb{R}^3} \dd u'\mathbf{k}(V(s',s),u') |\p_{x,u} X(s',s) h(X(s',s),u') (\nabla_{u} X(s',s))^{-1} | \label{kk_ibp_5} \\
& + \int^t_{t-\tb} \dd s w_{\tilde{\theta}}(v)e^{-\int^t_s \tilde{\nu}(X(\tau),V(\tau))\dd \tau} \int_{\mathbb{R}^3} \dd u \mathbf{k}(V(s),u)  |\p_{x,v} X(s)|  \int^{s-\delta}_{s-\tb(X(s),u)} \dd s' e^{-\int^s_{s'} \tilde{\nu}(X(\tau',s),V(\tau',s))\dd \tau'}
\notag \\
& \qquad \times \int^s_{s'} \dd s |\p_{u}[\tilde{\nu}(X(\tau',s),V(\tau',s))]| \dd \tau'   \int_{\mathbb{R}^3} \dd u' \mathbf{k}(V(s',s),u') |\p_{x,u} X(s',s)  h(X(s',s),u') (\nabla_{u} X(s',s))^{-1}|. 
\label{kk_ibp_6} 
\end{align}

For \eqref{kk_ibp_1}, we compute that
\begin{align*}
|\eqref{kk_ibp_1}| 
& \lesssim C(\delta,t)  \Vert wh\Vert_{L^\infty_{x,v}}\int^t_{t-\tb} \dd s w_{\tilde{\theta}}(v)e^{-\int^t_s \tilde{\nu}(X(\tau),V(\tau))\dd \tau} \int_{\mathbb{R}^3} \dd u  \int^{s-\delta}_{s-\tb(X(s),u)} \dd s' \\
& \qquad \times e^{-\int^s_{s'} \tilde{\nu}(X(\tau',s),V(\tau',s))\dd \tau'}  \int_{\mathbb{R}^3} \dd u' |\nabla_{u} [\mathbf{k}(V(s),u)\mathbf{k}(V(s',s),u') ]|  w^{-1}(u') \\
& \lesssim  C(\delta,t)\Vert wh\Vert_{L^\infty_{x,v}} w^{-\frac{1}{8}}(v).
\end{align*}
Here, we use Lemma \ref{lemma:v_variation} and Lemma \ref{lemma:k_theta} to obtain the following bounds for the $\dd u$ and $\dd u'$ integrations:
\begin{align*}
& \int_{\mathbb{R}^3} \mathbf{k}(V(s',s),u') w^{-1}(u') \dd u'
\lesssim w^{-1}(V(s',s)) \lesssim w^{-\frac{1}{2}}(u), \\
& \int_{\mathbb{R}^3} |\nabla_{u} \mathbf{k}(V(s',s),u')| w^{-1}(u') \dd u'
\lesssim w^{-1}(V(s',s)) \int_{\mathbb{R}^3} |\nabla_{u} \mathbf{k}(V(s',s),u')| \frac{w(V(s',s))}{w(u')} \dd u' \\
& \lesssim w^{-1}(V(s',s)) (1+|V(s',s)|^2) \int_{\mathbb{R}^3} \frac{\mathbf{k}_{\tilde{\varrho}}(V(s',s),u)}{|V(s',s)-u|} \dd u
\lesssim w^{-\frac{1}{2}}(V(s',s)) \lesssim w^{-\frac{1}{4}}(u).
\end{align*}
Therefore, we conclude 
\begin{align*}
    &  \Vert \eqref{kk_ibp_1} \Vert_{L^p_{x,v}} \lesssim \Vert wh\Vert_{L^\infty_{x,v}}.
\end{align*}

For \eqref{kk_ibp_2}, we use \eqref{nablau_X_lowerbdd} to control the term $| \nabla_u [(\p_u X(s',s))^{-1}] |$ in \eqref{kk_ibp_2} as follows:
\be \notag
| \nabla_u [(\p_u X(s',s))^{-1}] | \lesssim |\nabla_u X(s',s)|^{-1} |\nabla_u X(s',s)|^{-1} |\nabla_u^2 X(s',s)| \lesssim \frac{1}{\delta^2}|\nabla_u^2 X(s',s)|.
\ee
Similar to \eqref{minkowski_simplify}, applying Minkowski's inequality in $L^p_{x,v}$ and the H\"older's inequality in $\dd u$, we deduce
\begin{align*}
& \Vert \eqref{kk_ibp_2} \Vert_{L^p_{x,v}} \lesssim C(\delta, N,t) \Vert wh\Vert_{L^\infty_{x,v}} \Big\Vert w_{\tilde{\theta}}(v)\int_0^t \dd s e^{-\frac{\nu_0(t-s)}{4}} w^{-\frac{1}{2}}(V(s))\int_{\mathbb{R}^3} \dd u \mathbf{k}(V(s),u) \frac{w^{\frac{1}{2}}(V(s))}{w^{\frac{1}{2}}(u)}\frac{w^{\frac{1}{2}}(u)}{w^{\frac{3}{4}}(V(s',s))} \\
& \quad \times \int^{s-\delta}_0 \dd s' e^{-\frac{\nu_0(s-s')}{4}}\int_{\mathbb{R}^3} \dd u' \mathbf{k}(V(s',s),u') \frac{w(V(s',s))}{w(u')} w^{-\frac{1}{4}}(V(s',s))[|\nabla_{u}^2 X(s',s)| + |\nabla_{x}\nabla_{u}X(s',s)|] \Big\Vert_{L^p_{x,v}} \\
& \lesssim \Vert wh\Vert_{L^\infty_{x,v}} \int_0^t \dd s e^{-\frac{\nu_0(t-s)}{4}}\int_0^s\dd s' e^{-\frac{\nu_0(s-s')}{4}} \\
& \qquad \times \Big\Vert w^{-\frac{1}{8}}(V(s))\mathbf{k}(V(s),u) \frac{w^{\frac{1}{2}}(V(s))}{w^{\frac{1}{2}}(u)}w^{-\frac{1}{4}}(V(s',s))[|\nabla_{u}^2 X(s',s)| + |\nabla_{x}\nabla_{u}X(s',s)|]\Big\Vert_{L^p_{x,v,u}}.
\end{align*}
We now apply the change of variables $(X(s),V(s)) \mapsto (x,v)$, so that 
\[
(X(s',s), V(s',s)) \to (X(s';s,x,u),V(s';s,x,u)).
\]
We then apply a further change of variables $(X(s';s,x,u), V(s';s,x,u)) \mapsto (x,u)$ and obtain
\begin{align*}
& \Vert \eqref{kk_ibp_2} \Vert_{L^p_{x,v}}
\lesssim \Vert wh\Vert_{L^\infty_{x,v}} \int_0^t \dd s e^{-\frac{\nu_0(t-s)}{4}}\int_0^s\dd s' e^{-\frac{\nu_0(s-s')}{4}}\\
& \qquad \times \Big\Vert w^{-\frac{1}{8}}(v)\mathbf{k}(v,u) \frac{w^{\frac{1}{2}}(v)}{w^{\frac{1}{2}}(u)}w^{-\frac{1}{4}}(V(s';s,x,u))[|\nabla_{u}^2 X(s';s,x,u)| + |\nabla_{x}\nabla_{u}X(s';s,x,u)|]\Big\Vert_{L^p_{x,v,u}} \\
& \lesssim  \Vert wh\Vert_{L^\infty_{x,v}} \int_0^t \dd s e^{-\frac{\nu_0(t-s)}{4}}\int_0^s\dd s' e^{-\frac{\nu_0(s-s')}{4}}\big\Vert w^{-\frac{1}{8}}(u)[|\nabla_{u}^2 X(s';s,x,u)| + |\nabla_{x}\nabla_{u}X(s';s,x,u)]|\big\Vert_{L^p_{x,u}} \\
& \lesssim \Vert wh\Vert_{L^\infty_{x,v}} \big( \Vert \nabla_x h\Vert_{L^p_{x,v}} + \Vert wh\Vert_{L^\infty_{x,v}} {\color{blue}+ \Vert \nabla_x^3 \phi_E\Vert_{L^p_x}} \big),
\end{align*}
where in the third line the $\dd v$ integration is bounded since $p<3$, and in the last line we apply Lemma \ref{lemma:X_vv_Lp} and Lemma \ref{lemma:W3p}.

For $\eqref{kk_ibp_3}$, we apply Lemma \ref{lemma:deri_backward} and obtain
\begin{align*}
& \Vert \eqref{kk_ibp_3}\Vert_{L^p_{x,v}} \lesssim C(\delta,t,N)|wf_b|_{L^\infty_{\p\O,v}} \Big\Vert w_{\tilde{\theta}}(V(s)) \int^t_{t-\tb} \dd s e^{-\frac{\nu_0(t-s)}{2}} w^{-\frac{1}{4}}(V(s))\int_{u} \dd u    \mathbf{k}(V(s),u) \frac{w^{\frac{1}{4}}(V(s))}{w^{\frac{1}{4}}(u)} \\
& \qquad \times \frac{w^{\frac{1}{4}}(u)}{w^{\frac{1}{2}}(\vb(X(s),u))} \frac{te^{-\nu_0\tb(X(s),u)}w^{-\frac{1}{2}}(\vb(X(s),u))}{n(\xb(X(s),u))\cdot \vb(\xb(X(s),u))}\int_{\mathbb{R}^3} \dd u' \mathbf{k}(\vb(X(s),u)),u')\frac{w(\vb(X(s),u))}{w(u')}   \Big\Vert_{L^p_{x,v}} \\
& \lesssim | wf_b|_{L^\infty_{\p\O,v}}\Big\Vert \int^t_{t-\tb} \dd s e^{-\frac{\nu_0(t-s)}{2}}w^{-\frac{1}{8}}(V(s)) \Big\Vert \frac{e^{-\nu_0\tb(X(s),u)}w^{-\frac{1}{2}}(\vb(X(s),u))}{n(\xb(X(s),u))\cdot \vb(\xb(X(s),u))} \Big\Vert_{L^p_u}  \Big\Vert_{L^p_{x,v}} \\
&  \lesssim | wf_b|_{L^\infty_{\p\O,v}} \Big( \int_{x,v,u} \int^t_0  e^{-\frac{\nu_0(t-s)}{2}}w^{-\frac{1}{8}}(V(s)) \frac{e^{-p \nu_0\tb(X(s),u)}w^{-\frac{p}{2}}(\vb(X(s),u))}{|n(\xb(X(s),u))\cdot \vb(\xb(X(s),u))|^p}  \Big)^{1/p} \\
& \lesssim | wf_b|_{L^\infty_{\p\O,v}} \Big( \int_{x,v,u}\int_0^t e^{-\frac{\nu_0(t-s)}{2}} w^{-\frac{1}{8}}(v)\frac{e^{-p \nu_0\tb(x,u)}w^{-\frac{p}{2}}(\vb(x,u))}{|n(\xb(x,u))\cdot \vb(\xb(x,u))|^p}   \Big)^{1/p} 
\lesssim | wf_b|_{L^\infty_{\p\O,v}},
\end{align*}
where in the third and fourth lines we apply H\"older's inequality in $\dd u$ and $\dd s$, respectively, and in the last line we apply the change of variables $(X(s),V(s)) \mapsto (x,v)$ and Lemma \ref{lemma:integrate_nv}.

For \eqref{kk_ibp_4}, we control it in a similar way as in \eqref{kk_ibp_1} and obtain
\be \notag
\Vert \eqref{kk_ibp_4}\Vert_{L^p_{x,v}} \lesssim \Vert wh\Vert_{L^\infty_{x,v}}.
\ee

For \eqref{kk_ibp_5}, we apply Minkowski's inequality in $L^p_{x,v}$ and obtain
\begin{align*}
& \Vert \eqref{kk_ibp_5} \Vert_{L^p_{x,v}}
\lesssim C(\delta)\Vert wh\Vert_{L^\infty_{x,v}} \Big\Vert w_{\tilde{\theta}}(V(s)) w^{-\frac{1}{2}}(V(s))\int^t_{t-\tb} \dd s e^{-\int^t_s \tilde{\nu}(X(\tau),V(\tau))\dd \tau} \int_{|u|=N^{-1}} \dd u \mathbf{k}(V(s),u) w^{-\frac{1}{4}}(u) \\
& \qquad \times \frac{w^{\frac{1}{4}}(V(s))}{w^{\frac{1}{4}}(u)} \int^{s-\delta}_{s-\tb(X(s),u)} \dd s' \frac{w^{\frac{1}{2}}(u)}{w(V(s',s))} e^{-\int^s_{s'} \tilde{\nu}(X(\tau',s),V(\tau',s))\dd \tau'}\int_{\mathbb{R}^3} \dd u'\mathbf{k}(V(s',s),u') \frac{w(V(s',s))}{w(u)} \Big\Vert_{L^p_{x,v}} \\
& \lesssim  \Vert wh\Vert_{L^\infty_{x,v}} \Big\Vert  \int^t_{0} \dd s e^{-\frac{\nu_0(t-s)}{2}} \int_{|u|=N^{-1}} \dd u \mathbf{k}(V(s),u)\frac{w^{\frac{1}{4}}(V(s))}{w^{\frac{1}{4}}(u)} w^{-\frac{1}{4}}(u)\Big\Vert_{L^p_{x,v}} \\
& \lesssim \Vert wh\Vert_{L^\infty_{x,v}} \int_0^t \dd s e^{-\frac{\nu_0(t-s)}{2}} \int_{|u|=N^{-1}} \dd u \Big\Vert \mathbf{k}(V(s),u)\frac{w^{\frac{1}{4}}(V(s))}{w^{\frac{1}{4}}(u)}\Big\Vert_{L^p_{x,v}} w^{-\frac{1}{4}}(u) \\
& \lesssim  \Vert wh\Vert_{L^\infty_{x,v}} \int_0^t \dd s e^{-\frac{\nu_0(t-s)}{2}} \int_{|u|=N^{-1}} \dd u \Big\Vert \mathbf{k}(v,u)\frac{w^{\frac{1}{4}}(v)}{w^{\frac{1}{4}}(u)}\Big\Vert_{L^p_{x,v}} w^{-\frac{1}{4}}(u)\lesssim \Vert wh\Vert_{L^\infty_{x,v}},
\end{align*}
where in the last line we apply the change of variables $(X(s), V(s)) \mapsto (x,v)$.


For \eqref{kk_ibp_6}, using \eqref{nabla_nu} we compute that
\begin{align*}
& |\eqref{kk_ibp_6}| \lesssim \Vert wh\Vert_{L^\infty_{x,v}}\int_{t-\tb}^t \dd s w_{\tilde{\theta}}(v) e^{-\frac{\nu_0(t-s)}{4}} w^{-\frac{1}{2}}(V(s))\int_{\mathbb{R}^3} \dd u \mathbf{k}(V(s),u) \frac{w^{\frac{1}{2}}(V(s))}{w^{\frac{3}{4}}(u)} \frac{w^{\frac{1}{2}}(u)}{w(V(s',s))} \\
& \qquad \times \int^{s-\delta}_{s-\tb(X(s),u)} e^{-\frac{\nu_0(s-s')}{4}}w^{\frac{1}{4}}(u)\nu(u)[o(1)(\Vert \nabla_x h\Vert_{L^p_{x,v}}+\Vert \alpha_h \nabla_x h\Vert_{L^\infty_{x,v}}) +1] \\
& \qquad \times \int_{\mathbb{R}^3} \dd u' \mathbf{k}(V(s',s),u') \frac{w(V(s',s))}{w(u')} \\
& \lesssim w^{-\frac{1}{4}}(v)[o(1)(\Vert \nabla_x h\Vert_{L^p_{x,v}}+\Vert \alpha_h \nabla_x h\Vert_{L^\infty_{x,v}}) +\Vert wh\Vert_{L^\infty_{x,v}}].
\end{align*}
Therefore, we obtain
\begin{align*}
    &    \Vert |\eqref{kk_ibp_6}|\Vert_{L^p_{x,v}} \lesssim \Vert wh\Vert_{L^\infty_{x,v}}(\Vert \nabla_x h\Vert_{L^p_{x,v}}+\Vert \alpha_h \nabla_x h\Vert_{L^\infty_{x,v}}) +\Vert wh\Vert_{L^\infty_{x,v}}.
\end{align*}
Collecting the estimates from \eqref{kk_ibp_1}-\eqref{kk_ibp_6}, we conclude that
\be \label{kk_c}
\begin{split}
& \Vert \eqref{kk} \mathbf{1}_{s-s'>\delta \text{ and } |u|>N^{-1}} \Vert_{L^p_{x,v}} 
\\& \lesssim C(\delta,t,N) (\Vert wh\Vert_{L^\infty_{x,v}} + | wf_b|_{L^\infty_{\p\O,v}}) [1+\Vert \nabla_x h\Vert_{L^p_{x,v}} + \Vert \alpha_h \nabla_x h\Vert_{L^\infty_{x,v}}]. 
\end{split}
\ee

Combining \eqref{kk_a}, \eqref{kk_b}, and \eqref{kk_c}, we derive that
\be \notag
\Vert \eqref{kk} \Vert_{L^p_{x,v}}
\lesssim C(\delta,t,N) [\Vert wh\Vert_{L^\infty_{x,v}}+ | wf_b|_{L^\infty_{\p\O,v}}] + [o(1)+\Vert wh\Vert_{L^\infty_{x,v}} ](\Vert  w_{\tilde{\theta}}\nabla_x h\Vert_{L^p_{x,v}} +\Vert \alpha_h \nabla_x h\Vert_{L^\infty_{x,v}}). 
\ee
This, together with \eqref{k_A_bdd_2}, \eqref{k_B_bdd}, \eqref{minkowski_simplify}, and the fact that $\eqref{kk} = \eqref{k_nabla_express}_{\mathcal D}$, further implies that
\be \label{D_lp_bdd} 
\Vert \mathcal{D}\Vert_{L^p_{x,v}} \lesssim \Vert wh\Vert_{L^\infty_{x,v}}+ | wf_b|_{L^\infty_{\p\O,v}} +[o(1)+\Vert wh\Vert_{L^\infty_{x,v}} ](\Vert  w_{\tilde{\theta}}\nabla_x h\Vert_{L^p_{x,v}} +\Vert \alpha_h \nabla_x h\Vert_{L^\infty_{x,v}}). 
\ee
Finally, using the estimates for $\mathcal{A},\mathcal{B},\mathcal{C}$ in \eqref{A_bdd_lp}, \eqref{B_bdd_lp}, and \eqref{c_bdd_lp}, together with the assumption $\Vert wh\Vert_{L^\infty_{x,v}}+ | wf_b|_{L^\infty_{\p\O,v}}\ll 1$, we conclude the proof of the proposition.
\end{proof}

\section{
\texorpdfstring{A priori weighted $C^1_{x,v}$ estimate and $C^1_v$ estimate without weight}{Weighted C1xv estimate and C1v estimate without weight}
}
\label{sec:c1_estimate}

In this section, we establish weighted $C^1_{x,v}$ regularity and $C^1_v$ regularity without the $\alpha_h$-weight in Proposition \ref{prop:weight_C1} and Proposition \ref{prop:C1v}, respectively.

Similar to Section 4, the stationary setting $h,\phi_h,\alpha_h (x, v)$ defined in \eqref{eqn:h} and \eqref{alpha_weight_steady} are adopted while applying lemmas from Section \ref{sec:prelim}. 
Moreover, we impose the same a priori assumption \eqref{ap_assumption_steady} as in Proposition \ref{prop:weight_W1p}.
The smallness of $\nabla_x \phi_h (x)$ in \eqref{ap_assumption_steady} ensures the existence of $\O_\delta$ such that \eqref{sign_condition} is satisfied.

\smallskip

We start with the weighted $C^1_{x,v}$ estimate, using the $W^{1,p}$ estimate in Proposition \ref{prop:weight_W1p}, and bootstrap to weighted $C^1_{x,v}$ regularity.

\begin{proposition} \label{prop:weight_C1}

Suppose the assumption in Proposition \ref{prop:weight_W1p} holds. Then the following weighted $C^1_{x,v}$ estimate holds:
\be \notag
\Vert w_{\tilde{\theta}} \alpha_h \p_{x,v}h\Vert_{L^\infty_{x,v}} 
\lesssim \Vert w_{\tilde{\theta}} \p_{x,v} h\Vert_{L^p_{x,v}} +  \Vert wh\Vert_{L^\infty_{x,v}} + | wf_b|_{L^\infty_{\p\O,v}} + | w \p_{\mathbf{x}_p,v} f_b|_{L^\infty_{\p\O,v}}.
\ee
\end{proposition}

\begin{proof}

Since this proposition builds on Proposition \ref{prop:weight_W1p}, we adopt the same notation and framework as there.
Throughout the proof, we write
\be \notag
\begin{split}
(X(s), V(s)) & := (X(s;t,x,v), V(s;t,x,v)),
\\ (X(s',s), V(s',s)) & := (X(s';s,X(s),u), V(s';s,X(s),u)).
\end{split}
\ee
Following \eqref{eq1:weight_W1p}, we choose $t>0$ such that
\be \notag
\Vert \nabla_x (\phi_f + \phi_E) \Vert_{L^\infty_{x}} t +
\| \nabla_x^2 (\phi_h + \phi_E) \|_{L^\infty_x} (t)^2 e^{t} \ll 1
\ \text{ and } \
e^{-\frac{\nu_0}{4}t} \ll 1.
\ee
Moreover, we use the characteristic formula \eqref{rchara:nabla_0}–\eqref{rchara:nabla_phi_E_phi} for $w_{\tilde{\theta}}(u)\partial_{x, v} h(x, v)$, and adopt the same decomposition $\mathcal{A},\mathcal{B},\mathcal{C},\mathcal{D}$ as in \eqref{ABCD_notation} from the proof of Proposition \ref{prop:weight_W1p}. Specifically, we write
\be \notag
\alpha_h(x,v) w_{\tilde{\theta}} (v) \p_{x,v} h(x,v)
= \alpha_h(x,v) \mathcal{A} + \alpha_h(x,v) \mathcal{B} + \alpha_h(x,v) \mathcal{C} + \alpha_h(x,v) \mathcal{D}.
\ee


First, we estimate $\alpha_h(x,v) \mathcal{A}$. Since $\alpha_h \leq \delta' \ll 1$, it follows from \eqref{A_bdd_infty} that
\be \label{alpha_A_bdd_infty}
| \alpha_h(x,v) \mathcal{A}(x,v) | 
\lesssim o(1) ( \Vert \alpha_h \nabla_x h\Vert_{L^\infty_{x,v}} + \Vert \nabla_x h\Vert_{L^p_{x,v}} ) + \Vert wh \Vert_{L^\infty_{x,v}} + | wf_b|_{L^\infty_{\p\O,v}}.
\ee

Second, we estimate $\alpha_h(x,v) \mathcal{B}$. Recall from \eqref{ABCD_notation} that $\mathcal{B} = \eqref{rchara:nabla_0} + \eqref{rchara:nabla_gamma}$. 
Using Lemma \ref{lemma:velocity}, \eqref{nabla_0_compute}, and the fact that $e^{-\nu_0t/4} \ll 1$, we obtain
\begin{align*}
| \alpha_h(x,v) \eqref{rchara:nabla_0} | 
& \lesssim  e^{-\nu_0t/4}e^{C_0} \Vert \alpha_h(X(0),V(0))w_{\tilde{\theta}}(V(0))\p_{x,v} h(X(0),V(0))\Vert_{L^\infty_{x,v}} \\
& \lesssim o(1)\Vert \alpha_h(x,v) w_{\tilde{\theta}}(v)\p_{x,v}h(x,v)\Vert_{L^\infty_{x,v}}.
\end{align*}
We next use Lemma \ref{lemma:gamma} to estimate $\alpha_h(x,v) \eqref{rchara:nabla_gamma}$ by
\begin{align*}
& |\alpha_h(x,v) \eqref{rchara:nabla_gamma} | \\
& \lesssim \Vert wh\Vert_{L^\infty_{x,v}} \alpha_h(x,v)\int^t_{\max\{0,t-\tb\}} e^{-\frac{\nu(V(s))(t-s)}{4}} \int_{\mathbb{R}^3} \frac{w_{\tilde{\theta}}(V(s))}{w_{\tilde{\theta}}(u)} \mathbf{k}_{1}(V(s),u) w_{\tilde{\theta}}(u)|\p_{x,v}h(X(s),u)| \dd u \dd s \\
& \qquad + \alpha_h(x,v)\int^t_{\max\{0,t-\tb\}} e^{-\frac{\nu(V(s))(t-s)}{4}}  w_{\tilde{\theta}}(V(s)) \Gamma_v(h,h) \dd s \\
& \lesssim \Vert wh\Vert_{L^\infty_{x,v}}\Vert w_{\tilde{\theta}}\alpha_h \p_{x,v}h\Vert_{L^\infty_{x,v}} \alpha_h(x,v) \int^t_{\max\{0,t-\tb\}} e^{-\frac{\nu(V(s))(t-s)}{4}} \int_{\mathbb{R}^3} \frac{w_{\tilde{\theta}}(V(s))}{w_{\tilde{\theta}}(u)} \frac{\mathbf{k}_{1}(V(s),u)}{\alpha^2_h (X(s),u)} \dd u \dd s + \Vert wh\Vert_{L^\infty_{x,v}}^2 \\
& \lesssim \Vert wh\Vert_{L^\infty_{x,v}}\Vert w_{\tilde{\theta}} \alpha_h \p_{x,v}h\Vert_{L^\infty_{x,v}} + \Vert wh\Vert_{L^\infty_{x,v}}^2 \lesssim o(1)\Vert w_{\tilde{\theta}} \alpha_h \p_{x,v}h \Vert_{L^\infty_{x,v}} + \Vert wh\Vert_{L^\infty_{x,v}},
\end{align*}
where in the last line we apply Lemma \ref{lemma:nonlocal_to_local} and use $\Vert wh\Vert_{L^\infty_{x,v}}\ll 1$.
Combining the above two estimates, we conclude that
\be \label{alpha_B_bdd_infty}
| \alpha_h(x,v) \mathcal{B}| 
\lesssim o(1)\Vert w_{\tilde{\theta}} \alpha_h \p_{x,v}h\Vert_{L^\infty_{x,v}} + \Vert wh\Vert_{L^\infty_{x,v}}.
\ee

Third, we estimate $\alpha_h(x,v) \mathcal{C}$. It follows from \eqref{c_bdd_infty} and Lemma \ref{lemma:weight_singularity} that
\be \label{alpha_C_bdd_infty}
|\alpha_h(x,v)\mathcal{C}(x,v)| 
\lesssim \Vert wh\Vert_{L^\infty_{x,v}} + | wf_b|_{L^\infty_{\p\O,v}} + |w \p_{\mathbf{x}_{p},v} f_b|_{L^\infty_{\p\O,v}}.
\ee

Finally, we estimate the most challenging term $\alpha_h(x,v) \mathcal{D} = \alpha_h(x,v) \eqref{rchara:K_nabla}$. 
From \eqref{k_nabla_express}, we obtain
\begin{align}
| \eqref{rchara:K_nabla} |
& \lesssim \int^t_{t-\tb} e^{-\int^t_s \frac{\tilde{\nu}(X(\tau),V(\tau))\dd \tau}{4}} \int_{\mathbb{R}^3} \mathbf{k}(V(s),u) \frac{w_{\tilde{\theta}}(V(s))}{w_{\tilde{\theta}}(u)} w_{\tilde{\theta}}(u) |\p_{x} h(X(s),u)| \dd u \dd s 
\notag \\
& \lesssim \int^t_{t-\tb} e^{-\int^t_s \frac{\tilde{\nu}(X(\tau),V(\tau))\dd \tau}{4}}\int_{\mathbb{R}^3} \mathbf{k}(V(s),u) \frac{w_{\tilde{\theta}}(V(s))}{w_{\tilde{\theta}}(u)}  |w_{\tilde{\theta}}(u)\p_{x} h(X(s),u)| \dd u \dd s. \label{k_nabla_express2}
\end{align}
Analogous to the proof of Proposition \ref{prop:weight_W1p}, we expand $w_{\tilde{\theta}}(u)\partial_x h(X(s),u)$ in \eqref{k_nabla_express2} using the characteristic formula \eqref{rchara:nabla_0}-\eqref{rchara:nabla_phi_E_phi}, with $(t,x,v)$ replaced by $(s,X(s),u)$, and decompose the resulting expression into four parts: $\mathcal{A} (X(s),u), \mathcal{B} (X(s),u), \mathcal{C} (X(s),u), \mathcal{D} (X(s),u)$. 
We further define
\begin{align} \notag
\eqref{k_nabla_express2}_{\mathcal X}
:= \int^t_{t-\tb} e^{-\int^t_s \frac{\tilde{\nu}(X(\tau),V(\tau))\dd \tau}{4}}\int_{\mathbb{R}^3} \mathbf{k}(V(s),u) \frac{w_{\tilde{\theta}}(V(s))}{w_{\tilde{\theta}}(u)}  |\mathcal X (X(s),u) | \dd u \dd s
\ \text{ for }
\mathcal X \in \{ \mathcal{A}, \mathcal{B}, \mathcal{C}, \mathcal{D} \}. 
\end{align}
Accordingly, we bound \eqref{k_nabla_express2} by
\[
\eqref{k_nabla_express2}
\leq |\eqref{k_nabla_express2}_{\mathcal A}| + |\eqref{k_nabla_express2}_{\mathcal B}| + |\eqref{k_nabla_express2}_{\mathcal C}| + |\eqref{k_nabla_express2}_{\mathcal D}|.
\]

For $\alpha_h(x,v) \eqref{k_nabla_express2}_{\mathcal A}$, we apply \eqref{k_A_bdd} and the fact that $\alpha_h^2 < \alpha_h \leq \delta' \ll 1$ to obtain
\begin{align*}
|\alpha_h(x,v) \eqref{k_nabla_express2}_{\mathcal{A}}| 
    \lesssim  o(1) ( \Vert \alpha_h \nabla_x h\Vert_{L^\infty_{x,v}} + \Vert \nabla_x h\Vert_{L^p_{x,v}} ) + \Vert wh \Vert_{L^\infty_{x,v}} + | wf_b|_{L^\infty_{\p\O,v}}.
\end{align*}

For $\alpha_h(x,v) \eqref{k_nabla_express2}_{\mathcal B}$, we use \eqref{alpha_C_bdd_infty} and Lemma \ref{lemma:nonlocal_to_local} to deduce
\begin{align*}
& |\alpha_h(x,v) \eqref{k_nabla_express2}_{\mathcal{B}}| \\
& \lesssim \Vert \alpha_h \mathcal{B}\Vert_{L^\infty_{x,v}} \alpha_h(x,v)\int^t_{t-\tb} \dd s e^{-\int_s^t \frac{\tilde{\nu}(X(\tau),V(\tau))}{4}\dd \tau} \int_{\mathbb{R}^3} \dd u \mathbf{k}(V(s),u) \frac{w_{\tilde{\theta}}(V(s))}{w_{\tilde{\theta}}(u)} \frac{1}{\alpha_h(X(s),u)} \\
& \lesssim o(1)\Vert w_{\tilde{\theta}} \alpha_h \p_{x,v} h\Vert_{L^\infty_{x,v}}\alpha_h(x,v)\int^t_{t-\tb} \dd s e^{-\int_s^t \frac{\tilde{\nu}(X(\tau),V(\tau))}{4}\dd \tau} \int_{\mathbb{R}^3} \dd u \mathbf{k}(V(s),u) \frac{w_{\tilde{\theta}}(V(s))}{w_{\tilde{\theta}}(u)} \frac{1}{\alpha_h^2(X(s),u)} \\
& \lesssim o(1)\Vert w_{\tilde{\theta}} \alpha_h \p_{x,v}h\Vert_{L^\infty_{x,v}}.
\end{align*}

For $\alpha_h(x,v) \eqref{k_nabla_express2}_{\mathcal C}$, we apply Lemma \ref{lemma:velocity}, \eqref{c_bdd_infty}, and Lemma \ref{lemma:nonlocal_to_local} to have
\begin{align*}
| \alpha_h(x,v) \eqref{k_nabla_express2}_{\mathcal{C}}|  
& \lesssim [\Vert wh\Vert_{L^\infty_{x,v}} + | wf_b|_{L^\infty_{\p\O,v}} + |w \p_{\mathbf{x}_{p},v}f_b|_{L^\infty_{\p\O,v}}] \\
& \qquad \times \alpha_h(x,v)\int^t_{t-\tb} e^{-\int^t_s \frac{\tilde{\nu}(X(\tau),V(\tau))\dd \tau}{4}}\int_{\mathbb{R}^3} \mathbf{k}(V(s),u) \frac{w_{\tilde{\theta}}(V(s))}{w_{\tilde{\theta}}(u)} \frac{1}{\alpha^2_h (X(s),u)} \dd u \dd s \\
& \lesssim \Vert wh\Vert_{L^\infty_{x,v}} + | wf_b|_{L^\infty_{\p\O,v}}+ |w\p_{\mathbf{x}_{p},v}f_b|_{L^\infty_{\p\O,v}}.
\end{align*}

For $\alpha_h(x,v) \eqref{k_nabla_express2}_{\mathcal D}$, which is the most delicate part, we compute
\be \label{kk2} 
\begin{split}
& | \alpha_h(x,v) \eqref{k_nabla_express2}_{\mathcal{D}}| 
\\& \lesssim \int^t_{t-\tb} \dd s e^{-\int^t_s \frac{\tilde{\nu}(X(\tau),V(\tau))}{4}\dd \tau}\int_{\mathbb{R}^3} \dd u \mathbf{k}(V(s),u)\frac{w_{\tilde{\theta}}(V(s))}{w_{\tilde{\theta}}(u)}    \int^s_{s-\tb(X(s),u)} \dd s'  e^{-\int^s_{s'} \frac{\tilde{\nu}(X(\tau',s),V(\tau',s))}{4}\dd \tau'}
\\& \qquad \times \int_{\mathbb{R}^3} \dd u'\mathbf{k}(V(s',s),u')\frac{w_{\tilde{\theta}}(V(s',s))}{w_{\tilde{\theta}}(u')} | w_{\tilde{\theta}}(u') \nabla_x h(X(s',s),u')| \alpha_h(x,v). 
\end{split}
\ee
To compute \eqref{kk2}, we pick $\delta$ and $N$ such that
\[
0 < \delta \ll 1 \ll N,
\]
and we consider four cases: (1) $s-s'<\delta$, \ (2) $|u|>N$ or $|u'|>N$, \ (3) $|V(s)-u|<\frac{1}{N}$ or $|V(s',s)-u|<\frac{1}{N}$, \ (4) $s' \leq s-\delta$, $|u| \leq N$, $|u'| \leq N$, and $|V(s)-u| \geq \frac{1}{N}$, $|V(s',s)-u| \geq \frac{1}{N}$.

\smallskip

\textit{Case 1: $s-s'<\delta$.}
Applying Lemma \ref{lemma:nonlocal_to_local} and the fact that $s-s'<\delta$, we derive
\begin{align*}
& |\alpha_h(x,v)\eqref{rchara:K_nabla}_{\mathcal{D}} \mathbf{1}_{s-s'<\delta}| \\
& \lesssim \Vert w_{\tilde{\theta}} \alpha_h \p_{x,v}h\Vert_{L^\infty_{x,v}}  \int^t_{t-\tb} \dd s e^{-\int^t_s \frac{\tilde{\nu}(X(\tau),V(\tau))}{4}\dd \tau}\int_{\mathbb{R}^3} \dd u \mathbf{k}(V(s),u)\frac{w_{\tilde{\theta}}(V(s))}{w_{\tilde{\theta}}(u)} \alpha_h(x,v) \big( \frac{o(1)}{\alpha_h(X(s),u)} + o(1) \big) \\
& \lesssim \Vert w_{\tilde{\theta}} \alpha_h \p_{x,v}h\Vert_{L^\infty_{x,v}} o(1) \int^t_{t-\tb} \dd s e^{-\int^t_s \frac{\tilde{\nu}(X(\tau),V(\tau))}{4}\dd \tau}\int_{\mathbb{R}^3} \dd u \mathbf{k}(V(s),u)\frac{w_{\tilde{\theta}}(V(s))}{w_{\tilde{\theta}}(u)}  \frac{\alpha_h(x,v)}{\alpha^2_h (X(s),u)} \\
& \lesssim o(1)\Vert w_{\tilde{\theta}}\alpha_h \p_{x,v}h\Vert_{L^\infty_{x,v}}.
\end{align*}

\textit{Case 2: $|u|>N$ or $|u'|>N$.}
We first consider $|u'| > N$. 
If $|u|< \frac{N}{2}$, the assumption in Proposition \ref{prop:weight_W1p} implies that $V(s',s)< \frac{3 N}{4}$, and thus $|V(s',s) - u'| > \frac{N}{4}$. This leads to 
\be \notag
\mathbf{k}(V(s',s),u')\frac{w_{\tilde{\theta}}(V(s',s))}{w_{\tilde{\theta}}(u')} 
\lesssim o(1)\frac{e^{-C|V(s',s) - u'|^2}}{|V(s',s) - u'|}.
\ee
Otherwise, if $|u| \geq \frac{N}{2}$, we have $\tb(X(s),u)\ll 1$. By Lemma \ref{lemma:nonlocal_to_local}, we conclude that
\begin{align*}
    & |\alpha_h(x,v) \eqref{k_nabla_express2}_{\mathcal{D}} \mathbf{1}_{|u'|>N}| \lesssim o(1)\Vert w_{\tilde{\theta}} \alpha_h \p_{x,v}h\Vert_{L^\infty_{x,v}}.
\end{align*}
Similarly, when $|u|>N$, we obtain an analogous estimate as follows:
\begin{align*}
    & |\alpha_h(x,v)\eqref{k_nabla_express2}_{\mathcal{D}} \mathbf{1}_{|u|>N}| \lesssim o(1)\Vert w_{\tilde{\theta}} \alpha_h \p_{x,v}h\Vert_{L^\infty_{x,v}}.
\end{align*}

\textit{Case 3: $|V(s)-u|<\frac{1}{N}$ or $|V(s',s)-u|<\frac{1}{N}$.}
Using Lemma \ref{lemma:nonlocal_to_local}, we have
\be \notag
|\alpha_h(x,v) \eqref{k_nabla_express2}_{\mathcal{D}} \mathbf{1}_{|V(s',s)-u'|<N^{-1}\text{ or }|V(s)-u|<N^{-1}}| \lesssim o(1)\Vert w_{\tilde{\theta}} \alpha_h \p_{x,v}h\Vert_{L^\infty_{x,v}}.
\ee

\textit{Case 4: $s' \leq s-\delta$, $|u| \leq N$, $|u'| \leq N$, and $|V(s)-u| \geq \frac{1}{N}$, $|V(s',s)-u| \geq \frac{1}{N}$.}
This case satisfies the same conditions as in the estimate for \eqref{bootstrap}. However, instead of applying H\"older's inequality with $\frac12+\frac12=1$, we use the conjugate exponents $\frac1p+\frac1{p'}=1$ to obtain
\be \notag
|\alpha_h(x,v)\eqref{rchara:K_nabla}_{\mathcal D} \mathbf{1}_{s' \leq s-\delta, |u|<N, |u'|<N, \text{ and }|V(s',s)-u'|>N^{-1}, |V(s)-u|>N^{-1}}|
\lesssim \|\partial_{x,v} h\|_{L^p_{x,v}}.
\ee
We remark that, unlike in Proposition \ref{prop:weight_W1p}, we do not perform a change of variables to the partial derivative $\nabla_x$ or integration by parts, but instead directly apply H\"older's inequality with the argument to \eqref{bootstrap}.

Collecting the four cases, we obtain the estimate for $\alpha_h(x,v) \eqref{k_nabla_express2}_{\mathcal{D}}$ as follows:
\be \notag
|\alpha_h(x,v) \eqref{k_nabla_express2}_{\mathcal{D}}| 
\lesssim o(1)\Vert w_{\tilde{\theta}} \alpha_h \p_{x,v}h\Vert_{L^\infty_{x,v}} +\Vert \p_{x,v} h\Vert_{L^p_{x,v}}.
\ee
This, together with the estimates for $\alpha_h(x,v) \eqref{k_nabla_express2}_{\mathcal{A}}, \alpha_h(x,v) \eqref{k_nabla_express2}_{\mathcal{B}}$, and $\alpha_h(x,v) \eqref{k_nabla_express2}_{\mathcal{C}}$, implies that
\begin{align*}
    & |\alpha_h(x,v)\mathcal{D}(x,v)| \lesssim o(1)\Vert w_{\tilde{\theta}}\alpha_h \p_{x,v}h\Vert_{L^\infty_{x,v}}+ \Vert w_{\tilde{\theta}} \p_{x,v} h\Vert_{L^p_{x,v}} +  \Vert wh\Vert_{L^\infty_{x,v}} + | wf_b|_{L^\infty_{\p\O,v}} + | w \p_{\mathbf{x}_p,v} f_b|_{L^\infty_{\p\O,v}} .
\end{align*}
Combining this with the estimates \eqref{alpha_A_bdd_infty}, \eqref{alpha_B_bdd_infty}, and \eqref{alpha_C_bdd_infty}, and absorbing the $o(1)\Vert w_{\tilde{\theta}} \alpha_h \p_{x,v}h\Vert_{L^\infty_{x,v}}$ term into the left-hand side, we conclude the proof of the proposition.
\end{proof}

Next, we establish the $v$-derivative estimate without the $\alpha_h$-weight.

\begin{proposition} \label{prop:C1v}

Suppose the assumption in Proposition \ref{prop:weight_C1} holds. Then the following $C^1_{v}$ estimate holds:
    \begin{align*}
        &  \Vert w_{\tilde{\theta}} \p_v h \Vert_{L^\infty_{x,v}} \lesssim \Vert w_{\tilde{\theta}}\alpha_h \p_{x,v}h\Vert_{L^\infty_{x,v}} + \Vert w_{\tilde{\theta}} \p_{x,v} h\Vert_{L^p_{x,v}} +  \Vert wh\Vert_{L^\infty_{x,v}} + | wf_b|_{L^\infty_{\p\O,v}} + | w \p_{\mathbf{x}_p,v} f_b|_{L^\infty_{\p\O,v}} .
    \end{align*}
\end{proposition}

\begin{proof}

Since this proposition builds on Proposition \ref{prop:weight_W1p} and Proposition \ref{prop:weight_C1}, we adopt the same notation and framework.
Throughout the proof, we write
\be \notag
\begin{split}
(X(s), V(s)) & := (X(s;t,x,v), V(s;t,x,v)).
\end{split}
\ee
Analogous to Proposition \ref{prop:weight_C1}, we choose $t>0$ such that
\be \notag
\Vert \nabla_x (\phi_f + \phi_E) \Vert_{L^\infty_{x}} t +
\| \nabla_x^2 (\phi_h + \phi_E) \|_{L^\infty_x} (t)^2 e^{t} \ll 1
\ \text{ and } \
e^{-\frac{\nu_0}{4}t} \ll 1.
\ee

If $x\in \O$ and $X(s) \in \O \setminus \O_\delta$ for some $t-\tb \leq s \leq t$, then Lemma \ref{lemma:velocity} implies that
\[
\alpha_h(x,v) \gtrsim \alpha_h(X(s),V(s)) = \delta'.
\]
Hence, we concludes the $C^1_{v}$ estimate by
\[
| w_{\tilde{\theta}} \p_v h (x, v) |
\leq \frac{1}{\delta'} \Vert w_{\tilde{\theta}}\alpha_h \p_{x,v}h\Vert_{L^\infty_{x,v}} \lesssim \Vert w_{\tilde{\theta}}\alpha_h \p_{x,v}h\Vert_{L^\infty_{x,v}}.
\] 

Therefore, it suffices to consider the case that $X(s) \in \Omega_\delta$ for every $s \in [t - \tb, t]$.
From \eqref{t2_t1_bdd}, we have $|\tb| \lesssim \sqrt{\frac{\delta}{C_E}}$. Hence, let $\delta > 0$ be sufficiently small (depending on $t$), so that $\tb \leq t$ and thus
\[
\max\{t-\tb,0\} = t-\tb.
\]
Analogous to Proposition \ref{prop:weight_W1p}, we use the characteristic formula \eqref{rchara:nabla_0}–\eqref{rchara:nabla_phi_E_phi} for $w_{\tilde{\theta}}(u) \partial_{v} h(x, v)$.
Since $\tb \leq t$, the terms \eqref{rchara:nabla_0} and \eqref{rchara:nabla_nu_1} are excluded in the resulting expression. We then modify the decomposition \eqref{ABCD_notation} from the proof of Proposition \ref{prop:weight_W1p} accordingly. Specifically, we write
\be \label{ABCD_notation_C1v}
\begin{split}
\tilde{\mathcal{A}} & = \eqref{rchara:nabla_nu_2} + \eqref{rchara:nabla_K}+ \eqref{rchara:K_nabla_phi_E} + \eqref{rchara:nabla_nu_K}  + \eqref{rchara:nabla_nu_gamma} + \eqref{rchara:gamma_nabla_phi_E} + \eqref{rchara:nabla_E}  + \eqref{rchara:nabla_nu_phi}  + \eqref{rchara:nabla_V_phi} + \eqref{rchara:nabla_mu_phi} + \eqref{rchara:nabla_phi_E_phi},  
\\ \tilde{\mathcal{B}} & = \eqref{rchara:nabla_gamma}, 
\\ \tilde{\mathcal{C}} & = \eqref{rchara:nabla_tb} + \eqref{rchara:bdr} + \eqref{rchara:nabla_int} + \eqref{rchara:nabla_tb_gamma} + \eqref{rchara:nabla_tb_phi},
\\ \tilde{\mathcal{D}} & = \eqref{rchara:K_nabla}.
\end{split}
\ee
Moreover, since we consider the estimate for $w_{\tilde{\theta}}(u) \partial_{v} h(x, v)$, we focus on the $\p_{v}$-derivative rather than the $\p_{x, v}$-derivative, and the lower limit of the time integrals changes from $\max\{t-\tb,0\}$ to $t-\tb$.

For $\tilde{\mathcal{A}}$, similarly to \eqref{A_bdd_infty}, we obtain
\be \label{tilde_A_bdd_infty}
|\tilde{\mathcal{A}}| \lesssim o(1) \big( \Vert \alpha_h \nabla_x h\Vert_{L^\infty_{x,v}} + \Vert \nabla_x h\Vert_{L^p_{x,v}} \big) + \Vert wh\Vert_{L^\infty_{x,v}} + | wf_b|_{L^\infty_{\p\O,v}} .
\ee

For $\tilde{\mathcal{B}}$, we we apply Lemma \ref{lemma:gamma} and use $\Vert wh\Vert_{L^\infty_{x,v}}\ll 1$ to obtain
\begin{align}
| \tilde{\mathcal{B}}|  
& \lesssim \Vert wh\Vert_{L^\infty_{x,v}} \int^t_{\max\{0,t-\tb\}} e^{-\frac{\nu(V(s))(t-s)}{4}} \int_{\mathbb{R}^3} \frac{w_{\tilde{\theta}}(V(s))}{w_{\tilde{\theta}}(u)} \mathbf{k}_{1}(V(s),u) w_{\tilde{\theta}}(u)|\p_{v}h(X(s),u)| \dd u \dd s 
\notag \\ 
& \qquad + \int^t_{\max\{0,t-\tb\}} e^{-\frac{\nu(V(s))(t-s)}{4}} w_{\tilde{\theta}}(V(s))|\Gamma_v(h,h)(X(s),V(s))| \dd s 
\notag \\
& \lesssim \Vert wh\Vert_{L^\infty_{x,v}}\Vert w_{\tilde{\theta}} \p_{v}h\Vert_{L^\infty_{x,v}}  \int^t_{\max\{0,t-\tb\}} e^{-\frac{\nu(V(s))(t-s)}{4}} \int_{\mathbb{R}^3} \frac{w_{\tilde{\theta}} (V(s))}{w_{\tilde{\theta}}(u)} \mathbf{k}_1 (V(s),u) \dd u \dd s + \Vert wh\Vert_{L^\infty_{x,v}}^2 
\notag \\
& \lesssim \Vert wh\Vert_{L^\infty_{x,v}}\Vert w_{\tilde{\theta}}  \p_{v}h\Vert_{L^\infty_{x,v}} +\Vert wh\Vert_{L^\infty_{x,v}}^2 \lesssim o(1)\Vert w_{\tilde{\theta}}\p_{v}h\Vert_{L^\infty_{x,v}} + \Vert wh\Vert_{L^\infty_{x,v}}.
\label{tilde_B_bdd_infty}
\end{align}

For $\tilde{\mathcal{C}}$, since we focus on the $\p_{v}$-derivative, we use Lemma \ref{lemma:est_tf} to deduce that $\frac{\tb(x,v)}{ \vb \cdot n(\xb)} \lesssim 1$.The remaining computation is similar to \eqref{c_bdd_infty}, and we obtain
\be \label{tilde_C_bdd_infty}
|\tilde{\mathcal{C}}| \lesssim \Vert wh\Vert_{L^\infty_{x,v}} + | wf_b|_{L^\infty_{\p\O,v}} + | w\p_{\mathbf{x}_p,v}f_b|_{L^\infty_{\p\O,v}}.
\ee

For $\tilde{\mathcal{D}}$, we apply Lemma \ref{lemma:nonlocal_to_local} to obtain
\begin{align}
|\tilde{\mathcal{D}}| 
& \lesssim \Vert w_{\tilde{\theta}}\alpha_h \nabla_x h\Vert_{L^\infty_{x,v}}  \int^t_{t-\tb} e^{-\int_s^t \frac{\tilde{\nu}(X(\tau),V(\tau))}{2} \dd \tau} \int_{\mathbb{R}^3}  \mathbf{k}(V(s),u) \frac{w_{\tilde{\theta}}(V(s))}{w_{\tilde{\theta}}(u)\alpha_h(X(s),u)} |\p_v X(s)| \dd u \dd s 
\notag \\
& \lesssim \Vert w_{\tilde{\theta}} \alpha_h \nabla_x h\Vert_{L^\infty_{x,v}} \frac{\alpha_h(x,v)}{\alpha_h(x,v)} \lesssim \Vert w_{\tilde{\theta}} \alpha_h \nabla_x h\Vert_{L^\infty_{x,v}}.
\label{tilde_D_bdd_infty}
\end{align}
where in the second line we use $|\p_v X(s)| \lesssim \alpha_h (x,v)$ for any $t-\tb \leq s\leq t$ from Corollary \ref{cor:est_x_v}.
Collecting the estimates \eqref{tilde_A_bdd_infty}-\eqref{tilde_D_bdd_infty}, we conclude the proposition.
\end{proof}

\section{Uniqueness of steady solution}
\label{sec:stationary_uniqueness}

In this section, we establish the uniqueness of the stationary problem \eqref{eqn:h} in Proposition \ref{prop:stationary_uniqueness}. 
The proof relies on the a priori $L^2-L^\infty$ estimate and $W^{1,p}$ estimate established in Sections~\ref{sec:L2Linfty_estimate} and \ref{sec:w1p_estimate}, respectively. 

More importantly, Proposition \ref{prop:stationary_uniqueness} requires the smallness of $\Vert w_{\tilde{\theta}} \nabla_v h \Vert_{L^\infty_{x,v}}$ to obtain uniqueness, which motivates the derivation of the unweighted $C^1_v$ estimate in Proposition \ref{prop:C1v}.
Furthermore, the uniqueness argument developed here will be adapted in the well-posedness and existence of the stationary problem in Section~\ref{sec:existence}. 


\begin{proposition} \label{prop:stationary_uniqueness}

Let $h_1$ and $h_2$ be two solutions to \eqref{eqn:h} such that, for $i = 1, 2$,
\begin{align*}
\Vert wh_i\Vert_{L^\infty_{x,v}} 
& \lesssim | wf_b|_{L^\infty_{\p\O,v}} \ll 1, \\
\Vert w_{\tilde{\theta}} \nabla_v h_i\Vert_{L^\infty_{x,v}} 
& \lesssim | wf_b|_{L^\infty_{\p\O,v}} + | w \p_{\mathbf{x}_p,v} f_b|_{L^\infty_{\p\O,v}} \ll 1,
\end{align*}
then $h_1 = h_2$.
\end{proposition}

\begin{proof}

Since $h_1$ and $h_2$ are two solutions to \eqref{eqn:h}, the difference $h_1-h_2$ satisfies
\begin{align*}
\begin{cases}
& v \cdot \nabla_x (h_1-h_2) - (\nabla_x\phi_{h_1}+ \nabla_x \phi_E)\cdot \nabla_{v}(h_1-h_2) + \frac{v}{2}\cdot \nabla_x \phi_{h_1}(h_1-h_2) + e^{-\phi_E}\mathcal{L}(h_1-h_2) \\
& = \nabla_x \phi_{h_1-h_2}\cdot \nabla_{v} h_2 - \frac{v}{2}\cdot \nabla_x \phi_{h_1-h_2}h_2 + e^{-\phi_E/2}[\Gamma(h_1,h_1)-\Gamma(h_2,h_2)]-v\cdot \nabla_x \phi_{h_1-h_2} e^{-\phi_E/2}\sqrt{\mu}, \\
& (h_1-h_2)|_{\gamma_-} = 0, \\
& - \Delta_x \phi_{h_1-h_2} = e^{-\phi_E/2}\int_{\mathbb{R}^3} (h_1-h_2)\sqrt{\mu} \dd v \text{ in }\O, \ \phi_{h_1-h_2} = 0 \text{ on } \p\O, \\
&  -\p_n\phi_E > C_E >0 \text{ on } \p\O.
\end{cases}
\end{align*}
By multiplying the first equation by $h_1-h_2$ , we obtain the following $L^2$ energy estimate:
\begin{align}
& |(h_1-h_2)|_{L^2_{\gamma_+}}^2 + \Vert e^{-\phi_E/2}(\mathbf{I}-\mathbf{P})(h_1-h_2)\Vert_{L^2_{x,\nu}}^2
\lesssim \Vert \nabla_x \phi_{h_1}\Vert_{L^\infty_{x}} \Vert h_1-h_2\Vert_{L^2_{x,\nu}}^2 
\label{f_s-h_s_micro_l2_1} \\
& \qquad + \Vert \nabla_x \phi_{h_1-h_2}\Vert_{L^2_x}\Vert \nabla_{v}h_2\Vert_{L^{\infty}_{x}L^2_v} \Vert h_1-h_2\Vert_{L^2_{x,v}} + \Vert wh_2\Vert_{L^\infty_{x,v}} \Vert h_1-h_2\Vert_{L^2_{x,v}}\Vert \nabla_x \phi_{h_1-h_2}\Vert_{L^2_x} \label{f_s-h_s_micro_l2_2} \\
& \qquad+ \Vert e^{\phi_E/4}\nu^{-1/2}[\Gamma(h_1-h_2,h_1) + \Gamma(h_2,h_1-h_2)]\Vert_{L^2_{x,v}}^2 + o(1)\Vert e^{-\phi_E/2}(\mathbf{I}-\mathbf{P})(h_1-h_2) \Vert_{L^2_{x,\nu}}^2
\label{f_s-h_s_micro_l2_3} \\
& \qquad + \Big| \iint_{\O\times \mathbb{R}^3} v \cdot \nabla_x \phi_{h_1-h_2}e^{-\phi_E/2}\sqrt{\mu}(h_1-h_2)\dd x \dd v \Big|. \label{f_s-h_s_micro_l2_4}
\end{align}

\label{f_s-h_s_micro_l2}

For \eqref{f_s-h_s_micro_l2_1}, using Lemma \ref{lemma:phi_x_infinity} we derive that
\be \label{f_s-h_s_micro_l2_1_bdd}
\Vert \nabla_x \phi_{h_1}\Vert_{L^\infty_{x}} \Vert h_1-h_2\Vert_{L^2_{x,\nu}}^2
\lesssim \Vert w h_1 \Vert_{L^\infty_{x,v}} \Vert h_1-h_2\Vert_{L^2_{x,\nu}}^2.
\ee

For \eqref{f_s-h_s_micro_l2_2}, we again use Lemma \ref{lemma:phi_x_infinity} and the elliptic estimate to obtain that
\be \label{f_s-h_s_micro_l2_2_bdd}
\eqref{f_s-h_s_micro_l2_2}
\lesssim \Vert w_{\tilde{\theta}}\nabla_{v}h_2\Vert_{L^\infty_{x,v}}\Vert h_1-h_2\Vert_{L^2_{x,v}}^2 + \Vert wh_2\Vert_{L^\infty_{x,v}}\Vert h_1-h_2\Vert_{L^2_{x,v}}^2.
\ee

For the nonlinear term in \eqref{f_s-h_s_micro_l2_3}, we use Lemma \ref{lemma:gamma} to obtain that
\be \label{f_s-h_s_micro_l2_3_bdd}
\Vert \nu^{-1/2}[\Gamma(h_1-h_2,h_1)+\Gamma(h_2,h_1-h_2)]\Vert_{L^2_{x,v}}^2 
\lesssim \big( \Vert wh_2\Vert_{L^\infty_{x,v}}+\Vert wh_1\Vert_{L^\infty_{x,v}} \big)^2 \Vert h_1-h_2\Vert_{L^2_{x,\nu}}^2.
\ee

For \eqref{f_s-h_s_micro_l2_4}, similar to Lemma \ref{lemma:l2_energy}, we apply integration by parts and obtain
\begin{align*}
& \iint_{\O\times \mathbb{R}^3} v \cdot \nabla_x \phi_{h_1-h_2}e^{-\phi_E/2} \sqrt{\mu}(h_1-h_2)\dd x \dd v \\
& = - \iint_{\O\times \mathbb{R}^3}  \phi_{h_1-h_2}e^{-\phi_E/2} \sqrt{\mu} v \cdot \nabla_x (h_1-h_2) \dd x \dd v + \iint_{\O\times \mathbb{R}^3} \phi_{h_1-h_2}\frac{v \cdot \nabla_x \phi_E}{2} e^{-\phi_E/2}\sqrt{\mu}(h_1-h_2)\dd x \dd v \\
& \qquad + \int_{\gamma} \phi_{h_1-h_2} e^{-\phi_E/2} \sqrt{\mu}(h_1-h_2)(n(x)\cdot v) \dd S_x \dd v.
\end{align*}
The first term on the right-hand side can be computed using the equation for $h_1-h_2$ as
\begin{align*}
& \iint_{\O\times \mathbb{R}^3}  \phi_{h_1-h_2}e^{-\phi_E/2} \sqrt{\mu} v \cdot \nabla_x (h_1-h_2) \dd x \dd v \\
& = \iint_{\O\times \mathbb{R}^3} \phi_{h_1-h_2} e^{-\phi_E/2} \nabla_x \phi_{h_1}  \cdot \nabla_{v}((h_1-h_2)\sqrt{\mu})\dd x \dd v + \iint_{\O\times \mathbb{R}^3} \phi_{h_1-h_2}e^{-\phi_E/2} \sqrt{\mu} \nabla_x \phi_E \cdot \nabla_{v}(h_1-h_2) \dd x \dd v \\
& \qquad + \iint_{\O\times \mathbb{R}^3} \phi_{h_1-h_2}e^{-\phi_E/2} \nabla_x \phi_{h_1-h_2} \cdot \nabla_{v}(h_2\sqrt{\mu}) \dd x \dd v - \iint_{\O\times \mathbb{R}^3} \phi_{h_1-h_2}\sqrt{\mu}(v\cdot \nabla_x \phi_{h_1-h_2})\sqrt{\mu} \dd x \dd v  \\
& = \iint_{\O\times \mathbb{R}^3} \phi_{h_1-h_2}e^{-\phi_E/2} \sqrt{\mu} \nabla_x \phi_E \cdot \nabla_{v}(h_1-h_2) \dd x \dd v 
= - \iint_{\O\times \mathbb{R}^3} \phi_{h_1-h_2}\frac{v \cdot \nabla_x \phi_E}{2} e^{-\phi_E/2}\sqrt{\mu}(h_1-h_2)\dd x \dd v.
\end{align*}
Thus, we have
\be \label{f_s-h_s_micro_l2_4_bdd}
\eqref{f_s-h_s_micro_l2_4} 
\lesssim \Vert \nabla_x \phi_E\Vert_{L^\infty_x} \Vert \phi_{h_1-h_2}\Vert_{L^2_x} \Vert h_1-h_2\Vert_{L^2_{x,v}} \lesssim \Vert \nabla_x \phi_E\Vert_{L^\infty_x}\Vert h_1-h_2\Vert_{L^2_{x,v}}^2.
\ee

Collecting the estimates \eqref{f_s-h_s_micro_l2_1_bdd}–\eqref{f_s-h_s_micro_l2_4_bdd}, and together with $|e^{-\phi_E/2}-1|\ll 1$, we conclude that
\begin{align}
& |(h_1-h_2)|_{L^2_{\gamma_+}}^2 + \Vert (\mathbf{I}-\mathbf{P})(h_1-h_2)\Vert_{L^2_{x,\nu}}^2 \notag \\
& \lesssim  [\Vert wh_1\Vert_{L^\infty_{x,v}} + \Vert wh_2\Vert_{L^\infty_{x,v}} + \Vert \nabla_x \phi_E\Vert_{L^\infty_x} + \Vert w_{\tilde{\theta}} \nabla_{v}h_2\Vert_{L^\infty_{x,v}}]\Vert h_1-h_2\Vert_{L^2_{x,\nu}}^2. \label{micro_l2_stability}
\end{align}

For the macroscopic estimate, let $(a_{h_1}, \mathbf{b}_{h_1},  c_{h_1})$ and $(a_{h_2}, \mathbf{b}_{h_2}, c_{h_2})$ denote the macroscopic components of $\mathbf{P}h_1$ and $\mathbf{P}h_2$, respectively.
We use the same test functions $\psi_a, \psi_1,\psi_2,\psi_3, \psi_c$ as in Lemma \ref{lemma:macro_l2}, but with  $\phi_a, \phi_1, \phi_2, \phi_3, \phi_c$ solving
\begin{align*}
\Delta \phi_a & = a_{h_1}-a_{h_2}, \qquad \phi_a|_{\p\O} = 0, \\
\Delta \phi_i & = b_{h_1,i} - b_{h_2,i}, \qquad \phi_i|_{\p\O} = 0, \quad i=1,2,3, \\
\Delta \phi_c & = c_{h_1}-c_{h_2}, \qquad \phi_c|_{\p\O} = 0.
\end{align*}
The estimates on $(a_{h_1}, \mathbf{b}_{h_1},  c_{h_1})$ and $(a_{h_2}, \mathbf{b}_{h_2}, c_{h_2})$ are the same as those in Lemma \ref{lemma:macro_l2}, except for the following terms: for any $\psi \in \{ \psi_a,\psi_1,\psi_2,\psi_3,\psi_c \}$,
\begin{align*}
\Big| \iint_{\O\times \mathbb{R}^3} \nabla_x\phi_{h_1} \cdot \nabla_{v} (\frac{\psi}{\sqrt{\mu}}) (h_1-h_2) \sqrt{\mu} \dd x \dd v \Big|
&  \lesssim \Vert \nabla_x \phi_{h_1} \Vert_{L^\infty_{x}}  \Vert h_1-h_2\Vert_{L^2_{x,v}}^2 \lesssim \Vert wh_1\Vert_{L^\infty_{x,v}}  \Vert h_1-h_2\Vert_{L^2_{x,v}}^2, \\
\Big|\iint_{\O\times \mathbb{R}^3}  \nabla_x \phi_{h_1-h_2} \cdot \nabla_{v}(\frac{\psi}{\sqrt{\mu}})h_2\sqrt{\mu}\dd x \dd v   \Big| & \lesssim \Vert wh_2\Vert_{L^\infty_{x,v}} \Vert h_1-h_2\Vert_{L^2_{x,v}}^2.
\end{align*}
Therefore, we obtain the following macroscopic estimate:
\be \label{macro_l2_stability}
\begin{split}
& \Vert \mathbf{P}(h_1-h_2)\Vert_{L^2_{x,v}}^2+ \Vert \nabla_x \phi_{h_1-h_2}\Vert_{L^2_x}^2 
\\& \lesssim \Vert (\mathbf{I}-\mathbf{P})(h_1-h_2)\Vert_{L^2_{x,\nu}}^2 + |(h_1-h_2)|_{L^2_{\gamma_+}}^2 + [\Vert w h_1 \Vert_{L^\infty_{x,v}} + \Vert wh_2\Vert_{L^\infty_{x,v}} + \Vert \nabla_x \phi_E\Vert_{L^\infty_x}]\Vert h_1-h_2\Vert_{L^2_{x,v}}^2. 
\end{split}
\ee

Multiplying \eqref{macro_l2_stability} by a sufficiently small constant and adding it to \eqref{micro_l2_stability}, we obtain the following $L^2_{x,v}$ stability estimate:
\begin{align*}
& \Vert h_1-h_2\Vert_{L^2_{x,\nu}}^2 + \Vert \nabla_x \phi_{h_1-h_2}\Vert_{L^2_x}^2 + |(h_1-h_2)|_{L^2_{\gamma_+}}^2 \\
& \lesssim \big( \Vert wh_1\Vert_{L^\infty_{x,v}} + \Vert wh_2\Vert_{L^\infty_{x,v}} + \Vert \nabla_x \phi_E\Vert_{L^\infty_x} + \Vert w_{\tilde{\theta}} \nabla_{v}h_2\Vert_{L^\infty_{x,v}} \big) \Vert h_1-h_2\Vert_{L^2_{x,\nu}}^2.
\end{align*}
Therefore, when $\Vert w_{\tilde{\theta}} \nabla_{v}h_2\Vert_{L^\infty_{x,v}} \ll 1$, $\Vert w h_1 \Vert_{L^\infty_{x,v}}, \Vert w h_2 \Vert_{L^\infty_{x,v}} \ll 1$, and $\Vert \nabla_x \phi_E\Vert_{L^\infty_x} \ll 1$, we conclude that uniqueness holds.
\end{proof}

\section{Construction of steady solutions} \label{sec:existence}

In this section, we construct a solution to the steady problem \eqref{eqn:h} via the following iterative sequence: for any $\ell \in \N$,
\begin{align}
v \cdot \nabla_x h^{\ell+1} 
& - \nabla_x (\phi^\ell_h + \phi_E) \cdot \nabla_{v} h^{\ell+1} 
+ \frac{v \cdot \nabla_x \phi^\ell_h}{2} h^{\ell+1} + e^{-\phi_E} \mathcal{L} h^{\ell+1}
\notag \\
& = - (v \cdot \nabla_x \phi^{\ell+1}_h) e^{-\phi_E/2} \sqrt{\mu} + e^{-\phi_E/2} \Gamma(h^{\ell}, h^{\ell}), 
\label{eqtn:h^l} \\
h^{\ell+1} |_{\gamma_-} & = f_b (x, v), 
\label{bdry:h^l} \\
- \Delta \phi^i_h & = e^{-\phi_E/2} \int_{\R^3} h^{i} \sqrt{\mu} \dd v \text{ in } \O,  \ i\in \{\ell,\ell+1\},
\label{eqtn:phi^l} \\
\phi^i_h & = 0 \text{ on } \p\O, \ i\in \{\ell, \ell+1\},
\label{bdry:phi^l}
\end{align}
where the initial setting $h^0 = 0$ and $\nabla_x \phi^0_h = \mathbf{0}$.

We use the a priori estimates from Sections \ref{sec:L2Linfty_estimate}–\ref{sec:c1_estimate}, together with the uniqueness argument in Section \ref{sec:stationary_uniqueness}, to derive the regularity estimates in Section \ref{sec:regularity_construction} and to establish existence in Section \ref{sec:existence_construction}, respectively.
Moreover, since the construction \eqref{eqtn:h^l} involves $\phi^{\ell+1}_h$ on the right-hand side, it remains to establish the well-posedness of this construction. The proof is given in Section~\ref{sec:well_posedness_construction}.

Throughout this section, when applying the lemmas from Section \ref{sec:prelim}, we replace $f,\phi_f$ by the stationary solution $h,\phi_h$ defined in \eqref{eqn:h}, and for the kinetic weight lemmas from Section~\ref{sec:kinetic_weight}, we replace $\alpha (t,x,v)$ by the stationary weight $\alpha_h (x,v)$ defined in \eqref{alpha_weight_steady}. 
Moreover, we assume the inflow condition \eqref{inflow_condition} throughout all statements in this section.

\subsection{
\texorpdfstring{A priori $L^2-L^\infty$ estimate for the construction}{L2-Linfty estimate for the construction}
}
\label{sec:regularity_construction}

In this section, we establish a priori $L^2$ and $L^\infty$ estimates for the iterative construction \eqref{eqtn:h^l}-\eqref{bdry:phi^l} in Lemma \ref{lemma:Unif_steady}. 
The argument follows the a priori estimates developed in Section~\ref{sec:L2Linfty_estimate}.
The main difference is that the estimates obtained here are uniform with respect to the sequence $\{ h^{\ell+1} \}_{\ell=0}^{\infty}$.

\begin{lemma} \label{lemma:Unif_h_L2}

Suppose the inflow condition \eqref{inflow_condition} holds.
Under the construction \eqref{eqtn:h^l}-\eqref{bdry:phi^l}, $h^{\ell+1}$ satisfies that for any $\ell \in \N$,
\be \label{est:Unif_h_L2_1}
\begin{split}
| h^{\ell+1} |_{L^2_{\gamma_+}}^2 + \Vert (\mathbf{I}-\mathbf{P}) h^{\ell+1} \Vert_{L^2_{x,\nu}}^2  
& \lesssim |f_b|_{L^2_{\gamma_-}}^2 + \Vert \nu^{-1/2} \Gamma (h^{\ell}, h^{\ell})\Vert_{L^2_{x,v}}^2 + \Vert w h^{\ell} \Vert_{L^\infty_{x,v}} \Vert h^{\ell+1} \Vert_{L^2_{x,\nu}}^2
\\& \qquad + \Vert \nabla_x \phi_E \Vert_{L^\infty_x} \Vert h^{\ell+1} \Vert_{L^2_{x,v}}^2,
\end{split}
\ee
\be \label{est:Unif_h_L2_2}
\begin{split}
\Vert \mathbf{P} h^{\ell+1} \Vert_{L^2_{x,\nu}}^2
& \lesssim \Vert (\mathbf{I}-\mathbf{P}) h^{\ell+1} \Vert_{L^2_{x,\nu}}^2 + | h^{\ell+1} |_{L^2_{\gamma_+}}^2 + |f_b|_{L^2_{\gamma_-}}^2 + \Vert \nu^{-1/2} \Gamma (h^{\ell}, h^{\ell}) \Vert_{L^2_{x,v}}^2  
\\& \qquad + \delta_0 \Vert h^{\ell+1} \Vert_{L^2_{x,v}}^2 + \big( \Vert w h^{\ell} \Vert_{L^\infty_{x,v}} + \Vert \nabla_x \phi_E \Vert_{L^\infty_{x}} \big) \Vert h^{\ell+1} \Vert_{L^2_{x,v}}^2.
\end{split}
\ee
Furthermore, for any $\ell \in \N$,
\be \label{est:Unif_h_L2}
\begin{split}
| h^{\ell+1} |_{L^2_{\gamma_+}} + \Vert h^{\ell+1} \Vert_{L^2_{x,\nu}}
& \lesssim |f_b|_{L^2_{\gamma_-}} + \Vert \nu^{-1/2} \Gamma (h^{\ell}, h^{\ell}) \Vert_{L^2_{x,v}} + \Vert w h^{\ell} \Vert_{L^\infty_{x,v}}^{1/2} \Vert h^{\ell+1} \Vert_{L^2_{x,\nu}}
\\& \qquad + \sqrt{\delta_0} \Vert h^{\ell+1} \Vert_{L^2_{x,v}} + \big( \Vert w h^{\ell} \Vert_{L^\infty_{x,v}}^{1/2} + \Vert \nabla_x \phi_E \Vert_{L^\infty_x}^{1/2} \big) \Vert h^{\ell+1} \Vert_{L^2_{x,v}}.
\end{split}
\ee
\end{lemma}

\begin{proof}

First, we prove \eqref{est:Unif_h_L2_1}.
From \eqref{eqtn:h^l} and \eqref{bdry:h^l}, we obtain the following $L^2_{x,v}$ energy estimate:
\be \label{est1:Unif_h_L2_1}
\begin{split}
& | h^{\ell+1} |_{L^2_{\gamma_+}}^2 + \Vert e^{-\phi_E/2}(\mathbf{I}-\mathbf{P}) h^{\ell+1} \Vert_{L^2_{x,\nu}}^2 
\\& \lesssim |f_b|_{L^2_{\gamma_-}}^2 + o(1) \Vert e^{-\phi_E/2} (\mathbf{I}-\mathbf{P}) h^{\ell+1} \Vert_{L^2_{x,\nu}}^2 + \Vert \nu^{-1/2}e^{\phi_E/4}\Gamma (h^{\ell}, h^{\ell}) \Vert_{L^2_{x,v}}^2
\\& \qquad + \underbrace{\Vert \nabla_x \phi^\ell_h \Vert_{L^\infty_{x}} \Vert \nu^{1/2} h^{\ell+1} \Vert_{L^2_{x,v}}^2}_{\eqref{est1:Unif_h_L2_1}_1} 
+ \underbrace{\Big| \iint_{\O\times \mathbb{R}^3}  (v \cdot \nabla_x \phi^{\ell+1}_h ) e^{-\phi_E/2} \sqrt{\mu} h^{\ell+1} \dd x \dd v \Big|}_{\eqref{est1:Unif_h_L2_1}_2}.
\end{split}
\ee
From \eqref{eqtn:phi^l} and \eqref{bdry:phi^l}, together with Lemma \ref{lemma:phi_x_infinity}, we have
\be \label{est2:Unif_h_L2_1}
\eqref{est1:Unif_h_L2_1}_1 = \Vert \nabla_x \phi^{\ell}_h \Vert_{L^\infty_{x}} \Vert \nu^{1/2} h^{\ell+1} \Vert_{L^2_{x,v}}^2 
\lesssim \Vert w h^{\ell} \Vert_{L^\infty_{x,v}} \Vert h^{\ell+1} \Vert_{L^2_{x,\nu}}^2.
\ee
Similarly to the proof of Lemma~\ref{lemma:l2_energy}, the Dirichlet boundary condition \eqref{bdry:h^l}, together with the Poincaré inequality and the elliptic regularity estimate, implies that
\be \label{est5:Unif_h_L2_1}
\begin{split}
\eqref{est1:Unif_h_L2_1}_2
& = \Big| \iint_{\O\times \mathbb{R}^3} \phi^{\ell+1}_h \sqrt{\mu} (v \cdot \nabla_x \phi_E) e^{-\phi_E/2} h^{\ell+1} \dd x \dd v \Big|
\\& \lesssim \Vert \nabla_x \phi_E \Vert_{L^\infty_x} \Vert \phi^{\ell+1}_h \Vert_{L^2_x}\Vert h^{\ell+1} \Vert_{L^2_{x,v}} \lesssim \Vert \nabla_x \phi_E \Vert_{L^\infty_x} \Vert h^{\ell+1} \Vert_{L^2_{x,v}}^2.
\end{split}
\ee
Combining \eqref{est1:Unif_h_L2_1} with \eqref{est2:Unif_h_L2_1} and \eqref{est5:Unif_h_L2_1}, we conclude \eqref{est:Unif_h_L2_1}.

\smallskip

Second, we prove \eqref{est:Unif_h_L2_2}.
Given a test function $\psi$, the weak formulation for \eqref{eqtn:h^l} is
\begin{equation*}
\begin{split}
& - \iint_{\O\times \mathbb{R}^3} \mathbf{P} h^{\ell+1} (v\cdot \nabla_x \psi) \dd x \dd v - \iint_{\O\times \mathbb{R}^3} (\mathbf{I}-\mathbf{P}) h^{\ell+1} (v \cdot \nabla_x \psi) \dd x \dd v + \int_{\gamma} h^{\ell+1} \psi \dd \gamma
\\& + \iint_{\O\times \mathbb{R}^3} \sqrt{\mu} h^{\ell+1} \nabla_x \phi^{\ell}_h \cdot \nabla_{v} \big[\frac{1}{\sqrt{\mu}} \psi \big] \dd x \dd v + \iint_{\O\times \mathbb{R}^3} h^{\ell+1} \nabla_x \phi_E \cdot \nabla_{v} \psi \dd x \dd v + \iint_{\O\times \mathbb{R}^3} e^{-\phi_E}\mathcal{L} h^{\ell+1} \psi \dd x \dd v 
\\& = - \iint_{\O\times \mathbb{R}^3} (v \cdot \nabla_x \phi^{\ell+1}_h)e^{-\phi_E/2} \sqrt{\mu} \psi \dd x \dd v + \iint_{\O\times \mathbb{R}^3} e^{-\phi_E/2} \Gamma (h^{\ell}, h^{\ell}) \psi \dd x \dd v.  
\end{split}
\end{equation*}
Similarly to the proof of Lemma~\ref{lemma:macro_l2}, we write
\[
\mathbf{P} h^{\ell+1} = \big( a(x) + \mathbf{b} (x) \cdot v + c(x) \frac{|v|^2-3}{2} \big) \sqrt{\mu},
\]
and derive estimates for $a(x), \mathbf{b} (x)$, and $c(x)$ as follows.
For $c (x)$, we have
\begin{equation*} 
\Vert c \Vert_{L^2_{x}}^2 \lesssim \Vert (\mathbf{I}-\mathbf{P}) h^{\ell+1} \Vert_{L^2_{x,v}}^2 + \big( \Vert w h^{\ell} \Vert_{L^\infty_{x,v}} + \Vert \nabla_x \phi_E \Vert_{L^\infty_{x}} \big) \Vert h^{\ell+1} \Vert_{L^2_{x,v}}^2 + |h^{\ell+1}|_{L^2_{\gamma_+}}^2 + |f_b|_{L^2_{\gamma_-}}^2 + \Vert \nu^{-1/2} \Gamma (h^{\ell}, h^{\ell}) \Vert_{L^2_{x,v}}^2.
\end{equation*}
For $b_i$ with $i = 1, 2, 3$, we have
\begin{equation*} 
\Vert b_i \Vert_{L^2_{x}}^2 \lesssim \Vert (\mathbf{I}-\mathbf{P}) h^{\ell+1} \Vert_{L^2_{x,v}}^2 + \big( \Vert w h^{\ell} \Vert_{L^\infty_{x,v}} + \Vert \nabla_x \phi_E \Vert_{L^\infty_{x}} \big) \Vert h^{\ell+1} \Vert_{L^2_{x,v}}^2 + |h^{\ell+1}|_{L^2_{\gamma_+}}^2 + |f_b|_{L^2_{\gamma_-}}^2 + \Vert \nu^{-1/2} \Gamma (h^{\ell}, h^{\ell}) \Vert_{L^2_{x,v}}^2.
\end{equation*}
For $a (x)$, we have
\begin{equation*} 
\begin{split}
\Vert a \Vert_{L^2_{x}}^2 
& \lesssim \Vert (\mathbf{I}-\mathbf{P}) h^{\ell+1} \Vert_{L^2_{x,v}}^2 + \big( \Vert w h^{\ell} \Vert_{L^\infty_{x,v}} + \Vert \nabla_x \phi_E \Vert_{L^\infty_{x}} \big) \Vert h^{\ell+1} \Vert_{L^2_{x,v}}^2 + |h^{\ell+1}|_{L^2_{\gamma_+}}^2 + |f_b|_{L^2_{\gamma_-}}^2 
\\& \qquad + \delta_0 \Vert h^{\ell+1} \Vert_{L^2_{x,v}}^2 + \Vert \nu^{-1/2} \Gamma (h^{\ell}, h^{\ell}) \Vert_{L^2_{x,v}}^2.
\end{split}
\end{equation*}

Combining $a,b_i,c$ estimates above, we conclude \eqref{est:Unif_h_L2_2}.
Finally, we deduce \eqref{est:Unif_h_L2} by computing $2 \times \eqref{est:Unif_h_L2_1} + \eqref{est:Unif_h_L2_2}$.
\end{proof}

\begin{lemma} \label{lemma:Unif_wh_Linfty}
Suppose the inflow condition \eqref{inflow_condition} holds. Under the construction \eqref{eqtn:h^l}-\eqref{bdry:phi^l}, $h^{\ell+1}$ satisfies that for any $\ell \in \N$,
\be \label{est:Unif_wh_Linfty}
\Vert w h^{\ell+1} \Vert_{L^\infty_{x,v}} 
\lesssim \Vert \nu^{-1} w \Gamma (h^{\ell}, h^{\ell}) \Vert_{L^\infty_{x,v}} + | w f_b|_{L^\infty_{\p\O,v}} + \Vert h^{\ell+1} \Vert_{L^2_{x,v}}.
\ee
\end{lemma}

\begin{proof}

From the assumption that $\Vert \nabla_x \phi^\ell_h \Vert_{L^\infty_x} \ll 1$ and $|e^{-\phi_E}-1|\ll 1$, we have
\be \label{est1:Unif_wh_Linfty}
e^{-\phi_E}\nu(v) + \frac{v \cdot \nabla_x \phi^\ell_h}{2} \geq \frac{1}{2} \nu(v).
\ee
We rewrite \eqref{eqtn:h^l} as
\be \label{est2:Unif_wh_Linfty}
\begin{split}
& v \cdot \nabla_x h^{\ell+1} - \nabla_x (\phi^\ell_h + \phi_E) \cdot \nabla_{v} h^{\ell+1} + \Big( e^{-\phi_E}\nu(v) + \frac{v\cdot \nabla_x \phi^\ell_h}{2} \Big) h^{\ell+1} 
\\& = e^{-\phi_E} K h^{\ell+1} - (v \cdot \nabla_x \phi^{\ell+1}_h ) e^{-\phi_E/2}\sqrt{\mu} + e^{-\phi_E/2} \Gamma (h^{\ell}, h^{\ell}).
\end{split}
\ee
Define that $w(v) = e^{\theta |v|^2}$ with $0 < \theta < \frac{1}{4}$.
Similarly to the proof of Lemma~\ref{prop:wf_Linfty}, we apply the method of characteristics to \eqref{est2:Unif_wh_Linfty} and use \eqref{est1:Unif_wh_Linfty} to obtain that
\begin{align}
| w (v) h^{\ell+1} (x,v) | 
\leq & \mathbf{1}_{\tb(x,v)\geq t} e^{ - \int^t_0 \frac{\nu(V(s))}{4} \dd s} \big| w(V(0)) h^{\ell+1} (X(0),V(0)) \big|
\label{est3:Unif_wh_Linfty_1} \\
& + \mathbf{1}_{\tb(x,v)<t} e^{-\int^t_{t- \tb} \frac{\nu(V(s))}{4}\dd s} \big| w (\vb) f_b (\xb, \vb) \big| \label{est3:Unif_wh_Linfty_2} \\
& + \Big| \int^t_{\max\{0,t-\tb\}} e^{-\int^t_s \frac{\nu(V(\tau))}{4}\dd \tau} \int_{\mathbb{R}^3} \mathbf{k}(V(s),u) \frac{w(V(s))}{w(u)} w(u) h^{\ell+1} (X(s),u) \dd u \dd s \Big| \label{est3:Unif_wh_Linfty_3} \\
& + \int^t_{\max\{0,t-\tb\}} e^{-\int^t_s \frac{\nu(V(\tau))}{4}\dd \tau} \big| V(s) \cdot \nabla_x \phi^{\ell+1}_h (X(s)) \big| w(V(s))\sqrt{\mu(V(s))} \dd s \label{est3:Unif_wh_Linfty_4} \\
& + \int^t_{\max\{0,t-\tb\}} e^{-\int^t_s \frac{\nu(V(\tau))}{4}\dd \tau} w(V(s)) \big| \Gamma (h^{\ell} (X(s),V(s)), h^{\ell} (X(s),V(s)) ) \big| \dd s,
\label{est3:Unif_wh_Linfty_5}
\end{align}
where $(X(s), V(s)) := (X(s;t,x,v), V(s;t,x,v))$ with $\max\{0,t-\tb\} \leq s \leq t$.
Recall that $\nu(v) \geq \nu_0 > 0$ for any $v \in \mathbb{R}^3$. Now pick $t > 0$ such that 
\[
\nu_0 t \geq 10
\ \text{ and } \
(1 + t) e^{ - \frac{\nu_0}{4} t} \leq \frac{1}{4}.
\]

For \eqref{est3:Unif_wh_Linfty_1}, we have
\be \label{est4:Unif_wh_Linfty_1}
|\eqref{est3:Unif_wh_Linfty_1}|
\leq e^{ - \frac{\nu_0}{4} t} \big| w(V(0)) h^{\ell+1} (X(0),V(0)) \big|
\lesssim \frac{1}{4} \Vert w h^{\ell+1} \Vert_{L^\infty_{x,v}}.
\ee

For \eqref{est3:Unif_wh_Linfty_2}, we have
\be \label{est4:Unif_wh_Linfty_2}
|\eqref{est3:Unif_wh_Linfty_2}|
\leq \big| w (\vb) f_b (\xb, \vb) \big|
\leq | w f_b|_{L^\infty_{\p\O,v}}.
\ee

For \eqref{est3:Unif_wh_Linfty_4}, from \eqref{phi_f_C^1_2} we have
\be \label{est4:Unif_wh_Linfty_4}
|\eqref{est3:Unif_wh_Linfty_4}|
\leq \Vert \nabla_x \phi^{\ell+1}_h \Vert_{L^\infty_x} \int^t_{\max\{0,t-\tb\}} e^{- \nu_0 (t-s)/4} \dd s
\lesssim \frac{1}{8} \Vert w h^{\ell+1} \Vert_{L^\infty_{x,v}} + \Vert h^{\ell+1} \Vert_{L^2_{x,v}}.
\ee

For \eqref{est3:Unif_wh_Linfty_5}, we have
\be \label{est4:Unif_wh_Linfty_5}
|\eqref{est3:Unif_wh_Linfty_5}|
\leq \Vert \nu^{-1}w\Gamma (h^{\ell}, h^{\ell}) \Vert_{L^\infty_{x,v}} \int^t_{\max\{0,t-\tb\}} e^{-\int^t_s \frac{\nu(V(\tau))}{4}\dd \tau} \nu(V(s)) \dd s 
\lesssim \Vert \nu^{-1} w \Gamma (h^{\ell}, h^{\ell}) \Vert_{L^\infty_{x,v}}.
\ee

Finally, for \eqref{est3:Unif_wh_Linfty_3}, we apply the method of characteristics to $w (u) h^{\ell+1} (X(s), u)$ again. We denote that
\be \notag
V(s',s):= V(s';s,X(s),u)
\ \text{ and } \
X(s',s): = X(s';s,X(s),u).
\ee
Following from \eqref{est3:Unif_wh_Linfty_1}-\eqref{est3:Unif_wh_Linfty_5}, we obtain that
\be \notag
\eqref{est3:Unif_wh_Linfty_3} \leq \Big| \int^t_{\max\{0,t-\tb\}} e^{-\int^t_s \frac{\nu(V(\tau))}{4} \dd \tau} \int_{\mathbb{R}^3} \mathbf{k}(V(s),u) \frac{w(V(s))}{w(u)} w(u) h^{\ell+1} (X(s),u) \dd u \dd s \Big|,
\ee
with 
\begin{align}
& | w(u) h^{\ell+1} (X(s), u) | 
\notag \\
\leq & \mathbf{1}_{\tb(X(s),u)\geq s} e^{-\int^s_0 \frac{\nu(V(s',s))}{4} \dd s'} \big| w(V(0,s)) h^{\ell+1} (X(0,s),V(0,s)) \big| 
\label{est4:Unif_wh_Linfty_3_1} \\
& + \mathbf{1}_{\tb(X(s),u)<s}  e^{-\int^s_{s-\tb(X(s),u)} \frac{\nu(V(s',s))}{4} \dd s'} \big| w (\vb(X(s),u)) f_b (\xb(X(s),u), \vb(X(s),u)) \big|
\label{est4:Unif_wh_Linfty_3_2} \\
& + \Big| \int^s_{\max\{0,s-\tb(X(s),u)\}} e^{-\int^s_{s'} \frac{\nu(V(\tau',s))}{4} \dd \tau'} \int_{\mathbb{R}^3} \mathbf{k}(V(s',s),u') \frac{w(V(s',s))}{w(u')} w(u') h^{\ell+1} (X(s',s),u') \dd u' \dd s' \Big|
\label{est4:Unif_wh_Linfty_3_3} \\
& + \int^s_{\max\{0,s-\tb(X(s),u)\}} e^{-\int^s_{s'} \frac{\nu(V(\tau',s))}{4} \dd \tau'} \big| V(s',s) \cdot \nabla_x \phi^{\ell+1}_h (X(s',s)) \big| w(V(s',s)) \sqrt{\mu(V(s',s))} \dd s'  
\label{est4:Unif_wh_Linfty_3_4} \\
& + \int^s_{\max\{0,s-\tb(X(s),u)\}} e^{-\int^s_{s'} \frac{\nu(V(\tau',s))}{4} \dd \tau'} w(V(s',s)) \big| \Gamma(h^{\ell}, h^{\ell}) (X(s',s),V(s',s)) \big| \dd s'. 
\label{est4:Unif_wh_Linfty_3_5}
\end{align}

For \eqref{est4:Unif_wh_Linfty_3_1} and \eqref{est4:Unif_wh_Linfty_3_2}, using Lemma \ref{lemma:k_theta} we compute that
\be \label{est5:Unif_wh_Linfty_3_1}
\begin{split}
& \Big| \int^t_{\max\{0,t-\tb\}} e^{-\int^t_s \frac{\nu(V(\tau))}{4} \dd \tau} \int_{\mathbb{R}^3} \mathbf{k}(V(s),u) \frac{w(V(s))}{w(u)} \eqref{est4:Unif_wh_Linfty_3_1} \dd u \dd s \Big| 
\lesssim \frac{1}{4} \Vert w h^{\ell+1} \Vert_{L^\infty_{x,v}},
\\& \Big| \int^t_{\max\{0,t-\tb\}} e^{-\int^t_s \frac{\nu(V(\tau))}{4} \dd \tau} \int_{\mathbb{R}^3} \mathbf{k}(V(s),u) \frac{w(V(s))}{w(u)} \eqref{est4:Unif_wh_Linfty_3_2} \dd u \dd s \Big|
\lesssim | w f_b|_{L^\infty_{\p\O,v}}.
\end{split}
\ee

For \eqref{est4:Unif_wh_Linfty_3_4} and \eqref{est4:Unif_wh_Linfty_3_5}, analogous to \eqref{est4:Unif_wh_Linfty_4} and \eqref{est4:Unif_wh_Linfty_5}, Lemma \ref{lemma:k_theta} further implies that
\be \label{est5:Unif_wh_Linfty_3_4}
\begin{split}
\Big| \int^t_{\max\{0,t-\tb\}} e^{-\int^t_s \frac{\nu(V(\tau))}{4} \dd \tau} \int_{\mathbb{R}^3} \mathbf{k}(V(s),u) \frac{w(V(s))}{w(u)} \eqref{est4:Unif_wh_Linfty_3_4} \dd u \dd s \Big|
& \lesssim \frac{1}{8} \Vert w h^{\ell+1} \Vert_{L^\infty_{x,v}} + \Vert h^{\ell+1} \Vert_{L^2_{x,v}},
\\ \Big| \int^t_{\max\{0,t-\tb\}} e^{-\int^t_s \frac{\nu(V(\tau))}{4} \dd \tau} \int_{\mathbb{R}^3} \mathbf{k}(V(s),u) \frac{w(V(s))}{w(u)} \eqref{est4:Unif_wh_Linfty_3_5} \dd u \dd s \Big| 
& \lesssim \Vert \nu^{-1} w \Gamma (h^{\ell}, h^{\ell}) \Vert_{L^\infty_{x,v}}.
\end{split}
\ee

For \eqref{est4:Unif_wh_Linfty_3_3}, pick $0 < \delta \ll 1$ and $N \gg 1$ and consider four cases: 
$(a)$ $s - s' < \delta$, $(b)$ $|V(s',s) - u'|< \frac{1}{N} \text{ or } |u'|>N$, $(c)$ $|V(s)-u|<\frac{1}{N} \text{ or }|u|>N$, and $(d)$ $|V(s)-u| \geq \frac{1}{N}$, $|u| \leq N$, $s - s' \geq \delta$, $|u'| \leq N$, $|V(s',s)-u'| \geq \frac{1}{N}$.

\textit{Case $(a)$: $s - s' < \delta$.}
Using Lemma \ref{lemma:k_theta}, we compute that
\be \label{est5:Unif_wh_Linfty_3_3_a}
\begin{split}
& \Big| \int^t_{\max\{0,t-\tb\}} e^{-\int^t_s \frac{\nu(V(\tau))}{4} \dd \tau} \int_{\mathbb{R}^3} \mathbf{k}(V(s),u) \frac{w(V(s))}{w(u)} | \eqref{est4:Unif_wh_Linfty_3_3} \mathbf{1}_{s-s'<\delta} | \dd u \dd s \Big| 
\\& \lesssim \Vert w h^{\ell+1} \Vert_{L^\infty_{x,v}} \int^t_{0} e^{- \nu_0 (t-s)/4} \dd s \int_{\mathbb{R}^3} \mathbf{k}(V(s),u) \frac{w(V(s))}{w(u)} \dd u \int^s_{s-\delta} e^{- \nu_0 (s-s')/4} \dd s'
\lesssim \delta \Vert w h^{\ell+1} \Vert_{L^\infty_{x,v}}.
\end{split}
\ee

\textit{Case $(b)$: $|V(s',s) - u'|< \frac{1}{N}$ or $|u'|>N $.}
Using Lemma \ref{lemma:k_theta}, we compute that
\be \label{est5:Unif_wh_Linfty_3_3_b}
\begin{split}
& \Big| \int^t_{\max\{0,t-\tb\}} e^{-\int^t_s \frac{\nu(V(\tau))}{4} \dd \tau} \int_{\mathbb{R}^3} \mathbf{k}(V(s),u) \frac{w(V(s))}{w(u)} | \eqref{est4:Unif_wh_Linfty_3_3} \mathbf{1}_{|V(s',s) - u'|< \frac{1}{N} \text{ or } |u'|>N} | \dd u \dd s \Big| 
\\& \lesssim o(1) \Vert w h^{\ell+1} \Vert_{L^\infty_{x,v}} \int^t_{0} e^{- \nu_0 (t-s)/4} \dd s \int^s_{0} e^{- \nu_0 (s - s')/4} \dd s'  
\lesssim o(1) \Vert w h^{\ell+1} \Vert_{L^\infty_{x,v}}.
\end{split}
\ee

\textit{Case $(c)$: $|V(s)-u|<\frac{1}{N}$ or $|u|>N$.}
Similar to case $(b)$, Lemma \ref{lemma:k_theta} implies that
\be \label{est5:Unif_wh_Linfty_3_3_c}
\begin{split}
& \Big| \int^t_{\max\{0,t-\tb\}} e^{-\int^t_s \frac{\nu(V(\tau))}{4} \dd \tau} \int_{\mathbb{R}^3} \mathbf{k}(V(s),u) \frac{w(V(s))}{w(u)} | \eqref{est4:Unif_wh_Linfty_3_3} | \mathbf{1}_{|V(s)-u|< \frac{1}{N} \text{ or } |u|>N} \dd u \dd s \Big| 
\\& \lesssim o(1) \int^t_{0} e^{- \nu_0 (t-s)/4} \dd s \Vert w h^{\ell+1} \Vert_{L^\infty_{x,v}} \int^s_{0} e^{- \nu_0 (s - s')/4} \dd s'  
\lesssim o(1) \Vert w h^{\ell+1} \Vert_{L^\infty_{x,v}}.
\end{split}
\ee

\textit{Case $(d)$: $|V(s)-u| \geq \frac{1}{N}$, $|u| \leq N$, $|u'| \leq N$, $|V(s',s)-u'| \geq \frac{1}{N}$, $s - s' \geq \delta$.}
From \eqref{k_N_upper_bdd} in Lemma \ref{lemma:k_theta},
\be \notag
\mathbf{k}(V(s),u)\frac{w(V(s))}{w(u)}\mathbf{k}(V(s,s'),u')\frac{w(V(s',s))}{w(u')} w(u') \mathbf{1}_{|V(s)-u| \geq \frac{1}{N}, |u| \leq N, |u'| \leq N, |V(s',s)-u'| \geq \frac{1}{N}}
\leq (C_N)^2 e^{\theta N^2}.
\ee
Moreover, consider the following change of variable: for fixed $s', s, X(s)$, 
\be \notag
u \in \mathbb{R}^3 \mapsto X(s',s) = X(s';s,X(s),u) \in \O.
\ee
From Lemma \ref{lemma:deri_XV} and the assumption, we get
\be \label{est5:Unif_wh_Linfty_3_3_d}
\begin{split}
& \Big| \int^t_{\max\{0,t-\tb\}} e^{-\int^t_s \frac{\nu(V(\tau))}{4} \dd \tau} \int_{\mathbb{R}^3} \mathbf{k}(V(s),u) \frac{w(V(s))}{w(u)} | \eqref{est4:Unif_wh_Linfty_3_3} \mathbf{1}_{|V(s)-u| \geq \frac{1}{N}, |u| \leq N, s - s' \geq \delta, |u'| \leq N, |V(s',s)-u'| \geq \frac{1}{N}} | \dd u \dd s \Big| 
\\& \lesssim (C_N)^2 e^{\theta N^2} (N)^4 \int^t_{0} e^{- \nu_0 (t-s)/4} \dd s \int^s_{0} e^{- \nu_0 (s - s')/4} \dd s' \int_{\O} \big| h^{\ell+1} (X(s',s),u') \big| \frac{1}{\delta^3} \dd y 
\lesssim_{\delta,N} \Vert h^{\ell+1} \Vert_{L^2_{x,v}}.
\end{split}
\ee

Combining \eqref{est5:Unif_wh_Linfty_3_3_a}-\eqref{est5:Unif_wh_Linfty_3_3_d}, we obtain that 
\be \notag
\begin{split}
& \Big| \int^t_{\max\{0,t-\tb\}} e^{-\int^t_s \frac{\nu(V(\tau))}{4} \dd \tau} \int_{\mathbb{R}^3} \mathbf{k}(V(s),u) \frac{w(V(s))}{w(u)} \eqref{est4:Unif_wh_Linfty_3_3} \dd u \dd s \Big| 
\lesssim_{\delta,N} o(1) \Vert w h^{\ell+1} \Vert_{L^\infty_{x,v}} + \Vert h^{\ell+1} \Vert_{L^2_{x,v}}.
\end{split}
\ee
This, together with \eqref{est5:Unif_wh_Linfty_3_1} and \eqref{est5:Unif_wh_Linfty_3_4}, shows that
\be \notag
\eqref{est3:Unif_wh_Linfty_3} 
\lesssim \frac{3}{8} \Vert w h^{\ell+1} \Vert_{L^\infty_{x,v}} + | w f_b|_{L^\infty_{\p\O,v}} + \Vert \nu^{-1} w \Gamma (h^{\ell}, h^{\ell}) \Vert_{L^\infty_{x,v}} + 2 \Vert h^{\ell+1} \Vert_{L^2_{x,v}}.
\ee
Combining the above with \eqref{est4:Unif_wh_Linfty_1}-\eqref{est4:Unif_wh_Linfty_5}, we conclude \eqref{est:Unif_wh_Linfty}.
\end{proof}





\begin{lemma} \label{lemma:Unif_steady}

Suppose the inflow condition \eqref{inflow_condition} holds. 
Under the construction \eqref{eqtn:h^l}-\eqref{bdry:phi^l}, $(h^{\ell+1}, \nabla_x \phi^\ell_h )$ satisfies that for any $\ell \in \N$,
\begin{align}
& | h^{\ell+1} |_{L^2_{\gamma_+}} + \Vert h^{\ell+1} \Vert_{L^2_{x,\nu}}
\lesssim  |f_b|_{L^2_{\gamma_-}} + | w f_b|_{L^\infty_{\p\O,v}},
\label{Uest:h^l_L2} \\
& \Vert w h^{\ell+1} \Vert_{L^\infty_{x,v}} \lesssim  |f_b|_{L^2_{\gamma_-}} + | w f_b|_{L^\infty_{\p\O,v}}, \label{Uest:wh^l_Linfty} \\
& \Vert \nu^{-1/2} \Gamma (h^{\ell+1}, h^{\ell+1}) \Vert_{L^2_{x,v}} \lesssim \Vert \nu^{-1} w \Gamma (h^{\ell+1}, h^{\ell+1}) \Vert_{L^\infty_{x,v}} \lesssim |f_b|_{L^2_{\gamma_-}}^2 + | w f_b|_{L^\infty_{\p\O,v}}^2. \label{Uest:Gamma^l} \\
& \Vert \nabla_x \phi^\ell_h \Vert_{L^\infty_{x}} 
\lesssim  |f_b|_{L^2_{\gamma_-}} + | w f_b|_{L^\infty_{\p\O,v}}.
\label{Uest:DPhi^l}
\end{align} 
\end{lemma}

\begin{proof}

We first prove \eqref{Uest:h^l_L2} and \eqref{Uest:wh^l_Linfty} by induction.
From the construction \eqref{eqtn:h^l}-\eqref{bdry:phi^l} and the initial setting $h^0 = 0$ and $\nabla_x \phi^0_h = \mathbf{0}$, \eqref{Uest:h^l_L2} and \eqref{Uest:wh^l_Linfty} hold for $\ell = -1$.

Suppose that \eqref{Uest:h^l_L2} and \eqref{Uest:wh^l_Linfty} hold for all $-1 \leq \ell \leq k-1$.
Now consider the case when $\ell = k$. From \eqref{est:Unif_h_L2} in Lemma \ref{lemma:Unif_h_L2} and \eqref{est:Unif_wh_Linfty} in Lemma \ref{lemma:Unif_wh_Linfty}, we obtain that
\be \label{est1:Unif_steady}
\begin{split}
| h^{k+1} |_{L^2_{\gamma_+}} + \Vert h^{k+1} \Vert_{L^2_{x,\nu}}
& \lesssim |f_b|_{L^2_{\gamma_-}} + \Vert \nu^{-1/2} \Gamma (h^{k}, h^{k}) \Vert_{L^2_{x,v}} + \Vert w h^{k} \Vert_{L^\infty_{x,v}}^{1/2} \Vert h^{k+1} \Vert_{L^2_{x,\nu}} 
\\& \qquad + \sqrt{\delta_0} \Vert h^{k+1} \Vert_{L^2_{x,v}} + \big( \Vert w h^{k} \Vert_{L^\infty_{x,v}}^{1/2} + \Vert \nabla_x \phi_E \Vert_{L^\infty_x}^{1/2} \big) \Vert h^{k+1} \Vert_{L^2_{x,v}},
\end{split}
\ee
\be \label{est2:Unif_steady}
\Vert w h^{k+1} \Vert_{L^\infty_{x,v}} 
\lesssim \Vert \nu^{-1} w \Gamma (h^{k}, h^{k}) \Vert_{L^\infty_{x,v}} + | w f_b|_{L^\infty_{\p\O,v}} + \Vert h^{k+1} \Vert_{L^2_{x,v}}.
\ee
On the other hand, Lemma \ref{lemma:gamma} implies that
\be \label{est3:Unif_steady}
\Vert \nu^{-1/2} \Gamma (h^{k}, h^{k}) \Vert_{L^2_{x,v}} \lesssim \Vert \nu^{-1} w \Gamma (h^{k}, h^{k}) \Vert_{L^\infty_{x,v}} \lesssim \Vert w h^{k} \Vert_{L^\infty_{x,v}}^2.
\ee
Computing $2 \times \eqref{est1:Unif_steady} + \eqref{est2:Unif_steady}$ and using \eqref{est3:Unif_steady}, we derive 
\be \notag
\begin{split}
& 2 | h^{k+1} |_{L^2_{\gamma_+}} + \Vert h^{k+1} \Vert_{L^2_{x,\nu}}
+ \Vert w h^{k+1} \Vert_{L^\infty_{x,v}}
\\& \lesssim 2 |f_b|_{L^2_{\gamma_-}} + | w f_b|_{L^\infty_{\p\O,v}} + 3 \Vert w h^{k} \Vert_{L^\infty_{x,v}}^2 + 2 \Vert w h^{k} \Vert_{L^\infty_{x,v}}^{1/2} \Vert h^{k+1} \Vert_{L^2_{x,\nu}} 
\\& \qquad + \sqrt{\delta_0} \Vert h^{k+1} \Vert_{L^2_{x,v}} + 2 \big( \Vert w h^{k} \Vert_{L^\infty_{x,v}}^{1/2} + \Vert \nabla_x \phi_E \Vert_{L^\infty_x}^{1/2} \big) \Vert h^{k+1} \Vert_{L^2_{x,v}}.
\end{split}
\ee
Recall that \eqref{Uest:h^l_L2} and \eqref{Uest:wh^l_Linfty} hold for $\ell = k-1$.
Since $\Vert \nabla_x \phi_E \Vert^{1/2}_{L^\infty_x} + \delta_0 \ll 1$ and $\Vert h^{k+1} \Vert_{L^2_{x,\nu}} \lesssim \Vert w h^{k+1} \Vert_{L^\infty_{x,v}}$, we further obtain that
\begin{equation*}
2 | h^{k+1} |_{L^2_{\gamma_+}} + \frac{1}{2} \Vert h^{k+1} \Vert_{L^2_{x,\nu}} + \frac{1}{2} \Vert w h^{k+1} \Vert_{L^\infty_{x,v}}
\lesssim 2 |f_b|_{L^2_{\gamma_-}} + | w f_b|_{L^\infty_{\p\O,v}} + 3 \Vert w h^{k} \Vert_{L^\infty_{x,v}}^2.
\end{equation*}
This implies that
\begin{equation*}
\begin{split}
& 4 | h^{k+1} |_{L^2_{\gamma_+}} + \Vert h^{k+1} \Vert_{L^2_{x,\nu}} + \Vert w h^{k+1} \Vert_{L^\infty_{x,v}}
\lesssim  |f_b|_{L^2_{\gamma_-}} + | w f_b|_{L^\infty_{\p\O,v}},
\end{split}
\end{equation*}
and thus \eqref{Uest:h^l_L2} and \eqref{Uest:wh^l_Linfty} hold for all $\ell = k$. Therefore, we conclude \eqref{Uest:h^l_L2} and \eqref{Uest:wh^l_Linfty} for all $\ell \geq -1$.

Next, using \eqref{Uest:wh^l_Linfty} and \eqref{est3:Unif_steady}, we conclude \eqref{Uest:Gamma^l} for all $\ell \geq -1$.
Finally, we conclude \eqref{Uest:DPhi^l} for all $\ell \geq -1$ from \eqref{Uest:wh^l_Linfty} and 
Lemma \ref{lemma:phi_x_infinity}.
\end{proof}

\subsection{Well-posedness of the construction}
\label{sec:well_posedness_construction}

Recall that the construction \eqref{eqtn:h^l} has $\phi^{\ell+1}_h$ on the right-hand side. We now establish the well-posedness of the construction \eqref{eqtn:h^l}-\eqref{bdry:phi^l} in the following theorem.

Define $E_{h}^\ell = - \nabla_x \phi_h^\ell + \nabla_x \phi_E$.
Similar to the definition of $\alpha_h$ in \eqref{alpha_weight_steady}, we define $\alpha^{\ell}_h$ as follows: for every $\ell \in \N$,
\begin{equation} \notag
\begin{split}
\alpha^{\ell}_h :=
\begin{cases}
& \chi_{\delta'} \Big( \big[ |v\cdot \nabla_x\xi(x)|^2 + \xi^2(x) - 2 (v\cdot \nabla^2\xi(x) \cdot v)\xi(x) - 2(E_{h}^{\ell}(\tilde{x})\cdot \nabla_x\xi(\tilde{x}))\xi(x) \big]^{1/2} \Big), 
\ \text{ if } x \in \O_\delta, \\[5pt]
& \delta', 
\ \text{ if } x \in \O \setminus \O_\delta,
\end{cases} 
\end{split}
\end{equation} 
where $\chi_{\delta'}$ and $\xi$ are defined in \eqref{cut_off_function} and \eqref{xi_dist}, respectively.
Moreover, we set $\alpha^{-1}_h := 0$ for convenience.

\begin{theorem} \label{thm:well_poseness}

For any $\ell \in \N$, assume that $(h^{\ell}, \nabla_x \phi^\ell_h)$ from the construction \eqref{eqtn:h^l}-\eqref{bdry:phi^l}. 
Then there exists a unique solution $h^{\ell+1}$ such that
\be \label{eqtn:well_poseness}
\begin{split}
v \cdot \nabla_x h^{\ell+1} 
& - \nabla_x (\phi^\ell_h + \phi_E) \cdot \nabla_{v} h^{\ell+1} 
+ \frac{v \cdot \nabla_x \phi^\ell_h}{2} h^{\ell+1} + e^{-\phi_E} \mathcal{L} h^{\ell+1}
\\& = - (v \cdot \nabla_x \phi^{\ell+1}_h) e^{-\phi_E/2} \sqrt{\mu} + e^{-\phi_E/2} \Gamma(h^{\ell}, h^{\ell}),
\\ h^{\ell+1} |_{\gamma_-} & = f_b (x, v), \\
-\Delta \phi_h^{\ell+1} & =e^{-\phi_E/2} \int_{\mathbb{R}^3} h^{\ell+1} \sqrt{\mu}\dd v \text{ in }\O,  \ \ \ \ \phi^{\ell+1}_h|_{\p\O} = 0.
\end{split}
\ee
with the solution $h^{\ell+1}$ satisfying that 
\begin{align}
| h^{\ell+1} |_{L^2_{\gamma_+}} + \Vert h^{\ell+1} \Vert_{L^2_{x,\nu}} 
& \leq  |f_b|_{L^2_{\gamma_-}} + | w f_b|_{L^\infty_{\p\O,v}},
\label{reg:well_poseness_L2} \\
\Vert w h^{\ell+1} \Vert_{L^\infty_{x,v}} 
& \lesssim  |f_b|_{L^2_{\gamma_-}} + | w f_b|_{L^\infty_{\p\O,v}}, 
\label{reg:well_poseness_wLinfty} \\ 
\Vert w_{\tilde{\theta}} \p_{x,v} h^{\ell+1} \Vert_{L^p_{x,v}} + \Vert w_{\tilde{\theta}} \alpha^{\ell}_h \p_{x,v} h^{\ell+1} \Vert_{L^\infty_{x,v}} & \lesssim |f_b|_{L^2_{\gamma_-}} + | wf_b|_{L^\infty_{\p\O,v}} + | w \p_{\mathbf{x}_p,v} f_b|_{L^\infty_{\p\O,v}},
\label{reg:well_poseness_weight_w1p} \\
\Vert w_{\tilde{\theta}} \nabla_v h^{\ell} \Vert_{L^\infty_{x,v}} 
& \lesssim |f_b|_{L^2_{\gamma_-}} + | wf_b|_{L^\infty_{\p\O,v}} + | w \p_{\mathbf{x}_p,v} f_b|_{L^\infty_{\p\O,v}}.
\label{reg:well_poseness_weight_c1}
\end{align}
\end{theorem}

The proof of Theorem \ref{thm:well_poseness} proceeds in two steps.
\begin{itemize}
\item In \textbf{Step 1}, we introduce a prescribed function $g (x, v)$. Proposition~\ref{prop:steady_small_lambda_step1} shows that the steady problem \eqref{eqtn:h_step1}–\eqref{bdry:phi_step1}, involving a sufficiently small constant $0 < \lambda \ll 1$, admits a unique solution. 
We remark that the inclusion of a small $\lambda \ll 1$ is in the same spirit as in \cite{duan2019effects}, where $\lambda$ was used to handle the $\ell$-iteration of $K^\ell$.

\smallskip

\item In \textbf{Step 2}, Proposition \ref{prop:h_step2} applies an induction argument to extend the result from $\lambda$ to $1$.
By replacing the prescribed function \(g\) with the sequence \(\{ h^{\ell} \}_{\ell = 0}^{\infty}\) in \eqref{eqtn:h^l}, the well-posedness of the problem is established.
\end{itemize}

\smallskip

\textbf{Step 1.}
Suppose that a given function $g (x, v)$ satisfies that $g |_{\gamma_-} = f_b (x, v)$ and 
\be \label{condition:g_step1}
\begin{split}
&| g |_{L^2_{\gamma_+}} + \Vert g \Vert_{L^2_{x,\nu}} + \Vert w g \Vert_{L^\infty_{x,v}}
 \lesssim  |f_b|_{L^2_{\gamma_-}} + | w f_b|_{L^\infty_{\p\O,v}}, 
\end{split}
\ee
Moreover, define $\phi_g$ by
\be \label{condition:phi_g_step1}
\begin{cases}
& - \Delta \phi_g = e^{-\phi_E/2} \int_{\R^3} g \sqrt{\mu} \dd v 
\text{ in } \O, \\[5pt]
& \phi_g = 0 
\text{ on } \p\O.
\end{cases}
\ee

\begin{proposition} \label{prop:steady_small_lambda_step1}

For any non-negative $C \geq 0$, let $S_1(x,v)$ be a given function that satisfies
\begin{equation} \label{S1_assumption}
\begin{split}
\Vert S_1\Vert_{L^2_{x,v}} 
& \lesssim C[ |f_b|_{L^2_{\gamma_-}} +  | w f_b|_{L^\infty_{\p\O,v}}],
\\ \Vert w S_1 \Vert_{L^\infty_{x,v}} 
& \lesssim C[ |f_b|_{L^2_{\gamma_-}} +  | w f_b|_{L^\infty_{\p\O,v}}].
\end{split}
\end{equation}
Suppose the inflow condition \eqref{inflow_condition} holds. Assume that $(g, \phi_g)$ satisfy \eqref{condition:g_step1} and \eqref{condition:phi_g_step1}.
Then there exists a sufficiently small $0 < \lambda = \lambda (\Omega) \ll 1$ such that the following system
\begin{align}
v \cdot \nabla_x h
& - \nabla_x (\phi_g + \phi_E) \cdot \nabla_{v} h 
+ \frac{v \cdot \nabla_x \phi_g}{2} h + e^{-\phi_E} \mathcal{L} h
\notag \\
& = - \lambda (v \cdot \nabla_x \phi_h) e^{-\phi_E/2} \sqrt{\mu} + e^{-\phi_E/2} \Gamma(g, g) + S_1, 
\label{eqtn:h_step1} \\
h |_{\gamma_-} & = f_b (x, v), 
\label{bdry:h_step1} \\
- \Delta \phi_h & = e^{-\phi_E/2} \int_{\R^3} h \sqrt{\mu} \dd v \text{ in } \O, 
\label{eqtn:phi_step1} \\
\phi_h & = 0 \text{ on } \p\O.
\label{bdry:phi_step1}
\end{align}
admits a unique solution $h (x, v)$ satisfying that
\begin{align} 
| h |_{L^2_{\gamma_+}} + \Vert h \Vert_{L^2_{x,\nu}}
& \lesssim (C+1) [ |f_b|_{L^2_{\gamma_-}} + | w f_b|_{L^\infty_{\p\O,v}}].
\label{regularity:h_L2_step1}
\end{align} 
Moreover, the pair $(h, \phi_h)$ satisfies
\be \label{regularity:wh_Linfty_step1}
\begin{split}
\Vert w h \Vert_{L^\infty_{x,v}} 
& \lesssim (C+1) [  |f_b|_{L^2_{\gamma_-}} + | w f_b|_{L^\infty_{\p\O,v}} ],
\\ \Vert \nabla_x \phi_h \Vert_{L^\infty_{x}} 
& \lesssim (C+1) [ |f_b|_{L^2_{\gamma_-}} + | w f_b|_{L^\infty_{\p\O,v}}].
\end{split}
\ee
\end{proposition}

\begin{remark}
The system \eqref{eqtn:h_step1}-\eqref{bdry:phi_step1} is a Vlasov-Boltzmann system with a given potential field $\nabla_x (\phi_g+\phi_E)$ and additional source terms. Therefore, the well-posedness of the problem does not require regularity estimates for $h$, and we only need to control the term $-\lambda(v\cdot \nabla_x \phi_h) e^{-\phi_E/2}\sqrt{\mu}$ on the right-hand side.    
\end{remark}

\begin{proof}
The estimates in this proposition are analogous to those obtained in Sections \ref{sec:L2Linfty_estimate} and \ref{sec:stationary_uniqueness}.
Owing to the similarity of the arguments, we omit some of the details.
We remark that the main difference between this proposition and the estimates in the previous sections lies in the presence of the coefficient $\lambda$ in \eqref{eqtn:h_step1}, and in the fact that the pair $(g, \nabla_x \phi_g)$ is externally given rather than determined self-consistently.

\smallskip

\textbf{Step 1. Construction.}
Choose a constant $0 < \lambda < 1$, we construct solutions to the steady problem \eqref{eqtn:h_step1}-\eqref{bdry:phi_step1} via the following sequences: for any $\ell \in \N$,
\be \label{eq1:steady_small_lambda_step1}
\begin{split}
v \cdot \nabla_x h^{\ell+1} 
& - \nabla_x (\phi_g + \phi_E) \cdot \nabla_{v} h^{\ell+1} 
+ \frac{v \cdot \nabla_x \phi_g}{2} h^{\ell+1} + e^{-\phi_E} \mathcal{L} h^{\ell+1}
\\ & = - \lambda (v \cdot \nabla_x \phi^\ell_h) e^{-\phi_E/2} \sqrt{\mu} + e^{-\phi_E/2} \Gamma(g, g) + S_1, 
\\ h^{\ell+1} |_{\gamma_-} & = f_b (x, v), 
\\ - \Delta \phi^\ell_h & = e^{-\phi_E/2} \int_{\R^3} h^{\ell} \sqrt{\mu} \dd v \text{ in } \O, 
\\ \phi^\ell_h & = 0 \text{ on } \p\O,
\end{split}
\ee
where the initial setting $h^0 = 0$ and $\nabla_x \phi^0_h = \mathbf{0}$. 
From the assumption on $(g, \phi_g)$, the system \eqref{eq1:steady_small_lambda_step1} is a steady Vlasov–Boltzmann equation with given potential and sources. Its well-posedness can be justified by the same argument as in \cite{EGKM}, where $Kh^\ell$ was introduced for the solution's construction. Here we omit the details of this process and keep $\mathcal{L}h^{\ell+1}$ on the left-hand side and directly estimate the $L^2-L^\infty$ control of \eqref{eq1:steady_small_lambda_step1}. 
Moreover, the characteristic associated with \eqref{eq1:steady_small_lambda_step1} satisfies 
Lemmas \ref{lemma:W3p} and \ref{lemma:phi_x_infinity} in Section \ref{sec:estimate_characteristic}.

\smallskip

\textbf{Step 2. $L^2-L^\infty$ estimate.}
$(a)$
Analogously to Lemma \ref{lemma:l2_energy}, it follows from \eqref{eq1:steady_small_lambda_step1} that
\be \label{eq2:steady_small_lambda_step1}
\begin{split}
& | h^{\ell+1} |_{L^2_{\gamma_+}}^2 + \Vert e^{-\phi_E/2}(\mathbf{I}-\mathbf{P}) h^{\ell+1} \Vert_{L^2_{x,\nu}}^2 \\
& \lesssim |f_b|_{L^2_{\gamma_-}}^2 + o(1) \Vert e^{-\phi_E/2} (\mathbf{I}-\mathbf{P}) h^{\ell+1} \Vert_{L^2_{x,\nu}}^2 + \Vert \nu^{-1/2}e^{\phi_E/4} \Gamma (g, g) \Vert_{L^2_{x,v}}^2 + \Vert S_1\Vert_{L^2_{x,v}}^2 + o (1) \Vert h^{\ell+1} \Vert_{L^2_{x,v}}^2
\\& \qquad + \underbrace{\Vert \nabla_x \phi_g \Vert_{L^\infty_{x,v}} \Vert \nu^{1/2} h^{\ell+1} \Vert_{L^2_{x,v}}^2}_{\eqref{eq2:steady_small_lambda_step1}_1} 
+ \lambda \underbrace{\Big| \iint_{\O\times \mathbb{R}^3}  (v \cdot \nabla_x \phi^\ell_h ) e^{-\phi_E/2} \sqrt{\mu} h^{\ell+1} \dd x \dd v \Big|}_{\eqref{eq2:steady_small_lambda_step1}_2}.
\end{split}
\ee

It can be verified that Lemma \ref{lemma:phi_x_infinity} applies to $(g, \nabla_x \phi_g)$ in \eqref{condition:phi_g_step1}.
Hence, we have
\be \notag
\eqref{eq2:steady_small_lambda_step1}_1 \lesssim \Vert w g \Vert_{L^\infty_{x,v}} \Vert h^{\ell+1} \Vert_{L^2_{x,\nu}}^2.
\ee
Following the estimates in Lemma \ref{lemma:l2_energy}, we have
\be \notag
\eqref{eq2:steady_small_lambda_step1}_2
\lesssim \lambda \Vert \nabla_x \phi_E \Vert_{L^\infty_x} \big( \Vert h^{\ell} \Vert_{L^2_{x,v}}^2 + \Vert h^{\ell+1} \Vert_{L^2_{x,v}}^2 \big).
\ee
These, together with \eqref{eq2:steady_small_lambda_step1}, implies that for any $\ell \in \N$,
\be \label{eq3:steady_small_lambda_step1}
\begin{split}
| h^{\ell+1} |_{L^2_{\gamma_+}}^2 + \Vert (\mathbf{I}-\mathbf{P}) h^{\ell+1} \Vert_{L^2_{x,\nu}}^2  
& \lesssim |f_b|_{L^2_{\gamma_-}}^2 + \Vert \nu^{-1/2} \Gamma (g, g) \Vert_{L^2_{x,v}}^2 + \Vert S_1\Vert_{L^2_{x,v}}^2 + o (1) \Vert h^{\ell+1} \Vert_{L^2_{x,\nu}}^2 
\\& \qquad + \Vert w g \Vert_{L^\infty_{x,v}} \Vert h^{\ell+1} \Vert_{L^2_{x,\nu}}^2 + \lambda \Vert \nabla_x \phi_E \Vert_{L^\infty_x} \big( \Vert h^{\ell} \Vert_{L^2_{x,v}}^2 + \Vert h^{\ell+1} \Vert_{L^2_{x,v}}^2 \big).
\end{split}
\ee

Analogously to Lemma \ref{lemma:macro_l2}, given a test function $\psi$, the weak formulation for \eqref{eq1:steady_small_lambda_step1} is
\be \label{eq4:steady_small_lambda_step1}
\begin{split}
& - \iint_{\O\times \mathbb{R}^3} \mathbf{P} h^{\ell+1} (v\cdot \nabla_x \psi) \dd x \dd v - \iint_{\O\times \mathbb{R}^3} (\mathbf{I}-\mathbf{P}) h^{\ell+1} (v \cdot \nabla_x \psi) \dd x \dd v + \int_{\gamma} h^{\ell+1} \psi \dd \gamma 
\\& + \iint_{\O\times \mathbb{R}^3} \sqrt{\mu} h^{\ell+1} \nabla_x \phi_g \cdot \nabla_{v} \big[\frac{1}{\sqrt{\mu}} \psi \big] \dd x \dd v + \iint_{\O\times \mathbb{R}^3} h^{\ell+1} \nabla_x \phi_E \cdot \nabla_{v} \psi \dd x \dd v 
\\& + \iint_{\O\times \mathbb{R}^3} e^{-\phi_E}\mathcal{L} h^{\ell+1} \psi \dd x \dd v 
\\& = \underbrace{- \lambda \iint_{\O\times \mathbb{R}^3} (v \cdot \nabla_x \phi^\ell_h) e^{-\phi_E/2} \sqrt{\mu} \psi \dd x \dd v}_{\eqref{eq4:steady_small_lambda_step1}_1} + \iint_{\O\times \mathbb{R}^3} [e^{-\phi_E/2} \Gamma (g, g) + S_1] \psi \dd x \dd v.  
\end{split}
\ee

Similarly to the proof of Lemmas \ref{lemma:macro_l2} and \ref{lemma:Unif_h_L2}, we write
\[
\mathbf{P} h^{\ell+1} = \big( a(x) + \mathbf{b} (x) \cdot v + c(x) \frac{|v|^2-3}{2} \big) \sqrt{\mu}.
\]
The estimate on $c(x)$ follow from Lemma \ref{lemma:macro_l2} as follows:
\be \notag
\begin{split}
\Vert c \Vert_{L^2_{x}}^2 
& \lesssim \Vert (\mathbf{I}-\mathbf{P}) h^{\ell+1} \Vert_{L^2_{x,v}}^2 + \big( \Vert w g \Vert_{L^\infty_{x,v}} + \Vert \nabla_x \phi_E \Vert_{L^\infty_{x}} \big) \Vert h^{\ell+1} \Vert_{L^2_{x,v}}^2 
\\& \qquad + |h^{\ell+1}|_{L^2_{\gamma_+}}^2 + |f_b|_{L^2_{\gamma_-}}^2 + \Vert \nu^{-1/2} \Gamma (g, g) \Vert_{L^2_{x,v}}^2 + \Vert S_1\Vert_{L^2_{x,v}}^2.
\end{split}
\ee
The estimate on $\mathbf{b} (x)$ follow from Lemma \ref{lemma:macro_l2} as follows: for $i = 1, 2, 3$,
\be \notag
\begin{split}
\Vert b_i \Vert_{L^2_{x}}^2 
& \lesssim \Vert (\mathbf{I}-\mathbf{P}) h^{\ell+1} \Vert_{L^2_{x,v}}^2 + \big( \Vert w g \Vert_{L^\infty_{x,v}} + \Vert \nabla_x \phi_E \Vert_{L^\infty_{x}} \big) \Vert h^{\ell+1} \Vert_{L^2_{x,v}}^2 
\\& \qquad + |h^{\ell+1}|_{L^2_{\gamma_+}}^2 + |f_b|_{L^2_{\gamma_-}}^2 + \Vert \nu^{-1/2} \Gamma (g, g) \Vert_{L^2_{x,v}}^2 + \Vert S_1\Vert_{L^2_{x,v}}^2.
\end{split}
\ee
The estimate on $a (x)$ follow from Lemma \ref{lemma:macro_l2} except the term $\eqref{eq4:steady_small_lambda_step1}_1$. 
We choose
\be \notag
\begin{cases}
& \psi_a = - (|v|^2-10)\sqrt{\mu}(v \cdot \nabla_x  \phi_a), \\
& - \Delta_x \phi_a = a(x)  \text{ in } \O, \ \phi_a = 0 \text{ on } \p\O.
\end{cases}
\ee
Using the fact that $|e^{-\phi_E}-1|\ll 1$, we get
\be \notag
\begin{split}
| \eqref{eq4:steady_small_lambda_step1}_1 |
& = \big| - \lambda \iint_{\O\times \mathbb{R}^3} (v \cdot \nabla_x \phi^{\ell}_h)e^{-\phi_E/2} \sqrt{\mu} \psi \dd x \dd v \big|
\\& = \big| - 5 \lambda \int_{\O} \nabla_x \phi_a \cdot \nabla_x \phi^{\ell}_h e^{-\phi_E/2} \dd x \big| 
\lesssim \lambda \Vert \nabla_x \phi_a \Vert_{L^2_x} \Vert \nabla_x \phi^{\ell}_h e^{-\phi_E/2} \Vert_{L^2_x}
\lesssim \lambda \Vert a \Vert_{L^2_x} \Vert h^{\ell} \Vert_{L^2_{x, v}}.
\end{split}
\ee
This further implies that
\be \notag
\begin{split}
\Vert a \Vert_{L^2_{x}}^2 
& \lesssim \Vert (\mathbf{I}-\mathbf{P}) h^{\ell+1} \Vert_{L^2_{x,v}}^2 + \big( \Vert w g \Vert_{L^\infty_{x,v}} + \Vert \nabla_x \phi_E \Vert_{L^\infty_{x}} \big) \Vert h^{\ell+1} \Vert_{L^2_{x,v}}^2 + |h^{\ell+1}|_{L^2_{\gamma_+}}^2 + |f_b|_{L^2_{\gamma_-}}^2 
\\& \qquad + \lambda^2 \Vert h^{\ell} \Vert_{L^2_{x, v}}^2 + \Vert \nu^{-1/2} \Gamma (g, g) \Vert_{L^2_{x,v}}^2 + \Vert S_1\Vert_{L^2_{x,v}}^2.
\end{split}
\ee
Hence, we obtain that for any $\ell \in \N$,
\be \label{eq5:steady_small_lambda_step1}
\begin{split}
\Vert \mathbf{P} h^{\ell+1} \Vert_{L^2_{x,\nu}}^2
& \lesssim \Vert (\mathbf{I}-\mathbf{P}) h^{\ell+1} \Vert_{L^2_{x,\nu}}^2 + | h^{\ell+1} |_{L^2_{\gamma_+}}^2 + |f_b|_{L^2_{\gamma_-}}^2 + \Vert \nu^{-1/2} \Gamma (g, g) \Vert_{L^2_{x,v}}^2   + \Vert S_1\Vert_{L^2_{x,v}}^2
\\& \qquad + \lambda^2 \Vert h^{\ell} \Vert_{L^2_{x, v}}^2 + \big( \Vert w g \Vert_{L^\infty_{x,v}} + \Vert \nabla_x \phi_E \Vert_{L^\infty_{x}} \big) \Vert h^{\ell+1} \Vert_{L^2_{x,v}}^2.
\end{split}
\ee
Combining \eqref{eq3:steady_small_lambda_step1} with \eqref{eq5:steady_small_lambda_step1}, we conclude that for any $\ell \in \N$,
\be \label{eq6:steady_small_lambda_step1}
\begin{split}
& | h^{\ell+1} |_{L^2_{\gamma_+}} + \Vert h^{\ell+1} \Vert_{L^2_{x,\nu}}
\lesssim |f_b|_{L^2_{\gamma_-}} + \Vert \nu^{-1/2} \Gamma (g, g) \Vert_{L^2_{x,v}}+ \Vert S_1\Vert_{L^2_{x,v}} + \Vert w g \Vert_{L^\infty_{x,v}}^{1/2} \Vert h^{\ell+1} \Vert_{L^2_{x,\nu}}  
\\& \qquad + \sqrt{\lambda} \Vert \nabla_x \phi_E \Vert^{1/2}_{L^\infty_x} \Vert h^{\ell} \Vert_{L^2_{x, v}} + \lambda \Vert h^{\ell} \Vert_{L^2_{x, v}} + \big( \Vert w g \Vert_{L^\infty_{x,v}}^{1/2} + (1 + \sqrt{\lambda}) \Vert \nabla_x \phi_E \Vert_{L^\infty_x}^{1/2} \big) \Vert h^{\ell+1} \Vert_{L^2_{x,v}}.
\end{split}
\ee

It can be verified that Lemma \ref{lemma:gamma} applies to $(g, \nabla_x \phi_g )$.
Using the key estimate \eqref{eq6:steady_small_lambda_step1}, we obtain that for any $\ell \in \N$,
\be \notag
\begin{split}
| h^{\ell+1} |_{L^2_{\gamma_+}} + \Vert h^{\ell+1} \Vert_{L^2_{x,\nu}}
& \lesssim |f_b|_{L^2_{\gamma_-}} + \Vert w g \Vert_{L^\infty_{x,v}}^2 + \Vert S_1\Vert_{L^2_{x,v}} + \Vert w g \Vert_{L^\infty_{x,v}}^{1/2} \Vert h^{\ell+1} \Vert_{L^2_{x,\nu}} + \sqrt{\lambda} \Vert \nabla_x \phi_E \Vert^{1/2}_{L^\infty_x} \Vert h^{\ell} \Vert_{L^2_{x, v}} 
\\& \qquad + \lambda \Vert h^{\ell} \Vert_{L^2_{x,v}} + \big( \Vert w g \Vert_{L^\infty_{x,v}}^{1/2} + (1 + \sqrt{\lambda}) \Vert \nabla_x \phi_E \Vert_{L^\infty_x}^{1/2} \big) \Vert h^{\ell+1} \Vert_{L^2_{x,v}}.
\end{split}
\ee
From the initial setting $h^0 = 0$ and $\nabla_x \phi^0_h = \mathbf{0}$ and the assumption \eqref{condition:g_step1} on $g (x,v)$, there exists a sufficiently small $0 < \lambda = \lambda (\Omega) \ll 1$ such that
\be \notag
\begin{split}
| h^{\ell+1} |_{L^2_{\gamma_+}} + \Vert h^{\ell+1} \Vert_{L^2_{x,\nu}}
& \lesssim |f_b|_{L^2_{\gamma_-}} + \Vert w g \Vert_{L^\infty_{x,v}}^2 + \Vert S_1\Vert_{L^2_{x,v}}+ (\lambda + \Vert \nabla_x \phi_E \Vert^{1/2}_{L^\infty_x}) \Vert h^{\ell} \Vert_{L^2_{x, v}}.
\end{split}
\ee
This, together with the assumptions \eqref{condition:g_step1} and \eqref{S1_assumption}, it follows by induction that for any $\ell \in \N$,
\be \label{eq10:steady_small_lambda_step1}
\begin{split}
& | h^{\ell+1} |_{L^2_{\gamma_+}} + \Vert h^{\ell+1} \Vert_{L^2_{x,\nu}}
\lesssim (C+1) [ |f_b|_{L^2_{\gamma_-}} + | w f_b|_{L^\infty_{\p\O,v}}].
\end{split} 
\ee

\smallskip

$(b)$
Analogously to Proposition \ref{prop:wf_Linfty}, we rewrite \eqref{eq1:steady_small_lambda_step1} as
\be \label{eq7:steady_small_lambda_step1}
\begin{split}
& v \cdot \nabla_x h^{\ell+1} - \nabla_x (\phi_g + \phi_E) \cdot \nabla_{v} h^{\ell+1} + \Big( e^{-\phi_E}\nu(v) + \frac{v\cdot \nabla_x \phi_g}{2} \Big) h^{\ell+1} 
\\& = e^{-\phi_E} K h^{\ell+1} - \lambda (v \cdot \nabla_x \phi^\ell_h ) e^{-\phi_E/2}\sqrt{\mu} + e^{-\phi_E/2} \Gamma (g, g) + S_1. 
\end{split}
\ee
From the assumption on $(g, \phi_g)$ and $\nu(v) \geq \nu_0 > 0$, we get
\be \notag
\frac{1}{2} \nu(v) \geq e^{-\phi_E}\nu(v) + \frac{v \cdot \nabla_x \phi_g}{2} \geq \frac{1}{2} \nu(v) > \frac{\nu_0}{2}.
\ee
Define that $w(v) = e^{\theta |v|^2}$ with $0 < \theta < \frac{1}{4}$. We now pick $t > 0$ such that 
\[
\nu_0 t \geq 10
\ \text{ and } \
(1 + t) e^{ - \frac{\nu_0}{4} t} \leq \frac{1}{4}.
\]
Applying the method of characteristics on \eqref{eq7:steady_small_lambda_step1} and following Proposition \ref{prop:wf_Linfty}, we obtain that
\begin{align}
| w (v) h^{\ell+1} (x,v) | 
\leq & \mathbf{1}_{\tb(x,v)\geq t} e^{ - \int^t_0 \frac{\nu(V(s))}{4} \dd s} \big| w(V(0)) h^{\ell+1} (X(0),V(0)) \big|
\label{eq8:steady_small_lambda_step1_1} \\
& + \mathbf{1}_{\tb(x,v)<t} e^{-\int^t_{t- \tb} \frac{\nu(V(s))}{4}\dd s} \big| w (\vb) f_b (\xb, \vb) \big| \label{eq8:steady_small_lambda_step1_2} \\
& + \Big| \int^t_{\max\{0,t-\tb\}} e^{-\int^t_s \frac{\nu(V(\tau))}{4}\dd \tau} \int_{\mathbb{R}^3} \mathbf{k}(V(s),u) \frac{w(V(s))}{w(u)} w(u) h^{\ell+1} (X(s),u) \dd u \dd s \Big| 
\label{eq8:steady_small_lambda_step1_3} \\
& + \lambda \int^t_{\max\{0,t-\tb\}} e^{-\int^t_s \frac{\nu(V(\tau))}{4}\dd \tau} \big| V(s) \cdot \nabla_x \phi^\ell_h (X(s)) \big| w(V(s))\sqrt{\mu(V(s))} \dd s 
\label{eq8:steady_small_lambda_step1_4} \\
& + \int^t_{\max\{0,t-\tb\}} e^{-\int^t_s \frac{\nu(V(\tau))}{4}\dd \tau} w(V(s)) \big| \Gamma (g (X(s),V(s)), g (X(s),V(s)) ) \big| \dd s, 
\label{eq8:steady_small_lambda_step1_5} \\
& + \int^t_{\max\{0,t-\tb\}} e^{-\int^t_s \frac{\nu(V(\tau))}{4}\dd \tau} w(V(s)) | S_1 (X(s), V(s)) | \dd s, 
\label{eq8:steady_small_lambda_step1_6}
\end{align}
where $(X(s), V(s)) := (X(s;t,x,v), V(s;t,x,v))$ with $\max\{0,t-\tb\} \leq s \leq t$.

Similarly to Proposition \ref{prop:wf_Linfty}, we compute that
\be \notag
\begin{split}
& \eqref{eq8:steady_small_lambda_step1_1} + \eqref{eq8:steady_small_lambda_step1_2} + 
\eqref{eq8:steady_small_lambda_step1_3} + \eqref{eq8:steady_small_lambda_step1_5}
\\& \lesssim \frac{1}{4} \Vert w h^{\ell+1} \Vert_{L^\infty_{x,v}} + | w f_b|_{L^\infty_{\p\O,v}} + \frac{\lambda}{4} \Vert w h^{\ell} \Vert_{L^\infty_{x,v}} + \lambda \Vert h^{\ell} \Vert_{L^2_{x,v}} + \Vert \nu^{-1} w \Gamma (g, g) \Vert_{L^\infty_{x,v}} + \Vert h^{\ell+1} \Vert_{L^2_{x,v}}.
\end{split}
\ee
For \eqref{eq8:steady_small_lambda_step1_4}, from \eqref{phi_f_C^1_2} in Lemma \ref{lemma:phi_x_infinity} we have
\be \notag
\eqref{eq8:steady_small_lambda_step1_4}
\leq \lambda \Vert \nabla_x \phi^\ell_h \Vert_{L^\infty_x} \int^t_{\max\{0,t-\tb\}} e^{- \nu_0 (t-s)/4} \dd s
\lesssim \frac{\lambda}{4} \Vert w h^{\ell} \Vert_{L^\infty_{x,v}} + \lambda \Vert h^{\ell} \Vert_{L^2_{x,v}}.
\ee
For \eqref{eq8:steady_small_lambda_step1_6}, from the assumption \eqref{S1_assumption} we have
\be \notag
\eqref{eq8:steady_small_lambda_step1_6}
\leq \Vert w S_1 \Vert_{L^\infty_{x,v}} \int^t_{\max\{0,t-\tb\}} e^{- \nu_0 (t-s)/4} \dd s
\lesssim C [ |f_b|_{L^2_{\gamma_-}} + |w f_b|_{L^\infty_{\p\O,v}} ].
\ee
Collecting the above estimates, we conclude that for any $\ell \in \N$,
\be \label{eq9:steady_small_lambda_step1}
\begin{split}
& \Vert w h^{\ell+1} \Vert_{L^\infty_{x,v}} 
\\& \lesssim \frac{\lambda}{4} \Vert w h^{\ell} \Vert_{L^\infty_{x,v}} + \lambda \Vert h^{\ell} \Vert_{L^2_{x,v}} + \Vert \nu^{-1} w \Gamma (g, g) \Vert_{L^\infty_{x,v}} + (C+1) [ |f_b|_{L^2_{\gamma_-}} + | w f_b|_{L^\infty_{\p\O,v}}] + \Vert h^{\ell+1} \Vert_{L^2_{x,v}}.
\end{split}
\ee
Applying Lemma \ref{lemma:gamma} to $(g, \nabla_x \phi_g )$, and using \eqref{eq10:steady_small_lambda_step1}, \eqref{eq9:steady_small_lambda_step1} and Lemma \ref{lemma:phi_x_infinity}, it follows by induction that for any \(\ell \in \mathbb{N}\),
\be \label{eq11:steady_small_lambda_step1}
\begin{split}
\Vert w h^{\ell+1} \Vert_{L^\infty_{x,v}} 
& \lesssim (C+1) [ |f_b|_{L^2_{\gamma_-}} + | w f_b|_{L^\infty_{\p\O,v}}],
\\ \Vert \nabla_x \phi^\ell_h \Vert_{L^\infty_{x}} 
& \lesssim (C+1) [ |f_b|_{L^2_{\gamma_-}} + | w f_b|_{L^\infty_{\p\O,v}}]. 
\end{split} 
\ee

\smallskip

\textbf{Step 3. Existence.}
$(a)$
From \eqref{eq1:steady_small_lambda_step1}, $h^{\ell+1} - h^{\ell}$ satisfies that for any $\ell \in \N$,
\be \label{eq14:steady_small_lambda_step1}
\begin{split}
v \cdot \nabla_x (h^{\ell+1} - h^{\ell}) & - \nabla_x (\phi_g + \phi_E) \cdot \nabla_{v}(h^{\ell+1} - h^{\ell}) + \frac{v}{2} \cdot \nabla_x \phi_g (h^{\ell+1} - h^{\ell}) + e^{-\phi_E} \mathcal{L}(h^{\ell+1} - h^{\ell}) 
\\& = \- \lambda v \cdot \nabla_x (\phi^{\ell}_h - \phi^{\ell - 1}_h) e^{-\phi_E/2} \sqrt{\mu},
\\ h^{\ell+1} - h^{\ell} |_{\gamma_-} & = 0,
\\ - \Delta ( \phi^\ell_h - \phi^{\ell - 1}_h ) & = e^{-\phi_E/2} \int_{\R^3} ( h^{\ell} - h^{\ell - 1} )\sqrt{\mu} \dd v \text{ in } \O, 
\\ \phi^\ell_h - \phi^{\ell - 1}_h & = 0 \text{ on } \p\O.
\end{split}
\ee

From \eqref{eq14:steady_small_lambda_step1}, the $L^2_{x,v}$ energy estimate yields
\be \label{eq15:steady_small_lambda_step1}
\begin{split}
& | h^{\ell+1} - h^{\ell} |_{L^2_{\gamma_+}}^2 + \Vert e^{-\phi_E/2}(\mathbf{I}-\mathbf{P}) (h^{\ell+1} - h^{\ell}) \Vert_{L^2_{x,\nu}}^2 
\\& \lesssim \Vert \nabla_x \phi_g \Vert_{L^\infty_{x}} \Vert h^{\ell+1} - h^{\ell} \Vert_{L^2_{x,\nu}}^2 + 
\Big| \underbrace{ \iint_{\O\times \mathbb{R}^3} \lambda v \cdot \nabla_x (\phi^{\ell}_h - \phi^{\ell - 1}_h) e^{-\phi_E/2}\sqrt{\mu} (h^{\ell+1} - h^{\ell}) \dd x \dd v 
}_{\eqref{eq15:steady_small_lambda_step1}^*} \Big|.
\end{split}
\ee
Similarly to Lemma \ref{lemma:l2_energy}, from \eqref{eq14:steady_small_lambda_step1} we have
\begin{equation} \notag
\begin{split}
| \eqref{eq15:steady_small_lambda_step1}^* |
& = \big| \lambda \iint_{\O\times \mathbb{R}^3} (\phi^{\ell}_h - \phi^{\ell - 1}_h) v \cdot \nabla_x \phi_E e^{-\phi_E/2} \sqrt{\mu} (h^{\ell+1} - h^{\ell}) \dd x \dd v \big|
\\& \lesssim \lambda \Vert \nabla_x \phi_E \Vert_{L^\infty_x} \big( \Vert h^{\ell+1} - h^{\ell} \Vert_{L^2_{x,v}}^2 + \Vert h^{\ell} - h^{\ell - 1} \Vert_{L^2_{x,\nu}}^2 \big).
\end{split}
\end{equation}
From \eqref{eq15:steady_small_lambda_step1}, we get
\begin{equation} \label{eq17:steady_small_lambda_step1}
\begin{split}
& | h^{\ell+1} - h^{\ell} |_{L^2_{\gamma_+}}^2 + \Vert (\mathbf{I}-\mathbf{P}) (h^{\ell+1} - h^{\ell}) \Vert_{L^2_{x,\nu}}^2
\\& \lesssim \Vert \nabla_x \phi_g \Vert_{L^\infty_{x}} \Vert h^{\ell+1} - h^{\ell} \Vert_{L^2_{x,\nu}}^2 + \lambda \Vert \nabla_x \phi_E \Vert_{L^\infty_x} \big( \Vert h^{\ell+1} - h^{\ell} \Vert_{L^2_{x,v}}^2 + \Vert h^{\ell} - h^{\ell - 1} \Vert_{L^2_{x,\nu}}^2 \big).
\end{split}
\end{equation}

On the other hand, following part $(a)$ of the regularity estimates \eqref{eq4:steady_small_lambda_step1} and \eqref{eq5:steady_small_lambda_step1}, we obtain that
\begin{equation} \label{eq18:steady_small_lambda_step1}
\begin{split}
\Vert \mathbf{P} (h^{\ell+1} - h^{\ell}) \Vert_{L^2_{x,v}}^2
& \lesssim \Vert (\mathbf{I}-\mathbf{P}) (h^{\ell+1} - h^{\ell}) \Vert_{L^2_{x,\nu}}^2 + | h^{\ell+1} - h^{\ell} |_{L^2_{\gamma_+}}^2 + \lambda^2 \Vert h^{\ell} - h^{\ell-1} \Vert_{L^2_{x,v}}^2
\\& \qquad + \big( \Vert w g \Vert_{L^\infty_{x,v}} + \Vert \nabla_x \phi_E \Vert_{L^\infty_x} \big) \Vert h^{\ell+1} - h^{\ell} \Vert_{L^2_{x,v}}^2.
\end{split}
\end{equation}
By taking $\eqref{eq17:steady_small_lambda_step1} \times 2 + \eqref{eq18:steady_small_lambda_step1}$, $(h^{\ell+1}, \nabla_x \phi^\ell_h)$ satisfies that for any $\ell \in \N$,
\be \notag
\begin{split}
& | h^{\ell+1} - h^{\ell} |_{L^2_{\gamma_+}}^2 + \Vert h^{\ell+1} - h^{\ell} \Vert_{L^2_{x,\nu}}^2
\\& \lesssim \big( \Vert \nabla_x \phi_g \Vert_{L^\infty_{x}} + \Vert w g \Vert_{L^\infty_{x,v}} + \lambda \Vert \nabla_x \phi_E \Vert_{L^\infty_x} \big) \Vert h^{\ell+1} - h^{\ell} \Vert_{L^2_{x,v}}^2 + \big( \lambda \Vert \nabla_x \phi_E \Vert_{L^\infty_x} + \lambda^2 \big) \Vert h^{\ell} - h^{\ell-1} \Vert_{L^2_{x,v}}^2.
\end{split}
\ee
This, together with \eqref{eq10:steady_small_lambda_step1} and the elliptic regularity estimate, implies that  $\{ h^{\ell} \}^{\infty}_{\ell = 0}$ and $\{ \nabla_x \phi^\ell_h \}^{\infty}_{\ell = 0}$ are Cauchy sequences in $L^2_{x,\nu} (\O \times \R^3) \cap L^2 (\gamma_+)$ and $L^2 (\O)$,  respectively.

\smallskip

$(b)$ 
Since $\{ h^{\ell} \}^{\infty}_{\ell = 0}$ and $\{ \nabla_x \phi^\ell_h \}^{\infty}_{\ell = 0}$ from the construction \eqref{eq1:steady_small_lambda_step1} are Cauchy sequences in $L^2_{x,\nu} (\O \times \R^3) \cap L^2 (\gamma_+)$ and $L^2 (\O)$, respectively, there exist
\be \notag
h (x, v) \in L^2_{x,\nu} (\O \times \R^3) \cap L^2 (\gamma_+)
\ \text{ and } \
\nabla_x \phi_h (x) \in L^2 (\O)
\text{ with }
\phi_h = 0 \text{ on } \p\O,
\ee
such that
\be \label{eq22:steady_small_lambda_step1}
\begin{split}
h^{\ell} 
& \to h 
\ \text{ in } L^2_{x,\nu} (\O \times \R^3) \cap L^2 (\gamma_+)
\ \text{ as $\ell \to \infty$},
\\ \nabla_x \phi^{\ell}_h (x) 
& \to \nabla_x \phi_h (x)
\ \text{ in } L^2 (\O)
\ \text{ as $\ell \to \infty$}.
\end{split}
\ee
Combining the uniform-in-$\ell$ estimates for $\{ h^{\ell} \}^{\infty}_{\ell = 0}$ in \eqref{eq10:steady_small_lambda_step1}, we conclude that
\be \notag
| h |_{L^2_{\gamma_+}} + \Vert h \Vert_{L^2_{x,\nu}}
\lesssim (C+1) [ |f_b|_{L^2_{\gamma_-}} + | w f_b|_{L^\infty_{\p\O,v}}].
\ee
Moreover, the uniform-in-$\ell$ estimates for $\{ h^{\ell} \}^{\infty}_{\ell = 0}$ in \eqref{eq11:steady_small_lambda_step1} implies that
\be \label{eq25:steady_small_lambda_step1}
\begin{split}
\Vert w h \Vert_{L^\infty_{x,v}} 
& \lesssim (C+1) [  |f_b|_{L^2_{\gamma_-}} + | w f_b|_{L^\infty_{\p\O,v}} ],
\\ \Vert \nabla_x \phi_h \Vert_{L^\infty_{x}} 
& \lesssim (C+1) [ |f_b|_{L^2_{\gamma_-}} + | w f_b|_{L^\infty_{\p\O,v}}].
\end{split}
\ee

\smallskip

$(c)$ 
It remains to show that \((h, \nabla_x \phi_h)\) obtained in \eqref{eq22:steady_small_lambda_step1} is a solution of \eqref{eqtn:h_step1}–\eqref{bdry:phi_step1} in the sense of Definition \ref{def:weak_sol}.
From the construction \eqref{eq1:steady_small_lambda_step1}, for any $\ell \in \N$ and $\psi \in  C^\infty_{c} (\bar \O \times \R^3)$,
\Be \label{eq23:steady_small_lambda_step1}
\begin{split}
& \int_{ \gamma_+ } h^{\ell + 1} \psi \dd \gamma - \int_{ \gamma_- } f_b \psi \dd \gamma - \iint_{\O \times \R^3} h^{\ell + 1} v \cdot \nabla_x \psi \dd v \dd x
\\& + \iint_{\O \times \R^3} h^{\ell + 1} \nabla_x ( \phi_g + \phi_E) \cdot \nabla_v \psi \dd v \dd x + \iint_{\O \times \R^3} \frac{ v \cdot \nabla_x \phi_g }{2} h^{\ell + 1} \psi \dd v \dd x 
+ \iint_{\O \times \R^3} e^{-\phi_E} \mathcal{L} h^{\ell + 1} \psi \dd v \dd x
\\& = \iint_{\O \times \R^3} - \lambda ( v \cdot \nabla_x \phi^{\ell}_{h} )e^{-\phi_E/2}\sqrt{\mu} \psi \dd v \dd x + \iint_{\O \times \R^3} e^{-\phi_E/2} \Gamma (g, g) \psi \dd v \dd x +  \iint_{\O \times \R^3} S_1 \psi \dd v \dd x.
\end{split}
\Ee
Moreover, for any $\ell \in \N$ and $\varphi \in H^1_0 (\O) \cap C^\infty_c (\bar \O)$, 
\Be \label{eq24:steady_small_lambda_step1}
\int_{\O} \nabla_x \phi^{\ell}_h \cdot \nabla_x \varphi \dd x 
= \int_{\O} e^{-\phi_E/2} \big( \int_{\mathbb{R}^3} h^{\ell} \sqrt{\mu} \dd v \big) \varphi \dd x.
\Ee

For the first line in \eqref{eq23:steady_small_lambda_step1}, from $\psi \in  C^\infty_{c} (\bar \O \times \R^3)$ and $L^{2}$ convergence $h^{\ell} \to h$ in \eqref{eq22:steady_small_lambda_step1}, we derive
\Be \notag
\begin{split}
\int_{ \gamma_+ } h^{\ell + 1} \psi \dd \gamma
& \rightarrow 
\int_{ \gamma_+ } h \psi \dd \gamma
\ \text{ as } \ \ell \to \infty,
\\ \iint_{\O \times \R^3} h^{\ell + 1} v \cdot \nabla_x \psi \dd v \dd x
& \rightarrow 
\iint_{\O \times \R^3} h v \cdot \nabla_x \psi \dd v \dd x
\ \text{ as } \ \ell \to \infty.
\end{split}
\Ee
For the second line in \eqref{eq23:steady_small_lambda_step1}, using the $L^{2}$ convergence $\nabla_x \phi^{\ell}_h \to \nabla_x \phi_h$ in \eqref{eq22:steady_small_lambda_step1}, we further get
\Be \notag
\begin{split}
\iint_{\O \times \R^3} h^{\ell + 1} \nabla_x (\phi_g + \phi_E) \cdot \nabla_v \psi  \dd v \dd x
& \rightarrow 
\iint_{\O \times \R^3} h \nabla_x (\phi_g + \phi_E) \cdot \nabla_v \psi  \dd v \dd x
\ \text{ as } \ \ell \to \infty,
\\ \iint_{\O \times \R^3} \frac{ v \cdot \nabla_x \phi_{g} }{2} h^{\ell + 1} \psi \dd v \dd x
& \rightarrow 
\iint_{\O \times \R^3} \frac{ v \cdot \nabla_x \phi_{g} }{2} h \psi \dd v \dd x
\ \text{ as } \ \ell \to \infty.
\end{split}
\Ee
Moreover, the property that $K$ is bounded on $L^2_v$ from Lemma \ref{lemma:k_nu} shows that
\Be \notag
\iint_{\O \times \R^3} e^{-\phi_E} \mathcal{L} h^{\ell + 1} \psi \dd v \dd x
\rightarrow 
\iint_{\O \times \R^3} e^{-\phi_E} \mathcal{L} h \psi \dd v \dd x
\ \text{ as } \ \ell \to \infty.
\Ee
For the third line in \eqref{eq23:steady_small_lambda_step1}, the $L^{\infty}$ estimates on $\{ h^{\ell} \}^{\infty}_{\ell = 0}$ in \eqref{eq11:steady_small_lambda_step1}, \eqref{eq25:steady_small_lambda_step1} and Lemma \ref{lemma:gamma} imply that
\Be \notag
\begin{split}
\iint_{\O \times \R^3} - \lambda ( v \cdot \nabla_x \phi^{\ell+1}_{h} )e^{-\phi_E/2}\sqrt{\mu} \psi \dd v \dd x
& \rightarrow 
\iint_{\O \times \R^3} - \lambda ( v \cdot \nabla_x \phi_{h} )e^{-\phi_E/2}\sqrt{\mu} \psi \dd v \dd x
\ \text{ as } \ \ell \to \infty.
\end{split}
\Ee 
From the convergence of every term in \eqref{eq23:steady_small_lambda_step1}, $(h, \nabla_x \phi_h)$ solves \eqref{eqtn:h_step1} and \eqref{bdry:h_step1}.

For \eqref{eq24:steady_small_lambda_step1}, using the $L^{2}$ convergence $h^{\ell} \to h$ and $\nabla_x \phi^{\ell}_h \to \nabla_x \phi_h$ in \eqref{eq22:steady_small_lambda_step1}, together with $\varphi \in H^1_0 (\O) \cap C^\infty_c (\bar \O)$, we obtain 
\Be \notag
\begin{split}
\int_{\O} \nabla_x \phi^{\ell}_h \cdot \nabla_x \varphi \dd x
& \rightarrow 
\int_{\O} \nabla_x \phi_h \cdot \nabla_x \varphi \dd x
\ \text{ as } \ \ell \to \infty,
\\ \int_{\O} e^{-\phi_E/2} \big( \int_{\mathbb{R}^3} h^{\ell} \sqrt{\mu} \dd v \big) \varphi \dd x
& \rightarrow 
\int_{\O} e^{-\phi_E/2} \big( \int_{\mathbb{R}^3} h \sqrt{\mu} \dd v \big) \varphi \dd x
\ \text{ as } \ \ell \to \infty,
\end{split}
\Ee
and thus $(h, \nabla_x \phi_h)$ solves \eqref{eqtn:phi_step1} and \eqref{bdry:phi_step1}.
Therefore, $(h, \nabla_x \phi_h)$ is a weak solution of \eqref{eqtn:h_step1}–\eqref{bdry:phi_step1}.

\smallskip

\textbf{Step 4. Uniqueness.}
Let $(h_1, \nabla_x \phi_{h_1})$ and $(h_2, \nabla_x \phi_{h_2})$ be two solutions of \eqref{eqtn:h_step1}-\eqref{bdry:phi_step1} that satisfy \eqref{regularity:h_L2_step1}.
Therefore, the equation of $h_1 - h_2$ satisfies that  
\be \label{eq19:steady_small_lambda_step1}
\begin{split}
v \cdot \nabla_x (h_1 - h_2) & - \nabla_x (\phi_g + \phi_E) \cdot \nabla_{v} (h_1 - h_2) + \frac{v}{2} \cdot \nabla_x \phi_g (h_1 - h_2) + e^{-\phi_E} \mathcal{L} (h_1 - h_2) 
\\& = \- \lambda v \cdot \nabla_x (\phi_{h_1} - \phi_{h_2}) e^{-\phi_E/2} \sqrt{\mu},
\\ h_1 - h_2 |_{\gamma_-} & = 0,
\\ - \Delta ( \phi_{h_1} - \phi_{h_2} ) & = e^{-\phi_E/2} \int_{\R^3} ( h_1 - h_2 )\sqrt{\mu} \dd v \text{ in } \O, 
\\ \phi_{h_1} - \phi_{h_2} & = 0 \text{ on } \p\O.
\end{split}
\ee
Analogously to \eqref{eq15:steady_small_lambda_step1}-\eqref{eq17:steady_small_lambda_step1}, from \eqref{eq19:steady_small_lambda_step1} we have the following $L^2_{x,v}$ energy estimate:
\be \label{eq20:steady_small_lambda_step1}
\begin{split}
& | h_1 - h_2 |_{L^2_{\gamma_+}}^2 + \Vert e^{-\phi_E/2}(\mathbf{I}-\mathbf{P}) (h_1 - h_2) \Vert_{L^2_{x,\nu}}^2 
\\& \lesssim \Vert \nabla_x \phi_g \Vert_{L^\infty_{x}} \Vert h_1 - h_2 \Vert_{L^2_{x,\nu}}^2 + \Big| \iint_{\O\times \mathbb{R}^3} \lambda v \cdot \nabla_x (\phi_{h_1} - \phi_{h_2} ) e^{-\phi_E/2}\sqrt{\mu} (h_1 - h_2) \dd x \dd v \Big|
\\& \lesssim \Vert \nabla_x \phi_g \Vert_{L^\infty_{x}} \Vert h_1 - h_2 \Vert_{L^2_{x,\nu}}^2 + \Vert \nabla_x \phi_E \Vert_{L^\infty_x} \Vert h_1 - h_2 \Vert_{L^2_{x,v}}^2.
\end{split}
\ee
In addition, following the macroscopic estimates in Proposition \ref{prop:stationary_uniqueness} and Lemma \ref{lemma:macro_l2}, we obtain that
\begin{equation} \label{eq21:steady_small_lambda_step1}
\begin{split}
\Vert \mathbf{P} (h_1 - h_2) \Vert_{L^2_{x,v}}^2
& \lesssim \Vert (\mathbf{I}-\mathbf{P}) (h_1 - h_2) \Vert_{L^2_{x,\nu}}^2 + | h_1 - h_2 |_{L^2_{\gamma_+}}^2 + \delta_0 \Vert h_1 - h_2 \Vert_{L^2_{x,v}}^2
\\& \qquad + \big( \Vert w g \Vert_{L^\infty_{x,v}} + \Vert \nabla_x \phi_E \Vert_{L^\infty_x} \big) \Vert h_1 - h_2 \Vert_{L^2_{x,v}}^2,
\end{split}
\end{equation}
where $|e^{-\phi_E/2}-1| \leq \delta_0 \ll 1$.
By taking $\eqref{eq20:steady_small_lambda_step1} \times 2 + \eqref{eq21:steady_small_lambda_step1}$, we derive that
\be \notag
| h_1 - h_2 |_{L^2_{\gamma_+}}^2 + \Vert h_1 - h_2 \Vert_{L^2_{x,\nu}}^2
\lesssim \big( \Vert \nabla_x \phi_g \Vert_{L^\infty_{x}} + \Vert w g \Vert_{L^\infty_{x,v}} + \Vert \nabla_x \phi_E \Vert_{L^\infty_x} + \delta_0 \big) \Vert h_1 - h_2 \Vert_{L^2_{x,v}}^2.
\ee
From the assumption \eqref{condition:g_step1} on $g (x,v)$ and $0 < \Vert \nabla_x \phi_E \Vert_{L^\infty_x} + \lambda \ll 1$, we conclude that $h_1 = h_2$.
\end{proof}

\begin{corollary} \label{cor:steady_small_lambda_step1}

The constant $\lambda$ in Proposition \ref{prop:steady_small_lambda_step1} depends only on $\Omega$ and is independent of $g$ and $h$.
Moreover, under the same assumptions, the proposition remains valid for all constants $0 < \lambda^* \leq \lambda$.
\end{corollary}

Next, we prove the regularity of $h$ obtained in Proposition \ref{prop:steady_small_lambda_step1}, which will be applied to study the limit of the system \eqref{eqtn:h^l} - \eqref{bdry:phi^l} as $\ell\to\infty$.
\begin{proposition} \label{prop:steady_small_lambda_step1_reg}

Suppose the inflow condition \eqref{inflow_condition} holds, and let $0 < \lambda = \lambda(\Omega) \ll 1$ be sufficiently small as in Proposition \ref{prop:steady_small_lambda_step1}.
For any non-negative $C \geq 0$, assume the same hypotheses on the functions $S_1(x,v)$ and $(g, \phi_g)$ as in Proposition \ref{prop:steady_small_lambda_step1}. 
Define the weight function $\alpha_g$ by
\begin{equation} \label{alpha_weight_steady_step1}
\begin{split}
\alpha_g :=
\begin{cases}
& \chi_{\delta'} \Big( \big[ |v\cdot \nabla_x\xi(x)|^2 + \xi^2(x) - 2 (v\cdot \nabla^2\xi(x) \cdot v)\xi(x) - 2(E_g (\tilde{x}) \cdot \nabla_x\xi(\tilde{x}))\xi(x) \big]^{1/2} \Big), 
\ \text{ if } x \in \O_\delta, \\[5pt]
& \delta', 
\ \text{ if } x \in \O \setminus \O_\delta,
\end{cases} 
\end{split}
\end{equation} 
where $E_g (x) := - \nabla_x \phi_g (x) - \nabla_x \phi_E(x)$.
Further, assume that $g (x, v)$ and $S_1(x,v)$ satisfy
\begin{align}
\Vert w_{\tilde{\theta}} \p_{x,v} g \Vert_{L^p_{x,v}} + \Vert w_{\tilde{\theta}} \alpha_g \p_{x,v} g \Vert_{L^\infty_{x,v}} + \Vert w_{\tilde{\theta}} \nabla_v g \Vert_{L^\infty_{x,v}}
& \lesssim |f_b|_{L^2_{\gamma_-}} + | w f_b |_{L^\infty_{\partial\Omega,v}} + | w \p_{\mathbf{x}_p,v} f_b|_{L^\infty_{\p\O,v}}, 
\label{g_assumption_reg_step1} \\
\Vert w \p_{x,v} S_1 \Vert_{L^\infty_{x,v}}
& \lesssim C \big[ |f_b|_{L^2_{\gamma_-}} + | w f_b |_{L^\infty_{\partial\Omega,v}} + | w \p_{\mathbf{x}_p,v} f_b|_{L^\infty_{\p\O,v}} \big].
\label{S1_assumption_reg}
\end{align} 
Then the unique solution $h(x,v)$ to the system \eqref{eqtn:h_step1}–\eqref{bdry:phi_step1} satisfies
\begin{align}
\Vert w_{\tilde{\theta}} \p_{x,v} h \Vert_{L^p_{x,v}} + \Vert w_{\tilde{\theta}} \alpha_g \p_{x,v} h \Vert_{L^\infty_{x,v}} 
& \lesssim (C+1) \big[ |f_b|_{L^2_{\gamma_-}} + | w f_b |_{L^\infty_{\partial\Omega,v}} + | w \p_{\mathbf{x}_p,v} f_b|_{L^\infty_{\p\O,v}} \big],
\label{regularity:weight_w1p_step1} \\
\Vert w_{\tilde{\theta}}\nabla_v h \Vert_{L^\infty_{x,v}} 
& \lesssim (C+1) \big[ |f_b|_{L^2_{\gamma_-}} + | w f_b |_{L^\infty_{\partial\Omega,v}} + | w \p_{\mathbf{x}_p,v} f_b|_{L^\infty_{\p\O,v}} \big].
\label{regularity:weight_c1_step1}
\end{align}  
\end{proposition}

\begin{proof}

The strategy for proving this proposition follows that of Proposition \ref{prop:steady_small_lambda_step1}, and the estimates are analogous to those obtained in Sections \ref{sec:w1p_estimate} and \ref{sec:c1_estimate}. 
Owing to the similarity of the arguments, we omit some of the details.
For the reader's convenience, we repeat the construction from Proposition~\ref{prop:steady_small_lambda_step1}. 

\smallskip

\textbf{Step 1. Construction.}
For any $\ell \in \N$,
\be \label{eq1:steady_small_lambda_step1_reg}
\begin{split}
v \cdot \nabla_x h^{\ell+1} 
& - \nabla_x (\phi_g + \phi_E) \cdot \nabla_{v} h^{\ell+1} 
+ \frac{v \cdot \nabla_x \phi_g}{2} h^{\ell+1} + e^{-\phi_E} \mathcal{L} h^{\ell+1}
\\ & = - \lambda (v \cdot \nabla_x \phi^\ell_h) e^{-\phi_E/2} \sqrt{\mu} + e^{-\phi_E/2} \Gamma(g, g) + S_1, 
\\ h^{\ell+1} |_{\gamma_-} & = f_b (x, v), 
\\ - \Delta \phi^\ell_h & = e^{-\phi_E/2} \int_{\R^3} h^{\ell} \sqrt{\mu} \dd v \text{ in } \O, 
\\ \phi^\ell_h & = 0 \text{ on } \p\O,
\end{split}
\ee
where the initial setting $h^0 = 0$ and $\nabla_x \phi^0_h = \mathbf{0}$. 

Again, \eqref{eq1:steady_small_lambda_step1_reg} is a steady Vlasov Boltzmann equation with given potential and given sources. To rigorously justify that the regularity of the system is bounded, one can again rewrite the linear Boltzmann operator as $\nu h^{\ell+1}-Kh^{\ell}$ and prove the uniform in $\ell$ control of the regularity. Such a process has been done in \cite{CK}. For simplicity, we skip the proof of this procedure and directly estimate the regularity of \eqref{eq1:steady_small_lambda_step1_reg} in the next step.

\smallskip

\textbf{Step 2. $W^{1,p}$ estimate and weighted $C^1$ estimate.} 
$(a)$
Under the construction \eqref{eq1:steady_small_lambda_step1_reg}, we claim that for any $\ell \in \N$,
\be \label{eq3:steady_small_lambda_step1_reg}
\Vert w_{\tilde{\theta}} \p_{x,v} h^{\ell} \Vert_{L^p_{x,v}} + \Vert w_{\tilde{\theta}} \alpha_g \p_{x,v} h^{\ell} \Vert_{L^\infty_{x,v}} \lesssim (C+1) \big[ |f_b|_{L^2_{\gamma_-}} + | w f_b |_{L^\infty_{\partial\Omega,v}} + | w \p_{\mathbf{x}_p,v} f_b|_{L^\infty_{\p\O,v}} \big].
\ee

We prove \eqref{eq3:steady_small_lambda_step1_reg} by induction.
From the initial setting $h^0 = 0$, \eqref{eq3:steady_small_lambda_step1_reg} holds for $\ell = 0$.
Suppose that \eqref{eq3:steady_small_lambda_step1_reg} holds for all $0 \leq \ell \leq k$.
We now consider the case when $\ell = k+1$.
Applying the method of characteristics to \eqref{eqtn:h^l} and following \eqref{rchara:nabla_0}-\eqref{rchara:nabla_phi_E_phi} in Proposition \ref{prop:weight_W1p}, we obtain
\begin{align}
& w_{\tilde{\theta}}(v) \p_{x,v} h^{k+1} (x,v) 
\notag \\
= & \mathbf{1}_{\tb \geq t} e^{-\int^t_{0} \tilde{\nu}(X(s),V(s)) \dd s} \frac{w_{\tilde{\theta}}(v)}{w_{\tilde{\theta}}(V(0))}
\notag \\
& \qquad \times \big( \p_{x,v} X(0) \cdot w_{\tilde{\theta}}(V(0)) \p_{x} h^{k+1} (X(0),V(0)) + \p_{x,v} V(0) \cdot w_{\tilde{\theta}}(V(0)) \p_{v} h^{k+1} (X(0),V(0)) \big) 
\label{rchara:nabla_0_ell_step1} \\
& + \mathbf{1}_{\tb \geq t} e^{- \int_0^t \tilde{\nu}(X(s),V(s))\dd s} \Big( - \int^t_0 \p_{x,v} \big[ \tilde{\nu}(X(s),V(s)) \big] \dd s \Big) \frac{w_{\tilde{\theta}}(v)}{w_{\tilde{\theta}}(V(0))} w_{\tilde{\theta}}(V(0)) h^{k+1} (X(0),V(0)) 
\label{rchara:nabla_nu_1_ell_step1} \\
& +  \mathbf{1}_{\tb<t} e^{-\int^t_{t-\tb} \tilde{\nu}(X(s),V(s)) \dd s} \frac{w_{\tilde{\theta}}(v)}{w_{\tilde{\theta}}(\vb)}w_{\tilde{\theta}}(\vb) \Big[  \sum_{i=1}^2 \nabla_{x,v} \mathbf{x}_{p^1,i}^1 \p_{\mathbf{x}_{p^1,i}^1}f_b(\xb,\vb)+   \p_{x,v} \vb \cdot \nabla_v f_b(\xb,\vb)  \Big]
\label{rchara:bdr_ell_step1} \\
& + \mathbf{1}_{\tb<t} e^{-\int^t_{t-\tb} \tilde{\nu}(X(s),V(s)) \dd s  } \Big( - \int_{t-\tb}^t \p_{x,v} \big[ \tilde{\nu}(X(s),V(s)) \big] \dd s \Big) \frac{w_{\tilde{\theta}}(v)}{w_{\tilde{\theta}}(\vb)} w_{\tilde{\theta}}(\vb) f_b(\xb,\vb)
\label{rchara:nabla_nu_2_ell_step1} \\
& + \mathbf{1}_{\tb<t} e^{-\int^t_{t-\tb} \tilde{\nu}(X(s),V(s)) \dd s} \big( - \p_{x,v} \tb \tilde{\nu}(\xb,\vb) \big) \frac{w_{\tilde{\theta}}(v)}{w_{\tilde{\theta}}(\vb)} w_{\tilde{\theta}} (\vb) f_b (\xb,\vb) 
\label{rchara:nabla_tb_ell_step1} \\
& + \int^t_{\max\{ t-\tb, 0 \}} e^{-\int^t_{s} \tilde{\nu}(X(\tau), V(\tau)) \dd \tau} e^{-\phi_E(X(s))} \frac{w_{\tilde{\theta}}(v)}{w_{\tilde{\theta}}(V(s))} 
\notag \\
& \qquad \qquad \qquad \qquad \times \int_{\mathbb{R}^3} w_{\tilde{\theta}}(V(s)) \mathbf{k}(V(s),u) \p_{x,v} \big[ h^{k+1} (X(s),u) \big] \dd u \dd s 
\label{rchara:K_nabla_ell_step1} \\   
& + \int^t_{\max\{t-\tb,0\}} e^{-\int^t_s \tilde{\nu}(X(\tau),V(\tau)) \dd \tau} e^{-\phi_E(X(s))} \frac{w_{\tilde{\theta}}(v)}{w_{\tilde{\theta}}(V(s))}
\notag \\
& \qquad \qquad \qquad \qquad \times \int_{\mathbb{R}^3} w_{\tilde{\theta}}(V(s)) \big( \p_v \mathbf{k} (V(s),u) \cdot \p_{x,v} V(s) \big) h^{k+1} (X(s),u) \dd u \dd s  
\label{rchara:nabla_K_ell_step1} \\
&  + \int^t_{\max\{t-\tb,0\}} e^{-\int^t_s \tilde{\nu} (X(\tau),V(\tau)) \dd \tau} e^{-\phi_E(X(s))} \big( - \p_{x} \phi_E(X(s)) \cdot \p_{x,v} X(s) \big) \frac{w_{\tilde{\theta}}(v)}{w_{\tilde{\theta}}(V(s))} 
\notag \\
& \qquad \qquad \qquad \qquad \times \int_{\mathbb{R}^3} w_{\tilde{\theta}}(V(s)) \mathbf{k} (V(s),u) h^{k+1} (X(s),u) \dd u \dd s  
\label{rchara:K_nabla_phi_E_ell_step1} \\
& + \mathbf{1}_{\tb<t} ( \p_{x,v} \tb ) e^{-\int^t_{t-\tb} \tilde{\nu}(X(\tau),V(\tau)) \dd \tau} e^{-\phi_E(\xb)} \frac{w_{\tilde{\theta}}(v)}{w_{\tilde{\theta}}(\vb)}w_{\tilde{\theta}}(\vb) \int_{\mathbb{R}^3} \mathbf{k}(\vb,u) f_b (\xb,u) \dd u \dd s 
\label{rchara:nabla_int_ell_step1} \\
& + \int^t_{\max\{t-\tb,0\}} e^{-\int_s^t \tilde{\nu}(X(\tau),V(\tau))\dd \tau} \Big( - \int_s^t \p_{x,v} \big[ \tilde{\nu} (X(\tau), V(\tau)) \big] \dd \tau \Big) e^{-\phi_E(X(s))} \frac{w_{\tilde{\theta}}(v)}{w_{\tilde{\theta}}(V(s))} 
\notag \\
& \qquad \qquad \qquad \qquad \times \int_{\mathbb{R}^3}  w_{\tilde{\theta}}(V(s)) \mathbf{k}(V(s),u) h^{k+1} (X(s), u) \dd u \dd s \label{rchara:nabla_nu_K_ell_step1} \\
& + \int^t_{\max\{t-\tb,0\}} e^{-\int_s^t \tilde{\nu}(X(\tau),V(\tau))\dd \tau} e^{-\phi_E(X(s))/2} \frac{w_{\tilde{\theta}}(v)}{w_{\tilde{\theta}}(V(s))} w_{\tilde{\theta}}(V(s)) \p_{x,v} \big[ \Gamma (g, g) (X(s),V(s)) \big] \dd s
\label{rchara:nabla_gamma_ell_step1} \\
& + \mathbf{1}_{\tb<t} ( \p_{x,v} \tb ) e^{-\int^t_{t-\tb}\tilde{\nu}(X(\tau),V(\tau)) \dd \tau} e^{-\phi_E(\xb)/2} \frac{w_{\tilde{\theta}}(v)}{w_{\tilde{\theta}}(\vb)}w_{\tilde{\theta}}(\vb) \Gamma(f_b,f_b) (\xb,\vb) 
\label{rchara:nabla_tb_gamma_ell_step1} \\
& + \int^t_{\max\{t-\tb,0\}} e^{-\int_s^t \tilde{\nu}(X(\tau),V(\tau))\dd \tau} \Big( - \int_s^t   \p_{x,v} \big[ \tilde{\nu}(X(\tau),V(\tau)) \big] \dd \tau \Big) e^{-\phi_E(X(s))/2}
\notag \\
& \qquad \qquad \qquad \qquad \times \frac{w_{\tilde{\theta}}(v)}{w_{\tilde{\theta}}(V(s))} w_{\tilde{\theta}}(V(s)) \Gamma (g, g) (X(s),V(s)) \dd s \label{rchara:nabla_nu_gamma_ell_step1}  \\
& + \int^t_{\max\{t-\tb,0\}} e^{-\int_s^t \tilde{\nu}(X(\tau),V(\tau))\dd \tau} e^{-\phi_E(X(s))/2} \Big( \frac{1}{2} \nabla_x \phi_E(X(s)) \cdot \p_{x,v} X(s) \Big) 
\notag \\
& \qquad \qquad \qquad \qquad \times \frac{w_{\tilde{\theta}}(v)}{w_{\tilde{\theta}}(V(s))} w_{\tilde{\theta}}(V(s)) \Gamma (g, g) (X(s),V(s)) \dd s
\label{rchara:gamma_nabla_phi_E_ell_step1} \\
& + \lambda \int^t_{\max\{t-\tb,0\}} e^{-\int^t_{s} \tilde{\nu}(X(\tau),V(\tau)) \dd \tau} \Big( V(s) \cdot \p_{x,v} \big[ \nabla_x \phi^{k}_h (X(s)) \big] \Big) e^{-\phi_E(X(s))/2}
\notag \\
& \qquad \qquad \qquad \qquad \times \frac{w_{\tilde{\theta}}(v)}{w_{\tilde{\theta}}(V(s))} w_{\tilde{\theta}}(V(s)) \sqrt{\mu(V(s))} \dd s   
\label{rchara:nabla_E_ell_step1} \\
& + \lambda \int^t_{\max\{t-\tb,0\}} e^{-\int^t_s \tilde{\nu}(X(\tau),V(\tau)) \dd \tau} \Big( \p_{x,v} V(s) \cdot \nabla_x \phi^{k}_h (X(s)) \Big) e^{-\phi_E(X(s))/2}
\notag \\
& \qquad \qquad \qquad \qquad \times \frac{w_{\tilde{\theta}}(v)}{w_{\tilde{\theta}}(V(s))} w_{\tilde{\theta}}(V(s)) \sqrt{\mu(V(s))} \dd s
\label{rchara:nabla_V_phi_ell_step1} \\
& + \lambda \mathbf{1}_{\tb<t} (\p_{x,v} \tb) e^{-\int^t_{t-\tb} \tilde{\nu}(X(\tau),V(\tau)) \dd \tau} \frac{w_{\tilde{\theta}}(v)}{w_{\tilde{\theta}}(\vb)} w_{\tilde{\theta}}(\vb) \big( \vb \cdot \nabla_x \phi^{k}_h (\xb) \big) e^{-\phi_E(X(s))/2} \sqrt{\mu(\vb)} 
\label{rchara:nabla_tb_phi_ell_step1} \\
& + \lambda \int^t_{\max\{t-\tb,0\}} e^{-\int_s^t \tilde{\nu}(X(\tau),V(\tau))\dd \tau} \Big( - \int_s^t \p_{x,v} \big[ \tilde{\nu}(X(\tau),V(\tau)) \big] \dd \tau \Big) \big( V(s) \cdot \nabla_x \phi^{k}_h (X(s)) \big) 
\notag \\
& \qquad \qquad \qquad \qquad \times e^{-\phi_E(X(s))/2} \frac{w_{\tilde{\theta}}(v)}{w_{\tilde{\theta}}(V(s))} w_{\tilde{\theta}}(V(s)) \sqrt{\mu(V(s))} \dd s \label{rchara:nabla_nu_phi_ell_step1} \\
& + \lambda \int^t_{\max\{t-\tb,0\}} e^{-\int^t_s \tilde{\nu}(X(\tau),V(\tau)) \dd \tau} \big( V(s) \cdot \nabla_x \phi^{k}_h (X(s)) \big) e^{-\phi_E(X(s))/2}
\notag \\
& \qquad \qquad \qquad \qquad \times \frac{w_{\tilde{\theta}}(v)}{w_{\tilde{\theta}}(V(s))} w_{\tilde{\theta}}(V(s)) \sqrt{\mu(V(s))} \Big( - \frac{1}{2} \p_{x,v} V(s) \cdot V(s) \Big) \dd s   
\label{rchara:nabla_mu_phi_ell_step1} \\
& + \lambda \int^t_{\max\{t-\tb,0\}} e^{-\int^t_s \tilde{\nu}(X(\tau),V(\tau)) \dd \tau} \big( V(s) \cdot \nabla_x \phi^{k}_h (X(s)) \big) e^{-\phi_E(X(s))/2} \Big( \frac{1}{2} \nabla_x \phi_E(X(s)) \cdot \p_{x,v} X(s) \Big) 
\notag \\
& \qquad \qquad \qquad \qquad \times \frac{w_{\tilde{\theta}}(v)}{w_{\tilde{\theta}}(V(s))} w_{\tilde{\theta}}(V(s)) \sqrt{\mu(V(s))} \dd s
\label{rchara:nabla_phi_E_phi_ell_step1} \\
& + \int^t_{\max\{t-\tb,0\}} e^{-\int_s^t \tilde{\nu}(X(\tau),V(\tau))\dd \tau} \frac{w_{\tilde{\theta}}(v)}{w_{\tilde{\theta}}(V(s))} w_{\tilde{\theta}}(V(s)) \p_{x,v} \big[ S_1 (X(s),V(s)) \big] \dd s
\label{rchara:nabla_S1_step1} \\
& + \mathbf{1}_{\tb<t} ( \p_{x,v} \tb ) e^{-\int^t_{t-\tb}\tilde{\nu}(X(\tau),V(\tau)) \dd \tau} \frac{w_{\tilde{\theta}}(v)}{w_{\tilde{\theta}}(\vb)}w_{\tilde{\theta}}(\vb) S_1 (\xb,\vb) 
\label{rchara:nabla_tb_S1_step1} \\
& + \int^t_{\max\{t-\tb,0\}} e^{-\int_s^t \tilde{\nu}(X(\tau),V(\tau))\dd \tau} \Big( - \int_s^t   \p_{x,v} \big[ \tilde{\nu}(X(\tau),V(\tau)) \big] \dd \tau \Big) \frac{w_{\tilde{\theta}}(v)}{w_{\tilde{\theta}}(V(s))} w_{\tilde{\theta}}(V(s)) S_1 (X(s),V(s)) \dd s,
\label{rchara:nabla_nu_S1_ell_step1}
\end{align}
where $(X(s), V(s)) := (X(s;t,x,v), V(s;t,x,v))$ with $\max\{0,t-\tb\} \leq s \leq t$.

Similar to Proposition \ref{prop:weight_W1p}, the piece-wise formula above is a weak derivative to a mild formulation of \eqref{eq1:steady_small_lambda_step1_reg}.
We now denote 
\be \label{ABCD_notation_ell_step1}
\begin{split}
& \mathcal{A}(x,v) = \eqref{rchara:nabla_nu_1_ell_step1} + \eqref{rchara:nabla_nu_2_ell_step1} + \eqref{rchara:nabla_K_ell_step1}+ \eqref{rchara:K_nabla_phi_E_ell_step1} + \eqref{rchara:nabla_nu_K_ell_step1} + \eqref{rchara:nabla_nu_gamma_ell_step1}
\\& \qquad \qquad \quad + \eqref{rchara:gamma_nabla_phi_E_ell_step1} +  \eqref{rchara:nabla_E_ell_step1} + \eqref{rchara:nabla_V_phi_ell_step1}  + \eqref{rchara:nabla_nu_phi_ell_step1}  + \eqref{rchara:nabla_mu_phi_ell_step1} + \eqref{rchara:nabla_phi_E_phi_ell_step1} + \eqref{rchara:nabla_S1_step1} + \eqref{rchara:nabla_nu_S1_ell_step1}, 
\\& \mathcal{B}(x,v) = \eqref{rchara:nabla_0_ell_step1} + \eqref{rchara:nabla_gamma_ell_step1}, 
\\& \mathcal{C}(x,v) = \eqref{rchara:bdr_ell_step1} + \eqref{rchara:nabla_tb_ell_step1} + \eqref{rchara:nabla_int_ell_step1} + \eqref{rchara:nabla_tb_gamma_ell_step1} + \eqref{rchara:nabla_tb_phi_ell_step1} + \eqref{rchara:nabla_tb_S1_step1}, 
\\& \mathcal{D}(x,v) = \eqref{rchara:K_nabla_ell_step1}. 
\end{split}
\ee
We remark that, compared with the expansion \eqref{rchara:nabla_0}-\eqref{rchara:nabla_phi_E_phi} in Proposition~\ref{prop:weight_W1p}, three additional terms appear in \eqref{rchara:nabla_S1_step1}–\eqref{rchara:nabla_nu_S1_ell_step1} due to the presence of $S_1$.

From the assumption that $\Vert \nabla_x \phi_g \Vert_{L^\infty_x} + \Vert \nabla_x \phi_E \Vert_{L^\infty_x} \ll 1$, Lemma \ref{lemma:v_variation} implies that
\be \notag
\frac{w_{\tilde{\theta}} (v)}{w_{\tilde{\theta}} (V(s))} \lesssim e^{o(1) (t - s) |v|}
\ \text{ for any $\max\{0,t-\tb\} \leq s \leq t$}.
\ee
Recall that $\nu(v) \geq \nu_0 > 0$ for any $v \in \mathbb{R}^3$. Now fix a constant $t > 0$ such that 
\[
\nu_0 t \geq 10
\ \text{ and } \
(1 + t^2) e^{ - \frac{\nu_0}{4} t} \leq \frac{1}{10}.
\]
Since $\Vert \nabla_x \phi_E \Vert_{L^\infty_x} + |f_b|_{L^2_{\gamma_-}} + | w f_b|_{L^\infty_{\p\O,v}} \ll 1$, we further have 
\[
\big( |f_b|_{L^2_{\gamma_-}} + | w f_b|_{L^\infty_{\p\O,v}} \big) t^2 \ll 1.
\]
From Lemmas \ref{lemma:phi_x_infinity} and \ref{lemma:phi_C2}, together with the assumption \eqref{g_assumption_reg_step1} on $(g, \phi_g)$, we have
\[
\big( \| \nabla_x^2 \phi_g \|_{L^\infty_x} + \Vert \nabla_x \phi_E \Vert_{L^\infty_x} + \Vert \nabla_x \phi_g \Vert_{L^\infty_x} \big) t^2 \ll 1.
\]
Moreover,  the assumption \eqref{g_assumption_reg_step1} on $(g, \phi_g)$ implies
\be \label{eq4:steady_small_lambda_step1_reg}
\frac{3}{2} \nu(V(s)) > \tilde{\nu}(X(s),V(s)) := e^{-\phi_E(X(s))}\nu(V(s)) + \frac{V(s) \cdot \nabla_x \phi_g (X(s))}{2} \geq \frac{1}{2} \nu(V(s)) > \frac{\nu_0}{2}.
\ee
Thus, we compute that
\be \notag
\begin{split}
& \big| \nabla_{x, v} \big[ \tilde{\nu}(X(s),V(s)) \big] \big|
\\& \lesssim \big( \Vert \nabla_x \phi_E \Vert_{L^\infty_{x}} + o(1) [ \Vert \p_x g \Vert_{L^p_{x,v}} + \Vert \alpha_g \p_x g \Vert_{L^\infty_{x,v}} ] + \Vert w g \Vert_{L^\infty_{x,v}} \big) (1 + |v|) e^{ o(1) (t-s) } + e^{ o(1) (t-s) }.
\end{split}
\ee
Moreover, Lemma \ref{lemma:deri_XV} and Lemma \ref{lemma:phi_C2} imply that for any $\ell \in \N$,
\be \notag
\begin{split}
\big| \p_{x,v} \big[ \nabla_x \phi^{\ell}_h (X(s)) \big] \big| 
& \lesssim \big( o(1) \big[ \Vert \nabla_x h^{\ell} \Vert_{L^p_{x,v}} + \Vert \alpha_g \nabla_x h^{\ell} \Vert_{L^\infty_{x,v}} \big] + \Vert w h^{\ell} \Vert_{L^\infty_{x,v}} \big) e^{o(1)(t-s)}.
\end{split}
\ee

For $\mathcal{A} (x,v)$, similarly to the estimates \eqref{A_bdd_infty} and \eqref{A_bdd_lp} on $\mathcal{A} (x,v)$ in Proposition \ref{prop:weight_W1p}, together with the assumptions \eqref{S1_assumption} and \eqref{S1_assumption_reg} on $S_1$, we obtain that
\be \label{A_bdd_infty_ell_step1}
\begin{split}
| \mathcal{A}(x,v)| 
& \lesssim w^{-\frac{1}{4}}(v) \big( \Vert w h^{k+1} \Vert_{L^\infty_{x,v}} + o(1) \big[ \Vert \nabla_x h^{k} \Vert_{L^p_{x,v}} + \Vert \alpha_g \nabla_x h^{k} \Vert_{L^\infty_{x,v}} \big] + \Vert w h^{k} \Vert_{L^\infty_{x,v}} + | wf_b|_{L^\infty_{\p\O,v}} \big)
\\& \qquad + w^{-\frac{1}{4}}(v) \big( \Vert w \p_{x,v} S_1 \Vert_{L^\infty_{x,v}} + \Vert w S_1 \Vert_{L^\infty_{x,v}} \big),
\\ \Vert \mathcal{A}\Vert_{L^p_{x,v}} 
& \lesssim \Vert w h^{k+1} \Vert_{L^\infty_{x,v}} + o(1) \big[ \Vert \nabla_x h^{k} \Vert_{L^p_{x,v}} + \Vert \alpha_g \nabla_x h^{k} \Vert_{L^\infty_{x,v}} \big] + \Vert w h^{k} \Vert_{L^\infty_{x,v}} + | wf_b|_{L^\infty_{\p\O,v}}
\\& \qquad + C \big[ |f_b|_{L^2_{\gamma_-}} + | w f_b |_{L^\infty_{\partial\Omega,v}} + | w \p_{\mathbf{x}_p,v} f_b|_{L^\infty_{\p\O,v}} \big].
\end{split}
\ee

For $\mathcal{B} (x,v)$, similar to the estimates \eqref{nabla_0_compute} and \eqref{gamma_gain_est} in Proposition \ref{prop:weight_W1p}, from Lemma \ref{lemma:int_cov}, Lemma \ref{lemma:k_theta}, Lemma \ref{lemma:gamma} and  the assumptions \eqref{condition:g_step1} and \eqref{g_assumption_reg_step1} on $g$, we obtain that
\be \label{B_bdd_lp_ell_step1}
\Vert \mathcal{B} \Vert_{L^p_{x,v}} 
\lesssim \frac{1}{10} \Vert w_{\tilde{\theta}}(v)\p_{x,v} h^{k+1} \Vert_{L^p_{x,v}} + |f_b|_{L^2_{\gamma_-}} + | w f_b |_{L^\infty_{\partial\Omega,v}} + | w \p_{\mathbf{x}_p,v} f_b|_{L^\infty_{\p\O,v}}.
\ee

For $\mathcal{C} (x,v)$, similar to the estimates \eqref{c_bdd_infty} and \eqref{c_bdd_lp} in Proposition \ref{prop:weight_W1p}, from Lemma \ref{lemma:deri_backward}, Lemma \ref{lemma:integrate_nv} and the assumption \eqref{S1_assumption} on $S_1$, we derive that
\be \label{c_bdd_infty_ell_step1}
\begin{split}
| \mathcal{C}(x,v)| 
& \lesssim \big( \Vert w h^{k} \Vert_{L^\infty_{x,v}} + | w f_b |_{L^\infty_{\p\O,v}} + | w \p_{\mathbf{x}_p,v} f_b |_{L^\infty_{\p\O,v}} + \Vert w S_1 \Vert_{L^\infty_{x,v}} \big) \frac{e^{-\nu_0 \tb /2}}{| \vb \cdot n(\xb) |} w^{-\frac{1}{2}}(\vb), 
\\
\Vert \mathcal{C} \Vert_{L^p_{x,v}} 
& \lesssim \big( \Vert w h^{k} \Vert_{L^\infty_{x,v}} + | wf_b |_{L^\infty_{\p\O,v}} + | w \p_{\mathbf{x}_p,v} f_b |_{L^\infty_{\p\O,v}} + \Vert w S_1 \Vert_{L^\infty_{x,v}} \big) \Big\Vert e^{-\nu_0\tb/2} \frac{1}{|n(\xb) \cdot \vb|} w^{-\frac{1}{2}}(\vb) \Big\Vert_{L^p_{x,v}} 
\\& \lesssim \Vert w h^{k} \Vert_{L^\infty_{x,v}} + | w f_b |_{L^\infty_{\p\O,v}} + | w \p_{\mathbf{x}_p,v} f_b |_{L^\infty_{\p\O,v}} + |f_b|_{L^2_{\gamma_-}}.
\end{split}
\ee

For $\mathcal{D} (x,v)$, we rewrite \eqref{rchara:K_nabla_ell_step1} as follows:
\be \label{k_nabla_express_ell_step1}
\begin{split}
\eqref{rchara:K_nabla_ell_step1}
& = \int^t_{\max\{ t-\tb, 0 \}} e^{-\int^t_{s} \tilde{\nu}(X(\tau), V(\tau)) \dd \tau} e^{-\phi_E(X(s))} \frac{w_{\tilde{\theta}}(v)}{w_{\tilde{\theta}}(V(s))} 
\\& \qquad \qquad \qquad \qquad \times \int_{\mathbb{R}^3} \mathbf{k}(V(s),u) \frac{w_{\tilde{\theta}}(V(s))}{w_{\tilde{\theta}}(u)} 
\underbrace{w_{\tilde{\theta}}(u) \p_{x} h^{k+1} (X(s),u)}_{\eqref{k_nabla_express_ell_step1}^*}
\p_{x,v} X(s) \dd u \dd s.
\end{split}
\ee
Following the estimates on $\mathcal{D} (x,v)$ in Proposition \ref{prop:weight_W1p}, we apply the characteristic formula \eqref{rchara:nabla_0_ell_step1}-\eqref{rchara:nabla_nu_S1_ell_step1} on $\eqref{k_nabla_express_ell_step1}^* = w_{\tilde{\theta}}(u) \p_{x} h^{k+1} (X(s),u)$ by replacing $(t,x,v)$ with $(s,X(s),u)$.
Similar to \eqref{ABCD_notation_ell_step1}, we express $\eqref{k_nabla_express_ell_step1}^*$ as
\be \notag
\eqref{k_nabla_express_ell_step1}^*
= \eqref{k_nabla_express_ell_step1}^*_{\mathcal{A}} + \eqref{k_nabla_express_ell_step1}^*_{\mathcal{B}} + \eqref{k_nabla_express_ell_step1}^*_{\mathcal{C}} + \eqref{k_nabla_express_ell_step1}^*_{\mathcal{D}},
\ee
and we further denote that
\be \label{ABCD_notation_ell_D_step1}
\begin{split}
\eqref{rchara:K_nabla_ell_step1}_{\mathcal{A}, \mathcal{B}, \mathcal{C}, \mathcal{D}}
& := \int^t_{\max\{ t-\tb, 0 \}} e^{-\int^t_{s} \tilde{\nu}(X(\tau), V(\tau)) \dd \tau} e^{-\phi_E(X(s))} \frac{w_{\tilde{\theta}}(v)}{w_{\tilde{\theta}}(V(s))} 
\\& \qquad \qquad \qquad \qquad \times \int_{\mathbb{R}^3} \mathbf{k}(V(s),u) \frac{w_{\tilde{\theta}}(V(s))}{w_{\tilde{\theta}}(u)} \eqref{k_nabla_express_ell_step1}^*_{\mathcal{A}, \mathcal{B}, \mathcal{C}, \mathcal{D}} \p_{x,v} X(s) \dd u \dd s.
\end{split}
\ee

For $\eqref{rchara:K_nabla_ell_step1}_{\mathcal{A}}$, following the estimate \eqref{k_A_bdd} in Proposition \ref{prop:weight_W1p}, we have
\be \notag
\begin{split}
|\eqref{rchara:K_nabla_ell_step1}_{\mathcal{A}}|
& \lesssim \int^t_{\max\{ t-\tb, 0 \}} e^{-\int^t_s \frac{\tilde{\nu}(X(\tau),V(\tau))}{4} \dd \tau} | \p_{x,v} X(s) | \int_{\mathbb{R}^3} \mathbf{k}(V(s),u) \frac{w_{\tilde{\theta}}(V(s))}{w_{\tilde{\theta}}(u)}  |\eqref{k_nabla_express_ell_step1}^*_{\mathcal{A}}| \dd u \dd s 
\\& \lesssim w^{-\frac{1}{8}}(v) \big( o(1) \big[ \Vert \nabla_x h^{k} \Vert_{L^p_{x,v}} + \Vert \alpha_g \nabla_x h^{k} \Vert_{L^\infty_{x,v}} \big] + \Vert w h^{k} \Vert_{L^\infty_{x,v}} + \Vert w h^{k+1} \Vert_{L^\infty_{x,v}} 
\\& \qquad \qquad \qquad + | wf_b|_{L^\infty_{\p\O,v}} + \Vert w \p_{x,v} S_1 \Vert_{L^\infty_{x,v}} + \Vert w S_1 \Vert_{L^\infty_{x,v}} \big).
\end{split}
\ee
This further implies that
\be \label{k_A_bdd_ell_step1}
\begin{split}
\Vert \eqref{rchara:K_nabla_ell_step1}_{\mathcal{A}} \Vert_{L^p_{x,v}} 
& \lesssim \Vert w h^{k+1} \Vert_{L^\infty_{x,v}} + o(1) \big[ \Vert \nabla_x h^{k} \Vert_{L^p_{x,v}} + \Vert \alpha_g \nabla_x h^{k} \Vert_{L^\infty_{x,v}} \big] + \Vert w h^{k} \Vert_{L^\infty_{x,v}} + | wf_b|_{L^\infty_{\p\O,v}}
\\& \qquad + C \big[ |f_b|_{L^2_{\gamma_-}} + | w f_b |_{L^\infty_{\partial\Omega,v}} + | w \p_{\mathbf{x}_p,v} f_b|_{L^\infty_{\p\O,v}} \big].
\end{split}
\ee

For $\eqref{rchara:K_nabla_ell_step1}_{\mathcal{B}}$, following the estimate \eqref{k_B_bdd} in Proposition \ref{prop:weight_W1p}, we have
\be \label{k_B_bdd_ell_step1}
\begin{split}
\Vert \eqref{rchara:K_nabla_ell_step1}_{\mathcal{B}}\Vert_{L^p_{x,v}}  & \lesssim \frac{1}{10} \Vert w_{\tilde{\theta}}(v)\p_{x,v} h^{k+1} \Vert_{L^p_{x,v}} + \big( o(1) + \Vert w g \Vert_{L^\infty_{x,v}} \big) \Vert w_{\tilde{\theta}}(v) \p_{x,v} g \Vert_{L^p_{x,v}} + \Vert w g \Vert_{L^\infty_{x,v}}
\\& \lesssim \frac{1}{10} \Vert w_{\tilde{\theta}}(v)\p_{x,v} h^{k+1} \Vert_{L^p_{x,v}} + |f_b|_{L^2_{\gamma_-}} + | w f_b |_{L^\infty_{\partial\Omega,v}} + | w \p_{\mathbf{x}_p,v} f_b|_{L^\infty_{\p\O,v}}.
\end{split}
\ee

For $\eqref{rchara:K_nabla_ell_step1}_{\mathcal{C}}$, following the estimate \eqref{minkowski_simplify} in Proposition \ref{prop:weight_W1p}, we have
\be \label{k_C_bdd_ell_step1}
\begin{split}
\Vert \eqref{rchara:K_nabla_ell_step1}_{\mathcal{C}} \Vert_{L^p_{x,v}}  
& \lesssim \big( \Vert w h^{k} \Vert_{L^\infty_{x,v}} + | w f_b |_{L^\infty_{\p\O,v}} + | w \p_{\mathbf{x}_p,v} f_b |_{L^\infty_{\p\O,v}}  + \Vert w S_1 \Vert_{L^\infty_{x,v}} \big) \Big\Vert \frac{w^{-\frac{1}{4}}(\vb(x,u))e^{-\frac{\nu_0\tb(x,u)}{2}}}{n(\xb(x,u))\cdot \vb(x,u)} \Big\Vert_{L^p_{x,u}}  
\\& \lesssim \Vert w h^{k} \Vert_{L^\infty_{x,v}} + | w f_b |_{L^\infty_{\p\O,v}} + | w \p_{\mathbf{x}_p,v} f_b |_{L^\infty_{\p\O,v}} + |f_b|_{L^2_{\gamma_-}}.
\end{split}
\ee

For $\eqref{rchara:K_nabla_ell_step1}_{\mathcal{D}}$, we write it explicitly as follows for clarity:
\be \label{kk_ell_step1} 
\begin{split}
\eqref{rchara:K_nabla_ell_step1}_{\mathcal{D}}
& := \int^t_{\max\{ t-\tb, 0 \}} e^{-\int^t_{s} \tilde{\nu}(X(\tau), V(\tau)) \dd \tau} e^{-\phi_E(X(s))} \frac{w_{\tilde{\theta}}(v)}{w_{\tilde{\theta}}(V(s))} \dd s \int_{\mathbb{R}^3} \mathbf{k}(V(s),u) \frac{w_{\tilde{\theta}}(V(s))}{w_{\tilde{\theta}}(u)} \p_{x,v} X(s) \dd u
\\& \qquad \times \int^s_{\max\{ s - \tb(X(s),u), 0 \}} e^{-\int^s_{s'} \tilde{\nu}(X(\tau'), V(\tau')) \dd \tau'} e^{-\phi_E(X(s',s))} \frac{w_{\tilde{\theta}}(u)}{w_{\tilde{\theta}}(V(s',s))} \dd s' 
\\& \qquad \times \int_{\mathbb{R}^3} \mathbf{k}(V(s',s), u') \frac{w_{\tilde{\theta}}(V(s',s))}{w_{\tilde{\theta}}(u')} w_{\tilde{\theta}}(u') \p_{x} h^{k+1} (X(s',s),u') \p_{x,v} X(s',s) \dd u'.
\end{split}
\ee
Following the estimate on \eqref{kk_c} in Proposition \ref{prop:weight_W1p}, we have
\be \notag
\begin{split}
\Vert \eqref{kk_ell_step1} \Vert_{L^p_{x,v}} 
& \lesssim C(\delta,t,N) \big( \Vert w h^{k+1} \Vert_{L^\infty_{x,v}}+ | wf_b|_{L^\infty_{\p\O,v}} \big) + \frac{1}{10} \Vert w_{\tilde{\theta}}(v) \p_{x} h^{k+1} \Vert_{L^p_{x,v}}
\\& \qquad + \big( o(1) \big[ \Vert w_{\tilde{\theta}}(v) \p_x g \Vert_{L^p_{x,v}} + \Vert \alpha_g \p_x g \Vert_{L^\infty_{x,v}} \big] + \Vert w g \Vert_{L^\infty_{x,v}} \big).
\end{split}
\ee
This, together with \eqref{k_A_bdd_ell_step1}-\eqref{k_C_bdd_ell_step1} and the assumptions \eqref{condition:g_step1} and \eqref{g_assumption_reg_step1} on $g$, shows that
\be \label{D_bdd_lp_ell_step1}
\begin{split}
\Vert \mathcal{D} \Vert_{L^p_{x,v}} 
& \lesssim |f_b|_{L^2_{\gamma_-}} + | w f_b |_{L^\infty_{\p\O,v}} + | w \p_{\mathbf{x}_p,v} f_b |_{L^\infty_{\p\O,v}} + \frac{1}{10} \Vert w_{\tilde{\theta}}(v)\p_{x,v} h^{k+1} \Vert_{L^p_{x,v}} 
\\& \qquad + o(1) \big[ \Vert \nabla_x h^{k} \Vert_{L^p_{x,v}} + \Vert \alpha_g \nabla_x h^{k} \Vert_{L^\infty_{x,v}} \big] + \Vert w h^{k} \Vert_{L^\infty_{x,v}} + \Vert w h^{k+1} \Vert_{L^\infty_{x,v}} + \Vert w S_1 \Vert_{L^\infty_{x,v}}.
\end{split}
\ee

Combining \eqref{A_bdd_infty_ell_step1}-\eqref{c_bdd_infty_ell_step1} with \eqref{D_bdd_lp_ell_step1} and using Proposition \ref{prop:steady_small_lambda_step1}, we further derive that
\be \label{ABCD_bdd_lp_ell_step1}
\begin{split}
\Vert w_{\tilde{\theta}}(v) \p_{x,v} h^{k+1} \Vert_{L^p_{x,v}} 
& \lesssim o(1) \big[ \Vert w_{\tilde{\theta}}(v) \p_{x,v} h^{k} \Vert_{L^p_{x,v}} + \Vert \alpha_g \nabla_x h^{k} \Vert_{L^\infty_{x,v}} \big] + \frac{1}{5} \Vert w_{\tilde{\theta}}(v)\p_{x,v} h^{k+1} \Vert_{L^p_{x,v}} 
\\& \qquad + | w f_b |_{L^\infty_{\p\O,v}} + | w \p_{\mathbf{x}_p,v} f_b |_{L^\infty_{\p\O,v}} + |f_b|_{L^2_{\gamma_-}}.
\end{split}
\ee

\smallskip

$(b)$ 
Following the arguments in the proof of Proposition \ref{prop:weight_C1}, using the characteristic formula \eqref{rchara:nabla_0_ell_step1}-\eqref{rchara:nabla_nu_S1_ell_step1} and the notation $\mathcal{A},\mathcal{B},\mathcal{C},\mathcal{D}$ in \eqref{ABCD_notation_ell_step1}, we derive that
\be \label{eq5:steady_small_lambda_step1_reg}
\Vert w_{\tilde{\theta}} \alpha_g \p_{x,v} h^{k+1} \Vert_{L^\infty_{x,v}} 
\lesssim \Vert \alpha_g \mathcal{A}(x,v) \Vert_{L^\infty_{x,v}} + \Vert \alpha_g \mathcal{B}(x,v) \Vert_{L^\infty_{x,v}} + | \Vert \alpha_g \mathcal{C}(x,v) \Vert_{L^\infty_{x,v}} + \Vert \alpha_g \mathcal{D}(x,v) \Vert_{L^\infty_{x,v}}.
\ee

For $\alpha_g \mathcal{A}(x,v)$, the fact that $\alpha_g \lesssim 1$ and the estimate \eqref{A_bdd_infty_ell_step1} imply that
\be \label{alpha_A_bdd_infty_ell_step1}
\begin{split}
| \alpha_g \mathcal{A}(x,v)| 
& \lesssim o(1) \big[ \Vert \nabla_x h^{k} \Vert_{L^p_{x,v}} + \Vert \alpha_g \nabla_x h^{k} \Vert_{L^\infty_{x,v}} \big] + \Vert w h^{k+1} \Vert_{L^\infty_{x,v}} + \Vert w h^{k} \Vert_{L^\infty_{x,v}} + | wf_b|_{L^\infty_{\p\O,v}}
\\& \qquad + \Vert w \p_{x,v} S_1 \Vert_{L^\infty_{x,v}} + \Vert w S_1 \Vert_{L^\infty_{x,v}}.
\end{split}
\ee

For $\alpha_g \mathcal{B}(x,v)$, similar to the estimate \eqref{nabla_0_compute} in Proposition \ref{prop:weight_W1p}, we derive that
\be \notag
\begin{split}
| \alpha_g \eqref{rchara:nabla_0_ell_step1} |
& \lesssim e^{-\nu_0 t/4} e^{C_0} \Vert \alpha_g (X(0),V(0)) w_{\tilde{\theta}} (V(0))\p_{x,v} h^{k+1}(X(0),V(0))\Vert_{L^\infty_{x,v}} 
\\& \lesssim o(1) \Vert \alpha_g w_{\tilde{\theta}} (v) \p_{x,v} h^{k+1} \Vert_{L^\infty_{x,v}}.
\end{split}
\ee
Following Lemma \ref{lemma:gamma}, Lemma \ref{lemma:nonlocal_to_local}, and the estimates in Proposition \ref{prop:weight_W1p}, we derive that
\be \notag
| \alpha_g \eqref{rchara:nabla_gamma_ell_step1} |
\lesssim o(1) \Vert  \alpha_g w_{\tilde{\theta}} \p_{x,v} g \Vert_{L^p_{x,v}} + \Vert w g \Vert^2_{L^\infty_{x,v}}.
\ee
Thus, by the assumptions \eqref{condition:g_step1} and \eqref{g_assumption_reg_step1} on $g$, we obtain that
\be \label{alpha_B_bdd_lp_ell_step1}
| \alpha_g \mathcal{B}(x,v)|
\lesssim o(1) \Vert \alpha_g w_{\tilde{\theta}} (v) \p_{x,v} h^{k+1} \Vert_{L^\infty_{x,v}} + |f_b|_{L^2_{\gamma_-}} + | w f_b |_{L^\infty_{\partial\Omega,v}} + | w \p_{\mathbf{x}_p,v} f_b|_{L^\infty_{\p\O,v}}.
\ee

For $\alpha_g \mathcal{C}(x,v)$, the estimate \eqref{c_bdd_infty_ell_step1} and Lemma \ref{lemma:weight_singularity} imply that
\be \label{alpha_C_bdd_infty_ell_step1}
\begin{split}
| \alpha_g \mathcal{C}(x,v)| 
& \lesssim \big( \Vert w h^{k} \Vert_{L^\infty_{x,v}} + | w f_b |_{L^\infty_{\p\O,v}} + | w \p_{\mathbf{x}_p,v} f_b |_{L^\infty_{\p\O,v}} + \Vert w S_1 \Vert_{L^\infty_{x,v}} \big) \frac{e^{-\nu_0 \tb /2}}{| \vb \cdot n(\xb) |} w^{-\frac{1}{2}}(\vb) \alpha_g
\\& \lesssim \Vert w h^{k} \Vert_{L^\infty_{x,v}} + | w f_b |_{L^\infty_{\p\O,v}} + | w \p_{\mathbf{x}_p,v} f_b |_{L^\infty_{\p\O,v}} + |f_b|_{L^2_{\gamma_-}}.
\end{split}
\ee

For $\alpha_g \mathcal{D}(x,v)$, from \eqref{k_nabla_express_ell_step1} we obtain
\be \notag
\begin{split}
| \mathcal{D}(x,v) |
\lesssim \int^t_{t-\tb} e^{-\int^t_s \frac{\tilde{\nu}(X(\tau),V(\tau))\dd \tau}{2}}\int_{\mathbb{R}^3} \mathbf{k}(V(s),u) \frac{w_{\tilde{\theta}}(V(s))}{w_{\tilde{\theta}}(u)}  |w_{\tilde{\theta}}(u)\p_{x} h^{k+1} (X(s),u)| \dd u \dd s.
\end{split}
\ee
Similarly to the estimates in the proof of Proposition \ref{prop:weight_C1}, and using the estimates \eqref{ABCD_notation_ell_D_step1}-\eqref{alpha_C_bdd_infty_ell_step1} and the assumptions \eqref{S1_assumption} and \eqref{S1_assumption_reg} on $S_1$, we derive
\be \label{alpha_D_bdd_infty_ell_step1}
\begin{split}
| \alpha_g \mathcal{D}(x,v)| 
& \lesssim o(1) \big[ \Vert \nabla_x h^{k} \Vert_{L^p_{x,v}} + \Vert \alpha_g \nabla_x h^{k} \Vert_{L^\infty_{x,v}} \big] + \Vert w h^{k+1} \Vert_{L^\infty_{x,v}} + \Vert w h^{k} \Vert_{L^\infty_{x,v}} 
\\& \qquad + o(1) \Vert \alpha_g w_{\tilde{\theta}} (v) \p_{x,v} h^{k+1} \Vert_{L^\infty_{x,v}} + \Vert w_{\tilde{\theta}}(v) \p_{x,v} h^{k+1} \Vert_{L^p_{x,v}}
\\& \qquad + (C + 1) \big[ |f_b|_{L^2_{\gamma_-}} + | w f_b |_{L^\infty_{\partial\Omega,v}} + | w \p_{\mathbf{x}_p,v} f_b|_{L^\infty_{\p\O,v}} \big].
\end{split}
\ee
Combing \eqref{eq5:steady_small_lambda_step1_reg}-\eqref{alpha_D_bdd_infty_ell_step1} with Proposition \ref{prop:steady_small_lambda_step1}, we derive that
\be \label{alpha_linfty_ell_step1}
\begin{split}
\Vert w_{\tilde{\theta}} \alpha_g \p_{x,v} h^{k+1} \Vert_{L^\infty_{x,v}}  
& \lesssim o(1) \big[ \Vert \nabla_x h^{k} \Vert_{L^p_{x,v}} + \Vert \alpha_g \nabla_x h^{k} \Vert_{L^\infty_{x,v}} \big] + \Vert w_{\tilde{\theta}}(v) \p_{x,v} h^{k+1} \Vert_{L^p_{x,v}} 
\\& \qquad + (C + 1) \big[ |f_b|_{L^2_{\gamma_-}} + | w f_b |_{L^\infty_{\partial\Omega,v}} + | w \p_{\mathbf{x}_p,v} f_b|_{L^\infty_{\p\O,v}} \big] .
\end{split}
\ee
Since \eqref{eq3:steady_small_lambda_step1_reg} holds for $\ell = k$, \eqref{ABCD_bdd_lp_ell_step1} and \eqref{alpha_linfty_ell_step1} imply that \eqref{eq3:steady_small_lambda_step1_reg} holds for $\ell = k+1$, and therefore we conclude the uniform-in-$\ell$ estimates \eqref{eq3:steady_small_lambda_step1_reg} for every $\ell \geq 0$.

\smallskip

\textbf{Step 3. Weighted $C^1_v$ estimate.}
Under the construction \eqref{eq1:steady_small_lambda_step1_reg}, we claim that for any $\ell \in \N$,
\be \label{eq6:steady_small_lambda_step1_reg}
\Vert w_{\tilde{\theta}} \nabla_v h^{\ell} \Vert_{L^\infty_{x,v}} 
\lesssim (C+1) \big[ |f_b|_{L^2_{\gamma_-}} + | w f_b |_{L^\infty_{\partial\Omega,v}} + | w \p_{\mathbf{x}_p,v} f_b|_{L^\infty_{\p\O,v}} \big].
\ee
We omit the details, since the proof follows from arguments analogous to those used in the proof of Proposition~\ref{prop:C1v}.
Using induction, we obtain that for any $\ell \in \N$,
\be \notag
\Vert w_{\tilde{\theta}} \nabla_v h^{\ell} \Vert_{L^\infty_{x,v}} \lesssim \sum\limits^{\ell+1}_{m=\ell} \big( \Vert w_{\tilde{\theta}} \alpha_g \p_{x,v} h^{m} \Vert_{L^\infty_{x,v}} + o(1) \Vert \nabla_x h^{m} \Vert_{L^p_{x,v}} \big) + |f_b|_{L^2_{\gamma_-}} + | w f_b |_{L^\infty_{\partial\Omega,v}} + | w \p_{\mathbf{x}_p,v} f_b|_{L^\infty_{\p\O,v}}.
\ee
Combining with the estimate \eqref{eq3:steady_small_lambda_step1_reg}, we conclude the uniform-in-$\ell$ estimate \eqref{eq6:steady_small_lambda_step1_reg}.

\smallskip

\textbf{Step 4. Regularity: Proof of \eqref{regularity:weight_w1p_step1} and \eqref{regularity:weight_c1_step1}.}
From Proposition \ref{prop:steady_small_lambda_step1}, there exists a unique solution $h (x, v) \in L^2_{x,\nu} (\O \times \R^3) \cap L^2 (\gamma_+)$ to the system \eqref{eqtn:h_step1}–\eqref{bdry:phi_step1}, such that
\be \notag
\begin{split}
h^{\ell} \to h 
\ \text{ in } L^2_{x,\nu} (\O \times \R^3) \cap L^2 (\gamma_+)
\ \text{ as $\ell \to \infty$}.
\end{split}
\ee
This, together with the uniform-in-$\ell$ estimates \eqref{eq3:steady_small_lambda_step1_reg} and \eqref{eq6:steady_small_lambda_step1_reg}, implies \eqref{regularity:weight_w1p_step1} and \eqref{regularity:weight_c1_step1}.
\end{proof}

\begin{remark} \label{rmk:steady_small_lambda_step1}

Under the same assumptions, Proposition \ref{prop:steady_small_lambda_step1_reg} remains valid for all $0 \leq \lambda \leq 1$, not only for sufficiently small $\lambda$.
\end{remark}

\smallskip

\textbf{Step 2.}
To prove Theorem \ref{thm:well_poseness}, in the following proposition we increase $\lambda$ in Propositions \ref{prop:steady_small_lambda_step1} and \ref{prop:steady_small_lambda_step1_reg} to $1$ .

\begin{proposition} \label{prop:h_step2}

Suppose the inflow condition \eqref{inflow_condition} holds.
Assume the function $(g, \phi_g)$ satisfies that $g |_{\gamma_-} = f_b (x, v)$ and 
\be \label{g_assumption_reg_step2}
\begin{split}
| g |_{L^2_{\gamma_+}} + \Vert g \Vert_{L^2_{x,\nu}} + \Vert w g \Vert_{L^\infty_{x,v}}
& \lesssim  |f_b|_{L^2_{\gamma_-}} + | w f_b|_{L^\infty_{\p\O,v}},
\\ \Vert w_{\tilde{\theta}} \p_{x,v} g \Vert_{L^p_{x,v}} + \Vert w_{\tilde{\theta}} \alpha_g \p_{x,v} g \Vert_{L^\infty_{x,v}} + \Vert w_{\tilde{\theta}} \nabla_v g \Vert_{L^\infty_{x,v}}
& \lesssim |f_b|_{L^2_{\gamma_-}} + | w f_b |_{L^\infty_{\partial\Omega,v}} + | w \p_{\mathbf{x}_p,v} f_b|_{L^\infty_{\p\O,v}}, 
\end{split}
\ee
with $\phi_g$ defined by
\be \label{condition:phi_g_step2}
\begin{cases}
& - \Delta \phi_g = e^{-\phi_E/2} \int_{\R^3} g \sqrt{\mu} \dd v 
\text{ in } \O, \\[5pt]
& \phi_g = 0 
\text{ on } \p\O.
\end{cases}
\ee
For any non-negative $C \geq 0$, let $\{ S_k (x,v) \}^{\infty}_{k=1}$ be a sequence of functions satisfying that for all $k \in \mathbb{Z}_+$,
\begin{equation} \label{Sk_assumption}
\begin{split}
\Vert S_k \Vert_{L^2_{x,v}}
& \lesssim C[ |f_b|_{L^2_{\gamma_-}} + | w f_b|_{L^\infty_{\p\O,v}}],
\\ \Vert w S_k \Vert_{L^\infty_{x,v}} 
& \lesssim C [ |f_b|_{L^2_{\gamma_-}} +  | w f_b|_{L^\infty_{\p\O,v}}].
\end{split}
\end{equation}
Let $0 < \lambda = \lambda(\Omega) \ll 1$ be sufficiently small as in Proposition \ref{prop:steady_small_lambda_step1}. 
Then for all $k \in \mathbb{Z}_+$ with $k \lambda \leq 1$, the following system
\begin{align}
v \cdot \nabla_x h
& - \nabla_x (\phi_g + \phi_E) \cdot \nabla_{v} h 
+ \frac{v \cdot \nabla_x \phi_g}{2} h + e^{-\phi_E} \mathcal{L} h
\notag \\
& = - k \lambda (v \cdot \nabla_x \phi_h) e^{-\phi_E/2} \sqrt{\mu} + e^{-\phi_E/2} \Gamma(g, g) + S_k, 
\label{eqtn:h_step2} \\
h |_{\gamma_-} & = f_b (x, v), 
\label{bdry:h_step2} \\
- \Delta \phi_h & = e^{-\phi_E/2} \int_{\R^3} h \sqrt{\mu} \dd v \text{ in } \O, 
\label{eqtn:phi_step2} \\
\phi_h & = 0 \text{ on } \p\O.
\label{bdry:phi_step2}
\end{align}
admits a unique solution $h (x, v)$ such that
\be \label{h_L2_wh_Linfty_step2}
\begin{split}
| h |_{L^2_{\gamma_+}} + \Vert h \Vert_{L^2_{x,\nu}}
& \lesssim (C+1) [ |f_b|_{L^2_{\gamma_-}} + | w f_b|_{L^\infty_{\p\O,v}}],
\\ \Vert w h \Vert_{L^\infty_{x,v}} 
& \lesssim (C+1) [  |f_b|_{L^2_{\gamma_-}} + | w f_b|_{L^\infty_{\p\O,v}} ].
\end{split}
\ee
Furthermore, assume that the sequence of functions $\{ S_k (x,v) \}^{\infty}_{k=1}$ satisfies that for all $k \in \mathbb{Z}_+$,
\begin{equation} \label{Sk_assumption_reg}
\begin{split}
\Vert w \p_{x,v} S_k \Vert_{L^\infty_{x,v}}
& \lesssim C \big[ |f_b|_{L^2_{\gamma_-}} + | w f_b |_{L^\infty_{\partial\Omega,v}} + | w \p_{\mathbf{x}_p,v} f_b|_{L^\infty_{\p\O,v}} \big].
\end{split}
\end{equation}
Then, for all $k \in \mathbb{Z}_+$ with $k \lambda \leq 1$, the corresponding solution $h (x, v)$ satisfies
\be \label{regularity:wh_c1_step2}
\begin{split}
\Vert w_{\tilde{\theta}} \p_{x,v} h \Vert_{L^p_{x,v}} + \Vert w_{\tilde{\theta}} \alpha_g \p_{x,v} h \Vert_{L^\infty_{x,v}} 
& \lesssim (C+1) \big[ |f_b|_{L^2_{\gamma_-}} + | w f_b |_{L^\infty_{\partial\Omega,v}} + | w \p_{\mathbf{x}_p,v} f_b|_{L^\infty_{\p\O,v}} \big],
\\
\Vert w_{\tilde{\theta}}\nabla_v h \Vert_{L^\infty_{x,v}} 
& \lesssim (C+1) \big[ |f_b|_{L^2_{\gamma_-}} + | w f_b |_{L^\infty_{\partial\Omega,v}} + | w \p_{\mathbf{x}_p,v} f_b|_{L^\infty_{\p\O,v}} \big],
\end{split}
\ee
where the weight function $\alpha_g$ is defined in \eqref{alpha_weight_steady_step1}.
\end{proposition} 

\begin{proof}

The strategy for proving this proposition follows that of Propositions \ref{prop:steady_small_lambda_step1} and \ref{prop:steady_small_lambda_step1_reg}.
Owing to the similarity of the arguments, we omit some details.

We prove this Proposition by induction.
For the base case $k=1$, the result follows from Propositions \ref{prop:steady_small_lambda_step1} and \ref{prop:steady_small_lambda_step1_reg}.
Assume that the result holds for all $1 \leq k \leq n$.
We now consider the case $k = n+1$, and assume that $(n+1) \lambda \leq 1$.

\smallskip

\textbf{Step 1. Construction.}
We construct the solution to the system \eqref{eqtn:h_step2}–\eqref{bdry:phi_step2} for $k = n+1$ via the following sequences: for any $\ell \in \N$,
\be \label{eq2:steady_small_lambda_step2}
\begin{split}
v \cdot \nabla_x h^{\ell+1} 
& - \nabla_x (\phi_g + \phi_E) \cdot \nabla_{v} h^{\ell+1} 
+ \frac{v \cdot \nabla_x \phi_g}{2} h^{\ell+1} + e^{-\phi_E} \mathcal{L} h^{\ell+1}
\\ & = - n \lambda (v \cdot \nabla_x \phi^{\ell+1}_h) e^{-\phi_E/2} \sqrt{\mu} + e^{-\phi_E/2} \Gamma(g, g) - \lambda (v \cdot \nabla_x \phi_{h}^{\ell}) e^{-\phi_E/2} \sqrt{\mu} + S_{n+1},
\\ h^{\ell+1} |_{\gamma_-} & = f_b (x, v), 
\\ - \Delta \phi^\ell_h & = e^{-\phi_E/2} \int_{\R^3} h^{\ell} \sqrt{\mu} \dd v \text{ in } \O, 
\\ \phi^\ell_h & = 0 \text{ on } \p\O,
\end{split}
\ee
where the initial setting $h^0 = 0$ and $\nabla_x \phi^0_h = \mathbf{0}$. 

\smallskip

\textbf{Step 2. $L^2-L^\infty$ estimate, existence, and uniqueness.}
Under the construction \eqref{eq2:steady_small_lambda_step2}, we claim that for any $\ell \in \N$,
\be \label{induction_hypothesis_step2}
\begin{split}
| h^{\ell} |_{L^2_{\gamma_+}} + \Vert h^{\ell} \Vert_{L^2_{x,\nu}} 
& \lesssim (C+1) [|f_b|_{L^2_{\gamma_-}} + |wf_b|_{L^\infty_{x,v}}],
\\ \Vert w h^\ell \Vert_{L^\infty_{x,v}} 
& \lesssim (C+1) [  |f_b|_{L^2_{\gamma_-}} + | w f_b|_{L^\infty_{\p\O,v}} ].
\end{split}
\ee

We prove \eqref{induction_hypothesis_step2} by induction.
From the initial setting, \eqref{induction_hypothesis_step2} holds for $\ell = 0$.
Suppose that \eqref{induction_hypothesis_step2} holds for all $0 \leq \ell \leq m$.
We now consider the case when $\ell = m+1$.

Define that $\tilde{S}_{n+1} := - \lambda (v \cdot \nabla_x \phi_{h}^{m}) e^{-\phi_E/2} \sqrt{\mu} + S_{n+1}$.
Since \eqref{induction_hypothesis_step2} holds for $\ell = m$, it follows from the assumption \eqref{Sk_assumption} on $S_{n+1}$ and Lemma \ref{lemma:phi_x_infinity} that
\be \notag
\Vert \tilde{S}_{n+1} \Vert_{L^2_{x,v}} 
\lesssim \underbrace{\big( C + \lambda(C+1) \big)}_{:= C'} [ |f_b|_{L^2_{\gamma_-}}+|wf_b|_{L^\infty_{x,v}} ].
\ee
We rewrite the first equation in \eqref{eq2:steady_small_lambda_step2} as follows:
\be \label{eq1:steady_small_lambda_step2}
\begin{split}
v \cdot \nabla_x h^{m+1} 
& - \nabla_x (\phi_g + \phi_E) \cdot \nabla_{v} h^{m+1} 
+ \frac{v \cdot \nabla_x \phi_g}{2} h^{m+1} + e^{-\phi_E} \mathcal{L} h^{m+1}
\\ & = - n \lambda (v \cdot \nabla_x \phi^{m+1}_h) e^{-\phi_E/2} \sqrt{\mu} + e^{-\phi_E/2} \Gamma(g, g) + \tilde{S}_{n+1},
\end{split}
\ee
Since this proposition is assumed to hold for $k = n$ with any constant $C \geq 0$, then the construction \eqref{eq2:steady_small_lambda_step2} is well-posed, and \eqref{h_L2_wh_Linfty_step2} implies that
\be \notag
\begin{split}
| h^{m+1} |_{L^2_{\gamma_+}} + \Vert h^{m+1} \Vert_{L^2_{x,\nu}}
& \lesssim (C' + 1) [ |f_b|_{L^2_{\gamma_-}}+| w f_b|_{L^\infty_{x,v}} ],
\\ \Vert w h^{m+1} \Vert_{L^\infty_{x,v}} 
& \lesssim (C' + 1) [  |f_b|_{L^2_{\gamma_-}} + | w f_b|_{L^\infty_{\p\O,v}} ].
\end{split}
\ee

$(a)$
Analogously to the energy estimates \eqref{eq2:steady_small_lambda_step1} and \eqref{eq3:steady_small_lambda_step1}, and the macroscopic estimates \eqref{eq4:steady_small_lambda_step1} and \eqref{eq5:steady_small_lambda_step1} in Proposition \ref{prop:steady_small_lambda_step1} together with Lemma \ref{lemma:Unif_h_L2}, from \eqref{eq2:steady_small_lambda_step2} we obtain that 
\be \notag
\begin{split}
| h^{m+1} |_{L^2_{\gamma_+}}^2 + \Vert (\mathbf{I}-\mathbf{P}) h^{m+1} \Vert_{L^2_{x,\nu}}^2  
& \lesssim |f_b|_{L^2_{\gamma_-}}^2 + \Vert \nu^{-1/2} \Gamma (g, g) \Vert_{L^2_{x,v}}^2 + \Vert S_{n+1} \Vert_{L^2_{x,v}}^2 + o (1) \Vert h^{m+1} \Vert_{L^2_{x,\nu}}^2 
\\& \qquad + \Vert w g \Vert_{L^\infty_{x,v}} \Vert h^{m+1} \Vert_{L^2_{x,\nu}}^2 + n \lambda \Vert \nabla_x \phi_E \Vert_{L^\infty_x} \big( \Vert h^{m} \Vert_{L^2_{x,v}}^2 + \Vert h^{m+1} \Vert_{L^2_{x,v}}^2 \big),
\\ \Vert \mathbf{P} h^{m+1} \Vert_{L^2_{x,\nu}}^2 + n^2 \lambda^2 \Vert \nabla_x \phi_h^{m+1}\Vert_{L^2_x}^2
& \lesssim \Vert (\mathbf{I}-\mathbf{P}) h^{m+1} \Vert_{L^2_{x,\nu}}^2 + | h^{m+1} |_{L^2_{\gamma_+}}^2 + \Vert \nu^{-1/2} \Gamma (g, g) \Vert_{L^2_{x,v}}^2 + \Vert S_{n+1} \Vert_{L^2_{x,v}}^2
\\& \qquad + |f_b|_{L^2_{\gamma_-}}^2 + \lambda^2 \Vert h^{m} \Vert_{L^2_{x, v}}^2 + \big( \delta_0 + \Vert w g \Vert_{L^\infty_{x,v}} + \Vert \nabla_x \phi_E \Vert_{L^\infty_{x}} \big) \Vert h^{m+1} \Vert_{L^2_{x,v}}^2,
\end{split}
\ee
where the terms $n^2\lambda^2\Vert \nabla_x \phi_h^{m+1}\Vert_{L^2_x}^2$ and $\delta_0\Vert h^{m+1}\Vert_{L^2_{x,v}}^2$ appear in the second equation as a consequence of the same computation as in Lemma \ref{lemma:macro_l2}.
This, together with the fact that $(n+1) \lambda \leq 1$ and the assumptions \eqref{g_assumption_reg_step2} and \eqref{Sk_assumption} on $(g, \phi_g)$ and $S_{n+1}$, further implies that
\be \label{eq3:steady_small_lambda_step2}
\begin{split}
& | h^{m+1} |_{L^2_{\gamma_+}} + \Vert h^{m+1} \Vert_{L^2_{x,\nu}}
\\& \lesssim |f_b|_{L^2_{\gamma_-}} + \Vert \nu^{-1/2} \Gamma (g, g) \Vert_{L^2_{x,v}} + \Vert S_{n+1} \Vert_{L^2_{x,v}} + \Vert w g \Vert_{L^\infty_{x,v}}^{1/2} \Vert h^{m+1} \Vert_{L^2_{x,\nu}} + \lambda  \Vert h^{m} \Vert_{L^2_{x, v}}   
\\& \qquad + \Vert \nabla_x \phi_E \Vert^{1/2}_{L^\infty_x} \big( \Vert h^{m} \Vert_{L^2_{x,v}} + \Vert h^{m+1} \Vert_{L^2_{x,v}}  \big) + \big( \delta_0 + \Vert w g \Vert_{L^\infty_{x,v}}^{1/2} + \Vert \nabla_x \phi_E \Vert_{L^\infty_x}^{1/2} \big) \Vert h^{m+1} \Vert_{L^2_{x,v}}
\\& \lesssim (C+1) [|f_b|_{L^2_{\gamma_-}} + |wf_b|_{L^\infty_{x,v}}].
\end{split}
\ee
Thus, the first $L^2$ estimate in \eqref{induction_hypothesis_step2} holds for all $\ell \in \N$.

\smallskip

$(b)$
Analogously to the estimates \eqref{eq7:steady_small_lambda_step1}-\eqref{eq9:steady_small_lambda_step1} in Proposition \ref{prop:steady_small_lambda_step1} together with Lemma \ref{lemma:Unif_wh_Linfty}, from \eqref{eq2:steady_small_lambda_step2} we obtain that 
\be \notag
\begin{split}
\Vert w h^{m+1} \Vert_{L^\infty_{x,v}} 
& \lesssim \frac{\lambda}{4} \Vert w h^{m} \Vert_{L^\infty_{x,v}} + \lambda \Vert h^{m} \Vert_{L^2_{x,v}} + \Vert \nu^{-1} w \Gamma (g, g) \Vert_{L^\infty_{x,v}} + \Vert w S_{n+1} \Vert_{L^\infty_{x,v}} + \Vert h^{m+1} \Vert_{L^2_{x,v}}.
\end{split}
\ee
This, together with Lemma \ref{lemma:gamma} and the assumptions  \eqref{g_assumption_reg_step2} and \eqref{Sk_assumption}, further implies that
\be \notag
\begin{split}
\Vert w h^{m+1} \Vert_{L^\infty_{x,v}} 
& \lesssim (C+1) [ |f_b|_{L^2_{\gamma_-}} + | w f_b|_{L^\infty_{\p\O,v}}].
\end{split} 
\ee
Therefore, the second $L^\infty$ estimate in \eqref{induction_hypothesis_step2} holds for all $\ell \in \N$.

\smallskip

$(c)$
By direct computation, the equation of $h^{\ell+1} - h^{\ell}$ satisfies that for any $\ell \in \N$,
\be \notag
\begin{split}
v \cdot \nabla_x (h^{\ell+1} - h^{\ell}) & - \nabla_x (\phi_g + \phi_E) \cdot \nabla_{v}(h^{\ell+1} - h^{\ell}) + \frac{v \cdot \nabla_x \phi_g}{2} (h^{\ell+1} - h^{\ell}) + e^{-\phi_E} \mathcal{L}(h^{\ell+1} - h^{\ell}) 
\\& =  - n \lambda v \cdot \nabla_x (\phi^{\ell+1}_h - \phi^{\ell}_h) e^{-\phi_E/2} \sqrt{\mu}- \lambda v \cdot \nabla_x (\phi^{\ell}_h - \phi^{\ell - 1}_h) e^{-\phi_E/2} \sqrt{\mu},
\\ h^{\ell+1} - h^{\ell} |_{\gamma_-} & = 0,
\\ - \Delta ( \phi^\ell_h - \phi^{\ell - 1}_h ) & = e^{-\phi_E/2} \int_{\R^3} ( h^{\ell} - h^{\ell - 1} )\sqrt{\mu} \dd v \text{ in } \O, 
\\ \phi^\ell_h - \phi^{\ell - 1}_h & = 0 \text{ on } \p\O.
\end{split}
\ee
Following the estimates \eqref{eq17:steady_small_lambda_step1} and \eqref{eq18:steady_small_lambda_step1}, together with the energy estimate \eqref{eq3:steady_small_lambda_step2}, we obtain that for any $\ell \in \N$,
\be \notag
\begin{split}
& | h^{\ell+1} - h^{\ell} |_{L^2_{\gamma_+}}^2 + \Vert h^{\ell+1} - h^{\ell} \Vert_{L^2_{x,\nu}}^2 + n^2 \lambda^2 \Vert \nabla_x \phi_h^{\ell+1} - \nabla_x \phi_h^{\ell}\Vert_{L^2_x}
\\& \lesssim \big(\e+ \Vert \nabla_x \phi_g \Vert_{L^\infty_{x}} + \Vert w g \Vert_{L^\infty_{x,v}} + n \lambda \Vert \nabla_x \phi_E \Vert_{L^\infty_x} \big) \Vert h^{\ell+1} - h^{\ell} \Vert_{L^2_{x,v}}^2 + \big( n \lambda \Vert \nabla_x \phi_E \Vert_{L^\infty_x} + \lambda^2 \big) \Vert h^{\ell} - h^{\ell-1} \Vert_{L^2_{x,v}}^2.
\end{split}
\ee
This, together with \eqref{eq3:steady_small_lambda_step2} and the elliptic regularity estimate, implies that  $\{ h^{\ell} \}^{\infty}_{\ell = 0}$ and $\{ \nabla_x \phi^\ell_h \}^{\infty}_{\ell = 0}$ are Cauchy sequences in $L^2_{x,\nu} (\O \times \R^3) \cap L^2 (\gamma_+)$ and $L^2 (\O)$,  respectively. Furthermore, there exists
\be \notag
h (x, v) \in L^2_{x,\nu} (\O \times \R^3) \cap L^2 (\gamma_+)
\ \text{ and } \
\nabla_x \phi_h (x) \in L^2 (\O)
\text{ with }
\phi_h = 0 \text{ on } \p\O,
\ee
such that
\be \label{eq4:steady_small_lambda_step2}
\begin{split}
h^{\ell} 
& \to h 
\ \text{ in } L^2_{x,\nu} (\O \times \R^3) \cap L^2 (\gamma_+)
\ \text{ as $\ell \to \infty$},
\\ \nabla_x \phi^{\ell}_h (x) 
& \to \nabla_x \phi_h (x)
\ \text{ in } L^2 (\O)
\ \text{ as $\ell \to \infty$}.
\end{split}
\ee
Following the proof of Proposition \ref{prop:steady_small_lambda_step1}, the uniform-in-$\ell$ estimates for $\{ h^{\ell} \}^{\infty}_{\ell = 0}$ implies that $(h, \nabla_x \phi_h)$ is the unique weak solution of \eqref{eqtn:h_step2}-\eqref{bdry:phi_step2}.
Finally, combining the $L^2$ convergence in \eqref{eq4:steady_small_lambda_step2} with the uniform-in-$\ell$ estimates for $\{ h^{\ell} \}^{\infty}_{\ell = 0}$ in \eqref{induction_hypothesis_step2}, we conclude \eqref{h_L2_wh_Linfty_step2}.

\smallskip

\textbf{Step 3. $W^{1,p}$ estimate and weighted $C^1$ estimate.}
$(a)$
Under the construction \eqref{eq2:steady_small_lambda_step2}, we claim that for any $\ell \in \N$,
\be \label{induction_hypothesis_step2_reg}
\begin{split}
\Vert w_{\tilde{\theta}} \p_{x,v} h^{\ell} \Vert_{L^p_{x,v}} + \Vert w_{\tilde{\theta}} \alpha_g \p_{x,v} h ^{\ell}\Vert_{L^\infty_{x,v}} 
& \lesssim (C+1) \big[ |f_b|_{L^2_{\gamma_-}} + | w f_b |_{L^\infty_{\partial\Omega,v}} + | w \p_{\mathbf{x}_p,v} f_b|_{L^\infty_{\p\O,v}} \big],
\\
\Vert w_{\tilde{\theta}} \nabla_v h^{\ell} \Vert_{L^\infty_{x,v}} 
& \lesssim (C+1) \big[ |f_b|_{L^2_{\gamma_-}} + | w f_b |_{L^\infty_{\partial\Omega,v}} + | w \p_{\mathbf{x}_p,v} f_b|_{L^\infty_{\p\O,v}} \big].
\end{split}
\ee

Analogously to Step 2, we prove \eqref{induction_hypothesis_step2_reg} by induction.
From the initial setting, \eqref{induction_hypothesis_step2_reg} holds for $\ell = 0$. Suppose that \eqref{induction_hypothesis_step2_reg} holds for all $0 \leq \ell \leq m$. We now consider the case when $\ell = m+1$.

Recall that $\tilde{S}_{n+1} = - \lambda (v \cdot \nabla_x \phi_{h}^{m}) e^{-\phi_E/2} \sqrt{\mu} + S_{n+1}$.
Since \eqref{induction_hypothesis_step2_reg} holds for $\ell = m$, it follows from the assumption \eqref{Sk_assumption_reg} on $S_{n+1}$, \eqref{h_L2_wh_Linfty_step2} and Lemma \ref{lemma:phi_C2} that
\be \notag
\Vert w \p_{x,v} \tilde{S}_{n+1} \Vert_{L^2_{x,v}} 
\lesssim \underbrace{\big( C + \lambda(C+1) \big)}_{:= C''} \big[ |f_b|_{L^2_{\gamma_-}} + | w f_b |_{L^\infty_{\partial\Omega,v}} + | w \p_{\mathbf{x}_p,v} f_b|_{L^\infty_{\p\O,v}} \big]..
\ee
Since this proposition is assumed to hold for $k = n$ with any constant $C \geq 0$, then \eqref{regularity:wh_c1_step2} implies that
\be \label{eq5:steady_small_lambda_step2}
\begin{split}
\Vert w_{\tilde{\theta}} \p_{x,v} h^{m+1} \Vert_{L^p_{x,v}} + \Vert w_{\tilde{\theta}} \alpha_g \p_{x,v} h^{m+1} \Vert_{L^\infty_{x,v}} 
& \lesssim (C'' + 1) [ |f_b|_{L^2_{\gamma_-}}+| w f_b|_{L^\infty_{x,v}} ],
\\ \Vert w_{\tilde{\theta}}\nabla_v h^{m+1} \Vert_{L^\infty_{x,v}} 
& \lesssim (C'' + 1) [  |f_b|_{L^2_{\gamma_-}} + | w f_b|_{L^\infty_{\p\O,v}} ].
\end{split}
\ee

Similarly to Step 2, the remaining proof follows analogously from the estimates in Proposition \ref{prop:steady_small_lambda_step1_reg}, together with the assumption \eqref{Sk_assumption_reg} on $S_{n+1}$, Remark \ref{rmk:steady_small_lambda_step1}, and the fact that $(n+1) \lambda \leq 1$. Therefore, we conclude \eqref{induction_hypothesis_step2_reg} for all $\ell \in \N$.
Finally, combining the $L^2$ convergence in \eqref{eq4:steady_small_lambda_step2} with the uniform-in-$\ell$ estimates for $\{ h^{\ell} \}^{\infty}_{\ell = 0}$ in \eqref{induction_hypothesis_step2_reg}, we conclude \eqref{regularity:wh_c1_step2}.
\end{proof}

\begin{proof}[\textbf{Proof of Theorem \ref{thm:well_poseness}}]

First, choose a sufficiently large $N \in\mathbb{Z}^+$ such that $\lambda = \frac{1}{N}$ is small enough to satisfy the smallness assumption in Proposition \ref{prop:steady_small_lambda_step1}.
Assume that the function $(g, \phi_g)$ satisfies the hypotheses of Proposition \ref{prop:h_step2}. Let $S_N = 0$ and it satisfies
\begin{equation} \notag
\begin{split}
\Vert S_N \Vert_{L^2_{x,v}} + \Vert w S_N \Vert_{L^\infty_{x,v}} + \Vert w \p_{x,v} S_N \Vert_{L^\infty_{x,v}}
 = 0.
\end{split}
\end{equation}
Using Proposition \ref{prop:h_step2} and the fact that $N \lambda = 1$, there exists a unique solution to
\be \label{eq1:well_poseness}
\begin{split}
v \cdot \nabla_x h
& - \nabla_x (\phi_g + \phi_E) \cdot \nabla_{v} h 
+ \frac{v \cdot \nabla_x \phi_g}{2} h + e^{-\phi_E} \mathcal{L} h
\\& = - (v \cdot \nabla_x \phi_h) e^{-\phi_E/2} \sqrt{\mu} + e^{-\phi_E/2} \Gamma(g, g) + S_N,
\\ h |_{\gamma_-} & = f_b (x, v), 
\\ - \Delta \phi_h & = e^{-\phi_E/2} \int_{\R^3} h \sqrt{\mu} \dd v \text{ in } \O, 
\\ \phi_h & = 0 \text{ on } \p\O.
\end{split}
\ee
Moreover, the solution $h (x, v)$ satisfies
\be \label{eq2:well_poseness}
\begin{split}
| h |_{L^2_{\gamma_+}} + \Vert h \Vert_{L^2_{x,\nu}}
& \lesssim  |f_b|_{L^2_{\gamma_-}} + | w f_b|_{L^\infty_{\p\O,v}},
\\ \Vert w h \Vert_{L^\infty_{x,v}} 
& \lesssim  |f_b|_{L^2_{\gamma_-}} + | w f_b|_{L^\infty_{\p\O,v}},
\\ \Vert w_{\tilde{\theta}} \p_{x,v} h \Vert_{L^p_{x,v}} + \Vert w_{\tilde{\theta}} \alpha_g \p_{x,v} h \Vert_{L^\infty_{x,v}} 
& \lesssim |f_b|_{L^2_{\gamma_-}} + | w f_b |_{L^\infty_{\partial\Omega,v}} + | w \p_{\mathbf{x}_p,v} f_b|_{L^\infty_{\p\O,v}},
\\ \Vert w_{\tilde{\theta}}\nabla_v h \Vert_{L^\infty_{x,v}} 
& \lesssim |f_b|_{L^2_{\gamma_-}} + | w f_b |_{L^\infty_{\partial\Omega,v}} + | w \p_{\mathbf{x}_p,v} f_b|_{L^\infty_{\p\O,v}}.
\end{split}
\ee

Second, we replace $(g, \phi_g)$ by the sequence $\{ (h^{\ell}, \nabla_x \phi^\ell_h) \}^{\infty}_{\ell = 1}$ constructed in
\eqref{eqtn:h^l}–\eqref{bdry:phi^l}, and apply an iteration argument.
From the initial setting $h^0 = 0$ and $\nabla_x \phi^0_h = \mathbf{0}$, it is straightforward to verify that $(h^{0}, \nabla_x \phi^0_h)$ satisfies the hypotheses of Proposition \ref{prop:h_step2}. 
Then, by \eqref{eq1:well_poseness} and \eqref{eq2:well_poseness}, the construction
\eqref{eqtn:h^l}–\eqref{bdry:phi^l} yields the regularity estimates \eqref{reg:well_poseness_L2}–\eqref{reg:well_poseness_weight_c1} for $h^1$.
This further implies that $(h^{1}, \nabla_x \phi^1_h)$ also satisfies the hypotheses of Proposition \ref{prop:h_step2}.
Repeating this argument inductively, we obtain the uniform-in-$\ell$ estimates for $\{ h^{\ell} \}^{\infty}_{\ell = 0}$ in \eqref{reg:well_poseness_L2}-\eqref{reg:well_poseness_weight_c1}.
\end{proof}

\subsection{Existence of steady solutions and proof of Theorem \ref{thm:steady_wellpose}}
\label{sec:existence_construction}

We now combine the a priori estimates and well-posedness in Sections \ref{sec:regularity_construction} and \ref{sec:well_posedness_construction} to establish the existence of steady solutions and complete the proof of Theorem \ref{thm:steady_wellpose}.

We begin by applying the uniqueness argument developed in Section~\ref{sec:stationary_uniqueness} to show that the sequence $\{ h^{\ell} \}^{\infty}_{\ell = 0}$ is Cauchy in $L^2_{x,\nu} (\O \times \R^3) \cap L^2 (\gamma_+)$.

\begin{proposition} \label{prop:Unif_cauchy_seq}

Suppose the inflow condition \eqref{inflow_condition}.
Under the construction \eqref{eqtn:h^l}-\eqref{bdry:phi^l}, $(h^{\ell+1}, \nabla_x \phi^\ell_h)$ satisfies that for any $\ell \in \N$,
\be \label{est:Unif_cauchy_seq_h}
\begin{split}
& | h^{\ell+1} - h^{\ell} |_{L^2_{\gamma_+}}^2 + \Vert h^{\ell+1} - h^{\ell} \Vert_{L^2_{x,\nu}}^2
\\& \lesssim \Big[ \Vert w_{\tilde{\theta}}\nabla_{v} h^{\ell} \Vert_{L^p_{x,v}} + \Vert w h^{\ell+1} \Vert_{L^\infty_{x,v}} + \Vert w h^{\ell} \Vert_{L^\infty_{x,v}} + \Vert \nabla_x \phi_E \Vert_{L^\infty_x} + \big( \Vert w h^{\ell+1} \Vert_{L^\infty_{x,v}} + \Vert w h^{\ell} \Vert_{L^\infty_{x,v}} \big)^2 + \delta_0 \Big] 
\\& \qquad \times \Vert h^{\ell+1} - h^{\ell} \Vert_{L^2_{x,v}}^2 + \Big[ \Vert w_{\tilde{\theta}} \nabla_{v} h^{\ell} \Vert_{L^p_{x,v}} + \Vert w h^{\ell} \Vert_{L^\infty_{x,v}} \Big] \Vert h^{\ell} - h^{\ell-1} \Vert_{L^2_{x,v}}^2.
\end{split}
\ee
In addition, $\{ h^{\ell} \}^{\infty}_{\ell = 0}$ and $\{ \nabla_x \phi^\ell_h \}^{\infty}_{\ell = 0}$ are Cauchy sequences in $L^2_{x,\nu} (\O \times \R^3) \cap L^2 (\gamma_+)$ and $L^2 (\O)$,  respectively.
\end{proposition}

\begin{proof}

From \eqref{eqtn:h^l}-\eqref{bdry:phi^l}, the equation of $h^{\ell+1} - h^{\ell}$ satisfies that for any $\ell \in \N$,
\begin{equation} \label{eq:h^ell_diff}
\begin{split}
v \cdot \nabla_x (h^{\ell+1} - h^{\ell}) 
& - \nabla_x (\phi^\ell_h + \phi_E) \cdot \nabla_{v}(h^{\ell+1} - h^{\ell}) + \frac{v}{2} \cdot \nabla_x \phi^\ell_h (h^{\ell+1} - h^{\ell}) + e^{-\phi_E} \mathcal{L}(h^{\ell+1} - h^{\ell}) 
\\& = \nabla_x (\phi^{\ell}_h - \phi^{\ell - 1}_h) \cdot \nabla_{v} h^{\ell} - \frac{v}{2} \cdot \nabla_x (\phi^{\ell}_h - \phi^{\ell - 1}_h) h^{\ell} - v \cdot \nabla_x (\phi^{\ell+1}_h - \phi^{\ell}_h) e^{-\phi_E/2} \sqrt{\mu}
\\& \qquad + e^{-\phi_E/2} \big[ \Gamma (h^{\ell}, h^{\ell} ) - \Gamma (h^{\ell - 1}, h^{\ell - 1} ) \big],
\\ h^{\ell+1} - h^{\ell} |_{\gamma_-} & = 0,
\\ - \Delta ( \phi^\ell_h - \phi^{\ell - 1}_h ) & = e^{-\phi_E/2} \int_{\R^3} ( h^{\ell} - h^{\ell - 1} )\sqrt{\mu} \dd v \text{ in } \O, 
\\ \phi^\ell_h - \phi^{\ell - 1}_h & = 0 \text{ on } \p\O.
\end{split}
\end{equation}

From \eqref{eq:h^ell_diff}, the $L^2_{x,v}$ energy estimate yields
\begin{align}
& | h^{\ell+1} - h^{\ell} |_{L^2_{\gamma_+}}^2 + \Vert e^{-\phi_E/2}(\mathbf{I}-\mathbf{P}) (h^{\ell+1} - h^{\ell}) \Vert_{L^2_{x,\nu}}^2 
\lesssim \Vert \nabla_x \phi^{\ell}_h \Vert_{L^\infty_{x}} \Vert h^{\ell+1} - h^{\ell} \Vert_{L^2_{x,\nu}}^2 
\label{est1:h^ell_1} \\
& + \Vert \nabla_x (\phi^{\ell}_h - \phi^{\ell - 1}_h) \Vert_{L^2_x} \Vert \nabla_{v} h^{\ell} \Vert_{L^{\infty}_{x} L^2_v} \Vert h^{\ell+1} - h^{\ell} \Vert_{L^2_{x,v}} + \Vert w h^{\ell} \Vert_{L^\infty_{x,v}} \Vert h^{\ell+1} - h^{\ell} \Vert_{L^2_{x,v}} \Vert \nabla_x (\phi^{\ell}_h - \phi^{\ell - 1}_h) \Vert_{L^2_x}
\label{est1:h^ell_2} \\
& + \Big| \iint_{\O\times \mathbb{R}^3} v \cdot \nabla_x (\phi^{\ell+1}_h - \phi^{\ell}_h) e^{-\phi_E/2}\sqrt{\mu} (h^{\ell+1} - h^{\ell}) \dd x \dd v \Big|
\label{est1:h^ell_3} \\
& + \Vert e^{\phi_E/4}\nu^{-1/2} [ \Gamma (h^{\ell+1} - h^{\ell}, h^{\ell+1}) + \Gamma(h^{\ell}, h^{\ell+1} - h^{\ell}) ] \Vert_{L^2_{x,v}}^2 + o(1)\Vert e^{-\phi_E/2}(\mathbf{I}-\mathbf{P}) (h^{\ell+1} - h^{\ell}) \Vert_{L^2_{x,\nu}}^2.
\label{est1:h^ell_4}
\end{align}

First, we estimate \eqref{est1:h^ell_1}. 
From Lemma \ref{lemma:phi_x_infinity}, we have
\be \label{est2:h^ell_1}
\eqref{est1:h^ell_1} 
= \Vert \nabla_x \phi^{\ell}_h \Vert_{L^\infty_{x}} \Vert h^{\ell+1} - h^{\ell} \Vert_{L^2_{x,\nu}}^2 
\lesssim \Vert w h^{\ell} \Vert_{L^\infty_{x,v}} \Vert h^{\ell+1} - h^{\ell} \Vert_{L^2_{x,\nu}}^2.
\ee

Second, we estimate \eqref{est1:h^ell_2}.
By the Calderon-Zygmund inequality,
\be \label{est2:h^ell_2}
\Vert \nabla_x (\phi^{\ell}_h - \phi^{\ell - 1}_h) \Vert_{L^2_x} 
\lesssim \Vert h^{\ell} - h^{\ell - 1} \Vert_{L^2_{x,\nu}}.
\ee
This implies that
\be \notag
\Vert \nabla_x (\phi^{\ell}_h - \phi^{\ell - 1}_h) \Vert_{L^2_x} \Vert \nabla_{v} h^{\ell} \Vert_{L^\infty_{x} L^2_v} \Vert h^{\ell+1} - h^{\ell} \Vert_{L^2_{x,v}}
\lesssim \Vert w_{\tilde{\theta}}\nabla_{v} h^{\ell} \Vert_{L^{\infty}_{x,v}} \big( \Vert h^{\ell+1} - h^{\ell} \Vert_{L^2_{x,v}}^2 + \Vert h^{\ell} - h^{\ell - 1} \Vert_{L^2_{x,\nu}} \big).
\ee
and 
\be \notag
\Vert w h^{\ell} \Vert_{L^\infty_{x,v}} \Vert h^{\ell+1} - h^{\ell} \Vert_{L^2_{x,v}} \Vert \nabla_x (\phi^{\ell}_h - \phi^{\ell - 1}_h) \Vert_{L^2_x}
\lesssim \Vert w h^{\ell} \Vert_{L^\infty_{x,v}} \big( \Vert h^{\ell+1} - h^{\ell} \Vert_{L^2_{x,v}}^2 + \Vert h^{\ell} - h^{\ell - 1} \Vert_{L^2_{x,\nu}} \big).
\ee
Thus, we get
\be \label{est3:h^ell_2}
\eqref{est1:h^ell_2} 
\lesssim \big( \Vert w_{\tilde{\theta}}\nabla_{v} h^{\ell} \Vert_{L^{\infty}_{x,v}} + \Vert w h^{\ell} \Vert_{L^\infty_{x,v}} \big) \big( \Vert h^{\ell+1} - h^{\ell} \Vert_{L^2_{x,v}}^2 + \Vert h^{\ell} - h^{\ell - 1} \Vert_{L^2_{x,\nu}} \big).
\ee

Third, we estimate \eqref{est1:h^ell_3}.
Following Lemma \ref{lemma:l2_energy} and using \eqref{eq:h^ell_diff}, \eqref{est2:h^ell_2} and the Dirichlet boundary conditions of $\phi^{\ell+1}_h$ and $\phi^{\ell}_h$, we have
\be \label{est3:h^ell_3}
\eqref{est1:h^ell_3} 
\lesssim \Vert \nabla_x \phi_E \Vert_{L^\infty_x} \Vert h^{\ell+1} - h^{\ell} \Vert_{L^2_{x,v}}^2.
\ee

Last, we estimate \eqref{est1:h^ell_4}.
Using Lemma \ref{lemma:gamma}, we have 
\begin{equation} \label{est2:h^ell_4}
\begin{split}
& \Vert e^{\phi_E/4} \nu^{-1/2} [ \Gamma (h^{\ell+1} - h^{\ell}, h^{\ell+1}) + \Gamma(h^{\ell}, h^{\ell+1} - h^{\ell}) ]\Vert_{L^2_{x,v}}^2 
\\& \lesssim \Vert e^{\phi_E/4} \Vert_{L^\infty_x} \big( \Vert w h^{\ell+1} \Vert_{L^\infty_{x,v}} + \Vert w h^{\ell} \Vert_{L^\infty_{x,v}} \big)^2 \Vert h^{\ell+1} - h^{\ell} \Vert_{L^2_{x,v}}^2.
\end{split}
\end{equation}
Combining \eqref{est2:h^ell_1}, \eqref{est3:h^ell_2}, \eqref{est3:h^ell_3}, \eqref{est2:h^ell_4} with the assumption that $|e^{-\phi_E/2}-1|\ll 1$, we get
\begin{equation} \label{est1:h^ell_diff}
\begin{split}
& | h^{\ell+1} - h^{\ell} |_{L^2_{\gamma_+}}^2 + \Vert (\mathbf{I}-\mathbf{P}) (h^{\ell+1} - h^{\ell}) \Vert_{L^2_{x,\nu}}^2
\\& \lesssim \Big[ \Vert w_{\tilde{\theta}} \nabla_{v} h^{\ell} \Vert_{L^\infty_{x,v}} + \Vert w h^{\ell} \Vert_{L^\infty_{x,v}} + \Vert \nabla_x \phi_E \Vert_{L^\infty_x} + \big( \Vert w h^{\ell+1} \Vert_{L^\infty_{x,v}} + \Vert w h^{\ell} \Vert_{L^\infty_{x,v}} \big)^2 \Big] \Vert h^{\ell+1} - h^{\ell} \Vert_{L^2_{x,v}}^2
\\& \qquad + \Big[ \Vert w_{\tilde{\theta}} \nabla_{v} h^{\ell} \Vert_{L^\infty_{x,v}} + \Vert w h^{\ell} \Vert_{L^\infty_{x,v}} \Big] \Vert h^{\ell} - h^{\ell-1} \Vert_{L^2_{x,v}}^2.
\end{split}
\end{equation}

On the other hand, following the macroscopic estimates in Proposition \ref{prop:stationary_uniqueness} and Lemma \ref{lemma:Unif_h_L2}, we obtain that
\begin{equation} \label{est2:h^ell_diff}
\begin{split}
\Vert \mathbf{P} (h^{\ell+1} - h^{\ell}) \Vert_{L^2_{x,v}}^2
& \lesssim \Vert (\mathbf{I}-\mathbf{P}) (h^{\ell+1} - h^{\ell}) \Vert_{L^2_{x,\nu}}^2 + | h^{\ell+1} - h^{\ell} |_{L^2_{\gamma_+}}^2 + \delta_0 \Vert h^{\ell+1} - h^{\ell} \Vert_{L^2_{x,v}}^2
\\& \qquad + \big( \Vert w h^{\ell+1} \Vert_{L^\infty_{x,v}} + \Vert w h^{\ell} \Vert_{L^\infty_{x,v}} + \Vert \nabla_x \phi_E\Vert_{L^\infty_x} \big) \Vert h^{\ell+1} - h^{\ell} \Vert_{L^2_{x,v}}^2.
\end{split}
\end{equation}
By taking $\eqref{est1:h^ell_diff} \times 2 + \eqref{est2:h^ell_diff}$, we conclude \eqref{est:Unif_cauchy_seq_h} from Lemma \ref{lemma:Unif_steady}.
From the inflow condition \eqref{inflow_condition}, Theorem \ref{thm:well_poseness} shows that for any $\ell \in \N$,
\[
\Vert w_{\tilde{\theta}} \p_{x,v} h^{\ell} \Vert_{L^{\infty}_{x,v}} \lesssim |f_b|_{L^2_{\gamma_-}} + | wf_b|_{L^\infty_{\p\O,v}} + | w \p_{\mathbf{x}_p,v} f_b|_{L^\infty_{\p\O,v}} \ll 1.
\]
This, together with \eqref{est:Unif_cauchy_seq_h} and Lemma \ref{lemma:Unif_steady}, implies that for any $\ell \in \N$,
\be \notag
| h^{\ell+1} - h^{\ell} |_{L^2_{\gamma_+}}^2 + \Vert h^{\ell+1} - h^{\ell} \Vert_{L^2_{x,\nu}}^2
\leq \frac{1}{2} \Vert h^{\ell} - h^{\ell-1} \Vert_{L^2_{x,v}}^2,
\ee
and thus $\{ h^{\ell} \}^{\infty}_{\ell = 0}$ is a Cauchy sequence in $L^2_{x,\nu} (\O \times \R^3) \cap L^2 (\gamma_+)$.
Using the elliptic regularity estimate, we derive that for any $\ell \in \N$,
\be \notag
\Vert \nabla_x \phi^{\ell}_h - \nabla_x \phi^{\ell - 1}_h \Vert_{L^2_x} 
\lesssim \Vert h^{\ell} - h^{\ell - 1} \Vert_{L^2_{x, v}}.
\ee 
Therefore, $\{ \nabla_x \phi^\ell_h \}^{\infty}_{\ell = 0}$ forms a Cauchy sequence in $L^2 (\O)$.
\end{proof}

We are now ready to prove Theorem \ref{thm:steady_wellpose}. For the reader's convenience, we restate it below as Theorem \ref{thm:steady_solution}. 
The positivity of the solution, however, will be proved later in the proof of Theorem \ref{thm:dynamical_stability} in Section \ref{sec:dynamical_proof}.

\begin{theorem} \label{thm:steady_solution}

Suppose the inflow condition \eqref{inflow_condition} holds.
There exists a unique solution $h (x, v)$ to the steady problem \eqref{eqn:h} such that
\begin{align*}
v \cdot \nabla_x h
- \nabla_x (\phi_h + \phi_E) \cdot \nabla_{v} h 
+ \frac{v \cdot \nabla_x \phi_h}{2} h + e^{-\phi_E} \mathcal{L} h
& = - (v \cdot \nabla_x \phi_h) e^{-\phi_E/2} \sqrt{\mu} + e^{-\phi_E/2} \Gamma(h, h), 
\\
h |_{\gamma_-} & = f_b (x, v), 
\\
- \Delta \phi_h & = e^{-\phi_E/2} \int_{\R^3} h \sqrt{\mu} \dd v \text{ in } \O, 
\\
\phi_h & = 0 \text{ on } \p\O.
\end{align*}
Moreover, $(h, \nabla_x \phi_h )$ satisfies that 
\begin{align}
| h |_{L^2_{\gamma_+}} + \Vert h \Vert_{L^2_{x,\nu}}
& \lesssim  |f_b|_{L^2_{\gamma_-}} + | w f_b|_{L^\infty_{\p\O,v}},
\label{regularity:h_L2} \\ 
\Vert w h \Vert_{L^\infty_{x,v}} 
& \lesssim  |f_b|_{L^2_{\gamma_-}} + | w f_b|_{L^\infty_{\p\O,v}},
\label{regularity:wh_Linfty} \\ 
\Vert \nabla_x \phi_h \Vert_{L^\infty_{x}} 
& \lesssim  |f_b|_{L^2_{\gamma_-}} + | w f_b|_{L^\infty_{\p\O,v}},
\label{regularity:phi_x_Linfty} \\ 
\Vert w_{\tilde{\theta}} \nabla_v h \Vert_{L^\infty_{x,v}} 
& \lesssim |f_b|_{L^2_{\gamma_-}} + | wf_b|_{L^\infty_{\p\O,v}} + | w \p_{\mathbf{x}_p,v} f_b|_{L^\infty_{\p\O,v}}.
\label{regularity:nuh_v_Linfty}
\end{align}   
\end{theorem}

\begin{proof}

\textbf{Step 1. Regularity: Proof of \eqref{regularity:h_L2}-\eqref{regularity:phi_x_Linfty}.}
Using Lemma \ref{lemma:Unif_steady} and Proposition \ref{prop:Unif_cauchy_seq}, $\{ h^{\ell} \}^{\infty}_{\ell = 0}$ and $\{ \nabla_x \phi^\ell_h \}^{\infty}_{\ell = 0}$ from the construction \eqref{eqtn:h^l}-\eqref{bdry:phi^l} are Cauchy sequences in $L^2_{x,\nu} (\O \times \R^3) \cap L^2 (\gamma_+)$ and $L^2 (\O)$, respectively.
Thus, there exists
\be \notag
h (x, v) \in L^2_{x,\nu} (\O \times \R^3) \cap L^2 (\gamma_+)
\ \text{ and } \
\nabla_x \phi_h (x) \in L^2 (\O)
\text{ with }
\phi_h = 0 \text{ on } \p\O,
\ee
such that
\be \label{eq1:steady_small_lambda}
\begin{split}
h^{\ell} 
& \to h 
\ \text{ in } L^2_{x,\nu} (\O \times \R^3) \cap L^2 (\gamma_+)
\ \text{ as $\ell \to \infty$},
\\ \nabla_x \phi^{\ell}_h (x) 
& \to \nabla_x \phi_h (x)
\ \text{ in } L^2 (\O)
\ \text{ as $\ell \to \infty$}.
\end{split}
\ee
Combining the uniform-in-$\ell$ estimates for $\{ h^{\ell} \}^{\infty}_{\ell = 0}$ and $\{ \nabla_x \phi^\ell_h \}^{\infty}_{\ell = 0}$ in Lemma \ref{lemma:Unif_steady} with \eqref{eq1:steady_small_lambda}, we conclude \eqref{regularity:h_L2}, \eqref{regularity:wh_Linfty}, and \eqref{regularity:phi_x_Linfty}.

\smallskip

\textbf{Step 2. Existence.}
We now prove that $(h (x, v), \nabla_x \phi_h (x))$ obtained in \eqref{eq1:steady_small_lambda} is the solution to \eqref{eqn:h}
in the sense of Definition \ref{def:weak_sol}.
From the construction \eqref{eqtn:h^l}-\eqref{bdry:phi^l} and Lemma \ref{lemma:Unif_steady}, for any $\ell \in \N$, $h^{\ell + 1} (x,v)$ is the weak solution to \eqref{eqtn:h^l} and \eqref{bdry:h^l} with the field containing $\nabla_x \phi^{\ell}_h$.
Therefore, for any $\ell \in \N$ and $\psi \in  C^\infty_{c} (\bar \O \times \R^3)$,
\Be \label{weak_form_h^l}
\begin{split}
& \int_{ \gamma_+ } h^{\ell + 1} \psi \dd \gamma - \int_{ \gamma_- } f_b \psi \dd \gamma - \iint_{\O \times \R^3} h^{\ell + 1} v \cdot \nabla_x \psi \dd v \dd x
\\& + \iint_{\O \times \R^3} h^{\ell + 1} \nabla_x (\phi^{\ell}_h + \phi_E) \cdot \nabla_v \psi \dd v \dd x + \iint_{\O \times \R^3} \frac{ v \cdot \nabla_x \phi^{\ell}_{h} }{2} h^{\ell + 1} \psi \dd v \dd x 
+ \iint_{\O \times \R^3} e^{-\phi_E} \mathcal{L} h^{\ell + 1} \psi \dd v \dd x
\\& = \iint_{\O \times \R^3} - ( v \cdot \nabla_x \phi^{\ell+1}_{h} )e^{-\phi_E/2}\sqrt{\mu} \psi \dd v \dd x + \iint_{\O \times \R^3} e^{-\phi_E/2} \Gamma (h^{\ell}, h^{\ell}) \psi \dd v \dd x.
\end{split}
\Ee
Moreover, from \eqref{eqtn:phi^l} and \eqref{bdry:phi^l}, then for any $\ell \in \N$ and $\varphi \in H^1_0 (\O) \cap C^\infty_c (\bar \O)$, 
\Be \label{weak_form_phi^l}
\int_{\O} \nabla_x \phi^{\ell}_h \cdot \nabla_x \varphi \dd x 
= \int_{\O} e^{-\phi_E/2} \big( \int_{\mathbb{R}^3} h^{\ell} \sqrt{\mu} \dd v \big) \varphi \dd x.
\Ee

For the first line in \eqref{weak_form_h^l}, from $\psi \in  C^\infty_{c} (\bar \O \times \R^3)$ and $L^{2}$ convergence $h^{\ell} \to h$ in \eqref{eq1:steady_small_lambda}, we derive
\Be \notag
\begin{split}
\int_{ \gamma_+ } h^{\ell + 1} (x, v) \psi(x, v) \dd \gamma
& \rightarrow 
\int_{ \gamma_+ } h (x, v) \psi(x, v) \dd \gamma
\ \text{ as } \ \ell \to \infty,
\\ \iint_{\O \times \R^3} h^{\ell + 1} (x,v) v \cdot \nabla_x \psi(x, v) \dd v \dd x
& \rightarrow 
\iint_{\O \times \R^3} h (x,v) v \cdot \nabla_x \psi(x, v) \dd v \dd x
\ \text{ as } \ \ell \to \infty.
\end{split}
\Ee
For the second line in \eqref{weak_form_h^l}, using the $L^{2}$ convergence $\nabla_x \phi^{\ell}_h \to \nabla_x \phi_h$ in \eqref{eq1:steady_small_lambda}, together with $\psi \in  C^\infty_{c} (\bar \O \times \R^3)$, we get
\Be \notag
\begin{split}
\iint_{\O \times \R^3} h^{\ell + 1} (x, v) \nabla_x (\phi^{\ell}_h + \phi_E) \cdot \nabla_v \psi (x, v)  \dd v \dd x
& \rightarrow 
\iint_{\O \times \R^3} h (x, v) \nabla_x (\phi_h + \phi_E) \cdot \nabla_v \psi (x, v)  \dd v \dd x
\ \text{ as } \ \ell \to \infty,
\\ \iint_{\O \times \R^3} \frac{ v \cdot \nabla_x \phi^{\ell}_{h} }{2} h^{\ell + 1} \psi(x, v) \dd v \dd x
& \rightarrow 
\iint_{\O \times \R^3} \frac{ v \cdot \nabla_x \phi_{h} }{2} h \psi(x, v) \dd v \dd x
\ \text{ as } \ \ell \to \infty.
\end{split}
\Ee
%
Moreover, $L^{2}$ convergence $h^{\ell} \to h$ in \eqref{eq1:steady_small_lambda} and $\psi \in  C^\infty_{c} (\bar \O \times \R^3)$, together with the property that $K$ is bounded on $L^2_v$ from Lemma \ref{lemma:k_nu}, implies that
\Be \notag
\iint_{\O \times \R^3} e^{-\phi_E} \mathcal{L} h^{\ell + 1} \psi \dd v \dd x
\rightarrow 
\iint_{\O \times \R^3} e^{-\phi_E} \mathcal{L} h \psi \dd v \dd x
\ \text{ as } \ \ell \to \infty.
\Ee

For the third line in \eqref{weak_form_h^l}, using the $L^{2}$ convergence $h^{\ell} \to h$ and $\nabla_x \phi^{\ell}_h \to \nabla_x \phi_h$ in \eqref{eq1:steady_small_lambda}, together with the $L^{\infty}$ estimates on $\{ h^{\ell} \}^{\infty}_{\ell = 0}$ in Lemma \ref{lemma:Unif_steady} and \eqref{regularity:wh_Linfty}, Lemma \ref{lemma:gamma} and $\psi \in  C^\infty_{c} (\bar \O \times \R^3)$, we have
\Be \notag
\begin{split}
\iint_{\O \times \R^3} - ( v \cdot \nabla_x \phi^{\ell+1}_{h} )e^{-\phi_E/2}\sqrt{\mu} \psi \dd v \dd x
& \rightarrow 
\iint_{\O \times \R^3} - ( v \cdot \nabla_x \phi_{h} )e^{-\phi_E/2}\sqrt{\mu} \psi \dd v \dd x
\ \text{ as } \ \ell \to \infty,
\\ \iint_{\O \times \R^3} e^{-\phi_E/2} \Gamma (h^{\ell}, h^{\ell}) \psi \dd v \dd x
& \rightarrow 
\iint_{\O \times \R^3} e^{-\phi_E/2} \Gamma (h, h) \psi \dd v \dd x
\ \text{ as } \ \ell \to \infty.
\end{split}
\Ee 
From the convergence of every term in \eqref{weak_form_h^l}, we conclude that $(h, \nabla_x \phi_h)$ solves \eqref{weak_form_1}.

For \eqref{weak_form_phi^l}, using the $L^{2}$ convergence $h^{\ell} \to h$ and $\nabla_x \phi^{\ell}_h \to \nabla_x \phi_h$ in \eqref{eq1:steady_small_lambda}, together with $\varphi \in H^1_0 (\O) \cap C^\infty_c (\bar \O)$, we obtain 
\Be \notag
\begin{split}
\int_{\O} \nabla_x \phi^{\ell}_h \cdot \nabla_x \varphi \dd x
& \rightarrow 
\int_{\O} \nabla_x \phi_h \cdot \nabla_x \varphi \dd x
\ \text{ as } \ \ell \to \infty,
\\ \int_{\O} e^{-\phi_E/2} \big( \int_{\mathbb{R}^3} h^{\ell} \sqrt{\mu} \dd v \big) \varphi \dd x
& \rightarrow 
\int_{\O} e^{-\phi_E/2} \big( \int_{\mathbb{R}^3} h \sqrt{\mu} \dd v \big) \varphi \dd x
\ \text{ as } \ \ell \to \infty,
\end{split}
\Ee
and thus $(h, \nabla_x \phi_h)$ solves \eqref{weak_form_2}.
Therefore, we conclude that $(h, \nabla_x \phi_h)$ is a weak solution of \eqref{eqn:h}.

\smallskip

\textbf{Step 3. Regularity: Proof of \eqref{regularity:nuh_v_Linfty}.}
Theorem \ref{thm:well_poseness} shows that for any $\ell \in \N$,
\be \notag
\begin{split}
\Vert w_{\tilde{\theta}} \nabla_v h^{\ell} \Vert_{L^\infty_{x,v}} 
& \lesssim |f_b|_{L^2_{\gamma_-}} + | wf_b|_{L^\infty_{\p\O,v}} + | w \p_{\mathbf{x}_p,v} f_b|_{L^\infty_{\p\O,v}}.
\end{split}
\ee
Thus, there exists a subsequence $\{ \ell^k \}_{k=1}^{\infty}$ such that
\[
w_{\tilde{\theta}} \nabla_v h^{\ell^k} \xrightarrow[]{*} g 
\ \text{ in } 
L^\infty_{x,v}
\ \text{ as } \
k \to \infty.
\]
Recall from \eqref{eq1:steady_small_lambda} that $h^{\ell^k} \to h$ in $L^2_{x,\nu} (\O \times \R^3) \cap L^2 (\gamma_+)$. 
A standard result of integration by parts then implies that $g = w_{\tilde{\theta}} \nabla_v h$. Consequently,
\[
\Vert w_{\tilde{\theta}} \nabla_{v} h \Vert_{L^{\infty}_{x,v}} 
\lesssim |f_b|_{L^2_{\gamma_-}} + | wf_b|_{L^\infty_{\p\O,v}} + | w \p_{\mathbf{x}_p,v} f_b|_{L^\infty_{\p\O,v}},
\]
which yields \eqref{regularity:nuh_v_Linfty}.

\smallskip

\textbf{Step 4. Uniqueness.}
Since the solution $h (x, v)$ satisfies \eqref{regularity:h_L2}-\eqref{regularity:nuh_v_Linfty}, the uniqueness follows from the proof of Proposition \ref{prop:stationary_uniqueness}.
\end{proof}

\section{Dynamical problem} \label{sec:dynamical}

In this section, we prove Theorem~\ref{thm:dynamical_stability}, which establishes the existence, uniqueness, dynamical stability, and positivity of solutions to the dynamical problem \eqref{F_dynamical}. 
Theorem~\ref{thm:dynamical_stability} also implies the positivity of the stationary solution $F_s = e^{-\phi_E}\mu + e^{-\phi_E/2}\sqrt{\mu} h$ in Theorem~\ref{thm:steady_wellpose}. We refer to Remark~\ref{rmk:positivity} and the proof of Theorem~\ref{thm:dynamical_stability} for further discussion.

Throughout this section, when applying the lemmas from Section \ref{sec:prelim}, we treat $f$ and $\phi_f$ as the dynamical solution defined in \eqref{linear_f_dynamical}, and for the kinetic weight lemmas from Section~\ref{sec:kinetic_weight}, we replace $\alpha (t,x,v)$ by the dynamical weight $\alpha_f (t,x,v)$ defined in \eqref{alpha_weight_dyna}. 
Moreover, we let $(h, \phi_h)$ be a solution to the stationary problem \eqref{eqn:h}, satisfying the conclusions of Theorem~\ref{thm:steady_wellpose}.
In addition, we assume the inflow condition \eqref{inflow_condition}, the compatibility condition $f_0 |_{\gamma_-} = 0$, and the smallness condition in \eqref{initial_condition} from Theorem~\ref{thm:dynamical_stability}.

This section is organized as follows. Similar to the stationary problem, Sections~\ref{sec:dynamical_l2} and \ref{sec:dynamical_regularity} establish the a priori $L^2$--$L^\infty$ estimate and the a priori regularity estimate, respectively. To conclude the well-posedness, we apply the local argument in \cite{cao2019regularity} and establish local well-posedness in Section~\ref{sec:dynamical_local}. Finally, in Section~\ref{sec:dynamical_proof}, we combine the a priori estimates, the local well-posedness, and a continuity argument to conclude the global well-posedness and complete the proof of Theorem~\ref{thm:dynamical_stability}.

\subsection{\texorpdfstring{A priori $L^2-L^\infty$ estimate}{L2 estimate}} 
\label{sec:dynamical_l2}

In this section, we establish the a priori $L^2$--$L^\infty$ estimate for solutions to the dynamical problem \eqref{linear_f_dynamical}. The argument is analogous to the $L^2$--$L^\infty$ estimate for the stationary problem developed in Section~\ref{sec:L2Linfty_estimate}.

To derive the estimate, we introduce the exponential time weight $e^{\lambda t}$ into \eqref{linear_f_dynamical}. More precisely, we consider the equation for $e^{\lambda t} f$ obtained from \eqref{linear_f_dynamical}:
\be \label{eq:exp_f_dynamical}
\begin{split}
& \p_t (e^{\lambda t} f) - \lambda e^{\lambda t} f  + v \cdot \nabla_x (e^{\lambda t} f) - \nabla_x (\phi_f+\phi_E) \cdot \nabla_{v} ( e^{\lambda t} f) + \frac{v \cdot \nabla_x \phi_f}{2} e^{\lambda t} f + e^{-\phi_E}\mathcal{L} (e^{\lambda t} f) \\
& = - e^{\lambda t} e^{-\phi_E/2}v\cdot \nabla_x (\phi_f-\phi_{h}) \sqrt{\mu} - e^{\lambda t} \frac{v \cdot \nabla_x (\phi_f-\phi_{h})}{2}h + e^{\lambda t} \nabla_x (\phi_f-\phi_{h}) \cdot \nabla_{v} h \\
& \qquad +  e^{\lambda t} e^{-\phi_E/2}[\Gamma(f,f) + \Gamma(f,h)+\Gamma(h,f)].
\end{split}  
\ee

We first show an energy estimate on the microscopic component $(\mathbf{I}-\mathbf{P})e^{\lambda t}f$.

\begin{lemma} \label{lemma:l2_energy_dynamical}

Let $(f, \phi_f)$ be a solution to \eqref{linear_f_dynamical} on $[0, T]$. 
For $\lambda >0$, the following $L^2$ energy estimate holds:
\be \label{l2_energy_dynamical}
\begin{split}
&  \Vert e^{\lambda T} f(T)\Vert_{L^2_{x,v}}^2 + \Vert e^{\lambda T} \nabla_x (\phi_f-\phi_{h})(T)\Vert_{L^2_x}^2 + | e^{\lambda t}f|_{L^2_{T,\gamma_+}}^2 + \Vert (\mathbf{I}-\mathbf{P})e^{\lambda t}f \Vert_{L^2_{T,x,\nu}}^2 
\\& \lesssim \Vert f_0\Vert_{L^2_{x,v}}^2 + \Vert \nabla_x (\phi_f -\phi_{h})(0)\Vert_{L^2_{x}} + \lambda \Vert e^{\lambda t} \nabla_x (\phi_f-\phi_{h})\Vert_{L^2_{T,x}}^2+ \Vert \nu^{-1/2}e^{\lambda t}\Gamma(f,f)\Vert_{L^2_{T,x,v}}^2 
\\& \qquad + [\lambda + \Vert  wh\Vert_{L^\infty_{x,v}} + \Vert w_{\tilde{\theta}} \nabla_v h\Vert_{L^\infty_{x,v}} + \Vert \nabla_x \phi_E\Vert_{L^\infty_x} ] \Vert e^{\lambda t}f\Vert_{L^2_{T,x,\nu}}^2 + \Vert e^{\lambda t}wf\Vert_{L^\infty_{T,x,v}} \Vert e^{\lambda t/2}f\Vert_{L^2_{T,x,v}}^2 . 
\end{split}
\ee
\end{lemma}

\begin{proof}

We multiply \eqref{eq:exp_f_dynamical} by $e^{\lambda t}f$ and integrate over $t,x$, and $v$ to obtain the following energy estimate:
\be \label{dynamical_L2_int}
\begin{split}
& \Vert e^{\lambda T}f(T)\Vert_{L^2_{x,v}}^2 +  |e^{\lambda t}f|_{L^2_{T,\gamma_+}}^2 + \Vert e^{-\phi_E/2}(\mathbf{I}-\mathbf{P})e^{\lambda t}f\Vert_{L^2_{T,x,\nu}}^2 
\\& \lesssim \Vert f(0)\Vert_{L^2_{x,v}}^2 + \lambda \Vert e^{\lambda t}f\Vert_{L^2_{T,x,v}}^2 + o(1) \Vert e^{-\phi_E/2}(\mathbf{I}-\mathbf{P})e^{\lambda t}f\Vert_{L^2_{T,x,\nu}}^2 
\\& \quad + \Vert \nu^{-1/2}e^{\phi_E/4}e^{\lambda t}[\Gamma(f,f)+\Gamma(f,h)+\Gamma(h,f)]\Vert_{L^2_{T,x,v}}^2 
\\& \quad + \big\Vert \Vert  \nabla_x \phi_f(t) \Vert_{L^\infty_{x}}\nu^{1/2} e^{\lambda t} f \big\Vert_{L^2_{T,x,v}}^2 + \int_0^T e^{2\lambda t} \iint_{\O\times \mathbb{R}^3} - (v \cdot \nabla_x (\phi_f - \phi_{h} )) e^{-\phi_E/2}\sqrt{\mu} f \dd x \dd v \dd t
\\& \quad + \Vert  e^{\lambda t}f\Vert_{L^2_{T,x,v}} \Vert e^{\lambda t}\nabla_x (\phi_f-\phi_{h})\Vert_{L^2_{T,x}} \Vert vh\Vert_{L^\infty_{x}L^2_v} +  \Vert e^{\lambda t}f\Vert_{L^2_{T,x,v}} \Vert e^{\lambda t}\nabla_x (\phi_f-\phi_{h}) \Vert_{L^2_{T,x}} \Vert \nabla_v h\Vert_{L^\infty_x L^2_v }.
\end{split}
\ee

Using $\|wh\|_{L^\infty_{x,v}} \ll 1$ from Theorem~\ref{thm:steady_wellpose}, we control the contribution of $\Gamma(f,h)+\Gamma(h,f)$ in \eqref{dynamical_L2_int} by
\begin{align*}
& \Vert \nu^{-1/2}e^{\phi_E/4}e^{\lambda t} [\Gamma(f,h)+\Gamma(h,f)] \Vert_{L^2_{T,x,v}}^2 
\lesssim \Vert wh\Vert_{L^\infty_{x,v}} \Vert e^{\lambda t}f\Vert_{L^2_{T,x, \nu}}^2.
\end{align*}
From \eqref{phif_bdd_dynamical}, we control the contribution of $\Vert \nabla_x \phi_f(t)\Vert_{L^\infty_{x}}$ in \eqref{dynamical_L2_int} by
\begin{align*}
& \big\Vert \Vert \nabla_x \phi_f(t) \Vert_{L^\infty_{x}} \nu^{1/2} e^{\lambda t} f \big\Vert_{L^2_{T,x,v}}^2 
\lesssim \Vert wh\Vert_{L^\infty_{x,v}} \Vert e^{\lambda t}f\Vert_{L^2_{T,x,\nu}}^2 + \Vert e^{\lambda t}wf\Vert_{L^\infty_{x,v}} \Vert e^{\lambda t/2}f\Vert_{L^2_{T,x,\nu}}^2 .
\end{align*}
Using \eqref{phif_fs_bdd}, we control the last two terms in \eqref{dynamical_L2_int} by
\begin{align*}
& \Vert  e^{\lambda t}f\Vert_{L^2_{T,x,v}} \Vert e^{\lambda t}\nabla_x (\phi_f-\phi_{h})\Vert_{L^2_{T,x}} \Vert vh\Vert_{L^\infty_{x}L^2_v} +  \Vert e^{\lambda t}f\Vert_{L^2_{T,x,v}} \Vert e^{\lambda t}\nabla_x (\phi_f-\phi_{h}) \Vert_{L^2_{T,x}} \Vert \nabla_v h\Vert_{L^\infty_x L^2_v } \\
& \lesssim \Vert e^{\lambda t} f \Vert_{L^2_{T,x,v}}^2 \Vert wh\Vert_{L^\infty_{x,v}} + \Vert e^{\lambda t}f\Vert_{L^2_{T,x,v}}^2 \Vert w_{\tilde{\theta}}\nabla_v h\Vert_{L^\infty_{x,v}}.
\end{align*}
Finally, we compute the second term in the fourth line of \eqref{dynamical_L2_int}:
\be \label{eq1:l2_energy_dynamical}
\int_0^T e^{2\lambda t} \underbrace{\iint_{\O\times \mathbb{R}^3} - (v \cdot \nabla_x (\phi_f - \phi_{h} )) e^{-\phi_E/2}\sqrt{\mu} f \dd x \dd v}_{\eqref{eq1:l2_energy_dynamical}^*} \dd t.
\ee

From the zero Dirichlet boundary conditions of $\phi_f$ and $\phi_{h}$, integration by parts in $x$ yields
\begin{align*}
\eqref{eq1:l2_energy_dynamical}^* 
& = \iint_{\O\times \mathbb{R}^3} (\phi_f-\phi_{h})\sqrt{\mu} e^{-\phi_E/2} (v\cdot \nabla_x)f \dd x \dd v - \iint_{\O\times \mathbb{R}^3} (\phi_f - \phi_{h}) \sqrt{\mu} \frac{v \cdot \nabla_x \phi_E}{2} e^{-\phi_E/2}f \dd x \dd v.
\end{align*}
Applying the equation \eqref{linear_f_dynamical}, we obtain
\begin{align*}
&  \iint_{\O\times \mathbb{R}^3} (\phi_f - \phi_{h})\sqrt{\mu} e^{-\phi_E/2} (v \cdot \nabla_x)f \dd x \dd v \\
& = -\iint_{\O\times \mathbb{R}^3}  (\phi_f -\phi_{h}) e^{-\phi_E/2} \p_t f \sqrt{\mu} \dd x \dd v + \iint_{\O\times \mathbb{R}^3}  (\phi_f-\phi_{h})\sqrt{\mu}e^{-\phi_E/2}(\nabla_x (\phi_f+\phi_E)\cdot \nabla_{v}f) \dd x \dd v \\
& \qquad - \iint_{\O\times \mathbb{R}^3} (\phi_f-\phi_{h}) \sqrt{\mu}e^{-\phi_E/2} \frac{v\cdot \nabla_x \phi_f}{2}f  \dd x \dd v \\
& = -\int_{\O} (\phi_f-\phi_{h}) \p_t \Big(e^{-\phi_E/2} \int_{\mathbb{R}^3} (h+f-h)\sqrt{\mu}\dd v \Big) \dd x + \iint_{\O\times \mathbb{R}^3} (\phi_f-\phi_{h}) \sqrt{\mu} e^{-\phi_E/2} \nabla_x \phi_E \cdot \nabla_{v} f \dd x \dd v \\
& \qquad + \iint_{\O\times \mathbb{R}^3} (\phi_f - \phi_{h})e^{-\phi_E/2}\nabla_x \phi_f  \cdot \nabla_{v}(\sqrt{\mu}f)   \dd x \dd v \\
& = -\int_{\O} (\phi_f - \phi_{h}) \p_t (-\Delta_x(\phi_f - \phi_{h})) \dd x - \iint_{\O\times \mathbb{R}^3} (\phi_f - \phi_{h}) \sqrt{\mu}e^{-\phi_E/2} \frac{v\cdot \nabla_x \phi_E}{2}f\dd x\dd v \\
& = - \frac{1}{2} \int_{\O} \p_t (|\nabla_x (\phi_f - \phi_{h})|^2) \dd x - \iint_{\O\times \mathbb{R}^3} (\phi_f - \phi_{h}) \sqrt{\mu}e^{-\phi_E/2} \frac{v\cdot \nabla_x \phi_E}{2}f\dd x\dd v.
\end{align*}
Combining the above equality with \eqref{eq1:l2_energy_dynamical}, we derive
\be \notag
\begin{split}
& \eqref{eq1:l2_energy_dynamical} 
= - \frac{1}{2} \int_0^T \int_{\O} e^{2\lambda t} \p_t (|\nabla_x(\phi_f - \phi_{h})|^2) \dd x \dd t  - \int_0^T \iint_{\O\times \mathbb{R}^3} e^{2\lambda t} (\phi_f -\phi_{h}) \sqrt{\mu} e^{-\phi_E/2} v\cdot \nabla_x \phi_Ef \dd x \dd v \dd t
\\& \leq \lambda \Vert e^{\lambda t}\nabla_x (\phi_f-\phi_{h})\Vert_{L^2_{T,x}}^2 - \frac{1}{2} \Vert e^{\lambda T}\nabla_x (\phi_f-\phi_{h})(T)\Vert_{L^2_{x}}^2 + \frac{1}{2} \Vert \nabla_x (\phi_f-\phi_{h})(0)\Vert_{L^2_x}^2
\\& \qquad + \Vert e^{\lambda t}(\phi_f-\phi_{h})\Vert_{L^2_{T,x}} \Vert e^{\lambda t}f\Vert_{L^2_{T,x,v}} \Vert \nabla_x \phi_E\Vert_{L^\infty_x}
\\& \leq \lambda \Vert e^{\lambda t}\nabla_x (\phi_f-\phi_{h})\Vert_{L^2_{T,x}}^2 - \frac{1}{2} \Vert e^{\lambda T}\nabla_x (\phi_f-\phi_{h})(T)\Vert_{L^2_{x}}^2 + \frac{1}{2} \Vert \nabla_x (\phi_f-\phi_{h})(0)\Vert_{L^2_x}^2 + \Vert e^{\lambda t}f\Vert_{L^2_{T,x,v}}^2 \Vert \nabla_x \phi_E\Vert_{L^\infty_x},
\end{split}
\ee
where in the second line we apply integration by parts in $t$, and in the last line we use the elliptic regularity estimate.
Using the inflow condition \eqref{inflow_condition}, we conclude the lemma.
\end{proof}

Next, we control the macroscopic component $\mathbf{P} e^{\lambda t}f$ in the following lemma:

\begin{lemma} \label{lemma:macro_l2_dynamical}

Let $(f, \phi_f)$ be a solution to \eqref{linear_f_dynamical} on $[0, T]$. 
There exists $0< \lambda \ll 1$ such that the following estimate for $\Vert \mathbf{P}e^{\lambda t}f \Vert_{L^2_{T,x,v}}$ holds:
\be \label{macro_l2_dynamical}
\begin{split}
& \Vert \mathbf{P}e^{\lambda t}f \Vert_{L^2_{T,x,v}}^2+ \Vert e^{\lambda t}\nabla_x (\phi_f-\phi_{h})\Vert_{L^2_{T,x}}^2 
\\& \lesssim \Vert (\mathbf{I}-\mathbf{P})e^{\lambda t}f\Vert_{L^2_{T,x,\nu}}^2 + |e^{\lambda t}f|_{L^2_{T,\gamma_+}}^2  + \Vert \nu^{-1/2}e^{\lambda t}\Gamma(f,f)\Vert_{L^2_{T,x,v}}^2  + \Vert e^{\lambda t}wf\Vert_{L^\infty_{T,x,v}}\Vert e^{\lambda t/2}f\Vert_{L^2_{T,x,v}}^2 
\\& \qquad + \Vert e^{\lambda T}f(T)\Vert_{L^2_{x,v}}^2 + \Vert f_0\Vert_{L^2_{x,v}}^2. 
\end{split}
\ee
\end{lemma}

\begin{proof}


Given a test function $\psi \in L^{\infty} ([0, T]; L^2 (\Omega \times \mathbb{R}^3)) \cap L^2 ( [0, T] \times \partial \Omega \times \mathbb{R}^3; d\gamma)$ and $\p_t \psi + v\cdot \nabla_x \psi - \nabla_x (\phi_f + \phi_E)\cdot \nabla_x \psi \in L^2([0,T]\times \O\times \mathbb{R}^3)$, the weak formulation of \eqref{eq:exp_f_dynamical} is given by
{\small
\be \label{weak_formulation_dyna}
\begin{split}
& \underbrace{- \int_0^T \iint_{\O\times \mathbb{R}^3} e^{\lambda t} f \p_t \psi \dd x \dd v \dd t}_{\eqref{weak_formulation_dyna}_1} \ - \ \underbrace{\int_0^T\iint_{\O\times \mathbb{R}^3} \mathbf{P} e^{\lambda t}f (v\cdot \nabla_x \psi) \dd x \dd v \dd t}_{\eqref{weak_formulation_dyna}_2} \ - \ \underbrace{\int_0^T\iint_{\O\times \mathbb{R}^3} (\mathbf{I}-\mathbf{P})e^{\lambda t}f (v \cdot \nabla_x \psi) \dd x \dd v \dd t}_{\eqref{weak_formulation_dyna}_3}  
\\& \qquad + \underbrace{\int_0^T\int_{\gamma} \psi e^{\lambda t}f \dd \gamma \dd t}_{\eqref{weak_formulation_dyna}_4} \ + \ \underbrace{\int_0^T\iint_{\O\times \mathbb{R}^3} \sqrt{\mu} e^{\lambda t}f \nabla_x \phi_f \cdot  \nabla_{v} \big[\frac{1}{\sqrt{\mu}}\psi \big] \dd x \dd v \dd t}_{\eqref{weak_formulation_dyna}_5}
\\& \qquad + \underbrace{\int_0^T\iint_{\O\times \mathbb{R}^3} e^{\lambda t}f\nabla_x \phi_E \cdot \nabla_{v}\psi \dd x \dd v \dd t}_{\eqref{weak_formulation_dyna}_6} \ + \ \underbrace{\int_0^T\iint_{\O\times \mathbb{R}^3} e^{-\phi_E}\mathcal{L}e^{\lambda t}f \psi \dd x \dd v \dd t}_{\eqref{weak_formulation_dyna}_7} 
\\& = \underbrace{\int_0^T\iint_{\O\times \mathbb{R}^3} -\psi e^{\lambda t}(v \cdot \nabla_x (\phi_f-\phi_{h}))e^{-\phi_E/2} \sqrt{\mu} \dd x \dd v \dd t}_{\eqref{weak_formulation_dyna}_8} \ - \ \underbrace{ \int_0^T\iint_{\O\times \mathbb{R}^3} e^{\lambda t} \sqrt{\mu}h \nabla_x (\phi_f-\phi_{h}) \cdot \nabla_v(\frac{\psi}{\sqrt{\mu}}) \dd x \dd v \dd t}_{\eqref{weak_formulation_dyna}_9} 
\\& \qquad + \underbrace{\int_0^T\iint_{\O\times \mathbb{R}^3} \psi e^{\lambda t}(\lambda f +e^{-\phi_E/2}[\Gamma(f,f)+\Gamma(f,h)+\Gamma(h,f)] ) \dd x \dd v \dd t}_{\eqref{weak_formulation_dyna}_{10}} \ + \ \underbrace{\iint_{\O\times \mathbb{R}^3} [f(0)\psi(0) - e^{\lambda T} f(T)\psi(T)] \dd x \dd v}_{\eqref{weak_formulation_dyna}_{11}}. 
\end{split}
\ee
}

We recall \eqref{def:Ph} and apply $\mathbf{P}$ to the time-dependent function $e^{\lambda t}f$ to obtain
\[
\mathbf{P} e^{\lambda t}f
= a (t, x) \sqrt{\mu} + (\mathbf{b} (t, x) \cdot v)\sqrt{\mu} + c (t, x) \frac{|v|^2-3}{2}\sqrt{\mu}.
\]
Similar to the proof of Lemma~\ref{lemma:macro_l2}, we introduce the test functions $\psi_c$, $\psi_{b_1}$, $\psi_{b_2}$, $\psi_{b_3}$, and $\psi_a$ with the additional time weight $e^{\lambda t}$ to estimate the macroscopic coefficients $a$, $\mathbf{b} = (b_1, b_2, b_3)$, and $c$ as follows.
\begin{align*}
&   \begin{cases}
    & \psi_c = (|v|^2-5)\sqrt{\mu} (v \cdot \nabla_x \phi_c) \perp \ker \mathcal{L}, \\[5pt]
    & -\Delta_x \phi_c = e^{\lambda t}c , \quad \phi_c|_{\p\O} = 0.
    \end{cases} \\
 &  \begin{cases}
    & \psi_{b_1} =  \frac{3}{2}\Big(|v_1|^2 - \frac{|v|^2}{3} \Big)\sqrt{\mu} \p_{x_1} \phi_{b_1} + v_1v_2\sqrt{\mu} \p_{x_2}\phi_{b_1} + v_1v_3 \sqrt{\mu}\p_{x_3} \phi_{b_1}  \perp \ker \mathcal{L}, \\[5pt]
    & -(\p_{x_1}^2 \phi_{b_1} +\Delta_x \phi_{b_1} )  = e^{\lambda t}b_1, \quad \phi_{b_1}|_{\p\O} = 0.
    \end{cases} \\
&  \begin{cases}    
    & \psi_{b_2} =  v_1v_2\sqrt{\mu} \p_{x_1}\phi_{b_2} + \frac{3}{2}\Big(|v_2|^2 - \frac{|v|^2}{3} \Big)\sqrt{\mu} \p_{x_2} \phi_{b_2}  + v_2v_3 \sqrt{\mu}\p_{x_3} \phi_{b_2}  \perp \ker \mathcal{L}, \\[5pt]
    & -( \p_{x_2}^2  \phi_{b_2} + \Delta_x \phi_{b_2} )  = e^{\lambda t}b_2 , \ \phi_{b_2}|_{\p\O} = 0.
    \end{cases} \\    
&  \begin{cases}    
    & \psi_{b_3} =    v_1v_3\sqrt{\mu} \p_{x_1}\phi_{b_3} + v_2v_3 \sqrt{\mu}\p_{x_2} \phi_{b_3} + \frac{3}{2}\Big(|v_3|^2 - \frac{|v|^2}{3} \Big)\sqrt{\mu} \p_{x_3} \phi_{b_3}  \perp \ker \mathcal{L}, \\[5pt]
    & -( \p_{x_3}^2 \phi_{b_3} + \Delta_x \phi_{b_3}  ) = e^{\lambda t}b_3 , \quad \phi_{b_3}|_{\p\O} = 0.
    \end{cases} \\
&   \begin{cases}
    & \psi_a = -(|v|^2-10)\sqrt{\mu}(v \cdot \nabla_x  \phi_a), \\[5pt]
    & - \Delta_x \phi_a = e^{\lambda t} a, \quad \phi_a|_{\p\O} = 0.
    \end{cases}
\end{align*}
From the elliptic regularity and the trace theorem, we obtain
\be \notag
\Vert \phi_c\Vert_{H^2_x} + |\phi_c|_{H^1_{\p\O}} \lesssim \Vert e^{\lambda t}c\Vert_{L^2_{x}}, \quad
\sum_{i=1}^3 \Vert \phi_{b_i}\Vert_{H^2_x} + |\phi_{b_i}|_{H^1_{\p\O}} \lesssim \Vert e^{\lambda t}\mathbf{b}\Vert_{L^2_x}, \quad
\Vert \phi_a\Vert_{H^2_x} + |\phi_a|_{H^1_{\p\O}} \lesssim \Vert e^{\lambda t}a\Vert_{L^2_x}.
\ee
Now we choose the test function
\be \label{eq5:macro_l2_dynamical}
\psi = \delta \psi_a + \sum_{i=1}^3 \psi_{b_i} + \psi_c,
\ee
where $\delta>0$ will be specified at the end of the proof. Similar to the proof of Lemma~\ref{lemma:macro_l2}, all terms in \eqref{weak_formulation_dyna} can be controlled except for $\eqref{weak_formulation_dyna}_1$, $\eqref{weak_formulation_dyna}_8$, and $\eqref{weak_formulation_dyna}_{11}$:
\be \notag
\begin{split}
\eqref{weak_formulation_dyna}_2 & = \delta\Vert e^{\lambda t}a\Vert_{L^2_{T,x,v}}^2 + \Vert e^{\lambda t}\mathbf{b}\Vert_{L^2_{T,x,v}}^2 + \Vert e^{\lambda t}c\Vert_{L^2_{T,x,v}}^2, 
\\ |\eqref{weak_formulation_dyna}_3| + |\eqref{weak_formulation_dyna}_7| & \lesssim \Vert (\mathbf{I}-\mathbf{P})e^{\lambda t}f\Vert_{L^2_{T,x,v}}^2  + o(1)[\delta\Vert e^{\lambda t}a\Vert_{L^2_{T,x,v}}^2 + \Vert e^{\lambda t}\mathbf{b}\Vert_{L^2_{T,x,v}}^2 + \Vert e^{\lambda t}c\Vert_{L^2_{T,x,v}}^2], 
\\ |\eqref{weak_formulation_dyna}_4| & \lesssim |e^{\lambda t}f|_{L^2_{T,\gamma_+}}^2  + o(1)[\delta\Vert e^{\lambda t}a\Vert_{L^2_{T,x,v}}^2 + \Vert e^{\lambda t}\mathbf{b}\Vert_{L^2_{T,x,v}}^2 + \Vert e^{\lambda t}c\Vert_{L^2_{T,x,v}}^2], 
\\ |\eqref{weak_formulation_dyna}_5| + |\eqref{weak_formulation_dyna}_6| & \lesssim [ \Vert wh\Vert_{L^\infty_{x,v}} + \Vert \nabla_x \phi_E\Vert_{L^\infty_x}] \Vert e^{\lambda t}f\Vert_{L^2_{T,x,v}}^2 + \Vert e^{\lambda t}wf\Vert_{L^\infty_{T,x,v}} \Vert e^{\lambda t/2}f\Vert_{L^2_{T,x,v}}^2, 
\\ |\eqref{weak_formulation_dyna}_9| & \lesssim \Vert wh\Vert_{L^\infty_{x,v}} \Vert e^{\lambda t} \nabla_x (\phi_f-\phi_{h})\Vert_{L^2_{T,x}}\Vert e^{\lambda t}f\Vert_{L^2_{T,x,v}} \lesssim \Vert wh\Vert_{L^\infty_{x,v}} \Vert e^{\lambda t}f\Vert_{L^2_{T,x,v}}^2, 
\\ |\eqref{weak_formulation_dyna}_{10}| & \lesssim [\lambda + \Vert wh\Vert_{L^\infty_{x,v}}] \Vert e^{\lambda t}f\Vert_{L^2_{T,x,v}}^2 + \Vert \mu^{1/4}e^{\lambda t}\Gamma(f,f)\Vert_{L^2_{T,x,v}}^2 
\\& \qquad + o(1)[\delta\Vert e^{\lambda t}a\Vert_{L^2_{T,x,v}}^2 + \Vert e^{\lambda t}\mathbf{b}\Vert_{L^2_{T,x,v}}^2 + \Vert e^{\lambda t}c\Vert_{L^2_{T,x,v}}^2].
\end{split}
\ee

For $\eqref{weak_formulation_dyna}_{11}$, we apply H\"older's inequality to obtain
\be \notag
|\eqref{weak_formulation_dyna}_{11}| \lesssim \Vert e^{\lambda T}f(T)\Vert_{L^2_{x,v}} + \Vert f(0)\Vert_{L^2_{x,v}}.
\ee

For $\eqref{weak_formulation_dyna}_8$, we compute that
\begin{align}
\eqref{weak_formulation_dyna}_8 
& = -\int_0^T \iint_{\O\times \mathbb{R}^3} \delta\psi_a e^{\lambda t}(v\cdot \nabla_x (\phi_f -\phi_{h})) e^{-\phi_E/2}\sqrt{\mu} \dd x \dd v \dd t  
\notag \\
& = -5 \delta\int_0^T\int_{\O} \nabla_x e^{\lambda t}(\phi_f-\phi_{h}) \cdot \nabla_x \phi_a e^{-\phi_E/2} \dd x \dd t 
\notag \\
& = -5 \delta\int_0^T \int_{\O} e^{2\lambda t}|\nabla_x (\phi_f -\phi_{h})|^2 e^{-\phi_E/2} \dd x \dd t 
\label{eq1:macro_l2_dynamical} \\
& \qquad - 5 \delta\int_0^T\int_{\O} e^{\lambda t} \nabla_x (\phi_f-\phi_{h})\cdot (\nabla_x \phi_a - e^{\lambda t} \nabla_x (\phi_f-\phi_{h})) e^{-\phi_E/2} \dd x \dd t.
\label{eq2:macro_l2_dynamical}
\end{align}
From the inflow condition \eqref{inflow_condition}, \eqref{eq1:macro_l2_dynamical} satisfies that
\begin{align*}
    &  5 \delta\int_0^T\int_{\O} e^{2\lambda t}|\nabla_x (\phi_f -\phi_{h})|^2 e^{-\phi_E/2} \dd x \dd t \geq \delta\Vert e^{\lambda t}\nabla_x (\phi_f-\phi_{h})\Vert_{L^2_{T,x}}^2.
\end{align*}
On the other hand, $\phi_a - e^{\lambda t}(\phi_f-\phi_{h})$ in \eqref{eq2:macro_l2_dynamical} satisfies
\be \notag
- \Delta_x [\phi_a - e^{\lambda t}(\phi_f-\phi_{h})] = e^{\lambda t} a (1-e^{-\phi_E/2}), \qquad 
\phi_a - e^{\lambda t}(\phi_f-\phi_{h}) |_{\p\O} = 0.
\ee
The elliptic regularity estimate further implies
\begin{align*}
    & \Vert \phi_a - e^{\lambda t}(\phi_f-\phi_{h})\Vert_{L^2_{x}} \lesssim \Vert \phi_E\Vert_{L^\infty_x} \Vert e^{\lambda t}a\Vert_{L^2_x}.
\end{align*}
Thus, we derive that
\be \notag
| \eqref{eq2:macro_l2_dynamical} | \lesssim \delta\Vert \phi_E\Vert_{L^\infty_x} \Vert e^{\lambda t}\nabla_x (\phi_f-\phi_{h})\Vert_{L^2_{T,x}}^2 + \delta\Vert \phi_E\Vert_{L^\infty_x} \Vert e^{\lambda t}a\Vert_{L^2_x}^2.
\ee

Finally, we estimate $\eqref{weak_formulation_dyna}_1$, which involves the time derivative term 
\[
\p_t \psi = \delta \p_t  \psi_a + \sum_{i=1}^3 \p_t  \psi_{b_i} + \p_t \psi_c.
\]

For the contribution of $\psi_a$ in $\eqref{weak_formulation_dyna}_1$, we set $\Phi_a = \p_t \phi_a$. Then $\Phi_a$ satisfies
\be \notag
- \Delta_x \Phi_a = \p_t (e^{\lambda t}a), \qquad \Phi_a|_{\partial\Omega} = 0.
\ee
Multiplying the above equation by $\Phi_a$ and integrating by parts, we have
\be \label{eq3:macro_l2_dynamical}
\Vert \nabla_x \Phi_a\Vert_{L^2_{x}}^2 = \int_{\O} \p_t (e^{\lambda t}a)\Phi_a \dd x.
\ee
For $\p_t(e^{\lambda t}a)$, we apply the mass conservation law to \eqref{eq:exp_f_dynamical} and obtain
\be \notag
\p_t (e^{\lambda t}a) + \nabla_x \cdot (e^{\lambda t}\mathbf{b}) - \int_{\mathbb{R}^3} \sqrt{\mu}\nabla_x \phi_E\cdot \nabla_v (e^{\lambda t}f) \dd v = \lambda e^{\lambda t}a.
\ee
By the zero Dirichlet boundary condition $\Phi_a|_{\partial\Omega} = 0$, the Poincaré inequality shows
\[
\Vert \Phi_a\Vert_{L^2_x}\lesssim \Vert \nabla_x \Phi_a\Vert_{L^2_x}.
\]
This implies that
\be \notag
\begin{split}
\Big|\int_{\O} \p_t (e^{\lambda t}a)\Phi_a\dd x \Big| 
& \lesssim \int_{\O} |e^{\lambda t}\mathbf{b} \nabla_x \Phi_a| \dd x+ \Vert \nabla_x \phi_E\Vert_{L^\infty_x}\int_{\O} |\Phi_a| \int_{\mathbb{R}^3} \mu^{1/4} e^{\lambda t} f \dd v \dd x + \lambda \int_{\O} |e^{\lambda t}a \Phi_a| \dd x 
\\& \lesssim o(1)\Vert \nabla_x \Phi_a\Vert_{L^2_x}^2 + \Vert e^{\lambda t}\mathbf{b}\Vert_{L^2_x}^2 + [\Vert \nabla_x \phi_E\Vert_{L^\infty_x}+\lambda] \Vert e^{\lambda t}f\Vert_{L^2_{x,v}}^2 .
\end{split}
\ee
Combining with \eqref{eq3:macro_l2_dynamical}, we derive that
\be \notag
\Vert \nabla_x \Phi_a\Vert_{L^2_x}^2 \lesssim \Vert e^{\lambda t}\mathbf{b}\Vert_{L^2_x}^2 + [\Vert \nabla_x \phi_E\Vert_{L^\infty_x}+\lambda] \Vert e^{\lambda t}f\Vert_{L^2_{x,v}}^2.
\ee
Therefore, we control the contribution of $\psi_a$ in $\eqref{weak_formulation_dyna}_1$ by
\be \label{pt_psi_a_bdd}
\begin{split}
& \Big|\int_0^T \iint_{\O\times \mathbb{R}^3} e^{\lambda t}f \delta \p_t \psi_a \dd x \dd v \dd t  \Big|
\\& \lesssim \delta\Vert e^{\lambda t}\mathbf{b}\Vert_{L^2_{T,x}} \Vert \nabla_x \Phi_a\Vert_{L^2_{T,x}}  + \delta \Vert e^{\lambda t}(\mathbf{I}-\mathbf{P})f\Vert_{L^2_{T,x,v}} \Vert \nabla_x \Phi_a\Vert_{L^2_{T,x}} 
\\& \lesssim \delta[\Vert \nabla_x \phi_E\Vert_{L^\infty_x} + \lambda ]\Vert e^{\lambda t }f\Vert_{L^2_{T,x,v}}^2 + \delta\Vert e^{\lambda t}\mathbf{b}\Vert_{L^2_{T,x}}^2 + \delta \Vert e^{\lambda t}(\mathbf{I}-\mathbf{P})f\Vert_{L^2_{T,x,v}}^2. 
\end{split}
\ee

For the contribution of $\psi_{b_i}, \psi_c$ in $\eqref{weak_formulation_dyna}_1$, we set
$\Phi_{b_i} = \partial_t \phi_{b_i}$, $\Phi_c = \partial_t \phi_c$, and they satisfy
\be \notag
\begin{split}
- \Delta_x \Phi_{b_i} - \p_{x_i}^2\Phi_{b_i} = \p_t (e^{\lambda t}b_i), \qquad \Phi_{b_i}|_{\p\O} = 0, &
\\ - \Delta_x \Phi_c = \p_t (e^{\lambda t}c), \qquad  \Phi_c|_{\p\O} = 0. &
\end{split}
\ee
Using integration by parts, we have
\be \label{eq4:macro_l2_dynamical}
\Vert \nabla_x \Phi_{b_i}\Vert_{L^2_x}^2 + \Vert \p_{x_i} \Phi_{b_i}\Vert^2_{L^2_x}     =  \int_{\O} \p_t (e^{\lambda t} b_i) \Phi_{b_i}  \dd x , \qquad
\Vert\nabla_x \Phi_c\Vert^2_{L^2_x}    =  \int_{\O} \p_t (e^{\lambda t}c) \Phi_c  \dd x.
\ee

For $\p_t(e^{\lambda t} b_i)$ and $\p_t(e^{\lambda t} c)$, we apply the momentum and energy conservation laws to \eqref{eq:exp_f_dynamical}. Define
\[
\Theta_{ij}(f):= \int_{\mathbb{R}^3}(v_iv_j-1)f\sqrt{\mu} \dd v,
\qquad
\Lambda_{j}(f):= \int_{\mathbb{R}^3}\frac{1}{10}(|v|^2-5)v_j f\sqrt{\mu} \dd v.
\]
Then we obtain
\be \notag
\begin{split}
& \p_t (e^{\lambda t} b_i) + e^{\lambda t}\p_{x_i}(a+2c) + \nabla_x \cdot \Theta_{i}((\mathbf{I}-\mathbf{P})e^{\lambda t}f) 
\\& \qquad - \int_{\mathbb{R}^3} v_i\sqrt{\mu} e^{\lambda t}\nabla_x \phi_E \cdot \nabla_v f \dd v + \int_{\mathbb{R}^3} v_i\nabla_x \phi_f \cdot \nabla_v (e^{\lambda t}f\sqrt{\mu}) \dd v  
\\& = - \int_{\mathbb{R}^3} \mu v_i e^{-\phi_E/2} e^{\lambda t}v\cdot \nabla_x (\phi_f-\phi_{h})\mu \dd v + \int_{\mathbb{R}^3} v_i e^{\lambda t}\nabla_x (\phi_f -\phi_{h}) \cdot \nabla_v (h\sqrt{\mu})    +  \lambda e^{\lambda t}b_i,
\\& \p_t (e^{\lambda t}c) + \frac{e^{\lambda t}}{3}\nabla_x \cdot \mathbf{b} + \frac{1}{6} \nabla_x \cdot \Lambda((\mathbf{I}-\mathbf{P})e^{\lambda t}f) - \int_{\mathbb{R}^3} \frac{|v|^2-3}{2}\sqrt{\mu} \nabla_x \phi_E \cdot \nabla_v e^{\lambda t}f \dd v 
\\& \qquad + \int_{\mathbb{R}^3} \frac{|v|^2-3}{2} \nabla_x \phi_f \cdot \nabla_v (e^{\lambda t}f\sqrt{\mu}) \dd v 
= \int_{\mathbb{R}^3} \frac{|v|^2-3}{2}e^{\lambda t}\nabla_x (\phi_f -\phi_{h}) \cdot \nabla_v (h\sqrt{\mu}) \dd v + \lambda e^{\lambda t} c.
\end{split}
\ee
By the zero Dirichlet boundary conditions $\Phi_{b_i}|_{\p\O} = \Phi_c |_{\partial\Omega} = 0$, the Poincaré inequality shows
\be \notag
\Vert \Phi_{b_i} \Vert_{L^2_x} \lesssim \Vert \nabla_x \Phi_{b_i} \Vert_{L^2_x}, \qquad
\Vert \Phi_c \Vert_{L^2_x} \lesssim \Vert \nabla_x \Phi_c \Vert_{L^2_x}.
\ee
This implies that
\be \notag
\begin{split}
& \Big|\int_{\O} \p_t (e^{\lambda t}b_i)\Phi_{b_i} \dd x  \Big| 
\\& \lesssim  \int_{\O} |\nabla_x\Phi_{b_i}|[e^{\lambda t}(a+2c) + \Theta_i ((\mathbf{I}-\mathbf{P})e^{\lambda t}f)]  \dd x 
+ [\Vert \nabla_x \phi_E\Vert_{L^\infty_x}+ \Vert \nabla_x \phi_f\Vert_{L^2_x}]\int_{\O} \int_{\mathbb{R}^3} \mu^{1/4} e^{\lambda t} |f| \Phi_{b_i}\dd x \dd v 
\\& \qquad + \iint_{\O\times \mathbb{R}^3} \mu^{1/2} e^{\lambda t}|\nabla_x (\phi_f -\phi_{h})|\Phi_{b_i} \dd x \dd v + \lambda \iint_{\O\times \mathbb{R}^3} e^{\lambda t} b_i \Phi_{b_i} \dd x \dd v 
\\& \qquad + \Vert wh\Vert_{L^\infty_{x,v}} \iint_{\O\times \mathbb{R}^3} e^{\lambda t}|\nabla_x (\phi_f-\phi_{h})| \mu^{1/2} \Phi_{b_i} \dd x \dd v  
\\& \lesssim o(1)\Vert \nabla_x \Phi_{b_i}\Vert_{L^2_{x}}^2 + \Vert e^{\lambda t}f\Vert_{L^2_{x,v}}^2 + \Vert e^{\lambda t} \nabla_x (\phi_f -\phi_{h})\Vert_{L^2_x}^2,  
\end{split}
\ee
and
\be \notag
\begin{split}
\\& \Big| \int_{\O} \p_t (e^{\lambda t}c)\Phi_c \dd x \Big| 
\\& \lesssim \int_{\O}  |\nabla_x\Phi_{c}||e^{\lambda t}\mathbf{b} + \Lambda ((\mathbf{I}-\mathbf{P})e^{\lambda t}f)| \dd x + [\Vert \nabla_x \phi_E\Vert_{L^\infty_x}+ \Vert \nabla_x \phi_f\Vert_{L^2_x}]\int_{\O} \int_{\mathbb{R}^3} \mu^{1/4} e^{\lambda t} |f| \Phi_{c}\dd x \dd v  
\\& \qquad + \Vert wh\Vert_{L^\infty_{x,v}} \iint_{\O\times \mathbb{R}^3} e^{\lambda t}|\nabla_x (\phi_f-\phi_{h})| \mu^{1/2} \Phi_{c} \dd x \dd v + \lambda \iint_{\O\times \mathbb{R}^3} e^{\lambda t} c \Phi_{c} \dd x \dd v 
\\& \lesssim o(1)\Vert \nabla_x \Phi_{c}\Vert_{L^2_{x}}^2 + \Vert e^{\lambda t}f\Vert_{L^2_{x,v}}^2 + \Vert e^{\lambda t} \nabla_x (\phi_f -\phi_{h})\Vert_{L^2_x}^2.
\end{split}
\ee
Combining with \eqref{eq4:macro_l2_dynamical}, we derive that
\be \notag
\sum_{i=1}^3 \Vert \nabla_x \Phi_{b_i}\Vert_{L^2_x}^2 + \Vert \nabla_x \Phi_c\Vert_{L^2_x}^2 \lesssim  \Vert e^{\lambda t}f\Vert_{L^2_{x,v}}^2  +\Vert e^{\lambda t} \nabla_x (\phi_f -\phi_{h})\Vert_{L^2_x}^2.
\ee
Therefore, we control the contribution of $\psi_{b_i}, \psi_c$ in $\eqref{weak_formulation_dyna}_1$ by
\be \label{pt_psi_bc_bdd}
\begin{split}
& \int_0^T \iint_{\O\times \mathbb{R}^3} \Big( \sum_{i=1}^3|e^{\lambda t}(\mathbf{I}-\mathbf{P})f\p_t \psi_{b_i}| + |e^{\lambda t}(\mathbf{I}-\mathbf{P})f\p_t \psi_c| \Big) \dd x \dd v \dd t 
\\& \lesssim \Vert (\mathbf{I}-\mathbf{P})e^{\lambda t}f\Vert_{L^2_{T,x,v}} [\Vert e^{\lambda t}f\Vert_{L^2_{T,x,v}}  +\Vert e^{\lambda t} \nabla_x (\phi_f -\phi_{h})\Vert_{L^2_{T,x}}] 
\\& \lesssim o(1)\Vert \mathbf{P} e^{\lambda t}f\Vert_{L^2_{T,x,v}}^2 + o(1)\Vert e^{\lambda t}\nabla_x (\phi_f - \phi_{h})\Vert_{L^2_{T,x}}^2 + \Vert (\mathbf{I}-\mathbf{P})e^{\lambda t}f\Vert_{L^2_{T,x,v}}^2. 
\end{split}
\ee
Collecting \eqref{pt_psi_a_bdd} and \eqref{pt_psi_bc_bdd}, and using the test function $\psi$ in \eqref{eq5:macro_l2_dynamical}, we control $\eqref{weak_formulation_dyna}_1$ by
\be \notag
\begin{split}
|\eqref{weak_formulation_dyna}_1| 
& \lesssim \delta \Vert e^{\lambda t}\mathbf{b}\Vert_{L^2_{T,x}}^2 + \Vert (\mathbf{I}-\mathbf{P})e^{\lambda t}f\Vert_{L^2_{T,x,v}}^2 +   o(1)\Vert e^{\lambda t}\nabla_x (\phi_f - \phi_{h})\Vert_{L^2_{T,x}}^2 
\\& \qquad + [o(1)+\delta \Vert \nabla_x \phi_E\Vert_{L^\infty_x}+\delta\lambda] \Vert e^{\lambda t} \mathbf{P}f\Vert_{L^2_{T,x,v}}^2. 
\end{split}
\ee

Since $\Vert wh\Vert_{L^\infty_{x,v}} \ll 1$ from Theorem \ref{thm:steady_wellpose} and the inflow condition \ref{inflow_condition}, there exists $0<\lambda\ll1$ such that
\be \label{eq6:macro_l2_dynamical}
\lambda  + \Vert wh \Vert_{L^\infty_{x,v}} + \Vert \nabla_x \phi_E\Vert_{L^\infty_x} + \Vert \phi_E\Vert_{L^\infty_x}  \ll 1.
\ee
Collecting the estimates on $\eqref{weak_formulation_dyna}_1$--$\eqref{weak_formulation_dyna}_{11}$, together with the smallness condition \eqref{eq6:macro_l2_dynamical}, and absorbing the small terms into the left-hand side, we obtain
\begin{align*}
& \delta \Vert e^{\lambda t}a\Vert_{L^2_{T,x}}^2 + \Vert e^{\lambda t}\mathbf{b}\Vert_{L^2_{T,x}}^2 + \Vert e^{\lambda t}c\Vert_{L^2_{T,x}}^2 + \delta \Vert e^{\lambda t}\nabla_x (\phi_f-\phi_{h})\Vert_{L^2_{T,x}}^2 \\
& \lesssim \Vert (\mathbf{I}-\mathbf{P})e^{\lambda t}f\Vert_{L^2_{T,x,v}}^2 + |e^{\lambda t}f|_{L^2_{T,\gamma_+}}^2 + \Vert \nu^{-1/2}e^{\lambda t}\Gamma(f,f)\Vert_{L^2_{T,x,v}}^2 + \Vert e^{\lambda T}f(T)\Vert_{L^2_{x,v}}^2 + \Vert f_0\Vert_{L^2_{x,v}}^2 \\
& \qquad + \delta \Vert e^{\lambda t}\mathbf{b}\Vert_{L^2_{T,x}}^2 + \delta_1 \Vert e^{\lambda t}\nabla_x (\phi_f-\phi_{h})\Vert_{L^2_{T,x}}^2 + \delta_1 \Vert e^{\lambda t} \mathbf{P}f\Vert_{L^2_{T,x,v}}^2 + \Vert e^{\lambda t} wf\Vert_{L^\infty_{T,x,v}}\Vert e^{\lambda t/2}f\Vert_{L^2_{T,x,v}}^2,
\end{align*}
where the constant $0 < \delta_1 \ll 1$. 
We now choose $\delta_1 < \frac{\delta}{5} \ll 1$, which yields
\be \notag
\begin{split}
& \frac{\delta}{2} \Vert e^{\lambda t}a\Vert_{L^2_{T,x}}^2 + \frac{1}{2}\Vert e^{\lambda t}\mathbf{b}\Vert_{L^2_{T,x}}^2 + \frac{1}{2}\Vert e^{\lambda t}c\Vert_{L^2_{T,x}}^2 + \frac{\delta}{2} \Vert e^{\lambda t}\nabla_x (\phi_f-\phi_{h})\Vert_{L^2_{T,x}}^2 
\\& \lesssim \Vert (\mathbf{I}-\mathbf{P})e^{\lambda t}f\Vert_{L^2_{T,x,v}}^2 + |e^{\lambda t}f|_{L^2_{T,\gamma_+}}^2 + \Vert \nu^{-1/2}e^{\lambda t}\Gamma(f,f)\Vert_{L^2_{T,x,v}}^2
\\& \qquad + \Vert e^{\lambda T}f(T)\Vert_{L^2_{x,v}}^2 + \Vert f_0\Vert_{L^2_{x,v}}^2+ \Vert e^{\lambda t} wf\Vert_{L^\infty_{T,x,v}}\Vert e^{\lambda t/2}f\Vert_{L^2_{T,x,v}}^2.
\end{split}
\ee
Therefore, we conclude the lemma.
\end{proof}

Combining the energy estimate in Lemma \ref{lemma:l2_energy_dynamical} and the macroscopic estimate in Lemma \ref{lemma:macro_l2_dynamical}, we obtain the following $L^2$ estimate.

\begin{proposition} \label{prop:l2_dynamical}

Let $(f, \phi_f)$ be a solution to \eqref{linear_f_dynamical} on $[0, T]$. 
There exists $0 < \lambda \ll 1$ such that the following estimate holds:
\begin{align*}
& \Vert e^{\lambda T} f(T)\Vert_{L^2_{x,v}}^2 + \Vert e^{\lambda T} \nabla_x (\phi_f-\phi_{h})(T)\Vert_{L^2_x}^2 + | e^{\lambda t}f|_{L^2_{T,\gamma_+}}^2 + \Vert e^{\lambda t}f\Vert_{L^2_{T,x,\nu}}^2 + \Vert e^{\lambda t} \nabla_x (\phi_f-\phi_{h})\Vert_{L^2_{T,x}}^2 \\
& \lesssim \Vert f_0\Vert_{L^2_{x,v}}^2 + \Vert \nabla_x (\phi_f -\phi_{h})(0)\Vert_{L^2_{x}}^2 + \Vert \nu^{-1/2}e^{\lambda t}\Gamma(f,f)\Vert_{L^2_{T,x,v}}^2 + \Vert e^{\lambda t} wf\Vert_{L^\infty_{T,x,v}}\Vert e^{\lambda t/2}f\Vert_{L^2_{T,x,v}}^2.
\end{align*}
\end{proposition}

\begin{proof}

By Theorem~\ref{thm:steady_wellpose} and the inflow condition \eqref{inflow_condition}, there exists a constant $0 < \e < 1$ such that 
\[
\lambda + \Vert wh\Vert_{L^\infty_{x,v}} + \Vert w_{\tilde{\theta}} \nabla_v h\Vert_{L^\infty_{x,v}} + \Vert \nabla_x \phi_E\Vert_{L^\infty_x} \ll \e.
\]
Multiplying \eqref{macro_l2_dynamical} by $\e$ and adding the result to \eqref{l2_energy_dynamical}, we obtain
\begin{align*}
& \Vert e^{\lambda T}f(T)\Vert_{L^2_{x,v}}^2 + \Vert e^{\lambda T} \nabla_x (\phi_f-\phi_{h})(T)\Vert_{L^2_x} + |e^{\lambda t}f|_{L^2_{T,\gamma_+}}^2 +  \e \Vert e^{\lambda t}f\Vert_{L^2_{T,x,\nu}}^2 + \e \Vert e^{\lambda t}\nabla_x (\phi_f-\phi_{h})\Vert_{L^2_{T,x}}^2 \\
& \lesssim \Vert f_0\Vert_{L^2_{x,v}}^2 + \Vert \nabla_x (\phi_f-\phi_{h})(0)\Vert_{L^2_x}^2 + \Vert \nu^{-1/2}e^{\lambda t}\Gamma(f,f)\Vert_{L^2_{T,x,v}}^2 +  \Vert e^{\lambda t} wf\Vert_{L^\infty_{T,x,v}} \Vert e^{\lambda t/2}f\Vert_{L^2_{T,x,v}}^2 \\
& \qquad + [\lambda + \Vert wh\Vert_{L^\infty_{x,v}} + \Vert w_{\tilde{\theta}} \nabla_v h\Vert_{L^\infty_{x,v}} + \Vert \nabla_x \phi_E\Vert_{L^\infty_x} ] \Vert e^{\lambda t}f\Vert_{L^2_{T,x,\nu}}^2.
\end{align*}
Absorbing the last term on the right-hand side into the left-hand side yields the proposition.
\end{proof}

Similar to Proposition~\ref{prop:wf_Linfty}, we apply the $L^2$–$L^\infty$ bootstrap argument together with the a priori $L^2$ estimate in Proposition~\ref{prop:l2_dynamical} to derive the following a priori $L^{\infty}$ estimate.

\begin{proposition} \label{prop:linfty_dynamical}

Let $(f, \phi_f)$ be a solution to \eqref{linear_f_dynamical} on $[0, T]$. Suppose $\Vert wf \Vert_{L^\infty_{T,x,v}}<\infty$, then we have
\be \label{dynamical_bootstrap}
\Vert e^{\lambda t} wf \Vert_{L^\infty_{T,x,v}} 
\lesssim \Vert wf_0\Vert_{L^\infty_{x,v}} + \Vert e^{\lambda t} wf \Vert_{L^\infty_{T,x,v}}^2 + \Vert e^{\lambda t} f \Vert_{L^\infty_T L^2_{x,v}}. 
\ee
Consequently, if $\Vert wf(t)\Vert_{L^\infty_{x,v}}$ is continuous in $t$ and the assumption $\Vert wf_0\Vert_{L^\infty_{x,v}} < \delta_2 \ll 1$ in Theorem \ref{thm:dynamical_stability} holds, then there exists a constant $C_2 > 1$ such that
\be \label{linfty_bdd_dynamical}
\Vert e^{\lambda t} wf \Vert_{L^\infty_{T,x,v}} < C_2 \delta_2 \ll 1. 
\ee
\end{proposition}

\begin{proof}

The proof of \eqref{dynamical_bootstrap} is similar to the $L^\infty_{x,v}$ estimate in Proposition~\ref{prop:wf_Linfty}; see also the corresponding $L^\infty_{x,v}$ estimate in \cite{CKL}. For conciseness, we omit the details.

For \eqref{linfty_bdd_dynamical}, we apply Proposition~\ref{prop:l2_dynamical} to the term $\Vert e^{\lambda t} f \Vert_{L^\infty_T L^2_{x,v}}$ in \eqref{dynamical_bootstrap}, which implies
\be \notag
\begin{split}
\Vert e^{\lambda t} f \Vert_{L^\infty_T L^2_{x,v}} 
& \lesssim \Vert f_0\Vert_{L^2_{x,v}}^2 +  \Vert \nu^{-1} e^{-\lambda t} w \Gamma(e^{\lambda t}f,e^{\lambda  t}f)\Vert_{L^2_T L^\infty_{x,v}} + \Vert e^{\lambda t}wf\Vert_{L^\infty_{x,v}}^{1/2} \Vert e^{-\lambda t/2} e^{\lambda t}wf\Vert_{L^2_TL^\infty_{x,v}} \\& \lesssim  \Vert wf_0\Vert_{L^\infty_{x,v}} + \Vert e^{\lambda t}wf\Vert_{L^\infty_{T,x,v}}^2 + \Vert e^{\lambda t}wf\Vert_{L^\infty_{T,x,v}}^{3/2}. 
\end{split}
\ee
Then \eqref{linfty_bdd_dynamical} follows from a standard continuity argument under the smallness assumption on $\Vert wf_0\Vert_{L^\infty_{x,v}}$.
\end{proof}

\subsection{A priori uniform-in-time regularity estimate} \label{sec:dynamical_regularity}

In this section, we establish the a priori uniform-in-time regularity estimates for the dynamical problem \eqref{linear_f_dynamical} in Proposition~\ref{prop:W1p_dynamical}. 
The argument is analogous to the a priori regularity estimates for the stationary problem developed in Sections~\ref{sec:w1p_estimate} and \ref{sec:c1_estimate}.

Due to the lack of control on $\nabla^2_v h$ (see Remark~\ref{Remark:dynamical} for further details), instead of considering $F = e^{-\phi_E}\mu + e^{-\phi_E/2}\sqrt{\mu}(h+f)$ and the corresponding equation \eqref{linear_f_dynamical} for $f$, we rewrite the solution $F (t,x,v)$ as 
\begin{equation} \notag
F (t,x,v)= e^{-\phi_E}\mu + e^{-\phi_E/2} \sqrt{\mu} \mathfrak{f}
\ \text{ with } \
\mathfrak{f} = h + f. 
\end{equation} 
The equations for $\mathfrak{f}$ are given by
\be \label{eqn:mkf}
\begin{cases}
& \p_t \mathfrak{f} + v \cdot \nabla_x \mathfrak{f} - \nabla_x (\phi_\mathfrak{f} + \phi_E) \cdot \nabla_{v}  \mathfrak{f} + \frac{v\cdot \nabla_x\phi_\mathfrak{f}}{2}\mathfrak{f} + e^{-\phi_E}\mathcal{L}\mathfrak{f} = -(v\cdot \nabla_x\phi_\mathfrak{f})e^{-\phi_E/2}\sqrt{\mu}  + e^{-\phi_E/2}\Gamma(\mathfrak{f},\mathfrak{f}), \\
& \mathfrak{f}(x,v)|_{\gamma_-} = f_b, \quad \mathfrak{f}(0,x,v) = \mathfrak{f}_0 (x,v) = \frac{F(0,x,v) - e^{-\phi_E}\mu}{e^{-\phi_E/2}\sqrt{\mu}}, \\
&  -\Delta_x \phi_\mathfrak{f} = -\Delta_x \phi_f = e^{-\phi_E/2}\int_{\mathbb{R}^3}  \mathfrak{f} \sqrt{\mu} \dd v \text{ in }\O, \quad \phi_\mathfrak{f} = 0 \text{ on } \p\O, \\
& -\p_n \phi_E > C_E >0 \text{ on } \p\O.   
\end{cases}  
\ee

Note that the self-consistent fields in \eqref{linear_f_dynamical} and \eqref{eqn:mkf} are equal, namely, $\phi_{\mathfrak{f}} = \phi_f$.
Therefore, the potential field in \eqref{eqn:mkf} satisfies 
\[
E_\mathfrak{f} (t,x) := - \nabla_x (\phi_\mathfrak{f} + \phi_E) = -\nabla_x (\phi_f + \phi_E).
\]
Following the kinetic weight in \eqref{alpha_weight_dyna}, we continue to use the notation $\alpha_f(t,x,v)$ for the dynamical kinetic weight associated with \eqref{eqn:mkf}.
Moreover, we impose the following a priori assumption:
\begin{align}
& \Vert \nabla_x \phi_\mathfrak{f} \Vert_{L^\infty_{t,x}} \lesssim \Vert w \mathfrak{f} \Vert_{L^\infty_{t,x,v}} = \Vert w(h+f) \Vert_{L^\infty_{t,x,v}} \ll 1, \label{apriori_C1_dynamical} \\
& \Vert \nabla_x^2 \phi_\mathfrak{f}\Vert_{L^\infty_{t,x}} \ll 1.
\label{apriori_C2_dynamical}
\end{align}
Under these assumptions, $E_\mathfrak{f}(t,x)$ satisfies
\be \notag
E_\mathfrak{f} (t,x) \cdot n(\tilde{x}) > \frac{C_E}{2} 
\ \text{ for } 
x \in \O_\delta.
\ee
Here, the first assumption \eqref{apriori_C1_dynamical} follows from Proposition~\ref{prop:linfty_dynamical} and Theorem~\ref{thm:steady_wellpose}.
The second assumption \eqref{apriori_C2_dynamical} will be justified by Lemma~\ref{lemma:phi_C2} together with a continuity argument (see Section~\ref{sec:dynamical_proof} for further details).

\begin{proposition} \label{prop:W1p_dynamical}

Let $(\mathfrak{f}, \phi_\mathfrak{f})$ be a solution to \eqref{eqn:mkf}.
Assume the a priori assumptions \eqref{apriori_C1_dynamical} and \eqref{apriori_C2_dynamical}. Let $\delta_1,\delta_2\ll1$ be the constants defined in Theorem~\ref{thm:steady_wellpose} and Theorem~\ref{thm:dynamical_stability}. If the initial condition satisfies
\begin{align*}
    & \Vert w_{\tilde{\theta}} \p_{x,v}\mathfrak{f}_0\Vert_{L^p_{x,v}} + \Vert w_{\tilde{\theta}} \alpha_f \p_{x,v}\mathfrak{f}_0\Vert_{L^\infty_{x,v}} + \Vert w_{\tilde{\theta}} \nabla_v \mathfrak{f}_0\Vert_{L^\infty_{x,v}} < \delta_1+\delta_2 \ll 1,
\end{align*}
then for any $t>0$, there exists a time-independent constant $C_4>1$ such that
\be \label{h_uniform_regularity}
\Vert  w_{\tilde{\theta}}\p_{x,v} \mathfrak{f}(t)\Vert_{L^p_{x,v}} + \Vert w_{\tilde{\theta}} \alpha_f \p_{x,v}\mathfrak{f}(t)\Vert_{L^\infty_{x,v}} + \Vert w_{\tilde{\theta}} \nabla_v \mathfrak{f}(t)\Vert_{L^\infty_{x,v}} 
\leq C_4 [\delta_1 + \delta_2].  
\ee
\end{proposition}

\begin{proof}

For simplicity, we write $(X(s), V(s)) := (X(s;t,x,v), V(s;t,x,v))$ throughout the proof.
The proof of \eqref{h_uniform_regularity} is similar to the $W^{1, p}$ and $C^1$ estimates in Sections~\ref{sec:w1p_estimate} and \ref{sec:c1_estimate}.
In what follows, we focus on the key steps and the main differences from the previous arguments.

Analogous to \eqref{eq2:wf_Linfty} in the proof of Proposition \ref{prop:wf_Linfty}, we define the damping factor $\tilde{\nu}(t,x,v)$ for \eqref{eqn:mkf} by
\be \notag
\tilde{\nu} (t,x,v) := e^{-\phi_E(x)}\nu(v) + \frac{v\cdot \nabla_x \phi_\mathfrak{f}(t,x)}{2}.
\ee
From the assumptions \eqref{apriori_C1_dynamical} and \eqref{apriori_C2_dynamical}, there exists a constant $T_0 > 1$ such that 
\be \label{phi_T0_small}
\Vert \nabla_x (\phi_\mathfrak{f}+\phi_E)\Vert_{C^1_{t,x}}T_0 \ll 1.
\ee
Moreover, Lemma \ref{lemma:k_nu} implies that
\be \notag
\tilde{\nu} (t,x,v) > \frac{\nu(v)}{2}> \frac{\nu_0}{2}
\ \text{ for }
0 \leq t \leq T_0.
\ee
Analogously to \eqref{nabla_nu}, using \eqref{deri_XV_bdd} and Lemma \ref{lemma:phi_C2}, we derive that for $\max \{ 0, t - \tb \} \leq s \leq t \leq T_0$,
\be \label{nabla_nu_dynamical}
\begin{split}
& |\nabla_x \tilde{\nu}(s,X(s),V(s))| 
\\& \lesssim | \nabla_x \phi_E(X(s))| |\nabla_x X(s)| |\nu(V(s))| + |\nabla_v \nu(V(s))| |\nabla_x V(s)| 
\\& \qquad + |\nabla_x^2 \phi_\mathfrak{f} (s,X(s))| |\nabla_x X(s)||V(s)|+ |\nabla_x \phi_\mathfrak{f}(s,X(s)) | |\nabla_x V(s)| 
\\& \lesssim e^{o(1)(t-s)}\nu(v)[o(1)(\Vert \nabla_x \mathfrak{f}\Vert_{L^\infty_t L^p_{x,v}} +\Vert \alpha_f \nabla_x \mathfrak{f}\Vert_{L^\infty_{t,x,v}}) + \Vert w\mathfrak{f}\Vert_{L^\infty_{t,x,v}} + \Vert \nabla_x \phi_E\Vert_{L^\infty_{x}} + 1 ]. 
\end{split}
\ee

We now claim that it suffices to prove that for $0\leq t\leq T_0$,
\be \label{regularity_short_time}
\begin{split}
& \Vert w_{\tilde{\theta}}\alpha_f \p_{x,v}\mathfrak{f}(t)\Vert_{L^\infty_{x,v}} + \Vert w_{\tilde{\theta}} \p_{x,v}\mathfrak{f}(t)\Vert_{L^p_{x,v}} + \Vert w_{\tilde{\theta}} \nabla_v \mathfrak{f}(t)\Vert_{L^\infty_{x,v}} 
\\& \lesssim e^{-\frac{\nu_0t}{4}} \Big[ \Vert w_{\tilde{\theta}} \alpha_f \p_{x,v}\mathfrak{f}_0\Vert_{L^\infty_{x,v}} + \Vert w_{\tilde{\theta}} \p_{x,v}\mathfrak{f}_0\Vert_{L^p_{x,v}} + \Vert w_{\tilde{\theta}} \nabla_v \mathfrak{f}_0\Vert_{L^\infty_{x,v}} \Big] 
\\& \qquad + \Vert w\mathfrak{f}\Vert_{L^\infty_{T_0,x,v}} + |wf_b|_{L^\infty_{\p\O,v}} + |w \p_{\mathbf{x}_p,v}f_b |_{L^\infty_{\p\O,v}}.
\end{split}
\ee
Assuming \eqref{regularity_short_time}, by taking $t = T_0$ and choosing $T_0 > 1$ sufficiently large, we obtain
\be \label{regularity_T0}
\begin{split}
& \Vert w_{\tilde{\theta}}\alpha_f \p_{x,v}\mathfrak{f}(T_0)\Vert_{L^\infty_{x,v}} + \Vert w_{\tilde{\theta}} \p_{x,v}\mathfrak{f}(T_0)\Vert_{L^p_{x,v}} + \Vert w_{\tilde{\theta}} \nabla_v \mathfrak{f}(T_0)\Vert_{L^\infty_{x,v}} 
\\& \leq e^{-\frac{\nu_0T_0}{8}} \Big[ \Vert w_{\tilde{\theta}} \alpha_f \p_{x,v}\mathfrak{f}_0\Vert_{L^\infty_{x,v}} + \Vert w_{\tilde{\theta}} \p_{x,v}\mathfrak{f}_0\Vert_{L^p_{x,v}} + \Vert w_{\tilde{\theta}} \nabla_v \mathfrak{f}_0\Vert_{L^\infty_{x,v}} \Big] 
\\& \qquad + C\Big[ \Vert w\mathfrak{f}\Vert_{L^\infty_{T_0,x,v}} + |wf_b|_{L^\infty_{\p\O,v}} + |w \p_{\mathbf{x}_p,v}f_b |_{L^\infty_{\p\O,v}} \Big]. 
\end{split}
\ee
From the assumption \eqref{apriori_C1_dynamical}, we define
\[
C' := C \Big[ \sup_t\Vert w\mathfrak{f}(t)\Vert_{L^\infty_{x,v}}
+ |wf_b|_{L^\infty_{\p\O,v}} + |w \p_{\mathbf{x}_p,v}f_b |_{L^\infty_{\p\O,v}} \Big] < \infty.
\]
Iterating \eqref{regularity_T0}, we obtain for any $n\in \mathbb{Z}^+$,
\begin{align*}
& \Vert w_{\tilde{\theta}} \alpha_f \p_{x,v} \mathfrak{f}(nT_0)\Vert_{L^\infty_{x,v}} + \Vert w_{\tilde{\theta}} \p_{x,v}\mathfrak{f}(nT_0)\Vert_{L^p_{x,v}} + \Vert w_{\tilde{\theta}} \p_{v}\mathfrak{f}(nT_0)\Vert_{L^\infty_{x,v}} \\
& \leq e^{-\frac{\nu_0 T_0}{8}} [\Vert w_{\tilde{\theta}} \alpha_f \p_{x,v}\mathfrak{f}((n-1)T_0)\Vert_{L^\infty_{x,v}} + \Vert w_{\tilde{\theta}} \p_{x,v}\mathfrak{f}((n-1)T_0)\Vert_{L^p_{x,v}} + \Vert w_{\tilde{\theta}} \p_{v}\mathfrak{f}((n-1)T_0)\Vert_{L^\infty_{x,v}}] + C' \\
& \leq e^{-\frac{2\nu_0T_0}{8}} [\Vert w_{\tilde{\theta}} \alpha_f \p_{x,v}\mathfrak{f}((n-2)T_0)\Vert_{L^\infty_{x,v}} + \Vert w_{\tilde{\theta}} \p_{x,v}\mathfrak{f}((n-2)T_0)\Vert_{L^p_{x,v}} \\
& \qquad \qquad \qquad + \Vert w_{\tilde{\theta}} \p_{v}\mathfrak{f}((n-2)T_0)\Vert_{L^\infty_{x,v}}] + C'[1+e^{-\frac{\nu_0T_0}{8}}] \\
& \leq \ldots \leq e^{-\frac{n\nu_0 T_0}{8}} [\Vert w_{\tilde{\theta}} \alpha_f \p_{x,v}\mathfrak{f}_0\Vert_{L^\infty_{x,v}} + \Vert w_{\tilde{\theta}} \p_{x,v}\mathfrak{f}_0\Vert_{L^p_{x,v}}+ \Vert w_{\tilde{\theta}} \p_{v}\mathfrak{f}_0\Vert_{L^\infty_{x,v}}] + C' \sum_{j=0}^{n} e^{-\frac{j\nu_0T_0}{8}}.
\end{align*}
Hence, for every $n \in \mathbb{Z}^+$,
\begin{align*}
& \Vert w_{\tilde{\theta}} \alpha_f \p_{x,v}\mathfrak{f}(nT_0)\Vert_{L^\infty_{x,v}} + \Vert w_{\tilde{\theta}} \p_{x,v}\mathfrak{f}(nT_0)\Vert_{L^p_{x,v}} + \Vert w_{\tilde{\theta}} \p_v\mathfrak{f}(nT_0)\Vert_{L^\infty_{x,v}} \\
& \lesssim \Vert w_{\tilde{\theta}} \alpha_f \p_{x,v}\mathfrak{f}_0\Vert_{L^\infty_{x,v}} + \Vert w_{\tilde{\theta}} \p_{x,v}\mathfrak{f}_0\Vert_{L^p_{x,v}} + \Vert w_{\tilde{\theta}} \p_v\mathfrak{f}_0\Vert_{L^\infty_{x,v}} \\
& \qquad + \sup_t\Vert w\mathfrak{f}(t)\Vert_{L^\infty_{x,v}} + |wf_b|_{L^\infty_{\p\O,v}} + |w \p_{\mathbf{x}_p,v}f_b |_{L^\infty_{\p\O,v}},
\end{align*}
where the implicit constant in the inequality is independent of $n$.
Furthermore, for any $n \in \mathbb{Z}^+$ and $(n-1)T_0 \le t \le nT_0$, we apply \eqref{regularity_short_time} to obtain
\begin{align*}
& \Vert w_{\tilde{\theta}} \alpha_f \p_{x,v}\mathfrak{f}(t)\Vert_{L^\infty_{x,v}} + \Vert w_{\tilde{\theta}} \p_{x,v}\mathfrak{f}(t)\Vert_{L^p_{x,v}} + \Vert w_{\tilde{\theta}} \p_v\mathfrak{f}(t)\Vert_{L^\infty_{x,v}} \\
& \lesssim \Vert w_{\tilde{\theta}} \alpha_f \p_{x,v}\mathfrak{f}_0\Vert_{L^\infty_{x,v}} + \Vert w_{\tilde{\theta}} \p_{x,v}\mathfrak{f}_0\Vert_{L^p_{x,v}} + \Vert w_{\tilde{\theta}} \p_v\mathfrak{f}_0\Vert_{L^\infty_{x,v}} + \sup_t\Vert w\mathfrak{f}(t)\Vert_{L^\infty_{x,v}} + |wf_b|_{L^\infty_{\p\O,v}} + |w \p_{\mathbf{x}_p,v}f_b |_{L^\infty_{\p\O,v}},
\end{align*}
where the implicit constant in the inequality is independent of $n$ and $t$. 
Applying Proposition \ref{prop:linfty_dynamical} and Theorem \ref{thm:steady_wellpose}, we conclude the proposition for some constant $C_4 > 1$.

\smallskip

We now prove the claim \eqref{regularity_short_time}.
Let $0 \leq t \leq T_0$ and define $t^1(t,x,v) := t - \tb(t,x,v)$. 
First, we show that for any $0 \le t \le T_0$,
\be \label{W1p_short_time}
\begin{split}
\Vert w_{\tilde{\theta}} \p_{x,v}\mathfrak{f}(t)\Vert_{L^p_{x,v}} &\lesssim e^{-\frac{\nu_0 t}{4}} \Vert w_{\tilde{\theta}} \p_{x,v}\mathfrak{f}_0\Vert_{L^p_{x,v}} + o(1)[\Vert w_{\tilde{\theta}} \alpha_f \p_{x,v} \mathfrak{f} \Vert_{L^\infty_{t,x,v}} + \Vert w_{\tilde{\theta}} \p_{x,v}\mathfrak{f}\Vert_{L^\infty_t L^p_{x,v}}] 
\\& \qquad + \Vert w\mathfrak{f}\Vert_{L^\infty_{t,x,v}} + |wf_b|_{L^\infty_{\p\O,v}} + |w \p_{\mathbf{x}_p,v}f_b |_{L^\infty_{\p\O,v}}.
\end{split}
\ee

Applying the method of characteristics to \eqref{eqn:mkf}, we obtain that for any $0 \le t \le T_0$,
\be \label{f_dynamical_piecewise}
\begin{split}
\mathfrak{f}(t,x,v) 
& = \mathbf{1}_{\tb < t} e^{-\int^t_{t^1} \tilde{\nu}(s,X(s),V(s)) \dd s}  f_b(\xb,\vb) 
\\& \qquad + \mathbf{1}_{\tb \geq t} e^{-\int^t_0 \tilde{\nu}(s,X(s),V(s))\dd s} \mathfrak{f}_0(X(0),V(0))
\\& \qquad + \int^t_{\max\{t^1,0\}} e^{-\int^t_s \tilde{\nu}(\tau,X(\tau),V(\tau))\dd \tau} \int_{\mathbb{R}^3} \dd u e^{-\phi_E(X(s))}\mathbf{k}(V(s),u) \mathfrak{f}(s,X(s),u) \dd s 
\\& \qquad + \int^t_{\max\{t^1,0\}} e^{-\int^t_s \tilde{\nu}(\tau,X(\tau),V(\tau))\dd \tau} e^{-\frac{\phi_E(X(s))}{2}} \Gamma(\mathfrak{f},\mathfrak{f})(s,X(s),V(s)) \dd s
\\& \qquad + \int^t_{\max\{t^1,0\}} e^{-\int^t_s \tilde{\nu}(\tau,X(\tau),V(\tau))\dd \tau} V(s)\cdot \nabla_x \phi_\mathfrak{f} (s,X(s)) e^{-\frac{\phi_E(X(s))}{2}}\sqrt{\mu(V(s))} \dd s.
\end{split}
\ee
Taking the derivative of \eqref{f_dynamical_piecewise} and multiplying both sides by $w_{\tilde{\theta}}(v)$, we further obtain for any $0 \le t \le T_0$,
{\small
\begin{align}
& w_{\tilde{\theta}}(v) \p_{x,v} \mathfrak{f}(t,x,v) 
\notag \\
& = \mathbf{1}_{\tb>t} e^{-\int^t_{0}\tilde{\nu}(s,X(s),V(s))\dd s} \frac{w_{\tilde{\theta}}(v)}{w_{\tilde{\theta}}(V(0))}w_{\tilde{\theta}}(V(0))[ \p_{x,v} X(0) \p_{x}\mathfrak{f}_0(X(0),V(0)) + \p_{x,v} V(0)\cdot \p_{v}\mathfrak{f}_0(X(0),V(0)) ] 
\label{hchara:nabla_0} \\
& \quad + \mathbf{1}_{\tb > t} e^{-\int_0^t \tilde{\nu}(s,X(s),V(s))\dd s} \int^t_0 \p_{x,v}[\tilde{\nu}(s,X(s),V(s))] \dd s \frac{w_{\tilde{\theta}}(v)}{w_{\tilde{\theta}}(V(0))}w_{\tilde{\theta}}(V(0)) \mathfrak{f}_0(X(0),V(0)) 
\label{hchara:nabla_nu_1} \\
& \quad + \mathbf{1}_{\tb<t} e^{-\int^t_{t^1} \tilde{\nu}(s,X(s),V(s)) \dd s  } \int_{t^1}^t \p_{x,v} [\tilde{\nu}(s,X(s),V(s))] \dd s  \frac{w_{\tilde{\theta}}(v)}{w_{\tilde{\theta}}(\vb)}w_{\tilde{\theta}}(\vb) f_b(\xb,\vb)
\label{hchara:nabla_nu_2} \\
& \quad + \mathbf{1}_{\tb<t}\p_{x,v} \tb \tilde{\nu}(t^1,\xb,\vb) e^{-\int^t_{t^1} \tilde{\nu}(s,X(s),V(s)) \dd s}\frac{w_{\tilde{\theta}}(v)}{w_{\tilde{\theta}}(\vb)}w_{\tilde{\theta}}(\vb) f_b(\xb,\vb) 
\label{hchara:nabla_tb} \\
& \quad + \mathbf{1}_{\tb<t} e^{-\int^t_{t^1} \tilde{\nu}(s,X(s),V(s)) \dd s} \frac{w_{\tilde{\theta}}(v)}{w_{\tilde{\theta}}(\vb)}w_{\tilde{\theta}}(\vb) [  \p_{x,v} \xb \cdot \nabla_x f_b(\xb,\vb)+   \p_{x,v} \vb \cdot \nabla_v f_b(\xb,\vb)  ]  
\label{hchara:bdr} \\
& \quad + \int^t_{\max\{t^1,0\}} e^{-\int^t_{s}  \tilde{\nu}(\tau,X(\tau),V(\tau)) \dd \tau} \frac{w_{\tilde{\theta}}(v)}{w_{\tilde{\theta}}(V(s))}\int_{\mathbb{R}^3} \dd u e^{-\phi_E(X(s))} \mathbf{k}(V(s),u) w_{\tilde{\theta}}(V(s))\p_{x,v} [\mathfrak{f}(s,X(s),u)] \dd s 
\label{hchara:K_nabla} \\
& \quad + \int^t_{\max\{t^1,0\}} e^{-\int^t_s \tilde{\nu}(\tau,X(\tau),V(\tau)) \dd \tau} \frac{w_{\tilde{\theta}}(v)}{w_{\tilde{\theta}}(V(s))}\int_{\mathbb{R}^3} \dd u e^{-\phi_E(X(s))}w_{\tilde{\theta}}(V(s))\p_{x,v} V(s) \cdot \nabla_v \mathbf{k}(V(s),u)  \mathfrak{f}(s,X(s),u) \dd s  
\label{hchara:nabla_K} \\
& \quad + \int^t_{\max\{t^1,0\}} e^{-\int^t_s \tilde{\nu}(\tau,X(\tau),V(\tau)) \dd \tau} \frac{w_{\tilde{\theta}}(v)}{w_{\tilde{\theta}}(V(s))} 
\notag \\
& \qquad \qquad \times \int_{\mathbb{R}^3} \dd u \nabla_{x}\phi_E(X(s)) \p_{x,v}X(s) e^{-\phi_E(X(s))}w_{\tilde{\theta}}(V(s))\mathbf{k}(V(s),u)  \mathfrak{f}(s,X(s),u) \dd s  
\label{hchara:K_nabla_phi_E} \\
& \quad + \mathbf{1}_{\tb<t} \p_{x,v} \tb e^{-\int^t_{t^1} \tilde{\nu}(\tau,X(\tau),V(\tau)) \dd \tau} \frac{w_{\tilde{\theta}}(v)}{w_{\tilde{\theta}}(\vb)}w_{\tilde{\theta}}(\vb) \int_{\mathbb{R}^3} \dd u e^{-\phi_E(\xb)}\mathbf{k}(\vb,u) f_b(\xb,u) 
\label{hchara:nabla_int} \\
& \quad + \int^t_{\max\{t^1,0\}} e^{-\int_s^t \tilde{\nu}(\tau,X(\tau),V(\tau))\dd \tau} \int_s^t \dd \tau\p_{x,v} [\tilde{\nu}(\tau,X(\tau),V(\tau))] \frac{w_{\tilde{\theta}}(v)}{w_{\tilde{\theta}}(V(s))} 
\notag \\
& \qquad \qquad \times \int_{\mathbb{R}^3} \dd u e^{-\phi_E(X(s))} w_{\tilde{\theta}}(V(s))\mathbf{k}(V(s),u) \mathfrak{f}(s,X(s),u) \dd s 
\label{hchara:nabla_nu_K} \\
& \quad + \mathbf{1}_{\tb<t} \p_{x,v} \tb e^{-\int^t_{t^1}\tilde{\nu}(\tau,X(\tau),V(\tau)) \dd \tau} \frac{w_{\tilde{\theta}}(v)}{w_{\tilde{\theta}}(\vb)}w_{\tilde{\theta}}(\vb) e^{-\phi_E(\xb)/2} \Gamma(f_b,f_b)(\xb,\vb) 
\label{hchara:nabla_tb_gamma} \\
& \quad + \int^t_{\max\{t^1,0\}} e^{-\int_s^t \tilde{\nu}(\tau,X(\tau),V(\tau))\dd \tau} 
\notag \\
& \qquad \qquad \times \int_s^t \dd \tau  \p_{x,v} [\tilde{\nu}(\tau,X(\tau),V(\tau))] \frac{w_{\tilde{\theta}}(v)}{w_{\tilde{\theta}}(V(s))} w_{\tilde{\theta}}(V(s))e^{-\frac{\phi_E(X(s))}{2}}\Gamma(\mathfrak{f},\mathfrak{f})(s,X(s),V(s)) \dd s 
\label{hchara:nabla_nu_gamma} \\
& \quad + \int^t_{\max\{t^1,0\}} e^{-\int_s^t \tilde{\nu}(\tau,X(\tau),V(\tau))\dd \tau} \frac{w_{\tilde{\theta}}(v)}{w_{\tilde{\theta}}(V(s))} w_{\tilde{\theta}}(V(s))e^{-\frac{\phi_E(X(s))}{2}}\p_{x,v} [\Gamma(\mathfrak{f},\mathfrak{f})(s,X(s),V(s))] \dd s
\label{hchara:nabla_gamma} \\
& \quad + \int^t_{\max\{t^1,0\}} e^{-\int_s^t \tilde{\nu}(\tau,X(\tau),V(\tau))\dd \tau} \frac{w_{\tilde{\theta}}(v)}{w_{\tilde{\theta}}(V(s))} w_{\tilde{\theta}}(V(s))\frac{\nabla_x \phi_E(X(s))\p_{x,v}X(s)}{2} e^{-\frac{\phi_E(X(s))}{2}}\Gamma(\mathfrak{f},\mathfrak{f})(s,X(s),V(s)) \dd s
\label{hchara:gamma_nabla_phi_E} \\
& \quad + \int^t_{\max\{t^1,0\}} e^{-\int^t_{s} \tilde{\nu}(\tau,X(\tau),V(\tau)) \dd \tau}\frac{w_{\tilde{\theta}}(v)}{w_{\tilde{\theta}}(V(s))} w_{\tilde{\theta}}(V(s))  V(s) \cdot \p_{x,v}[  \nabla_x \phi_{\mathfrak{f}}(s,X(s))  ] e^{-\frac{\phi_E(X(s))}{2}}\sqrt{\mu(V(s))} \dd s   
\label{hchara:nabla_E} \\
& \quad + \mathbf{1}_{\tb<t} \p_{x,v} \tb e^{-\int^t_{t^1} \tilde{\nu}(\tau,X(\tau),V(\tau)) \dd \tau} \frac{w_{\tilde{\theta}}(v)}{w_{\tilde{\theta}}(\vb)} w_{\tilde{\theta}}(\vb)  \vb\cdot \nabla_x \phi_\mathfrak{f} (t^1,\xb)  e^{-\frac{\phi_E(\xb)}{2}}\sqrt{\mu(\vb)} 
\label{hchara:nabla_tb_phi} \\
& \quad + \int^t_{\max\{t^1,0\}} e^{-\int_s^t \tilde{\nu}(\tau,X(\tau),V(\tau))\dd \tau} \int_s^t \dd \tau \p_{x,v} [\tilde{\nu}(\tau,X(\tau),V(\tau))] 
\notag \\
& \qquad \qquad \times \frac{w_{\tilde{\theta}}(v)}{w_{\tilde{\theta}}(v)} w_{\tilde{\theta}}(V(s)) V(s) \cdot \nabla_x \phi_\mathfrak{f}(s,X(s))    e^{-\frac{\phi_E(X(s))}{2}}\sqrt{\mu(V(s))}  \dd s 
\label{hchara:nabla_nu_phi} \\
& \quad + \int^t_{\max\{t^1,0\}} e^{-\int^t_s \tilde{\nu}(\tau,X(\tau),V(\tau)) \dd \tau} \frac{w_{\tilde{\theta}}(v)}{w_{\tilde{\theta}}(V(s))} w_{\tilde{\theta}}(V(s)) \p_{x,v} V(s)\cdot\nabla_x \phi_\mathfrak{f}(s,X(s))    e^{-\frac{\phi_E(X(s))}{2}}\sqrt{\mu(V(s))} \dd s \label{hchara:nabla_V_phi}\\
& \quad + \int^t_{\max\{t^1,0\}} e^{-\int^t_s \tilde{\nu}(\tau,X(\tau),V(\tau)) \dd \tau} \frac{w_{\tilde{\theta}}(v)}{w_{\tilde{\theta}}(V(s))} w_{\tilde{\theta}}(V(s))   V(s)\cdot\nabla_x \phi_\mathfrak{f}(s,X(s)) 
\notag \\
& \qquad \qquad \times \p_{x,v} V(s)\cdot V(s)e^{-\frac{\phi_E(X(s))}{2}}\sqrt{\mu(V(s))} \dd s 
\label{hchara:nabla_mu_phi} \\
& \quad + \int^t_{\max\{t^1,0\}} e^{-\int^t_s \tilde{\nu}(\tau,X(\tau),V(\tau)) \dd \tau} \frac{w_{\tilde{\theta}}(v)}{w_{\tilde{\theta}}(V(s))} w_{\tilde{\theta}}(V(s)) V(s)  \cdot \nabla_x \phi_\mathfrak{f}(s,X(s)) 
\notag \\
& \qquad \qquad \times \frac{\nabla_x \phi_E(X(s))\p_{x,v}X(s)}{2} e^{-\frac{\phi_E(X(s))}{2}}\sqrt{\mu(V(s))} \dd s. \label{hchara:nabla_phi_E_phi}
\end{align}
}
We note that this piecewise derivative formula gives a weak derivative of \eqref{f_dynamical_piecewise} due to the compatibility condition $f_0|_{\gamma_-} = 0$ and $\mathfrak{f}(0,x,v)|_{\gamma_-} = h|_{\gamma_-} = f_b$; see \cite{GKTT,CK}.

Similar to \eqref{ABCD_notation} in Proposition \ref{prop:weight_W1p}, we split the terms \eqref{hchara:nabla_0}-\eqref{hchara:nabla_phi_E_phi} into four parts, denoted by $\mathcal{A}, \mathcal{B}, \mathcal{C}, \mathcal{D}$, as follows:
\be \label{ABCD_notation_h}
\begin{split}
\mathcal{A}(t,x,v) & = \eqref{hchara:nabla_nu_1} + \eqref{hchara:nabla_nu_2} + \eqref{hchara:nabla_K}+ \eqref{hchara:K_nabla_phi_E} + \eqref{hchara:nabla_nu_K}  + \eqref{hchara:nabla_nu_gamma} + \eqref{hchara:gamma_nabla_phi_E} 
\\& \qquad + \eqref{hchara:nabla_E}  + \eqref{hchara:nabla_nu_phi}  + \eqref{hchara:nabla_V_phi} + \eqref{hchara:nabla_mu_phi} + \eqref{hchara:nabla_phi_E_phi},
\\ \mathcal{B}(t,x,v) & = \eqref{hchara:nabla_0} + \eqref{hchara:nabla_gamma}, 
\\ \mathcal{C}(t,x,v) & = \eqref{hchara:nabla_tb} + \eqref{hchara:bdr} + \eqref{hchara:nabla_int} + \eqref{hchara:nabla_tb_gamma} + \eqref{hchara:nabla_tb_phi}, 
\\ \mathcal{D}(t,x,v) & = \eqref{hchara:K_nabla}.
\end{split}
\ee

For $\mathcal{A}(t,x,v)$, we follow \eqref{A_bdd_lp} by replacing $h(X(0),V(0))$ with $\mathfrak{f}_0(X(0),V(0))$ and conclude that
\be \notag
\Vert \mathcal{A}(t)\Vert_{L^p_{x,v}} 
\lesssim o(1) (\Vert w_{\tilde{\theta}} \p_{x,v}\mathfrak{f}\Vert_{L^\infty_t L^p_{x,v}} + \Vert \alpha_f \nabla_x \mathfrak{f}\Vert_{L^\infty_{t,x,v}}) + \Vert w\mathfrak{f}\Vert_{L^\infty_{t,x,v}} + |wf_b|_{L^\infty_{\p\O,v}} + \Vert w \mathfrak{f}_0\Vert_{L^\infty_{t,x,v}}
\ee

For $\mathcal{B}(t,x,v)$, we follow \eqref{B_bdd_lp} by replacing $h(X(0),V(0))$ with $\mathfrak{f}_0(X(0),V(0))$ and $o(1)$ with $e^{-\frac{\nu_0 t}{4}}$, and conclude that
\be \notag
\Vert \mathcal{B}(t)\Vert_{L^p_{x,v}} 
\lesssim e^{-\frac{\nu_0t}{4}}\Vert w_{\tilde{\theta}} \p_{x,v}\mathfrak{f}_0\Vert_{L^p_{x,v}} + \Vert w\mathfrak{f}\Vert_{L^\infty_{t,x,v}} + o(1)\Vert w_{\tilde{\theta}} \p_{x,v}\mathfrak{f}\Vert_{L^\infty_t L^p_{x,v}}. 
\ee

For $\mathcal{C}(t,x,v)$, we follow \eqref{c_bdd_lp} and conclude that
\be \notag
\Vert \mathcal{C}(t)\Vert_{L^p_{x,v}}\lesssim \Vert w\mathfrak{f}\Vert_{L^\infty_{t,x,v}} + |wf_b|_{L^\infty_{\p\O,v}} + |w\p_{\mathbf{x}_p,v}f_b|_{L^\infty_{\p\O,v}}.
\ee

For $\mathcal{D}(t,x,v)$, similarly to the proof of Proposition \ref{prop:weight_W1p}, we decompose the term $\p_{x,v} [\mathfrak{f}(s,X(s),u)]$ in \eqref{hchara:K_nabla} into four parts as in \eqref{ABCD_notation_h}. Consequently, \eqref{hchara:K_nabla} can be expanded as
\[
\eqref{hchara:K_nabla}
=
\eqref{hchara:K_nabla}_{\mathcal A} + \eqref{hchara:K_nabla}_{\mathcal B} + \eqref{hchara:K_nabla}_{\mathcal C} + \eqref{hchara:K_nabla}_{\mathcal D}.
\]
The estimates for most of the resulting terms follow from \eqref{D_lp_bdd}. The only difference arises from the contribution \eqref{hchara:nabla_0} associated with the initial data, denote by $\eqref{hchara:nabla_K}_{\eqref{hchara:nabla_0}}$. Similar to \eqref{k_B_bdd}, we compute
\begin{align*}
\Vert \eqref{hchara:nabla_K}_{\mathcal{B}}\Vert_{L^p_{x,v}} \lesssim & \Big\Vert  \int^t_{\max\{t-\tb,0\}} \dd s e^{-\frac{\nu(v)(t-s)}{4}} \int_{\mathbb{R}^3} \mathbf{k}(v,u) \frac{w_{\tilde{\theta}}(v)}{w_{\tilde{\theta}}(u)} e^{-\frac{\nu(u)s}{4}} \\
& \qquad \times w_{\tilde{\theta}}(u) [|\p_x \mathfrak{f}_0(X(0;s,x,u),V(0;s,x,u))| + |\p_v \mathfrak{f}_0(X(0;s,x,u),V(0;s,x,u))|]          \Big\Vert_{L^p_{x,v}} \\
& \lesssim e^{-\nu_0t/4}  \Vert w_{\tilde{\theta}}(v)\p_{x,v}\mathfrak{f}_0\Vert_{L^p_{x,v}}.
\end{align*}
Therefore, we conclude that 
\be \notag
\Vert \mathcal{D}(t)\Vert_{L^p_{x,v}} 
\lesssim \Vert w\mathfrak{f}\Vert_{L^\infty_{t,x,v}} + |wf_b|_{L^\infty_{\p\O,v}} + e^{-\nu_0 t/4} \Vert w_{\tilde{\theta}} \p_{x,v}\mathfrak{f}_0\Vert_{L^p_{x,v}} + o(1) \big( \Vert w_{\tilde{\theta}} \p_{x,v}\mathfrak{f}\Vert_{L^\infty_t L^p_{x,v}} + \Vert \alpha_f \nabla_x \mathfrak{f}\Vert_{L^\infty_{t,x,v}} \big).
\ee
Combining the above estimates for $\mathcal{A}, \mathcal{B}, \mathcal{C}$, and $\mathcal{D}$, we conclude \eqref{W1p_short_time}.

\smallskip

Second, we estimate $\Vert w_{\tilde{\theta}} \alpha_f \p_{x,v}\mathfrak{f}(t)\Vert_{L^\infty_{x,v}}$ and $\Vert w_{\tilde{\theta}} \p_v \mathfrak{f}(t)\Vert_{L^\infty_{x,v}}$. 
Compared to Proposition \ref{prop:weight_C1} and Proposition \ref{prop:C1v}, the only difference comes from the contribution \eqref{hchara:nabla_0} associated with the initial data.

For the contribution of \eqref{hchara:nabla_0} to $\Vert w_{\tilde{\theta}} \alpha_f \p_{x,v}\mathfrak{f}(t)\Vert_{L^\infty_{x,v}}$, we apply \eqref{deri_XV_bdd} together with \eqref{phi_T0_small} to obtain that for any $0 \leq t \leq T_0$,
\begin{align*}
& \alpha_f(t,x,v) \Big| e^{-\int^t_{0}\tilde{\nu}(X(s),V(s))\dd s} \frac{w_{\tilde{\theta}}(v)}{w_{\tilde{\theta}}(V(0))}w_{\tilde{\theta}}(V(0))[ \p_{x} X(0) \p_{x}\mathfrak{f}_0(X(0),V(0)) + \p_{x} V(0)\cdot \nabla_{v}\mathfrak{f}_0(X(0),V(0)) ]\Big| \\
& \lesssim e^{C_0}\alpha_f(0,X(0),V(0)) e^{-\nu_0t/4} |w_{\tilde{\theta}}(V(0))\p_{x,v}\mathfrak{f}_0(X(0),V(0))|
\lesssim e^{-\nu_0t/4}\Vert \alpha_f(0,x,v)w_{\tilde{\theta}}(v)\p_{x,v}\mathfrak{f}_0(x,v)\Vert_{L^\infty_{x,v}}.
\end{align*}

For the contribution of \eqref{hchara:nabla_0} to $\Vert w_{\tilde{\theta}} \p_v \mathfrak{f}(t)\Vert_{L^\infty_{x,v}}$, we obtain that for any $0 \leq t \leq T_0$,
\begin{align*}
    & \Big| e^{-\int^t_{0}\tilde{\nu}(X(s),V(s))\dd s} \frac{w_{\tilde{\theta}}(v)}{w_{\tilde{\theta}}(V(0))}w_{\tilde{\theta}}(V(0))[ \p_{v} X(0) \p_{x}\mathfrak{f}_0(X(0),V(0)) + \p_{v} V(0)\cdot \nabla_{v}\mathfrak{f}_0(X(0),V(0)) ]\Big| \\
    &\lesssim |\alpha_f(t,x,v)\p_x \mathfrak{f}_0(X(0),V(0)) + \p_v \mathfrak{f}_0(X(0),V(0))| e^{-\nu_0t/4}  \\
    & \lesssim    e^{-\nu_0t/4}[\Vert \alpha_f(0,x,v)w_{\tilde{\theta}}(v)\p_{x}\mathfrak{f}_0(x,v)\Vert_{L^\infty_{x,v}} + \Vert w_{\tilde{\theta}} \p_v \mathfrak{f}_0(x,v)\Vert_{L^\infty_{x,v}}].
\end{align*}
Combining the above estimates and \eqref{W1p_short_time}, we derive that
\begin{align*}
& \Vert  w_{\tilde{\theta}}\p_{x,v} \mathfrak{f}(t)\Vert_{L^p_{x,v}} + \Vert w_{\tilde{\theta}} \alpha_f \p_{x,v}\mathfrak{f}(t)\Vert_{L^\infty_{x,v}} + \Vert w_{\tilde{\theta}} \nabla_v \mathfrak{f}(t)\Vert_{L^\infty_{x,v}} \\
& \lesssim \Vert w_{\tilde{\theta}} \alpha_f \p_{x,v}\mathfrak{f}_0\Vert_{L^\infty_{x,v}} + \Vert w_{\tilde{\theta}} \p_{x,v}\mathfrak{f}_0\Vert_{L^p_{x,v}} + \Vert w_{\tilde{\theta}} \nabla_v \mathfrak{f}_0\Vert_{L^\infty_{x,v}} +   \sup_t\Vert w\mathfrak{f}(t)\Vert_{L^\infty_{x,v}} + |wf_b|_{L^\infty_{\p\O,v}} + |w \p_{\mathbf{x}_p,v}f_b |_{L^\infty_{\p\O,v}}.
\end{align*}
Therefore, we conclude the claim \eqref{regularity_short_time}.
\end{proof}

\subsection{Local well-posedness} \label{sec:dynamical_local}

In this section, we establish the local well-posedness for the dynamical problem \eqref{F_dynamical} in Proposition~\ref{prop:local_wellposed}. 
Similar to Section \ref{sec:dynamical_regularity}, due to the lack of control on $\nabla_v^2 h$, we avoid the decomposition
$F = F_s + e^{-\phi_E/2}\sqrt{\mu} f$ and the corresponding equation \eqref{linear_f_dynamical} for $f$. However, unlike Section \ref{sec:dynamical_regularity}, we now rewrite $F(t,x,v)$ as follows to adapt to the local well-posedness argument.
\begin{equation} \label{g_f_relationship} 
F (t,x,v)= \sqrt{\mu} g
\ \text{ with } \
g = \frac{F_s}{\sqrt{\mu}} + e^{-\phi_E/2} f.
\end{equation} 
The equations for $g$ are given by
\be \label{eqn:g}
\begin{cases}
& \p_t g + v \cdot \nabla_x g - \nabla_x (\phi_g + \phi_E) \cdot \nabla_{v}  g + \frac{v\cdot \nabla_x(\phi_g+\phi_E)}{2}g + \nu(\sqrt{\mu}g)g = e^{-\phi_E/2}\Gamma_{\text{gain}}(g,g), \\
& g(t,x,v)|_{\gamma_-} = \frac{F_b}{\sqrt{\mu}}, \quad g(0,x,v) = g_0(x,v) = \frac{F_0(x,v)}{\sqrt{\mu}}, \\
& - \Delta_x \phi_g = \int_{\mathbb{R}^3}  g \sqrt{\mu} \dd v - e^{-\phi_E} \text{ in } \O, \quad \phi_g = 0 \text{ on } \p\O, \\
& - \p_n\phi_E > C_E >0 \text{ on } \p\O,   
\end{cases}
\ee
where $\nu(\sqrt{\mu}g) := \int_{\mathbb{R}^3} \int_{\mathbb{S}^2}|(v-u)\cdot \omega| \sqrt{\mu(u)}g(u)\dd \omega \dd u$ denotes the collision frequency corresponding to the distribution $\sqrt{\mu}g$.
By direct computation and the definition of the kinetic weight in \eqref{alpha_weight_dyna}, it follows that 
\be \notag
\phi_g = \phi_f, \qquad 
\alpha_g(t,x,v) = \alpha_f(t,x,v).  
\ee

\begin{proposition} \label{prop:local_wellposed}

Assume that all assumptions in Theorem \ref{thm:steady_wellpose} and Theorem \ref{thm:dynamical_stability} hold.
There exist constants $\delta_3 \ll 1$ and $C_5 > 0$ such that if the initial data $F_0 = F_s + e^{-\phi_E}\sqrt{\mu} f_0 = \sqrt{\mu}g_0 \geq 0$ satisfies the compatibility condition $g_0(x,v)|_{\gamma_-} = \frac{F_b}{\sqrt{\mu}}$ and
\be \label{extra_assumption}
\Vert wf_0\Vert_{L^\infty_{x,v}}   < \delta_3, \qquad 
\Vert wg_0\Vert_{L^\infty_{x,v}} + \Vert  w_{\tilde{\theta}}\alpha_g \p_{x,v}g_0\Vert_{L^\infty_{x,v}} + \Vert \nabla_v g_0\Vert_{L^3_{x,v}} < C_5, 
\ee
then there exists a unique non-negative solution $F = \sqrt{\mu} g (t) \geq 0$ to \eqref{F_dynamical} on $[0,T']$, where $T' = T'(\delta_3) \ll 1$.
Moreover, there exists a constant $C_6 > 1$ such that for any $t \in [0,T']$,
\be \label{g_local_est}
\Vert w g (t) \Vert_{L^\infty_{x,v}} + \Vert w_{\tilde{\theta}} \p_{x,v} g (t) \Vert_{L^p_{x,v}} +  \Vert w_{\tilde{\theta}}\alpha_g \p_{x,v} g (t) \Vert_{L^\infty_{x,v}} + \Vert  w_{\tilde{\theta}}\nabla_v g (t) \Vert_{L^\infty_{x,v}} < C_6 C_5. 
\ee
Furthermore, $\Vert wg(t)\Vert_{L^\infty_{x,v}}, \Vert  w_{\tilde{\theta}}\nabla_v g(t)\Vert_{L^\infty_{x,v}},\Vert  w_{\tilde{\theta}}\p_{x,v} g(t)\Vert_{L^p_{x,v}},  \Vert w_{\tilde{\theta}}\alpha_g \p_{x,v} g(t)\Vert_{L^\infty_{x,v}}$ are continuous on $[0,T']$.
\end{proposition}

\begin{proof}

The proof is analogous to Theorem 4 in \cite{cao2019regularity}. We sketch the proof and highlight the main differences.

We construct a solution to \eqref{eqn:g} via the following iterative sequence: for any $m \in \N$,
\be \label{gm_eqn}
\begin{cases}
& \Big(\p_t + v\cdot \nabla_x - \nabla_x(\phi^m + \phi_E)\cdot \nabla_v + \frac{v\cdot \nabla_x(\phi^m+\phi_E)}{2} + \nu(\sqrt{\mu}g^m)\Big)g^{m+1} = \Gamma_{\text{gain}}(g^m,g^m), \\
& g^{m+1}|_{\gamma_-} = \frac{F_b}{\sqrt{\mu}}, \quad g^{m+1}(0,x,v) = g_0(x,v), \\
&  -\Delta_x \phi^m = \int_{\mathbb{R}^3} g^m \sqrt{\mu} \dd v - e^{-\phi_E} \text{ in } \O, \quad \phi^m = 0 \text{ on } \p\O.
\end{cases} 
\ee
where the initial setting $g^0 = \frac{F_b}{\sqrt{\mu}}$.

Compared to \cite{cao2019regularity}, we note the following two main differences.
One is that the boundary condition $g^{m+1}|_{\gamma_-} = \frac{F_b}{\sqrt{\mu}}$ in \eqref{gm_eqn} is of inflow type, so characteristics interact with the boundary at most once, which is simpler than the diffuse reflection case in \cite{cao2019regularity}. 
Another is that the Poisson equation in \eqref{gm_eqn} is equipped with a Dirichlet boundary condition $\phi^m |_{\p\O} = 0$, rather than a Neumann boundary condition.
Thus, to apply the framework of \cite{cao2019regularity}, we must ensure the favorable sign condition: for any $m \in \N$,
\be \label{gm_sign_condition}
-\p_n (\phi^m + \phi_E) > \frac{C_E}{2}
\text{ on } \p\O.
\ee

First, we establish a uniform-in-$m$ $L^\infty_{x,v}$ estimate for $\{ g^m \}^{\infty}_{m=1}$, which is not affected by the sign condition \eqref{gm_sign_condition}.
Using a standard characteristic argument (simpler than in \cite{cao2019regularity} due to the inflow boundary), there exists $C > 1$ and $T' = T'(\Vert wg_0\Vert_{L^\infty_{x,v}})\ll 1$ such that
\be \notag
\sup_m \sup_{0\leq t\leq T'}\Vert e^{(\theta-t)|v|^2}g^m(t,x,v)\Vert_{L^\infty_{x,v}}  < C \Big[\Vert e^{\theta |v|^2}g_0\Vert_{L^\infty_{x,v}} + \Big| w\frac{F_b}{\sqrt{\mu}} \Big|_{L^\infty_{\p\O,v}} \Big].
\ee
By \eqref{g_f_relationship}, $T'$ depends only on $\Vert wf_0\Vert_{L^\infty_{x,v}} < \delta_3$, i.e., $T'  = T'(\Vert w g_0 \Vert_{L^\infty_{x,v}})  = T'(\Vert w f_0\Vert_{L^\infty_{x,v}}) = T'(\delta_3)$.
Furthermore, the loss of the time-decaying velocity weight $e^{(\theta -t)|v|^2}$ arises from the following estimate of the nonlinear term:
\begin{align*}
    &  e^{(\theta-t)|v|^2}\int_0^t \Gamma_{\text{gain}}(g,g) \dd s\lesssim o(1) \Big( \sup_{0\leq s\leq t}\Vert e^{(\theta-s)|v|^2}g(s)\Vert_{L^\infty_{x,v}}\Big)^2.
\end{align*}

Second, we justify the favorable sign condition \eqref{gm_sign_condition} by investigating the perturbation equation. Specifically, we introduce the iterative sequence
\be \notag
\mathfrak{f}^{m+1} := e^{\phi_E/2} g^{m+1} - e^{-\phi_E / 2} \sqrt{\mu},
\ee
satisfying that for any $m \in \N$,
\be \label{hm_eqn}
\begin{cases}
& \p_t \mathfrak{f}^{m+1} + v\cdot \nabla_x \mathfrak{f}^{m+1} - \nabla_x (\phi^m + \phi_E)\cdot \nabla_v \mathfrak{f}^{m+1} + \frac{v\cdot \nabla_x \phi^m}{2}\mathfrak{f}^{m+1} + \nu^m \mathfrak{f}^{m+1} \\
& = e^{-\phi_E}K\mathfrak{f}^m-e^{-\phi_E/2}v\cdot \nabla_x \phi^m\sqrt{\mu}   + e^{-\phi_E/2}\Gamma_{\text{gain}}(\mathfrak{f}^m,\mathfrak{f}^m), \\
& \mathfrak{f}^{m+1}|_{\gamma_-} = f_b, \quad \mathfrak{f}^{m+1}(0,x,v) = \frac{F_0 - e^{-\phi_E}\mu}{e^{-\phi_E/2 }\sqrt{\mu}}, \\
& -\Delta_x \phi^m = e^{-\phi_E/2}\int_{\mathbb{R}^3}\mathfrak{f}^m\sqrt{\mu} \dd v \text{ in } \O, \quad \phi^m = 0  \text{ on } \p\O,
\end{cases} 
\ee
where $\nu^m : = e^{-\phi_E}\nu(v) + e^{-\phi_E/2} \nu(\sqrt{\mu}\mathfrak{f}^m)$.
We now prove that
\be \label{wfm_estimate}
\sup_{t\leq T'}\Vert w\mathfrak{f}^m(t)\Vert_{L^\infty_{x,v}} \ll 1.
\ee
Using the initial condition of $\mathfrak{f}^{m+1}$, together with the assumption $\Vert wf_0\Vert_{L^\infty_{x,v}} < \delta_3 \ll 1$, we obtain
\begin{align*}
\mathfrak{f}^{m+1}(0,x,v) = h + f_0,
\qquad
\Vert w\mathfrak{f}^{m+1}(0)\Vert_{L^\infty_{x,v}} \ll 1.
\end{align*}
We now apply the same characteristic argument as in \cite{cao2019regularity} to \eqref{hm_eqn}. The source terms on the right-hand side of \eqref{hm_eqn} are controlled by Lemma \ref{lemma:gamma} and Lemma \ref{lemma:phi_x_infinity}:
\be \label{RHS_wh}
\begin{split}
\int^t_0 e^{-\frac{\nu(v)}{2}(t-s)}[K\mathfrak{f}^m +\Gamma_{\text{gain}}(\mathfrak{f}^m,\mathfrak{f}^m)] \dd s & \lesssim T' \sup_{t\leq T} \Vert w \mathfrak{f}^m(t)\Vert_{L^\infty_{x,v}}^2 +  T'\sup_{t\leq T'}\Vert w\mathfrak{f}^m(t)\Vert_{L^\infty_{x,v}}, 
\\ \int_0^t e^{-\frac{\nu(v)}{2}(t-s)}|v\cdot \nabla_x \phi^m\sqrt{\mu} | \dd s 
& \lesssim T'\sup_{t\leq T'}\Vert \nabla_x \phi^m(t)\Vert_{L^\infty_{x}} \lesssim T'\sup_{t\leq T'}\Vert w\mathfrak{f}^m(t)\Vert_{L^\infty_{x,v}}.
\end{split}
\ee
Together with the smallness assumption on the boundary term $| wf_b|_{L^\infty_{\p\O,v}}\ll 1$, we conclude the following uniform-in-$m$ estimate:
\be \notag
\sup_m \sup_{0\leq t\leq T'}\Vert w \mathfrak{f}^m(t)\Vert_{L^\infty_{x,v}} \leq C\Vert w\mathfrak{f}(0)\Vert_{L^\infty_{x,v}} + C | wf_b|_{L^\infty_{\p\O,v}} \leq C [\delta_3 + | wf_b|_{L^\infty_{\p\O,v}}].
\ee
Since $C [\delta_3 + | wf_b|_{L^\infty_{\p\O,v}}] \ll 1$, it follows that $\nu^m \geq \frac{\nu(v)}{2}$ and we conclude \eqref{wfm_estimate}.
Moreover, applying Theorem \ref{thm:steady_wellpose}, we derive that
\be \notag
\sup_{t\leq T'}\Vert \nabla_x \phi^m(t)\Vert_{L^\infty_{x}} \lesssim  \sup_{t\leq T'}\Vert w\mathfrak{f}^m(t)\Vert_{L^\infty_{x,v}} \ll 1.
\ee
Consequently, the favorable sign condition \eqref{gm_sign_condition} holds. 

\smallskip

Third, we prove the existence and uniqueness of the solution to \eqref{eqn:g}.
From the sign condition \eqref{gm_sign_condition} and the initial assumption \eqref{extra_assumption}, we apply the same argument as in \cite{cao2019regularity} to obtain the following uniform-in-$m$ regularity estimate: for some $c>0$ and $C_6>1$,
\be \label{eq1:g_local_est}
\sup_{m}\sup_{0\leq t\leq T'}\Vert w_{\tilde{\theta}}e^{-c\nu t}\alpha_g^{m-1} \p_{x,v}g^m\Vert_{L^\infty_{x,v}} + \sup_m \sup_{0\leq t\leq T'}\Vert e^{-c\nu t}\nabla_v g^m\Vert_{L^3_x L^{1+\delta}_v} < C_6 C_5,
\ee
where the dynamical kinetic weight $\alpha_g^{m-1}(t,x,v)$ is defined in \eqref{alpha_weight_dyna}, with the self-consistent electric field 
\[
E_g^{m-1}(t,x) := -\big[\nabla_x \phi_g^{m-1}(t,x) + \nabla_x \phi_E(x)\big].
\]
Using \eqref{eq1:g_local_est} and following the local well-posedness argument in \cite{cao2019regularity}, we conclude that
\[
g^m \to g 
\ \text{ in }
L^\infty((0,T');L^{1+}(\O\times \mathbb{R}^3))
\ \text{ as }
m \to \infty,
\]
where $g$ satisfies \eqref{eqn:g}.
Moreover, since $g^m = e^{-\phi_E}\sqrt{\mu} + e^{-\phi_E/2} \mathfrak{f}^m$, we derive that 
\[
\mathfrak{f}^m \to \mathfrak{f} 
\ \text{ in }
L^\infty((0,T');L^{1+}(\O\times \mathbb{R}^3))
\ \text{ as }
m \to \infty,
\]
where $\mathfrak{f}$ satisfies \eqref{eqn:mkf}. 
The uniqueness of solutions to \eqref{eqn:mkf} is guaranteed in the space $L^{1+}_{x,v}$ under the additional assumption $\Vert \nabla_v g_0\Vert_{L^3_{x,v}} < \infty$; see \cite{cao2019regularity,CKL}.
From \eqref{wfm_estimate} and the relation $g = e^{-\phi_E}\sqrt{\mu} + e^{-\phi_E/2}\mathfrak{f}$, we obtain the estimate for $\|wg\|_{L^\infty_{x,v}}$ as in \eqref{g_local_est}, without the time-decaying velocity weight $e^{(\theta-t)|v|^2}$.

\smallskip

Fourth, we establish the regularity estimate \eqref{g_local_est}.
We work on the perturbation equation \eqref{hm_eqn}. Since $T' \ll 1$, we can follow a simpler argument than in Proposition \ref{prop:W1p_dynamical}, where the term $\mathcal{D}(t,x,v)$ in \eqref{hchara:K_nabla} can be directly estimated using the smallness of the time interval $[0, T']$. The remaining terms can be treated as in Proposition \ref{prop:W1p_dynamical}, so that we obtain the following estimates:
\begin{align*}
    & \| w_{\tilde{\theta}}(v)\eqref{hchara:K_nabla}\|_{L^p_{x,v}}
    \lesssim T' \sup_{0\leq s\leq T'}\| w_{\tilde{\theta}}\p_{x,v} \mathfrak{f}^m(s)\|_{L^p_{x,v}}, \\
    & \| w_{\tilde{\theta}}(v)\alpha_f^{m-1}(t,x,v)\eqref{hchara:K_nabla}\|_{L^p_{x,v}}
    \lesssim T' \sup_{0\leq s\leq T'}\| w_{\tilde{\theta}}\alpha_f^{m-1} \p_{x,v} \mathfrak{f}^m(s)\|_{L^\infty_{x,v}},
\end{align*}
where $\alpha_f^{m-1} := \alpha_g^{m-1}$.
Using the initial condition \eqref{initial_condition}, we obtain a uniform-in-$m$ estimate for \eqref{hm_eqn}:
\begin{align*}
\sup_{m}\sup_{0\leq t\leq T'}\Vert w_{\tilde{\theta}} \p_{x,v}\mathfrak{f}^m\Vert_{L^p_{x,v}}+     \sup_{m}\sup_{0\leq t\leq T'}\Vert w_{\tilde{\theta}}\alpha_f^{m-1} \p_{x,v}\mathfrak{f}^m\Vert_{L^\infty_{x,v}} + \sup_m \sup_{0\leq t\leq T'}\Vert w_{\tilde{\theta}} \nabla_v \mathfrak{f}^m\Vert_{L^\infty_{x,v}} < C_5 \delta_2.
\end{align*}
Since $\mathfrak{f}^m \to \mathfrak{f}$ strongly, we derive
\be \notag
\sup_{m}\sup_{0\leq t\leq T'}\Vert w_{\tilde{\theta}}\alpha_f \p_{x,v}\mathfrak{f}\Vert_{L^p_{x,v}}+     \sup_{m}\sup_{0\leq t\leq T'}\Vert w_{\tilde{\theta}}\alpha_f \p_{x,v}\mathfrak{f}\Vert_{L^\infty_{x,v}} + \sup_m \sup_{0\leq t\leq T'}\Vert w_{\tilde{\theta}}\nabla_v \mathfrak{f}\Vert_{L^\infty_{x,v}} < C_6\delta_2.
\ee
This further implies that
\be \notag
\sup_{m}\sup_{0\leq t\leq T'}\Vert w_{\tilde{\theta}} \p_{x,v}g\Vert_{L^p_{x,v}}+ \sup_{m}\sup_{0\leq t\leq T'}\Vert w_{\tilde{\theta}}\alpha_g \p_{x,v}g\Vert_{L^\infty_{x,v}} + \sup_m \sup_{0\leq t\leq T'}\Vert w_{\tilde{\theta}}\nabla_v g\Vert_{L^\infty_{x,v}} < C_6 C_5.
\ee

\smallskip

Finally, we prove the continuity of $\Vert wg(t)\Vert_{L^\infty_{x,v}}, \Vert  w_{\tilde{\theta}}\nabla_v g(t)\Vert_{L^\infty_{x,v}},  \Vert w_{\tilde{\theta}}\alpha_g \p_{x,v} g(t)\Vert_{L^\infty_{x,v}}$ in time $t$, and the positivity of the solution.
To prove the continuity, it suffices to prove that $w\mathfrak{f}$, $\alpha_f \p_{x,v}\mathfrak{f}$, and $\nabla_v \mathfrak{f}$ are continuous in time $t$. 
Indeed, these three terms are represented via the Duhamel formula along characteristics, as in \eqref{RHS_wh} and Proposition \ref{prop:W1p_dynamical}. The continuity then follows by comparing the representations at times $t$ and $t - \delta$ and letting $\delta \to 0$, which yields continuity in time. For a similar argument, we refer to Step 9 in Theorem 2 of \cite{CKL}.
The unique solution $F=\sqrt{\mu}g$ is non-negative since \eqref{gm_eqn} preserves positivity.
\end{proof}

\begin{remark}

The extra assumption $\Vert \nabla_v g_0\Vert_{L^3_{x,v}}$ in \eqref{extra_assumption} was introduced in \cite{cao2019regularity} to establish local well-posedness. 
This condition is automatically satisfied under the assumption \eqref{initial_condition} in Theorem \ref{thm:dynamical_stability}. In Proposition \ref{prop:local_wellposed}, we impose the additional assumption \eqref{extra_assumption} in order to apply the local well-posedness framework of \cite{cao2019regularity}.
\end{remark}

\subsection{Proof of Theorem \ref{thm:dynamical_stability}.} 
\label{sec:dynamical_proof}

Since the assumptions in Proposition \ref{prop:local_wellposed} are satisfied, there exists $T'\ll 1$ such that existence and uniqueness hold on $t\in [0, T']$. Since $\Vert wg(t)\Vert_{L^\infty_{x,v}}$ is continuous in time by Proposition \ref{prop:local_wellposed}, we obtain the continuity of $\Vert wf(t)\Vert_{L^\infty_{x,v}}$ in $t$. Hence, the assumptions in Proposition \ref{prop:linfty_dynamical} are satisfied, and \eqref{linfty_bdd_dynamical} yields
\be \notag
\Vert wf(T')\Vert_{L^\infty_{x,v}} < C_2 \delta_2 e^{- \lambda T'}.
\ee

We choose $\delta_3$ in Proposition \ref{prop:local_wellposed} such that $C_2\delta_2 < \delta_3$. Since the lifespan $T'$ of the local solution in Proposition \ref{prop:local_wellposed} depends only on $\delta_3$, and the regularity estimate at $T'$ satisfies
\be \notag
\Vert w_{\tilde{\theta}} \alpha_g \p_{x,v}g(T')\Vert_{L^\infty_{x,v}} + \Vert \nabla_v g(T')\Vert_{L^3_{x,v}} \lesssim \Vert w_{\tilde{\theta}} \alpha_g \p_{x,v}g(T')\Vert_{L^\infty_{x,v}} + \Vert w_{\tilde{\theta}}\nabla_v g(T')\Vert_{L^\infty_{x,v}} < C_6 C_5 .
\ee
Then we apply Proposition \ref{prop:local_wellposed} with initial condition at $t=T'$ to obtain existence and uniqueness of the solution on $[T',2T']$. From the well-posedness and continuity of $\Vert wf(t)\Vert_{L^\infty_{x,v}}$ on $[0,2T']$, Proposition \ref{prop:linfty_dynamical} implies that
\begin{align*}
    & \Vert wf(2T')\Vert_{L^\infty_{x,v}} < C_2 \delta_2 e^{-2\lambda T'}.
\end{align*}
Using \eqref{g_local_est}, we control the regularity at $t=2T'$ by
\begin{align*}
& \Vert w_{\tilde{\theta}}\p_{x,v} g(2T')\Vert_{L^p_{x,v}} + \Vert w_{\tilde{\theta}}\alpha_g \p_{x,v}g(2T')\Vert_{L^\infty_{x,v}} + \Vert w_{\tilde{\theta}}\nabla_v g(2T')\Vert_{L^\infty_{x,v}} \\
& < C_6 [\Vert w_{\tilde{\theta}}\p_{x,v} g(T')\Vert_{L^p_{x,v}} + \Vert w_{\tilde{\theta}}\alpha_g \p_{x,v}g(T')\Vert_{L^\infty_{x,v}} + \Vert w_{\tilde{\theta}}\nabla_v g(T')\Vert_{L^\infty_{x,v}}] < C_7|C_6|^2 C_5,
\end{align*}
where we use $\Vert \nabla_v g\Vert_{L^3_{x,v}} < C_7 \Vert w_{\tilde{\theta}}\nabla_v g\Vert_{L^\infty_{x,v}}$ from some constant $C_7>1$.
Inductively, for every $n \in \N$, existence and uniqueness hold on $[(n-1)T',nT']$. Moreover,
\begin{align}
& \Vert wf(nT')\Vert_{L^\infty_{x,v}} < C_2 \delta_2 e^{-\lambda nT'}, 
\notag \\
&  \Vert w_{\tilde{\theta}}\p_{x,v} g(nT')\Vert_{L^p_{x,v}} + \Vert w_{\tilde{\theta}}\alpha_g \p_{x,v}g(nT')\Vert_{L^\infty_{x,v}} + \Vert w_{\tilde{\theta}}\nabla_v g(nT')\Vert_{L^\infty_{x,v}} < |C_7|^{n-1}|C_6|^n C_5. 
\label{regularity_bdd}
\end{align}
This establishes the global well-posedness and the $L^\infty_{x,v}$ estimate \eqref{global_Linfty}.

Note that the regularity estimate \eqref{regularity_bdd} grows exponentially in time. To justify the uniform-in-time bound \eqref{global_regularity}, we turn to the a priori estimate in Proposition \ref{prop:W1p_dynamical}. The continuity of $g$ established in Proposition \ref{prop:local_wellposed} implies the continuity of $\Vert  w_{\tilde{\theta}}\nabla_v \mathfrak{f}(t)\Vert_{L^\infty_{x,v}},\Vert  w_{\tilde{\theta}}\p_{x,v} \mathfrak{f}(t)\Vert_{L^p_{x,v}},  \Vert w_{\tilde{\theta}}\alpha_f \p_{x,v} \mathfrak{f}(t)\Vert_{L^\infty_{x,v}}$.
Applying the continuity argument on $\Vert \nabla_x^2 \phi_\mathfrak{f}\Vert_{L^\infty_{t,x}}$ together with Lemma \ref{lemma:phi_C2}, we obtain \eqref{apriori_C2_dynamical}.
Hence, all assumptions in Proposition \ref{prop:W1p_dynamical} are satisfied. Consequently, the uniform-in-time estimates \eqref{h_uniform_regularity} and \eqref{global_regularity} follow from Proposition \ref{prop:W1p_dynamical}.

Finally, positivity holds for the local solution on $[0, T']$ by Proposition \ref{prop:local_wellposed}. By induction, for every $n \in \N$, positivity also holds on the interval $[(n-1)T', n T']$. Therefore, the global solution satisfies $F(t,x,v) \geq 0$.
This, together with the asymptotic stability of the dynamical solution $F(t,x,v)$, implies the positivity of the stationary solution $F_s (x,v) \geq 0$ in Theorem \ref{thm:steady_wellpose}.
This completes the proof of Theorem \ref{thm:dynamical_stability}.
\qed

\section*{Acknowledgments}

The authors thank Professor Chanwoo Kim for suggesting this problem and for many helpful discussions.

\bibliographystyle{siam}

\end{document}